\documentclass[11pt]{preprint}
\usepackage[utf8x]{inputenc}
\usepackage[T1]{fontenc}
\usepackage{amssymb}
\usepackage{amsmath}

\usepackage[greek, greek.ancient, english]{babel}

\usepackage{breakurl}
\usepackage{margincomment}
\usepackage{mhenvs}
\usepackage{mhequ}
\usepackage{mhsymb}
\usepackage{booktabs}
\usepackage{tikz}
\usepackage{mathrsfs}
\usepackage{times}
\usepackage{microtype}
\usepackage{slashed}
\usepackage{mathtools}
\usepackage{etoolbox}
\usepackage{centernot}
\usepackage{leftidx}
\usepackage{accents}
\usepackage{arydshln}
\usepackage{footnote}
\makesavenoteenv{tabular}
\usepackage{enumitem}
\usepackage{stackrel}
\usepackage{halloweenmath}
\usepackage{longtable}%\usepackage{scalerel}
\usepackage{array}
\usepackage{cprotect}
\usepackage{xstring}
\usepackage{ulem}
\usepackage{stmaryrd}
\usepackage{braket}
\usepackage{orcidlink}
\usepackage{wasysym}

\usepackage{trees}
\usepackage[shortcuts]{extdash}

\usepackage{hyperref}

\usetikzlibrary{external}
\tikzsetexternalprefix{tikz/}

\DeclareSymbolFont{timesoperators}{T1}{ptm}{m}{n}
\SetSymbolFont{timesoperators}{bold}{T1}{ptm}{b}{n}

\makeatletter
\renewcommand{\operator@font}{\mathgroup\symtimesoperators}
\makeatother

\colorlet{symbolsgrey}{blue!30!black!50}
\colorlet{testcolor}{green!60!black}
\definecolor{purple}{rgb}{0.55,0.05,0.8}
\definecolor{symbols}{rgb}{0.55,0.05,0.8}

\usetikzlibrary{shapes.misc}
\usetikzlibrary{shapes.symbols}
\usetikzlibrary{shapes.geometric}
\usetikzlibrary{snakes}
\usetikzlibrary{decorations}
\usetikzlibrary{decorations.markings}
\usetikzlibrary{quotes}
\usetikzlibrary{backgrounds}

\usetikzlibrary{calc}

\newtheorem{metatheorem}[lemma]{Metatheorem}
\newtheorem{assumption}[lemma]{Assumption}
\newtheorem{example}[lemma]{Example}

\let\oldskull\skull
\def\skull{\mathord{\oldskull}}

\newcommand{\hooklongrightarrow}{\lhook\joinrel\longrightarrow}

\def\wotimes{\mathbin{\widehat{{\otimes}}}}

\def\snabla{\slashed{\nabla}}

\DeclareMathAlphabet{\mathbbm}{U}{bbm}{m}{n}

\DeclareFontFamily{U}{BOONDOX-calo}{\skewchar\font=45 }
\DeclareFontShape{U}{BOONDOX-calo}{m}{n}{
  <-> s*[1.05] BOONDOX-r-calo}{}
\DeclareFontShape{U}{BOONDOX-calo}{b}{n}{
  <-> s*[1.05] BOONDOX-b-calo}{}
\DeclareMathAlphabet{\mcb}{U}{BOONDOX-calo}{m}{n}
\SetMathAlphabet{\mcb}{bold}{U}{BOONDOX-calo}{b}{n}
\setlist{noitemsep,topsep=4pt}

\makeatletter % Stolen from the internet to make a fat \cdot which isn't as fat as a \bullet
\newcommand*{\bigcdot}{}% Check if undefined
\DeclareRobustCommand*{\bigcdot}{%
  \mathbin{\mathpalette\bigcdot@{}}%
}
\newcommand*{\bigcdot@scalefactor}{.5}
\newcommand*{\bigcdot@widthfactor}{1.15}
\newcommand*{\bigcdot@}[2]{%
  \sbox0{$#1\vcenter{}$}% math axis
  \sbox2{$#1\cdot\m@th$}%
  \hbox to \bigcdot@widthfactor\wd2{%
    \hfil
    \raise\ht0\hbox{%
      \scalebox{\bigcdot@scalefactor}{%
        \lower\ht0\hbox{$#1\bullet\m@th$}%
      }%
    }%
    \hfil
  }%
}
\makeatother

\def\symbol#1{\textcolor{symbolsgrey}{#1}}
\def\1{\mathbf{\symbol{1}}}
\def\X{\symbol{X}}

\newcommand{\cev}[1]{\reflectbox{\ensuremath{\vec{\reflectbox{\ensuremath{#1}}}}}}

\def\b0{\mathbf{0}}
\def\bone{\mathbf{1}}

\def\bR{\boldsymbol{R}}
\def\bPsi{\boldsymbol{\Psi}}

\def\PPi{\boldsymbol{\Pi}}

\def\can{\textnormal{\scriptsize can}}

\def\BPHZ{\textnormal{\tiny \textsc{bphz}}}
\def\SDE{\textnormal{\tiny \textsc{sde}}}

\def\spacetime{\underline{\Lambda}}
\def\Gr{\mathrm{Gr}}

\newcommand\bilin[1]{B \big(#1 \big)}

\def\balpha{\boldsymbol{\alpha}}
\def\balphas{\boldsymbol{\alpha}^\dagger}
\def\bomega{\boldsymbol{\omega}}
\renewcommand{\digamma}{\text{\foreignlanguage{greek}{\ddigamma}}}

\let\f\frac

\colorlet{darkblue}{blue!90!black}
\colorlet{darkgreen}{green!50!black}
\let\comm\margincomment
\def\martin#1{}
\def\ajay#1{}
\def\martinp#1{}

\def\s{\mathfrak{s}}
\def\KK{\mathfrak{K}}

\def\reg{\mathrm{reg}}

\def\K{\mathfrak{K}}

\def\${|\!|\!|}
\def\Wick#1{\mathord{\kern0.1em{:}#1{:}\kern0.1em}}

\def\d{\mbox{d}}
\def\-{\mbox{-}}

\def\bXi{\boldsymbol{\Xi}}
\def\bxi{\boldsymbol{\xi}}

\def\reminit{ \CC^{\textnormal{\tiny{init}}}}

\def\CA{{\cal A}}
\def\CB{{\cal B}}
\def\CC{{\cal C}}
\def\CD{{\cal D}}

\def\CF{{\cal F}}
\def\CG{{\cal G}}
\def\CH{{\cal H}}
\def\CI{{\cal I}}
\def\CJ{{\cal J}}
\def\CK{{\cal K}}
\def\CL{{\cal L}}
\def\CM{{\cal M}}
\def\CN{{\cal N}}

\def\CP{{\cal P}}
\def\CQ{{\cal Q}}
\def\CR{{\cal R}}
\def\CS{{\cal S}}
\def\CT{{\cal T}}

\def\CW{{\cal W}}

\def\mfA{{\mathfrak A}}
\def\mfB{{\mathfrak B}}

\def\mfH{{\mathfrak H}}
\def\mfI{{\mathfrak I}}

\def\mfK{{\mathfrak K}}
\def\mfL{{\mathfrak L}}

\def\mfO{{\mathfrak O}}
\def\mfP{{\mathfrak P}}
\def\mfQ{{\mathfrak Q}}
\def\mfR{{\mathfrak R}}
\def\mfS{{\mathfrak S}}
\def\mfT{{\mathfrak T}}

\def\mfY{{\mathfrak Y}}
\def\mfZ{{\mathfrak Z}}

\def\mfa{{\mathfrak a}}
\def\mfb{{\mathfrak b}}

\def\mfd{{\mathfrak d}}
\def\mfe{{\mathfrak e}}
\def\mff{{\mathfrak f}}

\def\mfh{{\mathfrak h}}

\def\mfk{{\mathfrak k}}
\def\mfl{{\mathfrak l}}
\def\mfm{{\mathfrak m}}
\def\mfn{{\mathfrak n}}
\def\mfo{{\mathfrak o}}
\def\mfp{{\mathfrak p}}
\def\mfq{{\mathfrak q}}
\def\mfr{{\mathfrak r}}
\def\mfs{{\mathfrak s}}
\def\mft{{\mathfrak t}}

\def\mfy{{\mathfrak y}}

\def\cA{{\mathscr A}}
\def\cB{{\mathscr B}}
\def\cC{{\mathscr C}}
\def\cD{{\mathscr D}}

\def\cG{{\mathscr G}}

\def\cJ{{\mathscr J}}

\def\cL{{\mathscr L}}

\def\cR{{\mathscr R}}
\def\cS{{\mathscr S}}
\def\cT{{\mathscr T}}
\def\cU{{\mathscr U}}

\DeclareMathOperator{\Ran}{Ran}
\DeclareMathOperator{\sgn}{sgn}

\DeclareMathOperator{\Cl}{\textit{C}\ell}

\DeclareMathOperator{\GL}{GL}

\DeclareMathOperator{\spn}{spn}

\DeclareMathOperator{\cross}{cr}
\DeclareMathOperator{\sep}{sp}
\DeclareMathOperator{\crb}{\overline{cr}}

\DeclareMathOperator*{\bigwotimes}{\widehat{\bigotimes}}
\DeclareMathOperator*{\bigwotimesp}{{}_{\phantom{\pi}}\widehat{\bigotimes}_\pi}
\DeclareMathOperator*{\bigwotimesa}{{}_{\phantom{\alpha}}\widehat{\bigotimes}_\alpha}

\def\eqvdef{\ratio\Longleftrightarrow}

\numberwithin{equation}{section}

\def\dash{\leavevmode\unskip\kern0.18em--\penalty\exhyphenpenalty\kern0.18em}
\def\slash{\leavevmode\unskip\kern0.15em/\penalty\exhyphenpenalty\kern0.15em}

\def\lor{\,\mathrm{or}\,}
\def\land{\,\mathrm{and}\,}

\let\f\frac

\def\balpha{\boldsymbol{\alpha}}
\def\balphas{\boldsymbol{\alpha}^\dagger}
\def\bomega{\boldsymbol{\omega}}
\def\bpi{\boldsymbol{\pi}}
\def\bsigma{\boldsymbol{\sigma}}
\def\bpsi{\boldsymbol{\psi}}

\DeclareFontFamily{U}{matha}{\hyphenchar\font45}
\DeclareFontShape{U}{matha}{m}{n}{
<-6> matha5 <6-7> matha6 <7-8> matha7
<8-9> matha8 <9-10> matha9
<10-12> matha10 <12-> matha12
}{}
\DeclareSymbolFont{matha}{U}{matha}{m}{n}

\DeclareFontFamily{U}{mathx}{\hyphenchar\font45}
\DeclareFontShape{U}{mathx}{m}{n}{
<-6> mathx5 <6-7> mathx6 <7-8> mathx7
<8-9> mathx8 <9-10> mathx9
<10-12> mathx10 <12-> mathx12
}{}
\DeclareSymbolFont{mathx}{U}{mathx}{m}{n}

\DeclareMathDelimiter{\vvvert} {0}{matha}{"7E}{mathx}{"17}%

\makeatletter
\newcommand{\oset}[3][0ex]{%
  \mathrel{\mathop{#3}\limits^{
    \vbox to#1{\kern-2\ex@
    \hbox{$\scriptstyle#2$}\vss}}}}
\makeatother

\makeatletter

\DeclareRobustCommand{\TitleEquation}[2]{\texorpdfstring{\StrLeft{\f@series}{1}[\@firstchar]$\if%
b\@firstchar\boldsymbol{#1}\else#1\fi$}{#2}}

\makeatother

\makeatletter
\DeclareRobustCommand{\cev}[1]{%
  \mathpalette\do@cev{#1}%
}
\newcommand{\do@cev}[2]{%
  \fix@cev{#1}{+}%
  \reflectbox{$\m@th#1\vec{\reflectbox{$\fix@cev{#1}{-}\m@th#1#2\fix@cev{#1}{+}$}}$}%
  \fix@cev{#1}{-}%
}
\newcommand{\fix@cev}[2]{%
  \ifx#1\displaystyle
    \mkern#23mu
  \else
    \ifx#1\textstyle
      \mkern#23mu
    \else
      \ifx#1\scriptstyle
        \mkern#22mu
      \else
        \mkern#22mu
      \fi
    \fi
  \fi
}

\makeatother

\begin{document}

\title{Noncommutative Regularity Structures}

\author{Ajay Chandra\inst1\orcidlink{0000-0003-3690-1890}, Martin Hairer\inst{2,3}\orcidlink{0000-0002-2141-6561}, and Martin Peev\inst2\orcidlink{0000-0001-6191-8067}}

\institute{Purdue University, West Lafayette, USA \and Imperial College London, UK
\and EPFL, Lausanne, Switzerland\\[.5em]
\email{ajay.chandra@gmail.com, \\ \{m.hairer,m.peev21\}@imperial.ac.uk}}

\maketitle

\begin{abstract}
We extend the theory of regularity structures \cite{Hai14} to allow processes belonging to locally $m$-convex topological algebras.
This extension includes processes in the locally $C^{*}$-algebras of \cite{CHP23} used to localise singular stochastic partial differential equations involving fermions, as well as processes in Banach algebras such as infinite-dimensional semicircular \slash circular Brownian motion,
and more generally the $q$-Gaussians of \cite{BS91, BKS97,Boz99}.  

A new challenge we encounter in the $q$-Gaussian setting with $q \in (-1,1)$ are noncommutative renormalisation estimates where we must estimate operators in homogeneous $q$-Gaussian chaoses with arbitrary operator insertions.
We introduce a new Banach algebra norm on $q$-Gaussian operators that allows us to control such insertions; we believe this construction could be of independent interest. 
\end{abstract}

\setcounter{tocdepth}{2}
\tableofcontents

\section{Introduction}\label{sec:intro}
The last decade has seen rapid development in the theory of
nonlinear singular stochastic partial differential equations (SPDEs).
These equations are paradigmatic models for evolving random systems with infinitely many degrees of freedom,
very important examples being the large-scale behaviour of
models from statistical physics \cite{MR1462228,MourratWeber}
and Langevin stochastic quantisations \cite{ParisiWu} of
Euclidean \textit{bosonic} quantum field theories (QFTs) \cite{Gub21,HS21,CCHS2d}.

A fundamental obstacle to obtaining local well-posedness for nonlinear singular SPDE is that the roughness of the random forcing in these equations makes the solution so irregular that the nonlinear terms in the equation no longer have any canonical meaning.
A major breakthrough in overcoming this obstacle was the theory of regularity structures \cite{Hai14}, which provides a systematic framework for the analysis of parabolic singular SPDEs.

The theory of regularity structures belongs to a class of ``pathwise'' approaches to the analysis of rough stochastic ordinary and partial differential equations - the earliest example being Lyons's theory of rough paths \cite{Lyons}.
These pathwise approaches split the problem into two steps:
(i) a \textit{probabilistic step} where one constructs (possibly with renormalisation) certain multilinear functionals of the driving noise
and (ii) an \textit{analytic step} where one fixes a realisation of the noise (and corresponding multilinear functionals of the noise) and uses this datum to rewrite the equation in a form where it can be closed using purely deterministic arguments.
When carrying out step (i) within regularity structures, there are also complicated nonlinear constraints imposed
on these multilinear functionals (an instance of which, we call a ``model'') and how they can be renormalised \dash a robust description of this
spurred the development of an algebraic framework, in addition to the probabilistic and analytic steps described above.

The motivation of the present work is to apply this point of view beyond classical (commutative) probability,
to the setting of noncommutative probability theory.
Rough ordinary and partial differential equations in noncommutative probability theory
appear naturally in large $N$ limits of matrix models \cite{GS09} and the study of free entropy \cite{Dabrowski,BS01}.
Any probabilistic approach to constructing QFTs with fermions \cite{OS73,GJ71} also requires noncommutative probability theory.
Free probability has also been proposed as a mathematical language for studying ``master fields'' arising in matrix QFTs and gauge theories \cite{DG95}.

One contribution of the present article is to generalise the analytic theory of regularity structures to locally $m$-convex algebras, see Definition~\ref{def:mconvex}.
Applying this theory to a given noncommutative dynamic also requires one to topologise appropriate operator algebras with a topology that is both $m$-convex and behaves well with respect to renormalisation.
One of the main contributions of \cite{CHP23} was the development of such topologies in the case of fermions.
In this paper, we work within the wider setting of $q$-Gaussian processes, which encompass fermions, bosons, and free probability as special cases (see the next section).
Another contribution of the present article is the development of some special operator topologies to accommodate noncommutative renormalisation in this setting.

As mentioned earlier, there is an algebraic step and a probabilistic step in regularity structures, and we also make contributions in generalising these to the noncommutative setting. 
For the algebraic step, we develop a flexible framework for ``models'' that includes BPHZ renormalisation. 
However, we do not carry out the programme of studying how the deformation of products in a model in turn deforms the underlying equation \dash this is left to future work. 
Finally, we are also able to completely\footnote{Since the probabilistic argument we give in this article depends on the Gaussian structure, it does not allow for central limit theorems for $q$-Gaussian dynamics \dash we leave this to future work.} close the probabilistic step in the $q$-Gaussian setting by showing how any $q$-Gaussian moment estimate we require can be recovered from moment estimates for a carefully engineered commutative ($q=1$) model.

\subsection{Noncommutative Probability Theory}
In this section, we list the specific noncommutative probability spaces that we focus on and the different challenges they pose. A more detailed presentation can be found in Section~\ref{subsec:particles}.

By noncommutative probability space we mean a pair $(\CA, \omega)$ where $\CA$ is a unital $\ast$-algebra with involution $a \mapsto a^{\ast}$
and $\omega$ is a linear functional $\omega \colon \CA \rightarrow \C$ satisfying $\omega(\bone) = 1$ and $\omega(a^{\ast}a) \geqslant 0$ for all $a \in \CA$.
Here $\bone$ denotes the unit of $\CA$ and a map $\omega$ as described above is called a ``state''.
Since we adopt an analytic viewpoint, we will also ask that $\CA$ be a topological algebra.
Our starting point is a family of ``deformed Gaussian'' noncommutative probability spaces called $q$-Gaussians,
which were first introduced in \cite{FB70} and studied extensively by Bo\.{z}ejko \textit{et al.}, cf.\ \cite{BS91, BKS97,Boz99}.

In the commutative \slash classical case, given a Gaussian measure $\mu$ with
(real) reproducing kernel Hilbert space $\mfH$ realised as some process $\xi_1$,
we can formulate this classical probability space using a pair $(\CA_1,\omega_1)$ as follows.
The state $\omega_{1}$ is given by integration with respect to $\mu$, that is $\omega_{1}(a) = \int_{\Omega} a\ \d\mu$ where $\Omega$ is the underlying probability space of realisations of $\xi_1$.
One can then take\footnote{It will turn out that the commutative (bosonic) case will be an exception in our framework and we will not proceed by topologising $\CA_{1}$ as a locally $m$-convex algebra, see Section~\ref{sec:Bosons}.} $\CA_{1}$ to be the algebra generated by the random variables $\{ \xi_{1}(f): f \in \mfh\}$.

In the more general $q$-Gaussian setting, Gaussianity is characterised by a Wick rule for moments:
\begin{equ}[eq:WicksRuleIntro]
	\omega_{q}
	\Big(
	\prod_{i = 1}^{n} \xi_q(f_i)
	\Big)
	=
	\sum_{\bpi \in \CP_{[n]}}
	\prod_{b \in \bpi}
	q^{\mathrm{cr}(\bpi)} \scal{f_{b(1)}, f_{b(2)}}_{\mfH}
\end{equ}
where $\CP_{[n]}$ is the set of partitions of $[n]$ into $\frac{n}{2}$ pairs $b= \{b(1),b(2)\}$ and where $\mathrm{cr}(\bpi)$ is the number of crossings\footnote{See \eqref{eq:CrossingNumber}.}  in the partition $\bpi$. If $n$ is odd, then $\CP_{[n]} = \emptyset$ and this
expression vanishes.

For the commutative case above, we set $q=1$ and reproduce thereby the classical Wick rule. The family of $q$-Gaussians, parameterised by $q \in [-1,1]$, generalise Gaussians in that they
satisfy equation~\eqref{eq:WicksRuleIntro} for their specific value of $q$.

The choice of $q=-1$ in \eqref{eq:WicksRuleIntro} gives the fermionic Wick rule, while the choice\footnote{We adopt the convention that $0^{0} = 1$.} $q=0$ restricts the sum to non-crossing partitions, yielding the free Wick rule.
Note that for $q \neq 1$ this forces the underlying algebra $\CA$ to be noncommutative since the right-hand side of \eqref{eq:WicksRuleIntro} depends on the ordering of the $f_i$ in general. (In particular, when writing an expression like $\prod_{i=1}^n a_i$, this should always be interpreted as a shorthand for $a_1\cdots a_n$; we will use set notations like $\prod_{i\in I} a_i$ only when an implicit ordering of $I$ is assumed or in situations where the $a_i$ commute, as for the right-hand side
of \eqref{eq:WicksRuleIntro}.)

We now describe a concrete way to realise \eqref{eq:WicksRuleIntro} using a pair $(\CA_{q},\omega_{q})$.
For $q \in [-1,1)$, since $\CA_{q}$ is not a commutative algebra,
we cannot realise it as an algebra of real-valued functions over some probability space $\Omega$ as we do in classical probability.

A standard approach is to realise $\CA_{q}$ as an algebra of operators on the $q$-symmetrised Fock space $\CF_{q}$ generated
by $\mfH$, see Section~\ref{subsec:particles} where we describe this in detail.
In particular, one sets $\xi_{q}(f) = \alpha_{q}(f) + \alpha_{q}^{\dagger}(f)$ where $\alpha, \alpha^{\dagger}$
are annihilation and creation operators on $\CF_{q}$, satisfying, for $f,g \in \mfH$, the relation
\begin{equ}[eq:qrelations_Intro]
	\alpha_q(f) \alpha^\dagger_q(g) - q \alpha^\dagger_q(g) \alpha_q(f) = \scal{f,g}_{\mfH} \bone\;.
\end{equ}
The state $\omega_{q}$ is then
realised by setting $\omega_{q}(a) = \scal{ \1, a \1 }_{\CF_q}$ where $\1$ is\footnote{Note that we will continue this convention throughout the article, namely $\1$ will denote the vacuum vector in some Fock space and $\bone$ will denote the unit of some algebra.} the vacuum vector in $\CF_{q}$.
This construction applies in the $q=1$ case, giving an alternative way to realise the classic Gaussian case, in which case $\xi_1$ is called a bosonic field.

When $\mfH$ is a space of functions\slash distributions over a space-time domain $\spacetime = \R \times \T^{d}$,
the mapping $\mfH \ni f \mapsto  \xi_{q}(f) \in \CA_{q}$ can be thought of as an
``operator-valued distribution''.
Putting this all together, we see that, at least formally, we can interpret stochastic PDEs involving stochastic forcing by a $q$-Gaussian field
as operator-valued PDEs.
This is our starting point, but to go any further, we must answer the crucial question of how we should topologise the target operator algebras involved.

\subsubsection{An Illustrative Example: Algebra-Valued ODEs}
\label{sec:InitEx}

Before discussing the more complicated case of noncommutative singular nonlinear PDEs, we give a short description of the solution theory for algebra-valued nonlinear ODEs to motivate some of the topological assumptions we will make on the target algebra.

Let $\CA = \CM_n(\C)$ denote the algebra of $n \times n$ matrices.
This is a Banach algebra when equipped with the\footnote{The operator norm depends on a choice of norm on $\C^{n}$. However, particular choice is not important here due to the equivalence of norms on finite-dimensional vector spaces.} operator norm $\| \bigcdot \|$, in particular one has the submultiplicativity property
\begin{equ}
	\| A B \| \leqslant \| A \|  \| B \|
\end{equ}
for arbitrary $A,B \in \CA$.
This submultiplicative estimate is crucial for solving nonlinear ODEs with values in $\CA$.

Let $F \in \cC(\R; \CA)$ and let $P(X)$ be a single-variable polynomial. Consider the ODE for functions $\R \supset I \to \CA$, given by
\begin{equ}
	\dot x(t) = P(x(t)) + F(t) \;,
	\quad
	x(0) = x_{0} \in \CA \;.
\end{equ}
The standard Picard–Lindel\"{o}f argument for solving this equation is to rewrite it as a fixed-point problem for the operator $S \colon \cC([0,T]; \CA) \to \cC([0,T]; \CA)$ for some $T > 0$ where
the operator $S$ is defined by
\begin{equs}
	S[f](t) \eqdef x_0 +  \int\limits_0^t \left( P(f(s)) + F(s)\right) \d s \; .
\end{equs}
Thanks to the submultiplicativity of $\| \bigcdot \|$, we immediately get the bound
\begin{equ}
	\| P(f(t)) - P(g(t)) \| \lesssim \max\left\{ \|f(t)\|^{\deg(P)-1}, \|g(t)\|^{\deg(P)-1} , 1 \right\} \| f(t) - g(t) \| \;,
\end{equ}
which implies that $S$ is contractive on the space $\cC \left( [0,T] ; \CA\right)$ for small
enough $T$ (depending on $x_0 \in \CA$).
Clearly, this argument applies when the target space is any Banach algebra, even of infinite dimension.

However, applications will require us to solve equations in a wider class of topological algebras.
Even in the bosonic case, we will see that the algebra of random variables generated by the Gaussian field $\xi_{1}$ is not a Banach algebra.\footnote{The key issue is that Gaussian random variables are not in $L^{\infty}(\Omega,\mu)$.}
In the fermionic case, even though the corresponding ``Gaussian random variables'' do form a Banach algebra, renormalisation takes us outside of the Banach algebra setting.

A first possibility would be to consider locally convex algebras $\CA$,  topological
algebras equipped with a family of seminorms $\mfP$, such that multiplication is
continuous with respect to the topology induced by these seminorms.
Then, for each $\mfp \in \mfP$, there must exist (up to rescaling) two seminorms $\mfq, \mfr \in \mfP$, such that, for all $A,B \in \CA$,
\begin{equs}\label{eq:semiprop}
	\mfp (AB) \leqslant \mfq(A) \mfr(A) \; .
\end{equs}
However, since the seminorms change at each step of the iteration $f_{n+1} = S[f_n]$, we would
have difficulties closing the fixed point argument unless we have a uniform bound on all
seminorms simultaneously
\dash catapulting us back to the Banach algebra setting, which we were hoping to generalise beyond.

It is then natural to assume a condition stronger than \eqref{eq:semiprop} that is still more general than working in a Banach algebra \dash
we assume that $\mathcal{A}$ is not only locally convex, but locally $m$-convex.

\begin{definition}\label{def:mconvex}
	A topological algebra $\CA$ is said to be locally $m$-convex if it admits a family
	$\mfP$ of seminorms  inducing the topology on $\CA$
	with the property that, for all $\mfp \in \mfP$ and all $A,B \in \CA$,
	\begin{equs}
		\mfp(AB) \leqslant \mfp(A) \mfp(B) \; .
	\end{equs}
\end{definition}
We can now formulate a ``local existence'' theory that extends beyond the Banach algebra case.
In practice, we will specialise to $m$-Fr\'{e}chet algebras,
which are complete locally $m$-convex algebras that admit a countable family of submultiplicative seminorms inducing their topology.
We do this as non-metrisable topological vector spaces can behave quite pathologically when they interact with measure theory.

By standard arguments, for every $\mfp$, there now exists a time $T_{\mfp} > 0$ such that $f_n$ is a
Cauchy sequence for the seminorm
\begin{equ}
	\| f_n - f_m \|_{\mfp ; [0,T_\mfp]} = \sup_{t \in [0,T_\mfp]} \mfp \left( f_n(t) - f_m(t) \right)\;.
\end{equ}
This produces a Cauchy sequence in the Banach space of continuous functions with values in the Banach algebra $\CA_{\mfp} \eqdef \CA / \mfp^{-1}(0)$.

In the case where $\cA$ is an $m$-Fr\'{e}chet algebra, we may assume without loss of generality that the set  $\mfP = \{\mfp_n \,:\, n \in \N\}$ is totally ordered, i.e.\ $\mfp_n \leqslant \mfp_m$ if $n \leqslant m$, so that the corresponding sequence of stopping times $T_n \eqdef T_{\mfp_n}$ is monotone decreasing. In this case, $\CA$ is canonically the projective limit of the spaces $\CA_n  = \CA_{\mfp_n}$. 

If we are able to show that $\lim_{n\to\infty} T_n = T > 0$,
then $\varprojlim_{n} f_n = f \in \cC([0,T]; \CA)$ and we can view $f$ as a ``global'' solution up to time $T$.

\subsection{Topologising \TitleEquation{q}{q}-Gaussians}
The topological challenges we encounter with $q$-Gaussians vary as we change $q \in [-1,1]$.
Two properties we want our topology to have are (i) good multiplicative behaviour, since we are interested in nonlinear equations as described in the previous subsection,
and (ii) good behaviour with respect to renormalisation, since we work with rough equations.
For the second point, our probabilistic step will involve constructing renormalised fields as finite linear combinations of the homogeneous Wiener chaoses generated by an underlying
$q$-Gaussian field $\xi_{q}$.
The key question then is whether we can control objects in homogeneous Wiener chaoses as elements of the algebra where we pose the equation.

As already remarked, we can realise the algebras $\CA_{q}$ as algebras of operators on
a Hilbert space $\mathcal{F}_{q}(\mfH)$.
A natural choice of topology to try would be the corresponding operator norm
which is indeed a Banach algebra norm.
However, this Banach algebra norm will not be sufficient for our arguments for any choice of $q \in [-1,1]$.

\subsubsection{Bosons: \TitleEquation{q=1}{q=1}}\label{subsec:bosons}
We will often use the subscript $B$ rather than $q=1$ to distinguish the bosonic case.
In this case, the operators $\xi_{B}(f) = \xi_{1}(f)$ for $f \not = 0$ are unbounded operators, and the renormalised fields also give unbounded operators when smeared against test functions.
Going back to thinking about $\xi_{B}$ in the context of classical probability, a standard approach to solving random nonlinear equations is to
work ``pathwise'' in $\xi_B$: one fixes an arbitrary realisation, that is an element of $\Omega$, and obtains good analytic control with the realisation fixed.
Rough random equations are solved by leveraging the fact that renormalised fields are also well-defined and bounded for almost every realisation of $\xi_B$.
Our approach for bosons will be to essentially treat them in the framework of classical commutative probability.
For systems mixing different values of $q$ with $q=1$ (for instance, \eqref{eq:lin_sig_model}), we will formulate the dynamic in terms of (classically) random processes in noncommutative algebras, that is we work pathwise in the bosons.

%Re-interpreting this case in the language of noncommutative probability, we see that we can deal with unboundedness by using the elements of $\Omega$ to define multiplicative seminorms on $\CA_1$,
%which, ignoring sets of measure $0$, puts us in the setting of an $m$-convex algebra.

\subsubsection{Fermions: \TitleEquation{q=-1}{q=-1}}
We will often use the subscript $F$ rather than $q=-1$ to distinguish objects in the fermionic case.
In this case, the operators $\xi_{F}(f) = \xi_{-1}(f)$ are bounded operators, but renormalised fields can give rise to unbounded operators.
When working with unbounded operators coming from renormalised fields, we cannot directly apply the same strategy used to handle unboundedness in the bosonic case, as we cannot expect to realise renormalised fermionic fields as functions over some underlying $\Omega$.

However, in \cite{CHP23}, we were able to use the bosonic case as inspiration to ``localise'' fermions
using a family of finite-dimensional representations to create a family of multiplicative seminorms.
This allows us to realise fermions and their renormalised fields as taking values in a locally\footnote{A locally $C^{*}$-algebra is a complete topological $\ast$-algebra whose topology is generated by a family of sub-multiplicative seminorms, each satisfying the $C^{*}$-identity.} $C^{*}$-algebra $\cA_{F}$.
Heuristically, one should think of $\CA_{F}$ as an algebra of ``bounded random variables'' and of $\cA_{F}$ as the corresponding family of unbounded random variables.
We now state one of the main results of \cite{CHP23}, which summarises the construction of $\cA_{F}$.

\begin{theorem}[Theorem~1.12 of \cite{CHP23}]\label{thm:main_thm1_CHP23}
	Let $\mfH$ be a separable Hilbert space, let $\CA_{F}$ be the $C^{*}$-algebra of bounded operators generated by creation and annihilation operators on the antisymmetric Fock space $\mathcal{F}_{F}(\mfH)$, and let $\xi_{F}$ be an associated fermionic Dirac or Clifford field.
	Then there exists a locally $C^{*}$-algebra $\cA_{F}$, equipped with a family of $C^{*}$-seminorms $\big( \|\bigcdot\|_{n} \big)_{n=1}^{\infty}$ with the following properties.
	\begin{enumerate}
		\item $\cA_{F}$ contains a dense $C^*$-algebra $\mfA_{\infty} = \left\{ a \in  \cA_F \, \big| \, \sup_{n} \|a\|_{n} < \infty \right\}$, which carries a surjective $C^{*}$-homomorphism $\digamma \colon \mfA_{\infty}  \rightarrow \CA_{F}$.
		Under $\digamma$ vacuum states $\omega_{F}$ on $\CA_{F}$ pull back to give a densely defined state $\bomega_{F}$ on $\cA_{F}$ and one has canonical isomorphisms
		\begin{equ}
			\CL^{2}(\CA_{F},\omega_{F}) \simeq \CL^{2}(\mfA_{\infty},\bomega_{F})\;.
		\end{equ}
		\item Sequences of polynomials $P_{k}(\xi_{F}) \in \CA_{F}$ that converge in  $\CL^{2}(\CA_{F},\omega_{F})$ can be pulled back to polynomials $\mathbf{P}_{k}(\bxi_{F}) \in \mfA_{\infty}$ that converge in $\cA_{F}$.
	\end{enumerate}
\end{theorem}
The first part of the above theorem states that $\cA_{F}$ contains a bounded sub-algebra $\mfA_{\infty}$ that in essence contains the original bounded algebra $\CA_{F}$, and that this can be used to transfer the state $\omega_{F}$ to $\cA_{F}$.
The second part of the above theorem tells us that $\cA_{F}$ contains as elements the ``unbounded'' renormalised fields built from probabilistic limits of elements in $\CA_{F}$ \dash again, these limits would not in general belong to $\CA_{F}$.
However, allowing such unbounded elements means that we have to go from the Banach algebra setting of $\CA_{F}$ to the $m$-convex setting of $\cA_{F}$, as we forewarned the reader about in Section~\ref{sec:InitEx}.

In  \cite{CHP23}, the construction of Theorem~\ref{thm:main_thm1_CHP23} was applied to a singular SPDE involving fermions in the so-called Da Prato--Debussche regime.
While this equation could not be closed in the Banach algebra $\CA_{F}$ due to renormalisation, it could be closed in the space $\cA_{F}$.

In this article, we show that we can go beyond the Da Prato--Debussche regime for singular SPDEs with fermions by
formulating a theory of $\cA$-valued regularity structures (as a special case of a regularity structure valued in an $m$-convex space).

\subsubsection{Mezdons: \TitleEquation{q \in (-1,1)}{q \in (-1,1)}}\label{subsec:Intro_mezdonic}
To distinguish the $|q| < 1$ cases from the special cases of $q = \pm 1$, we will call the former setting the ``mezdonic'' setting and refer to the algebras as the algebras of $q$-mezdons.%\footnote{Coined in analogy with particle naming conventions suffixing the Ancient Greek neuter ending \foreignlanguage{greek}{-ον} and Bulgarian {\fontencoding{T2A}\selectfont\begin{otherlanguage}{bulgarian}Здравей\end{otherlanguage}} (me\v{z}du) in between.}

In the mezdonic case, the operator $\xi_{q}(f)$ is always bounded and \textit{ultracontractive estimates}
\cite{Boz99} show that elements of homogeneous Wiener chaoses can be realised as bounded operators.
This means that the Banach algebra $\CA_{q}$ will also include renormalised fields.
However, a new challenge for $|q| < 1$ is that we need new multiplicative estimates.
These estimates are similar to Young product estimates where a field of positive regularity
is intertwined\footnote{See the discussion around \eqref{eq:intertwined} below.} with a rougher renormalised field.
We control these intertwined products by defining a stronger norm \eqref{eq:small-q-AltNorm}.
We show that this new norm is multiplicative and it is easily seen to still contain all needed renormalised fields,
thus yielding a new Banach algebra $\cA_{q}$ -- which is dense in $\CA_{q}$ -- where we can solve singular SPDEs driven by $q$-noise.

\subsection{Main Results}\label{sec:mainresults}

To motivate the formalism developed in the article, we first give some specific examples of noncommutative singular SPDEs for which our new framework allows us to prove local well-posedness.
In Section~\ref{sec:genresults}, we will summarise some of the more general but technical results we prove, which allow us to treat these specific equations.

\subsubsection{The Mezdonic~\&~Cliffordian \TitleEquation{\Phi^4_3}{Phi^4_3}-Equations}\label{subsec:application1}

Given a spatial dimension $d$, we will write $\rho$ for a smooth compactly supported function on $\R_{<0} \times \R^{d}$
which is even in space and satisfies $\int \rho = 1$.
We also write, for $\eps > 0$, $\rho^{(\eps)}(t,x) = \eps^{-d-2} \rho(t/\eps^2, x/\eps)$.

\begin{theorem}[Mezdonic \TitleEquation{\Phi^4_3}{Phi^4_3}]\label{thm:Phi43}
	Let $q \in (-1,1)$, and let $\xi_{q}$ be a $q$-Gaussian space-time white noise over $\R \times \T^{3}$ \dash that is it satisfies
	\eqref{eq:WicksRuleIntro} with $\mfH = L^{2}(\R \times \T^{3})$.
	Then, one has local in time solutions $\phi_{q}$ to the singular equation
	\begin{equ}\label{eq:mezdonphi43}
		(\partial_{t} - \Delta + m^2) \phi_{q} = - \phi_{q}^{3} + \xi_{q}\;.
	\end{equ}
	In particular, for $\eps \in (0,1]$, there exist constants $C_\eps^{1}, C_\eps^{2}$ independent of $q$, with
	$|C_\eps^{1}| \lesssim \eps^{-1}$, $|C_\eps^{2}| \lesssim |{\log(\eps)}|$, and linear operators $\Delta^{(1)}_q, \Delta^{(2)}_q \in \CB(\cA_q) $ independent of $\eps$, such that one has the convergence as $\eps \downarrow 0$ of local in time solutions to the renormalised, regularised equations
	\begin{equ}
		(\partial_{t} - \Delta + m^2) \phi^{(\eps)}_{q} = - \Big( \left(\phi_{q}^{(\eps)}\right) ^{3} - \left(C_\eps^{1} \Delta_q^{(1)}  + C_\eps^{2} \Delta_q^{(2)}\right) \phi^{(\eps)}_{q} \Big) + \xi^{(\eps)}_{q}\;,
	\end{equ}
	where $\xi^{(\eps)}_{q} = \xi_{q} \ast \rho^{(\eps)}$.
	Here, the convergence takes place in a space of space-time distributions taking values in the Banach algebra $\cA_q$ as given in Theorem~\ref{thm:newBanachAlg}.
\end{theorem}

The $q=1$ case of Theorem~\ref{thm:Phi43} was first obtained in \cite{Hai14}, and our analysis for the $|q| < 1$ case follows roughly the same strategy but
using our theory of noncommutative regularity structures along with novel operator estimates on the algebra $\cA_{q}$.

The new mezdonic topology we use to define $\cA_q$ differs from the operator norm topology, as already alluded to when we mentioned the
new multiplicative estimates in Section~\ref{subsec:Intro_mezdonic}.
The definition of this operator algebra topology and proving its superior properties with respect to renormalisation are additional contributions of the present article.
We point to Theorems~\ref{thm:newBanachAlg}~and~\ref{thm:PolyWickestimate} for this construction.
The key role it plays in Theorem~\ref{thm:Phi43} is controlling the $\eps \downarrow 0$ limit of products such as
\begin{equ}\label{eq:intertwined}
	\<1>^{(\eps)}_{q}(z) v(z) \<1>^{(\eps)}_{q}(z) - C^1_\eps  \Delta_q\bigl(v(z)\bigr)\;,
\end{equ}
where $\<1>^{(\eps)}_{q}(x)$ is the solution to the linear equation $(\partial_{t} - \Delta) \<1>^{(\eps)}_{q} =  \xi^{(\eps)}_{q}$, $\Delta_q \in \CB(\cA_q)$ is the linear map which acts as multiplication by $q^n$ on the homogeneous Wiener chaos of order $n$, see \eqref{e:Deltaqdef}, $v$ is an $\cA_q$-valued function on space-time of sufficiently positive regularity,
and the convergence of $\<1>^{(\eps)}_{q}(z)^2  - C_{\eps}^1 \bone$ as $\eps \downarrow 0$ has previously been obtained using stochastic arguments.
See Lemma~\ref{lemma:2SYoung} where we explicitly show how this intertwined product is controlled. 

The novelty here is that noncommutativity makes \eqref{eq:intertwined} different from controlling the $\eps \downarrow 0$ limit of
$\left(\<1>^{(\eps)}_{q}(z)^2  - C_\eps^{1} \bone  \right) v(z)$,
which, once convergence of $\<1>^{(\eps)}_{q}(z)^2 - C_\eps^{1}   \bone $ is known, can be controlled with a scalar Young product estimate if $v$ is sufficiently regular. \martinp{Add Clifford}

\subsubsection{Mezdonic Rough Paths}

We prove in the mezdonic case the following results, which generalise the foundational theorems of rough path theory to the noncommutative setting. For the sake of simplicity, we restrict ourselves here to the one-dimensional statement. 
\begin{theorem}
	Let $q \in (-1,1)$ and $f, g, h \in \cC^\omega(\R; \R)$ be analytic functions. Let $\xi_q$ and $B_q$ be a $q$-$L^2(\R)$-white noise and $q$-Brownian motion in $\cA_q \eqdef \cA_q(L^2(\R))$, and let $X_0 \in \cA_q$ be an initial condition. Then, the rough differential equation
	\begin{equ}
		\begin{cases}
			d X  = f(X) + g(X) d B_q h(X) \\
			X(0) = X_0
		\end{cases}
	\end{equ}
	has a solution in $\CC^{\frac{1}{2}-}([0,T) ; \cA_q)$ for $T > 0$, unique up to a renormalisation constant, in the following sense.

	For every constant $C \in \R$ and sequence of $\xi^{(\eps)}_q \subset \cD(\R ; \cA_q)$ for $\eps \in (0,1]$ of mollifications of $\xi_q$ contained in the first chaos of $\cA_q$ and converging in $\CC^{-\frac{1}{2}-}(\R ; \cA_q)$, there exist a $T> 0$, a unique sequence of functions $X_{(\eps)} \in  \CC^{\frac{1}{2}-} \left( [0,T) ; \cA_q\right)$ and $X'_{(\eps)} \in  \CC^{\frac{1}{2}-} \left( [0,T) ; \left(\cA_q^{\wotimes_\pi 2} \right)\right)$ for all $\eps \in [0,1]$, converging as $\eps \downarrow 0$, s.t.\ for all $\eps \in [0,1]$ and all $s,t \in [0,T)$
	\begin{equ}
		\left\vvvert X_{(\eps)}(t) - X_{(\eps)}(s) - \Delta_q^{R; (1), \emptyset} \left(B^{(\eps)}_q(s,t) ; X'^{(\eps)}(s) \right)  \right\vvvert_{\cA_q^n} \lesssim |s-t|^{1-} \; .
	\end{equ}
	Here $\Delta_q^{R; (1), \emptyset} \in \CB(\cA_q)$ is a unique linear operator that only depends on $q$, but not on the choice of approximation or structure of the RDE.

	Furthermore, for all $\eps \in (0,1]$, $(X_{(\eps)}, X_{(\eps)}')$ satisfy on $[0,T)$
	\begin{equs}\label{eq:RenSDEInt}
		\partial_t X^{(\eps)} & = f^i(Y^{(\eps)})  + (X^{(\eps)}) \xi_q^{(\eps)} h(X^{(\eps)}) -\\
		& \qquad - C \Bigl( \left( D_{q}^R g(X^{(\eps)})[X'{(\eps)}] \right)  h(X^{(\eps)})  +  g(X^{(\eps)}) \left( D_q^L h(X^{(\eps)})[X'^j_{(\eps)}] \right) \Bigr)
	\end{equs}
	with $X_{(\eps)}(0) = X_0$. Here $D_q^{R/L}$ are $q$-commutative right and left derivatives applied to the Taylor expansions of $g$, $h$.
\end{theorem}
Here, the choice of $C$ corresponds to a choice of lift of the L\'evy area or renormalisation of the singular product $B_q \xi_q$. For example, setting 
\begin{equ}
	C \eqdef \lim_{\eps \downarrow 0} \lim_{T \to \infty} \omega_q \left( B^{(\eps)}_q(T) \xi_q^{(\eps)}(0) \right) = \frac{1}{2}  
\end{equ}
yields the It\^o lift. 

By changing the topology from the usual one determined by the $C^*$-structure to the new mezdonic topology, we were able to extend the solution theory to arbitrary $q \in (-1,1)$ and arbitrary lifts of the L\'evy area. In contrast, previous work \cite{BS98,DS18} could only deal with the It\^o, for $q=0$, and the Stratonovich, for $q \geqslant 0$, lifts respectively. However, this comes at the cost of having to restrict ourselves to analytic nonlinearities, as $\cA_q$ is not a $C^*$-algebra and thus does not possess a general functional calculus; for more detail, see Remark~\ref{rem:Comparison}.

Furthermore, we also establish an analogue of the It\^o formula for the mezdonic case for Brownian motion with the same methods applicable to more general processes. This extends a previous result of \cite{DS18} without the use of filtrations. 
\begin{theorem}
	In the above setting, we have for any analytic function $F$ the It\^o formula 
	\begin{equs}[eq:Ito]
		\partial_t F( B_{q} (t) ) &= F'(B_{q} (t))\left[ \xi_q \right]  + 2 C \Delta_q^{R; (1,1) , (1,2)} \left(  F''\left(B_q(t) \right) \right) \; ,
	\end{equs}
	where $\Delta_q^{R; (1,1) , (1,2)}$ is a specific bounded linear operator on $\cA_q$ independent of $F$. 
\end{theorem} 
Note that the $F'$ and $F''$ denote noncommutative derivatives, cf.\ for example \eqref{eq:NonComDer}.

\subsubsection{The Higgs--Yukawa\TitleEquation{_2}{2} Model}\label{subsec:application2}
Equation~\ref{eq:mezdonphi43} is not interesting in the case of Dirac fermions since the nonlinearity vanishes.
Instead, to show how our formalism can be applied to fermions, we apply it to the Langevin stochastic quantisation of the Higgs--Yukawa$_2$ model.
The Higgs--Yukawa$_2$ Euclidean state is formally given by
\begin{equ}
	\langle \,  \bigcdot \, \rangle_{S} = \frac{1}{Z} \langle \, \bigcdot \enskip e^{ -S(\phi,u,\bar{u}) } \rangle\;,
\end{equ}
with the action functional given by
\begin{equ}\label{eq:Yukawa_Action}
S(\phi,u,\bar{u})=	\int\limits_{\R^2} \Big(\frac{1}{2} |\nabla \phi|^2 + \frac{m^2}{2} \phi^2 + \bilin{\bar u, (-\snabla + M) u}+ g \phi \bilin{\bar u,u} + \frac{\lambda}{4} \phi^4 \Big)\,\d x\;,
\end{equ}
and where $\phi$ is a real bosonic field, $u$ and $\bar{u}$ are two independent, two-dimensional Euclidean Dirac spinor fields (corresponding to a fermionic particle and antiparticle), $\bilin{\bigcdot,\bigcdot}$ is the bilinear (not sesquilinear) extension of the usual scalar product on $\R^2$ to $\C^2$, $\snabla$ is the Dirac operator given by
\begin{equ}[eq:DiracOp]
	\snabla \eqdef \begin{pmatrix}
		 0 & -\partial_1+i \partial_2\\
		-\partial_1-i\partial_2 & 0
	\end{pmatrix} \; ,
\end{equ}
and the constants are $g \in \R$ and $m^{2}, M > 0$.
This type of Higgs--Yukawa$_2$ model in Euclidean signature was introduced in \cite{OS73}. We refer to \cite{CHP23} for a more detailed review of literature on the Higgs--Yukawa$_2$ model.

The system of singular SPDEs that we will study in the context of Higgs--Yukawa$_2$ are given by
\begin{equs}\label{eq:lin_sig_model}
    \partial_t \phi &= (\Delta-m^2) \phi - g  \bilin{(-\overline{\snabla}+M)\bar \upsilon , (\snabla+M)\upsilon} - \lambda \phi^3 +  \xi,	\\
    \partial_t \upsilon &=  (\Delta-M^2) \upsilon - g  \phi (\snabla+M) \upsilon  + \psi,\\
    \partial_t \bar\upsilon & =   (\Delta-M^2) \bar\upsilon- g \phi (-\overline{\snabla}+M)\bar\upsilon + \bar\psi\;,
\end{equs}
where  $\xi$ is a bosonic space-time white noise, and $(\psi, \bar{\psi})$ are a pair of space-time (Euclidean) free Dirac spinor fields, independent of $\xi$, satisfying
\begin{equs}[eq:fermion_noise_covar]
	\omega_F( \psi(s,x) \psi(t,y)) &= \omega_F( \bar{\psi}(s,x) \bar\psi(t,y)) = 0\;,\\
	\omega_F( \bar{\psi}(t,x) \psi(t',x')) &= \delta(t-t') \big(-\overline{\snabla}+M \big)^{-1}(x-x')\;.
\end{equs}
See Section~\ref{sec:HiggsYukawa2} for a more precise definition of the noise $(\psi, \bar{\psi})$.

Setting
 $\eta =  (\snabla +M)\upsilon$ and $\bar \eta= (-\snabla +M)\bar\upsilon$, one formally expects stationary solutions to
 \eqref{eq:lin_sig_model} to be such that  $(\phi(t),\eta(t),\bar\eta(t))^{*}\omega_F$ yields a Higgs--Yukawa$_2$ state at any fixed time.
We refer to \cite{CHP23} for a more detailed discussion of the relationship between the dynamic \eqref{eq:lin_sig_model} and the Higgs--Yukawa$_2$ state.

In \cite{CHP23} we studied local well-posedness for a regularised version of \eqref{eq:lin_sig_model}, which was required in order to be able to treat the problem using a Da Prato--Debussche argument.
Using the general results of the next section, we are now able to treat \eqref{eq:lin_sig_model} without any regularisation using noncommutative regularity structures.

\begin{theorem}\label{thm:Yukawa}
Let $\cA_F$ be the locally $C^{*}$-algebra as described in Theorem~\ref{thm:main_thm1_CHP23} associated to the Dirac fermionic space-time noise $(\psi,\bar{\psi}$)
appearing in \eqref{eq:lin_sig_model}.2
Then, there exist constants $\big(C^1_{\eps}\big)_{ \eps \in (0,1]}$, and central elements $\big(C^2_{\eps}\big)_{ \eps \in (0,1]}, \big(C^3_{\eps}\big)_{ \eps \in (0,1]} \subset Z(\cA_F)$ such that, for each $n \in \N$, the local-in-time solutions $\big(\phi_{\eps},( \upsilon_{\eps},\bar{\upsilon}_{\eps}) \big)$ to the regularised system of equations
\begin{equs}[eq:Yukawa_eps]
		\partial_t \phi_{\eps} &= (\Delta-m^2 + \lambda C^1_\eps - g^2 C^{3}_{\eps}) \phi_{\eps} \\
		{}& \qquad - g\Big(  \bilin{\bigl(-\overline{\snabla}+M \bigr) \bar \upsilon_{\eps} \; , \bigl(\snabla + M \bigr)\upsilon_{\eps}} - C^2_\eps \Big)
		  - \lambda \phi_\eps^3  +  \xi_{\eps},	\\
		\partial_t \upsilon_{\eps} &=  (\Delta-M) \upsilon_{\eps}- g  \phi_{\eps} \bigl( \snabla + M \bigr) \upsilon_{\eps}  + \bpsi_{\eps} \; ,\\
		\partial_t \bar\upsilon_{\eps} & =  (\Delta-M) \bar\upsilon_{\eps}- g \phi_{\eps} \bigl( -  \overline{\snabla} + M \bigr) \bar\upsilon_{\eps} + \bar\bpsi_{\eps}\;,
\end{equs}
seen as random fields with respect to the bosonic space-time white noise $\xi$ and taking values in $\cA_F$, project to random local-in-time solutions in $\cA_{n}$ that converge in probability on $\cA_{n}$ as $\eps \downarrow 0$.
Above, we write $(\xi_{\eps}, \bpsi_{\eps}, \bar{\bpsi}_{\eps}) = (\xi, \bpsi, \bar{\bpsi}) \ast \rho^{(\eps)}$ for the regularised noise fields, and $\cA_F$ is the algebra of operators on the Fock space generated by the fermionic white noise.

Here $\cA_{n}$ is the Banach space obtained by quotienting $\cA$  by the kernel of $\|\bigcdot\|_{n}$ and completing with respect to $\|\bigcdot\|_{n}$ where the $ \big( \|\bigcdot\|_{n} \big)_{n=1}^{\infty}$ are the family of seminorms appearing in Theorem~\ref{thm:main_thm1_CHP23}.
\end{theorem}
A key difference in the renormalisation needed for Theorem~\ref{thm:Yukawa} versus that of \cite[Theorem~1.4]{CHP23}
is the appearance of the non-local divergence, which puts the local well-posedness of \eqref{eq:lin_sig_model} outside the scope of the Da Prato--Debussche argument.
Written in terms of contractions of trees, the contributions to the constants $C^1_\eps$ and $C^2_\eps$ in \eqref{eq:Yukawa_eps} are in direct correspondence with contributions to the identically named constants in \cite[Theorem~1.4]{CHP23}.
We include $C^{3}_{\eps}$ for the contributions from the new non-local divergences.
See \eqref{eq:HiggsYukawa_constants} where this is made precise.

\subsubsection{General Results}\label{sec:genresults}

The machinery we develop in this paper that allows us to prove Theorems~\ref{thm:Phi43} and~\ref{thm:Yukawa} can be summarised in the following meta-theorem.

\begin{metatheorem}\label{metatheorem1}
The analytic theory of regularity structures generalises to the setting of $\mathcal{A}$-valued driving noises,
 where $\mathcal{A}$ is any locally $m$-convex algebra.

In particular, given a locally subcritical parabolic singular SPDE with $\mathcal{A}$-valued noises, one can generate a corresponding $\mathcal{A}$-regularity structure,
along with a corresponding space of models (including canonical lifts of smooth driving noises, and admitting renormalisation via preparation maps).

The Reconstruction, Multiplication, and Abstract Integration theorems hold on the associated spaces of modelled distributions.

For locally subcritical equations driven by $q$-Gaussian noises (with possibly different values of $q$) and satisfying a finite variance condition, one also has a BPHZ lift that is stable (in appropriate spaces of random models) under regularisation of the noise by smooth space-time mollification.
\end{metatheorem}
In comparison with the original analytic theory of regularity structures developed in \cite{Hai14},
our setting generalises the allowable ``target space'', namely the algebra of scalars in the target space is no longer required to be $\R$ or $\C$ but can instead be
any locally $m$-convex algebra $\mathcal{A}$.
As mentioned earlier, the stochastic step above is done via a post-processing of results from the commutative setting.

\subsection{Past Work and Related Literature}\label{sec:pastwork}

Noncommutative stochastic analysis has a long history at this point.
We will not attempt to give a complete overview of the literature, but point out some key developments
in this theory and describe how the present article fits into this picture.

The first article we have found on noncommutative Brownian motion is by Cockroft and Hudson \cite{CockroftHudson77} in 1977.
Starting in the early 1980s, a school of ``quantum stochastic calculus'' emerged \dash see for instance \cite{BSW1982} by Barnett, Streater, and Wilde for the first construction of fermionic stochastic integrals and followed in 1984 by the works of Hudson and Parthasarathy \cite{HudsonParthasarathy84} and Applebaum and Hudson \cite{ApplebaumHudson1984} which developed stochastic integrals taking values in operators on bosonic and fermionic Fock space, respectively along with It\^{o} formulae in these settings.
Useful references on stochastic calculus in the setting of bosonic and fermionic Fock spaces are \cite{parthasarathy2012introduction} and \cite{MeyerQuantum}.

More recently, Albeverio, Borasi, De Vecchi, and Gubinelli \cite{Gub20} developed a fermionic Itô calculus and an approach to fermionic Langevin stochastic quantisation in a non-tracial setting, where the authors worked at the level of Grassmann algebras rather than in the Fock space picture.
In \cite{Gub23}, the last two authors, along with Fresta and Gordina, extended this framework and built a more general $L^{p}$ and martingale theory in this non-tracial setting.

Beyond the bosonic and fermionic setting, K\"{u}mmerer and Speicher \cite{KS1992}, inspired by the approach of \cite{HudsonParthasarathy84}, developed an analogous theory of stochastic integration along with an It\^{o} formula in the setting of free Fock spaces.
Biane and Speicher \cite{BS98} further developed a theory of free It\^{o} calculus along with elements of a corresponding Malliavin calculus. In particular, their results were later employed by Kargin \cite{Kar11} to develop a local solution theory for SDEs driven by free white noise.
An extension of many of the results of \cite{BS98} to the mezdonic setting was obtained by Donati-Martin \cite{Donati-Martin2003}.
We note that the above results are focused on stochastic calculus with respect to Brownian motion \dash recently Jekel, Kemp, and Nikitopoulos \cite{JKN2023} developed a more general theory of stochastic calculus with martingales that includes mezdonic processes, see Remark~\ref{rem:Comparison} below for more details.

The story we have described in this subsection so far is closely analogous to the more probabilistic approach to commutative stochastic calculus.
As mentioned earlier, this article is much closer to the more analytic, pathwise approach to commutative stochastic calculus, i.e., the ``rough path approach''.

The present article is by no means the first attempt to develop a rough path approach in the setting of noncommutative probability theory.
In \cite{CapitaineDonatiMartin2001} Capitaine and Donati-Martin constructed a geometric rough path over free Brownian motion.
Deya and Schott \cite{DS18,DS19} extended this by constructing rough paths over $q \in [0,1)$ mezdonic Brownian motion and free fractional Brownian motion and used this to develop a local solution theory for SDEs driven by these processes.
In Section~\ref{sec:NonLinSDE} we apply our techniques to the problem of constructing rough paths for $q$-Brownian motion and compare our results to those of \cite{DS18}.
Finally, we mention \cite{BellingeriGilliers22} where Bellingeri and Gilliers develop a theory of signatures for smooth paths taking values in a general operator algebra.

Analogously to how the theory of regularity structures \cite{Hai14} gives an alternative formulation of the theory of branched rough paths in the commutative setting (which includes stochastic integration with respect to Brownian motion), the theory given in the present article gives an alternative approach to many of the stochastic integrals and rough path results described above.
It is completely novel, though, in being able to deal with the renormalisation and local well-posedness of very rough nonlinear singular SPDEs in the fermionic and mezdonic settings.

\subsection{Outline of the Paper}
In the remainder of Section~\ref{sec:intro}, we introduce and review some preliminary notation and background material. 
Section~\ref{sec:locally_t_convex} introduces some important definitions related to topologies on tensor products involving $m$-convex algebras, and in Section~\ref{sec:GNS}, we briefly review the GNS construction. 

Section~\ref{subsec:particles} describes $q$-Gaussians in more detail.
In Section~\ref{subsec:FockSpace} we define $q$-Gaussian algebras and states using their Fock space representation and also prove several key combinatorial formulae on $q$-Gaussian Hermite\slash Wick products.  
Sections~\ref{sec:Bosons} and~\ref{sec:Fermions} describe how bosons, Clifford fermions, and Dirac fermions fit into our framework \dash a more detailed description of the extended CAR algebra approach underlying our approach to Higgs-Yukawa$_2$ in Theorem~\ref{thm:Yukawa}.

Section~\ref{sec:qmezdons} gives crucial combinatorial formulae and analytic estimates needed for working with $q$-mezdons (that is, $q$-Gaussians with $|q| < 1$). 
Proposition~\ref{prop:OpBounds} states the key ultracontractive estimates first obtained by \cite{Boz99} for $q$-Mezdons, these allow us to realise renormalised products as bounded operators. 
Definition~\ref{def:newBanachAlg} and Theorem~\ref{thm:newBanachAlg} introduce a new Banach subalgebra $\cA_{q}$ of operators, and in Theorem~\ref{thm:PolyWickestimate} we show that the algebra $\cA_{q}$ allows us to control intertwined renormalised products such as \eqref{eq:intertwined}. 

We have included Section~\ref{sec:Besovestimates} as an intermission to describe how the framework developed so far can be used to give a noncommutative formulation of the Da Prato--Debussche argument.
In particular, we use some of our machinery for topologising tensor products along with a (soft) use of the Reconstruction theorem to show how we can generalise estimates for scalar Besov spaces, such as the Young product estimate Theorem~\ref{thm:YoungMult}, to the $m$-convex setting. 
Next, in order to provide a test case for the tools of Section~\ref{sec:qmezdons} we present a solution theory for $q$-Mezdon $\Phi^4_2$ in Theorem~\ref{thm:Phi42Ren}.

Section~\ref{sec:RS} then formulates an ``abstract'' theory of $\CA$-regularity structures for a given $m$-convex algebra $\CA$.  
In our formulation, the structure space for an $\CA$-regularity structure is more than a vector space, but in fact an $\CA$-bimodule \dash again, this is natural when considering that the classical theory of regularity structures can be thought of as a theory of $\R$-regularity structures. 
Section~\ref{sec:RS} also includes the corresponding definitions of a model and modelled distribution. 
The bulk of this section is then devoted to showing how the key analytic theorems of the classical theory of regularity structures \cite{Hai14} (reconstruction, abstract integration, modelled distribution product estimates, and the abstract fixed point map) generalise to the $\CA$-regularity structure setting.

Section~\ref{sec:TreeRegStruc} then specialises to $\mathcal{A}$-regularity structures generated by trees. 
This section introduces a class of decorated trees that are reminiscent of those introduced in \cite{Hai14,BHZ19} but with an additional $\CA$-decoration to allow for a bimodule structure and also a much richer extended decoration $\mfo$
that allows us to keep track of negative renormalisation in the noncommutative setting. 
To handle the noncommutativity in our setting, we work with a recursive formulation of the structure group, and for negative renormalisation, we adapt the preparation map approach of \cite{Br18}. We give an explicit inductive construction for a class of preparation maps and use this to give an elementary proof of the cointeraction property. 

We then show that, in the smooth setting, one has a corresponding canonical model and also a BPHZ lift. 
One gap in the present work is that we do not have a clear language and automated theorem for how renormalisation of the model induces counterterms in the $q$-Gaussian setting, in contrast with the commutative case \cite{BCCH21}. 
However, we do introduce a formalism of $\Delta_{q}$-operators associated to trees (giving more complicated analogues of the map $\Delta_{q}$ appearing in \eqref{eq:intertwined}) which allow us to encode how noncommutative extraction in negative renormalisations twists the resulting counterterm. 

Section~\ref{sec:StochEst} provides an ``API'' that allows us to use the black box theorem \cite{HS24} for stochastic estimates in the classical regularity structures to prove a similar black box result, Theorem~\ref{thm:BPHZ}, in the $q$-Gaussian setting.
The key observation is that, thanks to the Fock space structure enjoyed by all the $q$-Gaussians, stochastic estimates always reduce to estimates of sums of norms of deterministic kernels, modulo $q$-dependent factors and $q$-dependent symmetrisations. 
Section~\ref{subsec:FockSpaceAlg} introduces an algebra called the ``Fock space algebra'' $\mfA$ which serves as a bookkeeping tool to store kernel estimates so that control of the BPHZ lift of a specifically designed model for an $\mfA$-regularity structure gives control of BPHZ lifts of the associated models in all of the corresponding $q$-Gaussian regularity structures. 
Finally, we can obtain control of this BPHZ lift in the $\mfA$-regularity structure by applying \cite{HS24} in the classical commutative Gaussian case on a specially designed regularity structure and model. 

The paper concludes with Section~\ref{sec:Examples}, where we apply the theory we have developed throughout this paper to the following set of examples we mentioned above:
\begin{itemize}
	\item Rough Paths for mezdons for all $q \in (-1,1)$, a full local-in-time solution theory for rough SDE in this setting, as well as an It\^o-Formula, cf.\ Section~\ref{sec:NonLinSDE},
	\item Existence of local-in-time solutions for the mezdonic $\Phi^4_3$-equation for all $q \in (-1,1)$, cf.\ Section~\ref{sec:Phi43Mez},
	\item Existence of local-in-time, noncommutative ``pathwise'' solutions for the Cliffordian $\Phi^4_3$-equation, cf.\ Section~\ref{sec:eq:Phi43Clif},
	\item Existence of local-in-time, noncommutative ``pathwise'' solutions for the Euclidean Higgs-Yukawa$_2$-model, cf.\ Section~\ref{sec:HiggsYukawa2}. 
\end{itemize}

\subsection*{Acknowledgements}

{\small
AC gratefully acknowledges partial support by the EPSRC through EP/S023925/1.
MH gratefully acknowledges support by the Royal Society through a research professorship.
MP gratefully acknowledges support by EPSRC through the ``Mathematics of Random Systems'' CDT  EP/S023925/1.
}

\subsection{Notations and Conventions}\label{subsec:notation}

We use the notation $[n] \eqdef \{1,\dots, n\}$ for $n \in \N$.
For $k \in \N^d$, we write $|k| = \sum_{i = 1}^d k_i$.

We denote by $\mfS_n$\label{symbol:permutation} the permutation group for $n$ elements and $\mfS_{n,k} \eqdef \mfS_{n}/(\mfS_k \times \mfS_{n-k})$ \label{symbol:twopermutation}. When we write $\sigma \in \mfS_{n,k}$, we will always mean the representative of the equivalence class with the minimal number of inversions.

We denote by $\mathbb{K} \in \{\R,\C\}$ the base field of our vector spaces.
Accordingly, by ``scalar'' we will usually mean $\mathbb{K}$-valued. When dealing with inner products $\Braket{\bigcdot, \bigcdot}$ of complex vector spaces, we will always take them to be linear in the second argument and antilinear in the first. The symbol $\otimes$\label{symbol:tensor} will always denote the algebraic $\mathbb{K}$-tensor product of $\mathbb{K}$-vector spaces and $\wotimes$ a topological closure of this tensor product, the choice of which may be specified by a corresponding index. $\wotimes_\pi$ will denote the projective tensor products, $\wotimes_\eps$ the injective tensor product, $\wotimes_{C^*}$ the spatial tensor product of $C^*$-algebras, and $\wotimes_\alpha$ the Hilbert space tensor product. However, we will drop the index in the latter case unless necessary. For further details concerning these tensor products, see Appendix~\ref{appendix:FuncAna}. Furthermore, for a vector space $V$ and $k \in \N^d$ we will use the shorthand notation
\begin{equ}
	V^{\otimes k} \eqdef \bigotimes_{i = 1}^d V^{\otimes k_i}
\end{equ}
and analogously for the completions. For a finite index set $I$ and vector spaces $V_i$ for $i \in I$, we will also use the shorthand $f_{\otimes I} \eqdef \bigotimes_{i \in I} f_i \in \bigotimes_{i \in I} V_i$. We will refer to vectors of this form as simple tensors.   

For vector spaces $V,W$, let $L(V,W)$ be the space of linear maps $V \to W$. If $V,W$ are topological vector spaces, let $\CB(V,W)$ be the space of continuous linear maps. If $V = W$, we drop the second vector space in the notation. 
We write $\bone = \bone_{V}$ to denote the identity map on $V$. We overload notation as well for a unital algebra $A$, where we write $\bone = \bone_{A}$ for the unit of $A$. For $n \in \N$, we denote by $\CB^n(V; W)$ the space of separately continuous $n$-linear maps $V^n \to W$. We will often switch in notation between viewing $\phi \in \CB^n(V; W)$ as a multilinear map and a linear map $V^{\otimes n} \to W$ without remarking on it. However, whenever we perform extensions to topological tensor products, we shall justify why they are possible.

Whenever we consider a locally convex topological vector space (LCTVS)  $E$, we shall always assume that $E$ is Hausdorff and complete unless otherwise stated, and we will denote a set of continuous seminorms defining the topology on $E$ by $\mfP$.

Given any measure space $(S,\mu)$, and $q \in [1,\infty)$, we write $L^{q}(S,\mu)$ for the standard $L^{q}$ space of (equivalence classes of) scalar functions.
We define the vector space $L^{\infty -}(S,\mu) \eqdef \bigcap_{q\geqslant 1} L^{q}(S,\mu)$, which is a
locally convex topological vector space (LCTVS) when equipped with the family of norms $( \| \bigcdot \|_{L^q})_{q\geqslant 1}$. We also drop $\mu$ from the notation when it is clear from context.

% For a LCTVS space $(\CF(U), \left( \| \bigcdot \|_{\CF,\mfq} \right)_{\mfq \in \mfQ})$ of functions on $U \subset \R^d$, we will often consider a corresponding space of $E$-valued functions $\CF(U;E)$ with one seminorm $\| \bigcdot\|_{\CF , \mfq ; \mfp}$ for every $\mfp \in \mfP$ and function seminorm $\| \bigcdot \|_{\CF , \mfq}$. We say that a linear map $\phi \in \left( \CF(U;E) , \CF'(V;E) \right) $ is $t$-continuous\footnote{The $t$ stands for tensorial, as the seminorms are typically cross seminorms.} if for all $\mfp \in \mfP$ and all $f \in \CF(U;E)$
% \begin{equ}
% 	\left\| \phi(f) \right\|_{\CF' , \mfq' ; \mfp} \leqslant C(\mfq', \mfp) \left\| f \right\|_{\CF, \mfq ; \mfp} \;  .
% \end{equ}

When $E$ is a Fr\'echet space, we write $L^{q}(S,\mu ; E)$ for the closure of the algebraic tensor product $L^{q}(S,\mu) \otimes E$, with respect to the topology induced by the seminorms
\begin{equ}
	\| f \|_{L^q \otimes \mfp} \eqdef \biggl( \int_S \bigl(\mfp(f(x))\bigr)^q \d \mu (x) \biggr)^{\frac{1}{q}}
\end{equ}
for all $\mfp \in P$.
%\footnote{One should note this is not the space of all appropriately integrable $E$-valued measurable functions, as when $E$ is not separable the latter space will be strictly larger. The functions we defined above are usually called strongly measurable, cf.\ \cite{Hyt16} \ajay{If we indeed need to use this definition for $E$ not separable, add a pointer to where it happens}}
We again write $L^{\infty - }(S,\mu; E) = \bigcap_{q\geqslant 1} L^q(S,\mu; E)$.

% We write $\spacetime \eqdef \R \times \T^{2}$ for the primary space-time domain we work on in this article.
% We sometimes also work on $\R^3$ as our space-time, where the first coordinate denotes time,
% so the corresponding spatial domain $\T^2$ is replaced with $\R^2$. We also introduce the subsets $\spacetime_T \eqdef [0,T] \times \T^2$ and $\spacetime_+ \eqdef \R_+ \times \T^2$.

Given an open subset $U$, we write  $\cD(U)$ for the space of smooth, compactly supported scalar-valued functions on $U$.
For a LCTVS $E$, we let $\cD'(U;E)\eqdef \CB( \cD(U) ; E )$.
We write $\cC(U)$ for the space of continuous scalar functions on $U$ and, for $r \in \N$, $\cC^r(U)$ for the space of $r$ times continuously differentiable functions on $U$, and analogously for $\cC^r(U ; E)$.
By $\cD'(\R^n \times \T^m ; E)$, we denote the subspace of $\cD'(\R^{n+m} ; E)$ of distributions $\xi$ that are invariant under the natural action of $\Z^m$ by translation on the last $m$coordinates.
We write $H^s(\R^d)$ for the Sobolev space of index $s \in \R$.

In the context of SPDEs, it often happens that we need to scale different space-time directions differently. We thus introduce a scaling $\s = (\s_1,\dots, \s_d) \in \N^d_+$ and, for $(x_1,\dots, x_d) \in \R^d$,
we define its rescaling by a factor $\lambda > 0$ by
\begin{equ}
	\lambda^\s (x_1,\dots, x_d) \eqdef  (\lambda^{\s_1} x_1, \dots,  \lambda^{\s_d} x_d)\;.
\end{equ}
We use a corresponding scaled ``norm'' on $\R^d$ by setting, for $x  \in \R^d$, $|x|_\s \eqdef |x_1|^{1/\s_1} + \cdots + |x_d|^{1/\s_d}$, so that $|\lambda^\s x|_{\s} = \lambda |x|_{\s}$. By abuse of notation we also define, for $k \in \N^d$, $|k|_{\s} \eqdef \sum_{i = 1}^n \s_i k_i$.

With this scaling, we define the scaling dimension $d_{\s} \eqdef \sum_i \s_i$ of our space, in particular the function $|x|_{\s}^{j}$ is locally integrable near $0$ if and only if $j > -d_{\s}$.
For $x \in \R^d$, $r>0$ we define the scaled balls $B_{\s}(x,r) \eqdef \left\{ y \in \R^d \, \middle| \, |y-x|_{\s} < r \right\}$.

For a continuous function $\phi \in \cC(\R^d)$, we define for $x \in \R^d$, $\lambda \in (0,1]$ the recentred rescaled version of $f$ by
\begin{equ}
	\CS^\lambda_{\s, x} \phi (y) = \lambda^{-|\s|} \phi \left( \lambda^{-\s} (y-x)\right)\;.
\end{equ}

We fix a mollifier $\rho \in \cD(\R^d)$, that satisfies $\int\rho = 1$, $\rho \geqslant 0$, and $\rho(-x) = \rho(x)$ for all $x \in \R^d$. For a distribution $\xi \in \cD'(\R^d)$, we will denote its mollification by
\begin{equ}
	\xi^{(\eps)} (x) \eqdef \xi\left( \CS^\eps_{\s, x} \rho  \right)
\end{equ}
for $\eps \in (0,1]$.

The scaling $\s = (1,\dots, 1)$ will be called Euclidean, and whenever a scaling is not mentioned in the notation, it will be assumed to be Euclidean. E.g.\ $B(x,r)$ will denote the usual ($\ell^1$-)ball around $x$ of radius $r$.

\subsubsection{Locally \TitleEquation{t}{t}-Convex Spaces}\label{sec:locally_t_convex}
When solving (nonlinear) equations in locally topological vector spaces, we will want to work with stronger continuity conditions on maps between such spaces.
Rather than the standard topological notion of continuity, we will want to impose more quantitative conditions that provide estimates on seminorms in the codomain in terms of estimates on specified seminorms in the domain \dash the strengthening here is analogous to how $m$-convexity is stronger than just being a topological algebra.

In order to formulate this, it is natural to assume some choice of correspondences between seminorms on different locally topological vector spaces as input.
\begin{definition}[Locally \TitleEquation{t}{t}-Convex]
	Let $\mfP$ be an indexing set and $(F, \mfQ)$ a locally convex topological vector space.
	A $t$-convex realisation of $(F, \mfQ)$ over $\mfP$ is a partition $(\mfQ_{\mfp})_{\mfp \in \mfP}$ of $\mfQ$ (or an equivalent set of seminorms) into disjoint non-empty sets.
	We say that $(F, (\mfQ_\mfp)_{\mfp \in \mfP})$ is a locally $t$-convex space over $\mfP$.
	We call seminorms in $\mfQ_\mfp$ the $\mfp$-seminorms of $F$.
	We sometimes drop the mention of $\mfP$ and simply call $(F, (\mfQ_\mfp)_{\mfp \in \mfP})$ a locally $t$-convex space.
\end{definition}

\begin{remark}
	The symbol $t$ in $t$-convex stands for tensorial, as the $t$-convex spaces $F$ we will be considering are typically $E$-valued function spaces where $(E,\mfP)$ is a locally convex vector space and we use the set of seminorms $\mfP$ of $E$ to partition those of $F$. 
	
	Such function spaces can be realised as a tensor product of a scalar-valued function space with $E$, and the seminorms in $\mfQ_\mfp$ are then the seminorms that can be realised as crossnorms containing $\mfp$.
	In such a situation, we will also sometimes say that $(F,\mfQ)$ is $t$-convex over $(E, \mfP)$ or just $E$, instead of $\mfP$.
\end{remark}

We then have the promised strengthened continuity condition.

\begin{definition}

	A continuous linear map $\phi \colon F_1 \to F_2$ between two spaces $(F_1, \mfQ^1)$ and $(F_2, \mfQ^2)$ that are locally $t$-convex over $\mfP$ is said to be $t$-continuous if, for every $\mfq^2 \in \mfQ^2_{\mfp}$, there exists $\mfq^1 \in \mfQ^1_{\mfp}$ and a constant $C \geqslant 0$, such that, for all $v \in F_1$,
	\begin{equ}
		\mfq_2\left( \phi(v)\right) \leqslant C \mfq_1(v) \; .
	\end{equ}

	We drop the parentheses above for linear maps when no confusion may arise.
Analogously, for a multilinear map $\phi \colon F_1 \times \cdots \times F_n \to G$ of locally $t$-convex spaces $(F_i , \mfQ^i)$, $(G, \mfY)$, we say that it is $t$-continuous, if for all $\mfp \in \mfP$, $i \in [n]$ and $\mfy \in \mfY_\mfp$, there exists $\mfq_i \in \mfQ_{\mfp}^i$ and a constant $C>0$ such that, for all $v_i \in F_i$,
	\begin{equ}
		\mfy \left( \phi(v_1, \dots, v_n) \right) \leqslant C \mfq_1(v_1) \cdots \mfq_n(v_n) \; .
	\end{equ}
\end{definition}

\begin{remark}
	If we have two  locally $t$-convex spaces $F_1$ and $F_2$ that are $E$-valued function spaces, then $t$-continuity of $\phi \colon F_1 \rightarrow F_2$ imposes a continuity condition where how we measure sizes in $E$ is consistent between functions in $F_1$ and $F_2$.
\end{remark}

\begin{example}
	A simple example of a $t$-convex space is the space $\cC(\R^d ; V)$ of continuous functions  $\R^d \to V$ where $V$ is a Banach space, equipped with pointwise convergence. Then $\cC(\R^d; V)$ is locally $t$-convex over $\cC(\R^d)$. The set of seminorms on both spaces is indexed by compacta $ \K \Subset \R^d$ and given by
	\begin{equ}
		\|f\|_{\K} \eqdef \sup_{x \in \K} |f(x)| \; , \qquad \|g\|_{\K} \eqdef \sup_{x \in \K} \| g(x)\|
	\end{equ}
	for $f \in \cC(\R^d)$ and $g \in \cC(\R^d; V)$ respectively. For any bounded linear operator $A \colon V \to V$, the map $f \mapsto A \circ f$ is $t$-continuous, whereas, for example, convolution with a kernel $K \colon \R^d \to \R$ would not be $t$-continuous.
\end{example}

\begin{remark}
	A locally convex space $(E,\mfP)$ is always locally $t$-convex over itself by taking the partition $\mfP_\mfp \eqdef \{ \mfp\}$. Furthermore, for a pair $(F_1, \mfQ^1)$, $(F_2, \mfQ^2)$ of $t$-convex spaces over $E$, the product space $F_1 \times F_2$ becomes a $t$-convex space when equipped with the topology induced by the seminorms $\left( \mfQ^1_\mfp \oplus \mfQ^2_\mfp \right)_{\mfp \in \mfP}$ with
	\begin{equ}
		\mfQ^1_\mfp \oplus \mfQ^2_\mfp \eqdef \left\{ \mfq_1 \oplus \mfq_2 \, \big| \, \mfq_1 \in \mfQ^1_\mfp \, , \, \mfq_2 \in \mfQ^2_\mfp   \right\} \; .
	\end{equ}
\end{remark}

\begin{definition}[\TitleEquation{t}{t}-Projective Tensor Product]
\label{def:tProjTens}
	Let $(F_1, \mfQ^1)$, $(F_2, \mfQ^2)$ be two locally $t$-convex spaces over $(E, \mfP)$.

	We define the $t$-projective tensor product topology on $F_1 \otimes F_2$ to be given by the collection of seminorms $\mfq_1 \otimes \mfq_2$ with $\mfq_1 \in \mfQ^1_\mfp$ and $\mfq_2 \in \mfQ^2_\mfp$ for $\mfp \in \mfP$ where
	\begin{equ}
		\big(\mfq_1 \otimes \mfq_2 \big) ( A ) \eqdef \inf \sum_{i}  \mfq_1(v_i)\mfq_2(w_i)\;,
	\end{equ}
	with the infimum being taken over all finite subsets of $\left\{ \left(v_i, w_i\right) \right\}_{i} \subset F_1 \times F_2$, s.t.\
	\begin{equ}
		A = \sum_{i} v_i \otimes w_i  \; .
	\end{equ}
	%This sequence of topologies is strongly $m$-compatible with $\mfP$.
	By a slight abuse of notation we will denote the completion of $F_1 \otimes F_2$, with respect to this topology by $F_1 \wotimes_\pi F_2$ and this tensor product in general by $\wotimes_\pi$.\footnote{The reason this is an abuse of notation is that the $t$-projective tensor product topology is slightly weaker than the traditional projective tensor product as we do not allow mixed products of seminorms. The only time we will be using the ``normal'' projective tensor product is when we take products of Banach spaces where we only have one norm, so that Definition~\ref{def:tProjTens} and the standard definition coincide.}
\end{definition}

By the same arguments used for the projective tensor product, \cite[Proposition~43.4]{Trev67}, one can prove the following statement.
\begin{proposition}
\label{prop:ProjTensExt}
	Any $t$-continuous bilinear map $\phi \colon F_1 \times F_2 \to G$ of locally $t$-convex spaces $(F_i , \mfQ^i)$, $(G, \mfY)$ extends to a $t$-continuous linear map $\phi \colon F_1 \wotimes_\pi F_2 \to G$.
\end{proposition}

Beyond just locally $m$-convex algebras $\CA$, we will also need to deal with different types of bimodules over these algebras. Let $(\CA, \mfP)$ be a locally $m$-convex algebra.
\begin{definition}[Locally $\boldsymbol{t}$- and $\boldsymbol{m}$-Convex Modules]
\label{def:tmconvex}
	A locally convex $\CA$\-/bimodule $(\CM, \mfQ)$  is called a locally $t$-convex module if $\CM$ is locally $t$-convex over $(\CA,\mfP)$ and the multiplication map $\CA \times \CM \times \CA \to \CM$ is locally $t$-convex of norm $1$, i.e.\ for all $m \in \CM$, $a,b \in \CA$, $\mfp \in \mfP$, and $\mfq \in \mfQ_\mfp$
	\begin{equ}
		\mfq(amb) \leqslant \mfp(a) \mfq(m) \mfp(b) \; .
	\end{equ}
	We say that $\CM$ is locally $m$-convex if and only if $\CM$ is locally $t$-convex and each $\mfQ_\mfp = \{ \mfp_{\CM} \}$ consists of exactly one seminorm.
\end{definition}

\begin{remark}
	 The locally $m$-convex modules we will encounter are going to be of the form $\CA^{\wotimes_\pi n}$, see above for the exact definition of $\wotimes_\pi$.
\end{remark}

\begin{definition}[$\boldsymbol{t}$-Continuous Morphisms]
	Let $(\CM_1, \mfQ^1)$, \dots, $(\CM_n, \mfQ^n)$, and $(\CN,
	\mfY)$ be $n+1$ locally $t$-convex $\CA$-bimodules. We say that a map $\phi \colon \CM_1 \times \cdots \times \CM_n \to \CN$ is an $\CA$-multilinear morphism, if $\phi$ is $\mathbb{K}$-multilinear, and for all $a,b \in \CA$, $m_i \in \CM_i$
	\begin{equ}
		\phi(am_1, m_2, \dots, m_{n-1}, m_n b) = a\phi(m_1, m_2, \dots, m_{n-1}, m_n )b \; ,
	\end{equ}
	and
	\begin{equ}
		\phi(m_1, \dots, m_i a , m_{i+1}, \dots , m_n ) = \phi(m_1, \dots, m_i , a m_{i+1}, \dots , m_n ) \; .
	\end{equ}

	By abuse of notation we shall denote by $\CB(\CM)$ the space of $t$-continuous, $\CA$-bimodule morphisms $\CM \to \CM$.
\end{definition}

\begin{example}
	For a locally $m$-convex algebra $\CA$, $\CA^{\wotimes_\pi n}$ is naturally an $m$-convex $\CA$-bimodule for all $n$. Furthermore, all the possible internal multiplications $\CA^{\wotimes_\pi n } \to \CA^{\wotimes_\pi m}$ for $m < n$ are $t$-continuous.
\end{example}

\subsubsection{GNS Construction \texorpdfstring{\dash}{-} Noncommutative \TitleEquation{\CL^2}{L2}-Space}\label{sec:GNS}

We now recall the Gel'fand--Naimark--Segal (GNS) construction for an arbitrary unital $C^{*}$-algebra $\CA$, which will be fundamental for our definition of renormalised products \slash trees. For a
detailed discussion of this construction, see \cite[Section~2.3.2]{BR87}.

Throughout this subsection, we fix a $C^*$-algebra $\CA$ and a state $\omega$, that is a continuous 
linear functional on $\CA$ that satisfies\footnote{We shall use $\dagger$ instead of $*$ for the adjoint operation on our $C^*$-algebras in keeping with the customary notation for creation and annihilation operators.} $\omega(A^\dagger A)> 0$ for all $A \in \CA \setminus \{0\}$ and is normalised such that 
$\omega(\bone_{\CA}) = 1$.
This allows us to define a positive semi-definite inner product on $\CA$
\begin{equs}[e:scalarProd]
	\scal{\bigcdot, \bigcdot} \colon 	\CA \times \CA & \to \C\\
		(A,B) & \mapsto \omega(A^\dagger B)\;.
\end{equs}
By quotienting out the left ideal
\begin{equ}[e:def_of_null_ideal]
	\CN_\omega \eqdef \big\{ A \in \CA \, \big| \, \omega(A^\dagger A) = 0 \big\}
\end{equ}
and completing the resulting space under the norm determined by \eqref{e:scalarProd}, one obtains a Hilbert space we shall write as
\begin{equ}
	\mathcal{L}^2(\CA, \omega) \eqdef \overline{\CA/\CN_\omega}\;.
\end{equ}
Writing $\CA \ni B \mapsto [B] \in \overline{\CA/\CN_\omega}$ for the quotient map, we note that the algebra $\CA$ naturally acts on this Hilbert space via $A[B] \eqdef [AB]$.

If one considers $\CA$ to be the $C^*$-algebra of bounded measurable functions on a measure space $(\Omega,\Sigma)$ and $\omega$ to be integration against a probability measure $\mu$, then the above construction is exactly the same as the construction of the classical $L^2(\Omega,\Sigma,\mu)$ space.

\section{\TitleEquation{q}{q}-Gaussians for \TitleEquation{q \in [-1,1]}{}}
\label{subsec:particles}

All of the noise processes we focus on in this article are generalisations of Gaussians, namely, they can be realised using operators acting on a Fock space representation and satisfying an analogue of Wick's rule.

In this section, we first give a quick review of algebraic $q$-Fock space representations for $q \in [-1,1]$ and some of the combinatorial identities they satisfy. We use this to reformulate Gaussians from classical probability theory, which we refer to as the ``bosonic'' setting, and free fermionic \dash ``anticommutative'' \dash algebras as concrete topological algebras. These are respectively the cases $q = 1$ and $-1$. Next, we review the localised approach for fermions we proposed in \cite{CHP23}.

We postpone the discussion of topologising the Mezdonic case $|q| < 1$ to Section~\ref{sec:qmezdons} as we introduce a novel topology on the algebra different from the usual one provided by the $C^*$-algebra structure.

Our presentation below partially follows \cite{Boz99,Kemp05}.

\subsection{The Fock Space Picture}\label{subsec:FockSpace}

All three types of algebras mentioned above can be realised as algebras of operators acting on a type of Hilbert space called a Fock space. For $q \in (-1,1)$ the following definitions are adapted from \cite{Boz99}.

Let $\mfH$ be a Hilbert space over a field $\mathbb{K} \in \{\R , \C\}$ with inner product $\scal{\bigcdot,\bigcdot}$, and let $\1$ be a unit vector in $\mathbb K$. If $\mfH$ is complex, we also fix a real structure, i.e.\ an antiunitary involution $\kappa \colon \mfH \to \mfH$.\footnote{When  $\mfH$ is a space of functions, 
this would typically just be complex conjugation.}

We then define the algebraic Fock space $\CF(\mfH)$ by setting
\begin{equ}
	\CF(\mfH) \eqdef \bigoplus_{n \in \N} \mfH^{\otimes n}\;,
\end{equ}
where we adopt the convention that $\mfH^{\otimes 0} = \mathbb{K}$. We denote by $\Braket{\bigcdot , \bigcdot}_{\CF_0}$ the sum of the usual tensor inner products on $\mfH^{\otimes n}$ and by $\| \bigcdot \|_{\CF_0}$ the corresponding norm.

For $q \in [-1,1]$, we define the $q$-symmetrisation operator by
\begin{equ}
	P_q (f_1 \otimes \cdots \otimes f_n) \eqdef \sum_{\sigma \in \mfS_n} q^{|\sigma|} f_{\sigma(1)} \otimes \cdots \otimes f_{\sigma(n)}
\end{equ}
for $f_1, \dots, f_n \in \mfH$, and $P_q \1 \eqdef \1$, extended by linearity to $\CF(\mfH)$. Here, $|\sigma|$ is the number of inversions in the permutation $\sigma$, that is
\begin{equ}
	|\sigma| \eqdef \big| \left\{ (i,j) \in [n] \times [n]  \, \big| \,( i < j )\land (\sigma(j) < \sigma(i)) \right\} \big|  \;.
\end{equ}
Throughout the paper we adopt the convention that $0^{0} \eqdef 1$, so that $P_0 = \bone_{\CF(\mfH)}$. For $q=-1$, $q^{|\sigma|} = \sgn(\sigma)$ is the sign of the permutation $\sigma$.

The symmetrisation operator satisfies the following algebraic identity for all $n, k \in \N$ and all $f_1, \dots, f_n \in \mfH$
\begin{equ}
	P_q(f_1 \otimes \cdots \otimes f_n) = \sum_{\sigma \in \mfS_{n,k}} q^{|\sigma|} P_{q} \left( f_{\sigma(1)} \otimes \cdots \otimes f_{\sigma(k)} \right) \otimes P_{q} \left( f_{\sigma(k+1)} \otimes \cdots \otimes f_{\sigma(n)} \right)
\end{equ}
We remind the reader here that by abuse of notation, we mean by $\sigma \in [\sigma] \subset \mfS_{n}$ the unique representative of $[\sigma] \in \mfS_{n,k}$, s.t.\
\begin{equ}
	|\sigma| = \min_{\pi \in [\sigma]} |\pi| \; ,
\end{equ}
see \cite[Chapter~1.10]{Hum90} for more details.

% where for $\sigma \in \mfS_{n,k} = \mfS_{n}/(\mfS_k \times \mfS_{n-k})$

% ,\martin{Why do you sometimes use $|\sigma|$ and sometimes $\inv(\sigma)$?}\martinp{Abuse of notation on lhs $\sigma$ is an equivalence class, rhs a specific permutation}
% \begin{equ}
% 	|\sigma| = \inv(\sigma) \; .
% \end{equ}

\begin{example}
	For example setting $n = 4$ and $k=2$, the equivalence classes of $\mfS_{4,2}$ are in one-line notation
	\begin{equs}
	{}  [(1234)] &= \{ (1234), (2134), (1243), (2143) \} \\
	{}	[(1324)] &= \{ (1324), (2314), (1423), (2413) \} \\
	{}	[(3124)] &= \{ (3124), (3214), (4123), (4213) \} \\
	{}	[(1342)] &= \{ (1342), (2341), (1432), (2431) \} \\
	{}	[(3142)] &= \{ (3142), (3241), (4132), (4231) \} \\
	{}	[(3412)] &= \{ (3412), (3421), (4312), (4321) \}
	\end{equs}
	where we have ordered the permutations in each equivalence by increasing inversion number. We would thus identify the equivalence class $[(3142)]$ with the permutation $(3142)$ which has inversion number $3$, whereas the other permutations in this class have inversion numbers $4$ and $5$.

	Similarly, we see that $\mfS_{3,1} = \left\{ (123), (213), (231) \right\}$ and thus
	\begin{equs}
		P_q(f_1 \otimes f_2 \otimes f_3) & = f_1 \otimes f_2 \otimes f_3 + q f_1 \otimes f_3 \otimes f_2 +\\
		&\qquad + q \left( f_2 \otimes f_1 \otimes f_3 + q f_3 \otimes f_1 \otimes f_2 \right) +\\
		&\qquad + q^2 \left( f_2 \otimes f_3 \otimes f_1 + q f_3 \otimes f_2 \otimes f_1 \right) = \\
		& = \sum_{\sigma \in \mfS_{3,1}} q^{|\sigma|} \left( f_{\sigma(1)} \otimes f_{\sigma(2)} \otimes f_{\sigma(3)} + q f_{\sigma(1)} \otimes f_{\sigma(3)} \otimes f_{\sigma(2)} \right) = \\
		& = \sum_{\sigma \in \mfS_{3,1}} q^{|\sigma|} P_q \left( f_{\sigma(1)} \right) \otimes P_q \left( f_{\sigma(2)} \otimes f_{\sigma(3)} \right) \; .
	\end{equs}
\end{example}

In \cite[Lemmata~2~\&~3]{BS91} it was shown that $P_q \geqslant 0$ for $q \in [-1,1]$ and that $P_q > 0$ for $q \neq \pm 1$. Thus, $\scal{\bigcdot,\bigcdot}_{\CF_{q}} \eqdef \Braket{\,\bigcdot\, , P_q \, \bigcdot \,}$ is positive definite and defines an inner product on
\begin{equ}
	\mathring \CF_q(\mfH) \eqdef \CF(\mfH) / \ker(P_q) \; .
\end{equ}

\begin{definition}
Let $\CF_{q}(\mfH)$ be the completion of $\mathring \CF_q(\mfH) $ under $\scal{\bigcdot,\bigcdot}_{\CF_{q}}$. We will call $\CF_{q}(\mfH)$ the $q$-Fock space. The $n$-particle subspace, $\CF_{q,n}(\mfH)$, is the closure of $\mfH^{\otimes n}/\ker(P_q \big|_{\mfH^{\otimes n}}) $ in $ \CF_{q}(\mfH)$.
\end{definition}

\begin{remark}
	For $q \in (-1,1)$, and all $n \in \N$,  $\CF_{0,n}(\mfH) \subset \CF_{q,n}(\mfH)$. In fact, we show in Corollary~\ref{cor:nPartEq} that they are equal as Banach spaces.
\end{remark}

\begin{remark}
	For $f_1, \dots, f_m , g_1, \dots, g_n \in \mfH$
	\begin{equ}\label{eq:q-innerproduct}
		\scal{f_1 \otimes \cdots \otimes f_m , g_1 \otimes \cdots \otimes g_n }_{\CF_{q}} \eqdef
		\mathbbm{1}_{\{m = n\}}
		\sum_{\sigma \in \mfS_{n}}
		q^{|\sigma|}
		\prod_{j=1}^{n} \scal{f_j,g_{\sigma(j)}}_{\mfH}\;.
	\end{equ}
\end{remark}

% We also note that, for $q \in (-1,1)$, we have
% \begin{equ}
% 	\CF_{0}(\mfH) = \bigwoplus_{n \in \N} \mfH^{\wotimes n} \subset \CF_{q}(\mfH)\;,
% \end{equ}
% see \cite{??.}
% \end{remark}

For $f \in \mfH$, we define creation and annihilation operators $\alpha_q(f)$ and $\alpha_q(f)^{\dagger}$ on $\CF(\mfH)$ by setting, $\alpha_q^{\dagger}(f)\1 = f$, $\alpha_q(f)\1 = 0$, and, for $f_1 \otimes \cdots \otimes f_n \in \mfH^{\otimes n}$,
\begin{equs}
\alpha^\dagger_q(f) \left( f_1 \otimes \cdots \otimes f_n \right) & \eqdef f \otimes f_1 \otimes \cdots \otimes f_n\;,\\
\alpha_q(f) \left( f_1 \otimes \cdots \otimes f_n \right) & \eqdef\sum_{i = 1}^n q^{i-1} \scal{f, f_i} f_1 \otimes \cdots \otimes \widehat{f_i} \otimes \cdots \otimes f_n\;,
\end{equs}
where $\widehat{f_i}$ indicates that we have removed the factor $f_i$ from the tensor product. $\alpha_q^\dagger(f)$ and $\alpha_q(f)$ descend to well-defined operators on $\mathring\CF_q(\mfH)$ and are adjoints of each other with respect to  the inner product $\scal{\bigcdot,\bigcdot}_{\CF_{q}}$, cf.\ \cite{BS91}.

\begin{remark}
	Here the assignment $f \mapsto \alpha_q^\dagger(f)$ is linear whereas $f \mapsto \alpha_q(f)$ is antilinear, when $\mathbb{K} = \mathbb{C}$.
\end{remark}

\begin{remark}
	For $q \in [-1,1)$ the creation and annihilation operators extend to bounded operators on $\CF_q(\mfH)$, see \eqref{eq:CAROpBnd} and \eqref{eq:qbound}.

	%Technically, $\alpha_1(f)$ and $\alpha_1^\dagger(f)$ are closable unbounded operators, which can defined for example on the domain $\mathring{\CF}_1(\mfH)$, which these operators map into itself. The algebraic identities we shall discuss should be taken as identities for linear operators $\mathring{\CF}_1(\mfH) \to \mathring{\CF}_1(\mfH)$.
\end{remark}

The creation and annihilation operators satisfy the following generalised canonical commutation relations: for any $f,g \in \mfH$ we have
\begin{equ}\label{eq:qrelations}
	\alpha_q(f) \alpha^\dagger_q(g) - q \alpha^\dagger_q(g) \alpha_q(f) = \scal{f,g}_{\mfH} \bone\;,
\end{equ}
where $\bone$ denotes the identity operator on $\CF_{q}(\mfH)$. (Recall that our convention
for complex scalar products is that $\scal{f,g}_{\mfH}$ is linear in $g$ and antilinear in $f$.)

\begin{definition}[\TitleEquation{q}{q}-\TitleEquation{\mfH}{H}-Noise Operator]
\label{def:qNoise}
	For a Hilbert space $\mfH$ and $q \in [-1,1]$, let $\xi_q \colon \mfH \to L \bigl( \mathring{\CF}_0(\mfH)\bigr)$ be the map
	\begin{equ}
		f \longmapsto \xi_q(f) = \alpha_q^\dagger(f) + \alpha_q(\kappa f) \;,
	\end{equ}
	where $\kappa = \bone_{\mfH}$ if $\mfH$ is real and its real structure otherwise. 

	We will call $\xi_q$ the $q$-$\mfH$-noise operator. If $q = 1$, this will also be called the bosonic noise operator, for $q=-1$, the Clifford\footnote{The natural operators for Dirac fermions will not be self-adjoint, see Definition~\ref{def:DiracNoise}} noise operator and, for $|q|<1$, the mezdonic noise operator. If $\mfH = L^2(\R^d)$, these will also be called white noises.
\end{definition}

\begin{definition}[\TitleEquation{q}{q}-Field Algebras]
	For a Hilbert space $\mfH$ and $q \in [-1,1]$, let $\mfA_q(\mfH)$ denote the $*$-algebra generated by $\left\{ \xi_q(f) \, \big| \, f \in \mfH \right\}$.
	In the cases $q = 1$, $q = -1$, and $|q| < 1$, $\mfA_q(\mfH)$ will be called the bosonic, Clifford, and mezdonic algebra, respectively.

	Let $\omega_q \colon \mfA_q(\mfH) \to \mathbb{K}$ be the state given by
	\begin{equ}
		A \longmapsto \Braket{\1, A \1}_{\CF_q} \; .
	\end{equ}
	The state $\omega_q$ will be called the \textit{vacuum} state.
\end{definition}

\subsubsection{Wick Products}
\label{sec:WickRenorm}

Let $\CL^2(\mfA_q(\mfH), \omega_q)$ denote the space $\mfA_q(\mfH)$ equipped with the topology induced by the inner product
\begin{equ}
\scal{A,B}_{\CL^2} = \omega_q(A^\dagger B) \; .
\end{equ}
$\CL^2(\mfA_q(\mfH), \omega_q)$ isometrically embeds into $\mathring\CF_q(\mfH)$ via the extension of the mapping $\mfA_q(\mfH) \ni A \mapsto A \1 \in \mathring\CF_q(\mfH)$.

We now define Wick-ordered operators, which we use to describe the inverse isometry, mapping  $\mathring\CF_q(\mfH)$ into $\mfA_q(\mfH)$, equipped with the norm $\| \bigcdot \|_{\CL^2}$.

\begin{definition}
For each $n \in \N$,
we define a Wick ordering map $\xi_q^{\diamond n}\colon \mathring\CF_{q,n}(\mfH) \to \mfA_{q}(\mfH)$
as follows.
For $n\in \{0,1\}$, we set
\begin{equ}
	\xi_q^{\diamond 0}(c) = c \bone, \quad \xi_q^{\diamond 1} (f) \eqdef \xi_q(f) \; .
\end{equ}
For $n \geqslant 2$, we inductively define $\xi_q^{\diamond n}(f_1 \otimes \cdots \otimes f_n)$ by the formula
\begin{equs}
	\label{eq:WickDef1}
	\xi_q^{\diamond n}(f_1 \otimes \cdots &\otimes f_n) \eqdef \xi_q(f_1) \xi_q^{\diamond (n-1)}(f_2 \otimes \cdots \otimes f_n) - \\
	& \qquad -\sum_{i = 2}^n q^{i-2} \scal{\kappa f_1,f_i}_\mfH \xi_q^{\diamond(n-2)} \bigl(f_2 \otimes \cdots \otimes \widehat{f_i} \otimes \cdots \otimes f_n\bigr)    \; ,
\end{equs}
where $\widehat{f_i}$ denotes the absence of the factor $f_i$ in the tensor product.
\end{definition}

\begin{remark}
	For $q = \pm 1$, $\xi_q^{\diamond n}$ is \textit{a priori} only well-defined $F \in \mfH^{\otimes n}$. 
	However, it is not difficult to check that $\xi_{q}^{\diamond n}$ satisfies the appropriate (anti)symmmetry properties and that these maps thus descend to the quotient spaces. 
\end{remark}

\begin{remark}
Given a finite totally ordered set $I$, it will be notationally convenient in the sequel to 
write $\xi_q^{\diamond I} \colon \mfH^{\otimes I} \to \mfA_{q}(\mfH)$ for the map
defined exactly as above via the unique identification of $I$ with $\{1,\ldots,|I|\}$.
Similarly, we define the map $\xi_q^{I} \colon \mfH^{\otimes I} \to \mfA_{q}(\mfH)$
by setting $\xi_q^{I}(f_{\otimes I}) = \prod_{i \in I} \xi_q(f_i)$.
\end{remark}

Note that we can rewrite \eqref{eq:WickDef1} as
\begin{equs}
	\label{eq:WickDef2}
	\xi_q^{\diamond n} & (f_1 \otimes \cdots \otimes f_n) = \xi_q^{\diamond (n-1)}(f_1 \otimes \cdots \otimes f_{n-1}) \xi_q(f_n)  - \\
	& \quad -\sum_{i = 1}^{n-1} q^{n-1-i} \scal{\kappa f_i,f_n}_\mfH \xi_q^{\diamond(n-2)} \bigl(f_1 \otimes \cdots \otimes \widehat{f_i} \otimes \cdots \otimes f_{n-1}\bigr)   \; .
\end{equs}

We call any element of the image of $\xi_q^{\diamond n}$, for fixed $n$, a ``Wick-ordered operator'' or ``Wick product''.
The following proposition states the isometry property of these Wick ordering maps.

\begin{proposition}\label{prop:WickOrthog}
For any $n \in \N$ and $F \in \mathring\CF_{q,n}(\mfH)$,
\begin{equ}
	\xi_q^{\diamond n}(F) \1 = F\;.
\end{equ}
In particular, for $m \in \N$ and $G \in \mathring\CF_{q,m}(\mfH)$ we have
\begin{equ}
\scal{ \xi_q^{\diamond n}(F),\xi_q^{\diamond m}(G) }_{\CL^2} = \scal{F,G}_{\CF_q} \; .
\end{equ}
\end{proposition}
\begin{proof}
We prove the statement by induction on $n$, with the base cases $n=0,1$ being immediate.
Assuming $n \geqslant 2$ and the result holds for $n-1$, we have
\begin{equs}
\xi_q^{\diamond n}(f_{1} \otimes \cdots \otimes f_n) \1
&=
\big( \alpha_q^\dagger(f_1) + \alpha_q(f_1) \big) \big( f_2 \otimes \cdots \otimes f_n \big) - \\
{}& \qquad -
	\sum_{i = 2}^n q^{i-2} \scal{f_1,f_i}_\mfH f_2 \otimes \cdots \otimes \widehat{f_i} \otimes \cdots \otimes f_n =\\
&=  \alpha_q^\dagger(\kappa f_1) \big( f_2 \otimes \cdots \otimes f_n \big) = f_{1} \otimes \cdots \otimes f_{n}\;,
\end{equs}
as claimed.
\end{proof}

\begin{remark}
In the physics literature, Wick-ordered operators are often called ``normal-ordered'' operators.
The reason for this name is that one can view $\xi_q^{\diamond n}(f_1 \otimes \cdots \otimes f_n)$ as being obtained by starting with the product
\[
\xi_q(f_1) \cdots \xi_q(f_n)
=
\prod_{j=1}^{n} \big(\alpha_q^\dagger(f_j) + \alpha_q(\kappa f_j) \big)\;,
\]
and then moving all creation operators appearing in each summand in the expansion of $\xi_q(f_1) \cdots \xi_q(f_n)$ to the left of all annihilation operators,
as if we had the relation
\begin{equ}
\alpha_q(f)\alpha_q^\dagger(g) - q \alpha_q^\dagger(g) \alpha_q(f) \text{ ``$=$'' } 0 \; .
\end{equ}
\end{remark}

In the spirit of the above remark, we state the following lemma.
\begin{lemma}\label{lem:WickBlockDecomposition}
For each $k, \ell \in \N$,
we define a ``Wick block'' operator
\[
W^{k,\ell}_q \colon \mathring\CF_{q,k+\ell}(\mfH) \to \mfA_q(\mfH)
\]
as the extension of the $(k+\ell)$-multilinear map on $\mfH$ given by mapping
\begin{equs}[eq:WickBlockExp]
	\mfH^{k+\ell}	\ni (f_1, \dots, f_{k+\ell}) \longmapsto \hspace*{-0.4cm} \sum_{\sigma \in \mfS_{k+\ell,k}} \hspace*{-0.3cm} q^{|\sigma|} \alpha_{q}^\dagger( & f_{\sigma^{-1}(1)}) \cdots \alpha_{q}^\dagger( f_{\sigma^{-1}(k)})\\
	& \; \alpha_{q}( \kappa f_{\sigma^{-1}(k+1)}) \cdots \alpha_{q}( \kappa f_{\sigma^{-1}(k+\ell)}) \;.
\end{equs}
We then have, for any $n \in \N$ and $f_1, \dots, f_n \in \mfH$, the ``Wick block decomposition''
\begin{equ}
	\label{eq:BlckDecomp}
	\xi_q^{\diamond n}(f_1 \otimes \cdots \otimes f_n) = \sum_{k = 0}^n W^{n-k,k}_q(f_1 \otimes \cdots \otimes f_n) \; . 
\end{equ}
\end{lemma}
\begin{proof}
	See \cite[Proposition~1.1]{Boz99}.
\end{proof}

\begin{remark}
For fixed $F \in \mathring\CF_{q,k+\ell}(\mfH)$, the operator $W_q^{k,\ell}(F)$ annihilates the first $\ell$ particles and then creates $k$ new ones,
that is, for $j \geqslant \ell$,
	\begin{equ}
		W_q^{k,\ell}(F) \colon \mathring\CF_{q, j} \longrightarrow \mathring\CF_{q, j - \ell + k} \; .
	\end{equ}
In particular, each operator in the decomposition \eqref{eq:BlckDecomp} maps a fixed $j$-particle sector into mutually orthogonal particle sectors.
\end{remark}

Proposition~\ref{prop:WickOrthog} shows that $\iota_q = \bigoplus_{n=0}^{\infty} \xi_q^{\diamond n}$ isometrically embeds $\mathring\CF_q(\mfH)$ into $\mathcal{L}^{2} \big( \mfA_q(\mfH), \omega_{q} \big)$, we now verify it indeed extends to an isomorphism between $\CF_q(\mfH)$ and $\mathcal{L}^{2} \big( \mfA_q(\mfH), \omega_{q} \big)$ by verifying that an arbitrary product $\xi_q(f_1)\cdots \xi_q(f_n) \in \mfA_q(\mfH)$ can be written as a sum of Wick products.
We begin by defining a notion of ``partially contracted'' Wick products.

\begin{definition}
	Let $I$ be a totally ordered, finite set, $k \in \N$ with $ 2k \leqslant |I|$, and let $\CP_{I,k}$ be the set of collections of $k$ pairwise disjoint pairs in $I$
	\dash that is $\bpi \in \CP_{I,k}$ if $\bpi = \{ \pi_1 , \dots, \pi_k\}$ with $\pi_{\ell} = (i,j)$, $i,j \in I$ and $i < j$, and $\pi_{\ell} \cap \pi_{\ell'} = \emptyset$ for $\ell \neq \ell'$. For $I = [n]$ with $n \in \N$, we set $\CP_{[n],k} \eqdef \CP_{[n],k}$. Let $\CP_{I} \eqdef \bigcup_{2k \le  |I|} \CP_{I, k}$.
	Finally, by a slight abuse of notation, for $\bpi \in \CP_I$, we set
	\begin{equ}
		\bigcup \bpi \eqdef \bigcup \left\{ \{i,j\} \, \big| \, (i,j) \in \bpi \right\}\;,\qquad
		I \setminus \bpi \eqdef I \setminus \bigcup \bpi \; .
	\end{equ}
We will also write $\emptyset_I$ for the empty contraction of $I$.
\end{definition}
\begin{definition}
\label{def:ContWickProd}
	For $\bpi \in \CP_{I}$, we define $\xi_q^{\diamond \bpi} \colon \mfH^{\otimes I} \rightarrow \mfA_{q}(\mfH)$ by setting
	\begin{equ}[eq:AbsContr]
		\xi_q^{\diamond \bpi} \left(  f_{\otimes I} \right) \eqdef q^{\crb(\bpi)} \xi_q^{\diamond (I \setminus \bpi) } \left( f_{\otimes I \setminus \bpi} \right)  \prod_{(s,t) \in \bpi}  \scal{\kappa f_s,f_t}_\mfH
	\end{equ}
	%\martinp{Add remark below that this extends to Hilbert space tensor product}
	where $\crb(\bpi) = \cross(\bpi) + \sep(\bpi)$ is the intertwining number of $\bpi$
	with $\mathrm{cr}(\bpi)$ the crossing number of $\bpi$, that is	
	\begin{equ}\label{eq:CrossingNumber}
		\cross(\bpi) \eqdef \sum_{(i,j) \in \bpi} \left|\left\{ (k,l) \in \bpi \, \big| \,  i < k < j < l\right\}\right| \;,
	\end{equ}
	and
	\begin{equ}
		\sep( \bpi) \eqdef \sum_{(s,t) \in \bpi} \left| [s,t] \cap \left( I \setminus  \bpi\right) \right|
		= \sum_{i \in I \setminus  \bpi} |\{(s,t) \in \bpi\,:\, i \in [s,t]\}|
	\end{equ}
	the separation number of $\bpi$.
\end{definition}

\begin{example}
	We illustrate the intertwining number using the example of $\bpi =\{(1,4),(2,5)\}$ inside of $[6]$. We can represent this contraction using the diagram
	\begin{equ}
		\scalebox{0.8}
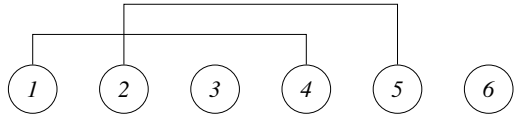
{
			\begin{tikzpicture}[every node/.style={circle,draw,minimum size=8mm}]
			% nodes
			\node (n1) at (0,0) {1};
			\node (n2) at (1.5,0) {2};
			\node (n3) at (3,0) {3};
			\node (n4) at (4.5,0) {4};
			\node (n5) at (6,0) {5};
			\node (n6) at (7.5,0) {6};

			% bracket-like arcs
			\draw (n1.north) |- ++(0,0.5) -| (n4.north);
			\draw (n2.north) |- ++(0,1.0) -| (n5.north);
			\end{tikzpicture} 
		} 
	\end{equ}
	The contraction $\bpi$ has exactly one crossing, and the pairings $(1,4)$ and $(2,5)$ are each separated only by the element $3$, as $2$ and $4$ are not counted per definitionem. Thus, the separation number of $\bpi$ is $2$. In total, $\crb(\bpi) = 3$.
\end{example}

\begin{remark}\label{rem:crb}
	The reason for the notation $\crb(\bpi)$ is that one has $\crb(\bpi) = \f12 \cross(\bar\bpi)$, where $\bar\bpi$ 
	is obtained by ``duplicating'' $\bpi$, with the second copy being its mirror image, and connecting all free slots to their opposite.  Another interpretation is that $\crb(\bpi)$ counts the crossings if the ``unpaired'' slots are connected to a
	point at infinity.

	In the above example, this amounts to the following diagram 
	\begin{equ}
	\scalebox{0.8}
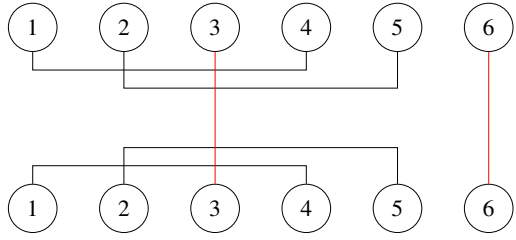
{
	\begin{tikzpicture}[every node/.style={circle,draw,minimum size=8mm}]
		% reflected copy above
		\begin{scope}[yshift=2.5cm, yscale=-1]
			% nodes
			\node (n1r) at (0,0) {1};
			\node (n2r) at (1.5,0) {2};
			\node (n3r) at (3,0) {3};
			\node (n4r) at (4.5,0) {4};
			\node (n5r) at (6,0) {5};
			\node (n6r) at (7.5,0) {6};

			% bracket-like arcs
			\draw (n1r.south) |- ++(0,0.3) -| (n4r.south);
			\draw (n2r.south) |- ++(0,0.6) -| (n5r.south);
		\end{scope}

		% original diagram
		% nodes
		\node (n1) at (0,0) {1};
		\node (n2) at (1.5,0) {2};
		\node (n3) at (3,0) {3};
		\node (n4) at (4.5,0) {4};
		\node (n5) at (6,0) {5};
		\node (n6) at (7.5,0) {6};

		% bracket-like arcs
		\draw (n1.north) |- ++(0,0.3) -| (n4.north);
		\draw (n2.north) |- ++(0,0.6) -| (n5.north);

		\draw[red] (n3) -- (n3r);
		\draw[red] (n6) -- (n6r);
	\end{tikzpicture}
	}
	\end{equ} 
	where we have added the new contractions in red. It contains exactly $6 = 2 \crb(\bpi)$ crossings.  
\end{remark}

\begin{theorem}
\label{thm:WickTerms}
For any totally ordered set $I$ and $(f_{i})_{i \in I} \in \mfH^{I}$
	\begin{equ}\label{eq:Wickcontraction}
		\xi^{I}_{q}(f_{\otimes I})
		=
		\prod_{i \in I} \xi_q(f_i)
		=\sum_{\bpi\in \CP_{I}} \xi_q^{\diamond \bpi} (f_{\otimes I})\;.
\end{equ}

\end{theorem}
\begin{remark}
	We remind the reader that for an ordered, finite index set $I$, $f_{\otimes I}$ means $\bigotimes_{i \in I} f_i$.
\end{remark}
\begin{proof}
	We write $I = [n]$, and proceed by induction over $n$.
	When $n = 1,2$, the desired claim is a quick computation using the inductive definition
	\eqref{eq:WickDef1}.
	For the inductive step we use \eqref{eq:WickDef2} and observe that, for fixed $\bpi \in \CP_{[n],k}$,
	\begin{equs}\label{eq:qProdInd}
		\xi_q^{\diamond(n-2k)} & \Big(  f_{\otimes [n]  \setminus \bpi}\Big) \xi_q(f_{n+1})
		= \xi_q^{\diamond(n+1-2k)} \Big( f_{\otimes  [n+1] \setminus \bpi} \Big) + \\
		& \qquad +  \sum_{ j \in [n+1] \setminus \bpi }  q^{\#(j, \bpi)} \scal{\kappa f_j,f_{n+1}}_\mfH \xi_q^{\diamond(n-1-2k)} \Big( f_{\otimes I(j,\bpi) } \Big) \; .
	\end{equs}
	Here $I(j,\bpi) =  [n+1] \setminus  ( \bpi \sqcup \{(j,n+1)\})$ and $\#(j, \bpi )$ is the number of elements of $[n+1] \setminus  \bpi$ between $j$ and $n+1$. Any contraction of $[n+1]$ is either already a contraction of $[n]$ or can be written as $\bpi \sqcup \{(j,n+1)\}$ for some $\bpi \in \CP_{[n]}$ and $j \in [n] \setminus \bpi$. Therefore, \eqref{eq:qProdInd} produces all possible contractions of the integers $[n+1]$ by summing over $\bpi \in \CP_{[n]}$. Thus, we only need to check that
	\begin{equ}
		\label{eq:CrsNum}
		\crb\left(\bpi \sqcup \{(j,n+1)\}\right)  = \crb(\bpi) + \#(j, \bpi) \; .
	\end{equ}
	Note that changing the base set from $[n]$ to $[n+1]$ does not impact the crossing or separation number of $\bpi$.

	If $j \in [s,t]$ for some pair $(s,t) \in \bpi$, then the separation number of $(s,t)$ is reduced by 1 but we introduce a new crossing of the pairs $(s,t)$ and $(j,n+1)$, thus the contribution of $(s,t)$ to the sums of the separation and crossing number remains the same on both sides of \eqref{eq:CrsNum}. Therefore, the only term appearing when we add $(j,n+1)$ is the separation number of $(j,n+1)$ relative to $\bpi \sqcup \{(j,n+1)\}$. However, this is exactly $\#(j,\bpi)$.
\end{proof}

Note that by Proposition~\ref{prop:WickOrthog}, for $n > 0$ and $F \in \mathring{\CF}_{q,n}$ we have
\[
\omega_{q} \big( \xi^{\diamond n}_{q}(F) \big)
= \scal{\xi^{\diamond 0}(1),\xi^{\diamond n}_{q}(F)}_{\CL^2} = 0\;,
\]
so that $\omega_q \left( \xi_q^{\diamond n , \bpi} (f_1 \otimes \cdots \otimes f_n) \right) = 0$ unless $\bigcup \bpi = [n]$.
Therefore, Theorem~\ref{thm:WickTerms} yields the following ``$q$-Wick'' rule as a corollary, see for instance also \cite[Corollary~2.1]{EffrosPopa03}.
\begin{corollary}
\label{cor:q-Wick}
	For all $f_i \in \mfH$ we have
	\begin{equs}
		\omega_q\bigl(\xi_q(f_1) \cdots \xi_q(f_{2n})\bigr) &= \sum_{\bpi \in \CP_{[2n],n}}  q^{\cross(\bpi)}
		\prod_{(i,j)  \in \bpi} \scal{\kappa f_{i},f_{j}}_\mfH \;,\\
		\omega_q\bigl(\xi_q(f_1) \cdots \xi_q(f_{2n+1})\bigr) &= 0\;.
	\end{equs}
\end{corollary}
Let $\Delta_q \colon \mfA_q(\mfH) \to \mfA_q(\mfH)$ be the linear (but not multiplicative) map defined by
\begin{equ}[e:Deltaqdef]
	\Delta_q(\xi_q^{\diamond n} (F) ) \eqdef q^n \xi_q^{\diamond n} (F) \; .
\end{equ}
Let $f_1, \dots, f_n \in \mfH$ and $\bpi \in \CP_{[n]}$. Setting $I \eqdef [n] \setminus \bpi$, we define
\begin{equ}[eq:RltvContr]
	\Delta^{\bpi}_q \Big( \xi^{I}_{q}(f_{\otimes I}) \Big) \eqdef  \sum_{\bsigma \in \CP_{I}} q^{\crb(\bpi , \bsigma)} \xi_q^{\diamond \bsigma}  \left( f_{\otimes I}\right)
\end{equ}
where $\crb(\bpi , \bsigma) \eqdef \crb(\bpi \cup \bsigma) - \crb(\bsigma)$.

\begin{remark}
	The map $\Delta_q$ is usually denoted by $\Gamma_q$ or $\Gamma(q)$ in the literature\footnote{See for instance \cite[Def.~4.2]{DS18}} as it is the second quantisation of the map $q \colon \mfH \to \mfH$ which just acts by multiplication with $q$. However, since we will be making use of more complicated versions of $\Delta_q$, such as $\Delta^{\bpi}_q$ and $\Delta_q^{R; \boldsymbol{k}, \bpi}$ below, and we reserve $\Gamma$ for the structure groups of regularity structures, we introduced this separate notation.
\end{remark}

% We can write $[n] \setminus \bigcup\bpi = \bigcup_{i = 1}^\ell I_i$ where each set $I_i$ is a ``continuous'' sequence of integers, i.e.\ $i, i+2 \in I_i$ implies that $i+1 \in I_i$. We define
% \begin{equ}
% 	\Delta^{\bpi}_q \Big( \prod_{i \in [n]} \xi_i \Big) \eqdef q^{\mathrm{cr}(\bpi)} \prod_{i = 1}^\ell \Delta_q^{r(I_i, \bpi)} \Big( \prod_{j \in I_i} \xi_j \Big)
% \end{equ}
% where $\mathrm{cr}(\bpi)$ is the crossing number of the set of pairs $\bpi$ and $r(I_i, \bpi)$ is the number of $\pi_j = (s,t) \in \bpi$, such that $s < \min I_i < \max I_i < t$.

We can now state the following Wick product theorem.
\begin{proposition}\label{prop:wickprop}
	With the above notation
	\begin{equ}
		\Delta^{\bpi}_q\Big( \xi^{I}_{q}(f_{\otimes I}) \Big) \prod_{(\kappa s,t) \in \bpi} \scal{f_s,f_t}_\mfH  = \sum_{\substack{\bsigma \in \CP_{I} \\ \bsigma \supset \bpi }}  \xi_q^{\diamond \bsigma} \left( f_{\otimes [n]} \right) \;  .
	\end{equ}
\end{proposition}

\begin{remark}
	This means that $\Delta^{\bpi}_q(  \xi_q(f_1) \cdots \xi_q(f_n)) $ contains all terms of the Wick expansion of $\xi_q(f_1) \cdots \xi_q(f_n)$ where the terms with $i \in \bigcup \bpi$ have been contracted.
\end{remark}

\begin{proof}
	The assertion follows by direct inspection of formulae \eqref{eq:AbsContr} and \eqref{eq:RltvContr}.
\end{proof}

% \begin{proof}
% 	Let $\bsigma \supset \bpi$. If $\sigma \in \bsigma$ is contained in one of the continuous intervals $I_i \supset \sigma$, then contraction of $\sigma$ removes two of the fields in $I_i$. This reduces the separation number of $\bpi$ by 2 but does not introduce any new crossings. In the same way $\Delta_q^{r(I_i, \bpi)}$ multiplies the elements in the Wick expansion which have $\sigma$ contracted with two instances of $q$ less than the ones that do not have $\sigma$ contracted.

% 	Now we consider the case that $\sigma \in \bsigma$ connects two continuous intervals $I_i$ and $I_j$, where we assume w.l.o.g.\ that $\max I_i < \min I_j$. We note that $\sigma$ crosses a pair $\pi \in \bpi$ if and only if $\pi$ contributes to either $r(I_i, \bpi)$ or $r(I_j, \bpi)$. This is the case because otherwise, $\max I_i < \min \pi < \max \pi < \min I_j$, in which case $\sigma$ envelops $\pi$, or $\max \pi < \min I_i$ or $\max I_j < \min \pi$ in which cases $\sigma$ and $\pi$ are disjoint.

% 	This means that although $\sigma$ formally reduces the separation number of $\pi$ by 1, it contributes $1$ to the overall crossing number, thus not changing the effect $\Delta^{r(I_i, \bpi)}_q$ and $\Delta^{r(I_j, \bpi)}_q$.
% \end{proof}

We finish this section with a proposition that allows us to interpret the multiplication of Wick-ordered expressions as summing over contractions between the separate factors. 
Given two totally ordered sets $I$ and $J$, we write $I\sqcup J$ for the disjoint union, endowed
with the total order such that $i < j$ for all $i \in I$ and $j \in J$. Given
$\bpi \in \CP_I$ and $\bsigma\in\CP_J$, we also write $\bpi \sqcup \bsigma$ for the 
corresponding element of $\CP_{I\sqcup J}$. We then have

\begin{lemma}\label{lem:doubleWick}
In the above setting, let $\CP(\bpi,\bsigma) \subset \CP_{I\sqcup J}$ denote the 
set of all pairings $\bar \bpi$ such that $\bpi \sqcup \bsigma \subset \bar\bpi$ and such
that, for all $(s,t) \in \bar\bpi \setminus (\bpi \sqcup \bsigma)$, one has
$s \in I$ and $t \in J$. Then, for all $F \in \mfH^{\otimes I}$ and $G \in \mfH^{\otimes J}$
one has the identity
	\begin{equs} \label{eq:WickProdDef}
		\xi^{\diamond \bpi}_q \bigl( F \bigr)
		\xi^{\diamond \bsigma}_q \bigl( G \bigr) = \sum_{\bar\bpi \in \CP(\bpi,\bsigma)} \xi_q^{\diamond \bar\bpi} \bigl( F \otimes G \bigr)
	\end{equs}
\end{lemma}

\begin{proof}
Given $\bar\bpi \in \CP(\bpi,\bsigma)$, consider $\bar \bpi \setminus (\bpi \sqcup \bsigma)$
as a pairing of $(I\setminus \bpi)\sqcup (J \setminus \bsigma)$. With this identification, it
is straightforward from Remark~\ref{rem:crb} to see that one has the identity
\begin{equ}
\crb(\bar\bpi) = \crb(\bpi) + \crb(\bsigma) + \crb \bigl(\bar \bpi \setminus (\bpi \sqcup \bsigma)\bigr)\;.
\end{equ}
Combining this with \eqref{eq:AbsContr}, we can reduce ourselves to the case 
where $\bpi$ and $\bsigma$ are empty, so we only need to show that
\begin{equ}
		\xi^{\diamond I}_q \bigl( F \bigr)
		\xi^{\diamond J}_q \bigl( G \bigr) = \sum_{\bpi \in \CP(I,J)} \xi_q^{\diamond \bpi} \bigl( F \otimes G \bigr)\;,
\end{equ}
where $\CP(I,J) \subset \CP_{I\sqcup J}$ denotes the 
set of all pairings $\bpi$ such that, for all $(s,t) \in \bpi$, one has
$s \in I$ and $t \in J$.

We now assume loss of generality that $J = [\ell]$ for some $\ell \in \N$ and we proceed by induction over $\ell$. For $\ell = 0$, the statement is trivially true and 
	for $\ell = 1$ this is the definition of Wick ordering, \eqref{eq:WickDef2}. 
	For the inductive step, consider $G = \bar G \otimes g_\ell = \bigotimes_{i=1}^\ell g_i$ with $g_i \in \mfH$.
	Then
	\begin{equs}
		\xi_{q}^{\diamond I}(F) \xi_q^{\diamond \ell}(G) &= \xi_{q}^{\diamond I}(F) \xi_q^{\diamond (\ell-1)}(\bar G) \xi_q(g_\ell) -\\
		& \qquad - \sum_{i \in [\ell-1]} q^{\ell-1-i} \scal{g_i,g_\ell}   \xi_{q}^{\diamond I}(F) \xi_q^{\diamond(\ell-2) }(g_{[\ell]\setminus \{i,\ell\}}) =\\
		&= \sum_{\bpi \in \CP(I,[\ell-1])} \xi_q^{\diamond \bpi} (F \otimes \bar G) \xi_q(g)  - \sum_{i \in [\ell-1]}   \xi_{q}^{\diamond I}(F) \xi_q^{\diamond \pi_i}(G) \;,
	\end{equs}
where we write $\pi_i$ for the pairing of $[\ell]$ given by $\pi_i = \{(i,\ell)\}$.
	Here we used the definition of Wick ordering in the first line, applied the induction hypothesis in the second equality and rewrote the contractions produced by the Wick ordering in terms of $\xi^{\diamond  \bpi}_q$ notation. 
	
Combining the definition \eqref{eq:AbsContr} of $\xi_q^{\diamond \bpi}$ with 
the recursion \eqref{eq:WickDef2} for the Wick product, we observe that one has the identity
\begin{equ}
	\xi_q^{\diamond \bpi} (F \otimes \bar G) \xi_q(g)
	= \xi_q^{\diamond \bpi} (F \otimes G)
	+ \sum_{i \in (I \sqcup [\ell-1])\setminus \bpi} \xi_q^{\diamond (\bpi \cup (i,\ell))} (F \otimes G)\;,
\end{equ}
where we view $\bpi$ as a pairing of $J \sqcup [\ell]$ in the right-hand side.
Furthermore, using the induction hypothesis for the last term, we have the identity
\begin{equ}
\xi_{q}^{\diamond I}(F) \xi_q^{\diamond \pi_i}(G)
= \sum_{\bpi \in \CP(\emptyset_I,\pi_i)} \xi_q^{\diamond \bpi} (F \otimes  G)\;.
\end{equ}
It follows that
\begin{equs}
\xi_{q}^{\diamond I}(F) \xi_q^{\diamond \ell}(G) &=
\sum_{\bpi \in \CP(I,[\ell-1])}  \xi_q^{\diamond \bpi} (F \otimes G) + \\
&\quad	+ \sum_{\bpi \in \CP(I,[\ell-1])}  \sum_{i \in (I \sqcup [\ell-1])\setminus \bpi} \xi_q^{\diamond (\bpi \cup (i,\ell))} (F \otimes G)  -\\
&\quad - \sum_{i \in [\ell-1]}  \sum_{\bsigma \in \CP(\emptyset_I,\pi_i)} \xi_q^{\diamond \bsigma} (F \otimes  G)\;.
\end{equs}
It remains to note that the elements $\bsigma \in \CP(\emptyset_I,\pi_i)$ are precisely those 
such that there exists  $\bpi \in \CP(I,[\ell-1])$ with $\bsigma = \bpi \cup (i,\ell)$
and that $\CP(I,[\ell])$ consists of those $\bsigma$ for which there exist 
$\bpi \in \CP(I,[\ell-1])$ such that either
$\bsigma = \bpi$ or $\bsigma = \bpi \cup (i,\ell)$
for some $i \in I \cap ((I \sqcup [\ell-1])\setminus \bpi)$. 
\end{proof}

To state the announced result regarding products of Wick products with more than to factors, it will be convenient to have a notion of ``$m$-partitioned set'', that is, a finite totally ordered set $I$ 
that comes equipped with a totally ordered ``partition'' $\{I_1,\dots,I_m\}$ such that we can write $I = I_1 \sqcup I_2 \sqcup\ldots \sqcup I_m$ as 
a disjoint union of $m$ consecutive\footnote{Consecutive means that the ordering of blocks respects the ordering of $I$, that is if $s \in I_j$ and $t \in I_{j+1}$ then $s < t$ in $I$.} intervals\footnote{Here the term interval means that if $s,u,t \in I$ and $s,t \in I_j$ for a block $I_j$ of $I$, then $u \in I_j$.} that we call blocks of $I$. 
We wrote the word partition in quotes since we allow some of the $I_j$'s to be empty, so that, for example, $\emptyset \sqcup [n]$ and $[n]\sqcup \emptyset$ are distinct\footnote{Since the orderings of the partitions are different.} $2$-partitioned sets of size $n$. 
We use the notation $I_j \lhd I$ to indicate that $I_j$ is a block of $I$. 
Note that, since blocks come with an order, it makes sense to use the notation $\prod_{J \lhd I}$ to indicated an ordered product. 

 The concatenation of an $m$-partitioned set and an $o$-partitioned set naturally yields an $(m+o)$-partitioned sets, and, given $J \subset I$, $I\setminus J$ is again an $m$-partitioned set.
 (It is important here to allow empty blocks.)

For $m \in \N$ we denote by $\mfI_m$ to the collection of all $m$-partitioned sets (realised concretely as subsets of $\N$)\footnote{The choice of the particular infinite set $\N$ is not so important here as fundamentally we used $m$-partitioned sets as labellings in operations of a finite nature, we choose to make these sets concrete just to avoid any reference to a set of all sets.}. 
For a given $m$-partitioned set $I$, we write $\widehat\CP_{I}\subset \CP_{I}$ for the set of
all contractions $\bpi$ such that, for every $(s,t) \in \bpi$, $s$ and $t$ belong to different
elements of the given partition of $I$.
For an $m$-partitioned set $I$ we write $\|I\| = m$ for the number of elements in the partition of $I$.
When we do not want to specify $\|I\|$ we will simply say that $I$ is a partitioned set and we write $\mfI = \sqcup_{m \in \N} \mfI_m$ for the collection of all partitioned sets.
Note that we still write $|I|$ for the cardinality of the underlying set $I$. 

Partitioned sets are useful labels for working with products of Wick products in a non-commmutative setting. 
We state the following proposition on Wick products. 
\begin{proposition}
\label{prop:WickProd}
Let $I$ be a partitioned set and fix $(f_{i})_{ i \in I} \in \mfH^{I}$, then 
	\begin{equ} \label{eq:WickProd}
		\prod_{J \lhd I} \xi^{\diamond J}_q \left( f_{\otimes J} \right) = \sum_{\bpi \in \widehat{\CP}_{I}} \xi_q^{\diamond \bpi} \left(  f_{\otimes I} \right)\;.
	\end{equ}
\end{proposition}

\begin{proof}
We proceed by induction on $\|I\|$. For $\|I\| \in \{0,1\}$, $\widehat{\CP}_{I}$ only contains the empty contraction, so the statement is true. 
We now take $ \|I\| \geqslant 2$, and assume the desired statement holds for all partitioned sets with strictly fewer blocks.
Let $\bar{I}$ be the partitioned set obtained by removing the last-most block from $I$, and denote this block by $\bar{J}$. 
Using the induction hypothesis and Lemma~\ref{lem:doubleWick}, we then have
\begin{equs}
	\prod_{J \lhd I } \xi^{\diamond J}_q \left( f_{\otimes J} \right) 
	&= \sum_{\bar\bpi \in \widehat{\CP}_{\bar{I}}} \xi_q^{\diamond \bar\bpi} \left( F_{\otimes \bar{I}} \right)\xi^{\diamond \bar{J}}_q \left( f_{\otimes \bar{J}}\right) = \\
	&=  \sum_{\bar\bpi \in \widehat{\CP}_{\bar{I}}} \sum_{\bsigma \in \CP(\bar\bpi,\emptyset_{\bar{J}})} \xi_q^{\diamond \bsigma} \left( f_{\otimes I} \right)\;.
\end{equs}
We conclude the proof by observing that, for every $\bpi \in \widehat{\CP}_{I}$ there exists a unique $\bar \bpi \in \CP_{\bar{I}}$ such that $\bpi \in \CP(\bar\bpi,\emptyset_{\bar{J}})$ and that furthermore $\CP(\bar\bpi,\emptyset_{\bar{J}}) \subset \widehat{\CP}_{I}$.
\end{proof}

\subsection{Bosons}
\label{sec:Bosons}
We begin our discussion of the bosonic case, $q = 1$, by noting that $\mfA_1(\mfH)$ is a commutative $*$-algebra.

We write $[A,B]_{-} = AB - BA$ for the standard commutator.
To see that $\mfA_1(\mfH)$ is commutative we observe that, for all $f,g \in \mfH$, the image of $[\alpha^\dagger_1(f), \alpha^\dagger_1(g)]_-$ is in the kernel of $P_1$.
Thus, as operators on $\mathring{\CF}_1(\mfH)$,
\begin{equ}[eq:BosCom1]
	[\alpha_1(f), \alpha_1(g)]_- = [\alpha^\dagger_1(f), \alpha^\dagger_1(g)]_-  = 0  \; .
\end{equ}
Therefore,
\begin{equs}[eq:BosCom2]
	{}[\xi_q(f), \xi_q(g)]_- &= \left[ \alpha_1^\dagger(f) ,  \alpha_1(\kappa g) \right]_- + \left[ \alpha_1 (\kappa f) , \alpha^\dagger(g) \right]_-  = \\
	& = - \Braket{\kappa g,f}_{\mfH} + \Braket{\kappa f,g}_{\mfH} = \\
	&=- \Braket{\kappa f,g}_{\mfH} + \Braket{\kappa f,g}_{\mfH}=  0 \; .
\end{equs}
In fact, one can prove that the operators $ \big( \xi_q(f): f \in \mfH \big)$ are a family of essentially self-adjoint operators on the domain $\mathring{\CF}_1 \subset \CF_1$, and their closures strongly commute.
Thus, they may be realised as an algebra of functions on a suitable probability space.
In particular, Corollary~\ref{cor:q-Wick} implies that, under the state $\omega_1$,
 $\left( \xi_1(f) \right)_{f \in \mfH}$ are a family of jointly Gaussian random variables indexed by $\mfH$ with the
 covariance of $\xi_1(f)$  and $\xi_1(g)$ given by $\Braket{f , g }_{\mfH}$.

When $\mfH = L^2(\R^d)$, then $\xi_1 = \xi_B$ where $\xi_B$ is the standard $d$-dimensional white noise.
In the context of classical stochastic analysis, the annihilation operator $\alpha_1(f)$ is often called the Malliavan derivative in the direction $f$.% and $\alpha^\dagger_1(f)$ is called the Skorokhod integral.\martin{Sort of. The way they are related is that $\alpha^\dagger_1(f) X = \delta(f X)$ for any random variable $X$.}

Unfortunately, as the random variables $\xi_1(f)$  are unbounded operators, they do not form a Banach algebra.
The natural topological algebra structure that can be put on a space of unbounded functions is pointwise convergence, which equips it with a locally $m$-convex topology.  However, since we are working with equivalence classes of measurable functions modulo functions that are $0$ almost everywhere, we have to work with pointwise convergence almost everywhere.
 In particular, we shall denote by $\CA_B(\mfH)$
 the algebra of measurable functions generated by $\left(\xi_1(f)\right)_{f \in \mfH}$ completed with respect to convergence almost everywhere.
  We will identify $\CA_B(\mfH)$ with $L^0(\Omega, \mu)$ the space of (equivalence classes of) measurable functions over a suitable probability space $(\Omega, \mu)$. 

Defining convergence on $\CA_B(\mfH)$ using convergence almost everywhere only equips it with a ``convergence structure'' rather than a classical topology.
However this point of view is useful since it allows one to almost view $\CA_B(\mfH)$ as a locally $m$-convex algebra. 
Speaking formally rather than rigorously, each point $\mcb{p} \in \Omega$ in the probability space induces the submultiplicative seminorm
\begin{equ}
	\| \phi \|_{\mcb{p}} = |\phi(\mcb{p})|
\end{equ}
for $\phi \in \CA_B(\mfH)$, and a given sequence of elements of $\CA_{B}(\mfH)$ are said to converge in  $\CA_{B}(\mfH)$ if they converge in $\| \bigcdot \|_{\mcb{p}}$ for almost every $\mcb{p} \in \Omega$. 

For polynomials of our underlying bosonic Gaussian it will be convenient to work in the stronger space  $\cA_B(\mfH) \eqdef  L^{\infty -}(\Omega , \mu) = \bigcap_{q =1}^{\infty} L^{q}(\Omega,\mu) \subset \CA_B(\mfH)$, in particular we say sequences in $\cA_B(\mfH) $ converge if they converge in $\| \bigcdot \|_{L^q}$ for all $q \in [1, \infty)$.
This space will be important when we work with (Gaussian) models in regularity structures.

An importact fact we will use, which follows from a combination of the It\^{o} isometry and Gaussian hypercontractivity, is that for any $n \in \N$ and $p \in [1, \infty)$, there exists $C_{n,p} < \infty$ such that, for all $F \in \mfH^{\wotimes_\alpha n}$, 
\begin{equ}[eq:BosHyper]
	\| \xi^{\diamond n}_1(F) \|_{L^p} \leqslant C_{n,p} \| F \|_{\mfH^{\wotimes_\alpha n} } \; .
\end{equ}
In particular, $\xi_1^{\diamond n}$ extends to a continuous map $\mfH^{\wotimes_\alpha n} \to \cA_B(\mfH)$. 

\begin{remark}
% However, we can combine these two by retreating to the topology of convergence in measure. While this turns $\CA_B(\mfH)$ into a topological algebra, it is neither $m$-convex nor even locally convex. When stating a continuity result involving $\CA_B(\mfH)$, it will always be at least with respect to the topology of convergence in measure. \martinp{Push to BPHZ Section, replace with why you want to be ambiguous }
The algebra $\CA_B(\mfH)$ ``almost'' fits into our framework of noncommutative regularity structures based on locally $m$-convex algebras, but it falls short of fitting in completely since we would need to use convergence almost everywhere instead of pointwise convergence. 
Nonetheless, from a conceptual point of view, imagining $\CA_B(\mfH)$ as an $m$-convex algebra will provide the reader with the useful intuition for some of our constructions. 

However, whenever making a mathematically rigorous and precise statement, we will equip $\CA_B(\mfH)$ with the topology of convergence in measure, which is neither $m$-convex nor even locally convex. 
We allow ourselves a level of ambiguity as to which mode of convergence (convergence almost everywhere vs convergence in measure) in some of our discussions, but in explicit statements such as theorems and propositions we refer to convergence in measure. See also the discussion in Remarks~\ref{rem:RanRS}~\&~\ref{rem:MixConv}.

\end{remark}

Whenever we have to define \dash or renormalise \dash a singular product appearing in a (noncommutative) SPDE, we do this by using the formulae appearing in Section~\ref{sec:WickRenorm} to rewrite it as a sum of Wick-ordered terms and suitably subtracting diverging ones. However, this always requires that we be able to extend $\xi^{\diamond n}_q$ to the Hilbert space tensor product $\mfH^{\wotimes_\alpha n}$.
%This has been extensively studied in the bosonic case, essentially using \eqref{eq:BosHyper}, cf.\ \cite{} \ajay{missing reference}. 
In Sections~\ref{sec:Fermions}~ and~\ref{sec:qmezdons} we will explain how we can find analogous estimates for $q \in [-1,1)$.

\subsubsection{Mixed Systems}
\label{sec:MixSys}

In order to describe mixed systems of bosons and other particles, the latter collectively described by an algebra $\CA$, we will need to topologise and complete $\CA_B(\mfH) \otimes \CA$ and $\cA_B(\mfH) \otimes \CA$. Whenever $\CA$ is metrisable, which we will assume in our examples, we can identify these tensor products with subspaces of $L^0(\Omega , \mu ; \CA)$ \dash the space of (equivalence classes of) measurable functions $\Omega \to \CA$, where $\CA$ is equipped with the Borel $\sigma$-algebra induced by its metric. 
We therefore define 
\begin{equ}
	\cA_B(\mfH) \wotimes \CA \eqdef L^{\infty-}(\Omega , \mu ; \CA)
\end{equ}
and
\begin{equ}
	\CA_B(\mfH) \wotimes \CA \eqdef L^0(\Omega , \mu ; \CA)
\end{equ}
where the space $\CA_B(\mfH) \wotimes \CA$ should be thought of as equipped with convergence almost everywhere or convergence in measure (and in formal statements, always the latter).
The spaces above are called random algebras (where the randomness here is referring to the randomness in the bosonic component).

\subsection{Fermions}
\label{sec:Fermions}

In this section, we will discuss the fermionic case, $q = -1$. As with the case $q=1$, $P_{-1}$ has a non-trivial kernel, which leads to the following additional anticommutation relations
\begin{equ}[e:anticomm]
	\left[ \alpha_{-1}(f), \alpha_{-1}(g) \right]_+ = \bigl[ \alpha^\dagger_{-1}(f), \alpha^\dagger_{-1}(g) \bigr]_+ = 0 \;,
\end{equ}
where the anticommutator is defined as $\left[ A, B \right]_+ \eqdef AB + BA$.
Furthermore, one can show \cite[Ch.~5.2]{BR87} that the relation $\bigl[ \alpha_{-1}(f), \alpha^\dagger_{-1}(g) \bigr]_+ = \Braket{f,g} \1$ implies the operator bound
\begin{equ}
\label{eq:CAROpBnd}
	\bigl\| \alpha^\dagger_{-1}(f) \bigr\| = \left\| \alpha_{-1}(f) \right\| = \|f\|_{\mfH} \; .
\end{equ}
We denote by $\CA_F(\mfH)$ the $C^*$-algebra generated by $\alpha_{-1}^\dagger(f)$ for $f \in \mfH$. This $C^*$-algebra naturally acts on $\CF_{-1}(\mfH)$.

It is natural to split the fermionic setting into two further sub-cases that are of interest,
the self-adjoint case (``Clifford'' fermions) and a non self-adjoint case (``Dirac'' fermions).

The self-adjoint case corresponds to the $-1$-$\mfH$-noise $\xi_F \eqdef \xi_{-1}$ introduced in Definition~\ref{def:qNoise}.
We will call $\xi_F$ the Clifford noise, as a simple calculation shows that it generates the Clifford algebra of $\mfH$, which we will denote by $\CA_F^{\Cl}(\mfH)$. This is also the closure of $\mfA_{-1}(\mfH)$ in $\CA_F(\mfH)$. Furthermore, since we assume that $\mfH$ is separable, $\CA_F^{\Cl}(\mfH)$ is also isomorphic to the hyperfinite $\text{II}_1$ factor, cf.\ \cite{Tak04}. % \martinp{Add reference to the hyperfinite $\text{II}_1$ }%, cf. \cite{?} for further information on Clifford algebras. \ajay{Missing reference.}

\subsubsection{Dirac Fermions}
\label{sec:DirFerm}
The Dirac case requires us to have additional structure and assumptions.
In particular, we assume that $\mfH$ is a complex Hilbert space equipped with a
conjugation (namely an antilinear involution) $\kappa$ and that we have chosen a $\kappa$-antisymmetric antiunitary map $U \colon \mfH \to \mfH$, namely $U$ satisfies
\begin{equ}
	\kappa U^\dagger \kappa = -U \; .
\end{equ}

\begin{definition}[Dirac Noise Operator]\label{def:DiracNoise}
	The Dirac noise with correlation function $U$ is the operator-valued distribution $\Psi \colon \mfH \to \CA_F(\mfH)$ given by
	\begin{equ}
		\Psi(f) \eqdef \alpha_{-1}^\dagger(f) + \alpha_{-1}( \kappa U f) \; .
	\end{equ}
%	for some $\mu > 0$.
\end{definition}
Due to $\kappa$-antisymmetry one readily checks that for all $f, g \in \mfH$
\begin{equ}
	\left[ \Psi(f), \Psi(g) \right]_+ = 0 \; .
\end{equ}

Motivated by models from physics, we will chiefly be interested in Hilbert spaces that can be written 
as $\mfH = \mfh \otimes \C^2 \cong \mfh \oplus \mfh$ for another Hilbert space $\mfh$.  In this case we view $\Psi$ as an element of $\CB(\mfH ; \CA_F(\mfH)) \cong \CB(\mfh ; \CA_F(\mfH))^2$ with components $\Psi = (\psi, \bar\psi)$.

To construct a correlation matrix $U$, we can start with a conjugation $\bar\kappa \colon \mfh \to \mfh$ and an arbitrary unitary $V \colon \mfh \to \mfh$, and set
\begin{equ}
	U \eqdef \begin{pmatrix}
		0 & V \\
		-\bar\kappa V^\dagger \bar\kappa & 0
	\end{pmatrix} \; .
\end{equ}
The linear map $U$ is then unitary and $\kappa$-antisymmetric for $\kappa = \bar\kappa \oplus \bar\kappa$. 
Furthermore, we fix $\lambda \in \left(0, \frac{1}{2}\right)$ and the corresponding quasi-free state $\omega_\lambda$ which is faithful on $\CA_F(\mfH)$ as defined in \cite[Section~2.5]{CHP23}, cf.\ also \cite{Gub23}.

\subsubsection{The Extended CAR Algebra}\label{sec:extendedcar}
In both the Clifford and Dirac settings, the operators $\left( \Psi(f) \right)_{f \in \mfH}$ generate an anticommutative or Grassmann algebra inside of $\CA_F(\mfH)$, the closure of which we shall denote by $\CG_F(\mfH)$. Analogously to \eqref{eq:WickDef1}, Wick powers of $\Psi(f)$ are defined inductively by setting $\Psi^{\diamond 0}(c) \eqdef c \bone$, $\Psi^{\diamond 1}(f) \eqdef \Psi(f)$, and
\begin{equs}
	\label{eq:WickDefDirac}
	\Psi^{\diamond n}(f_1 \otimes \cdots &\otimes f_n) \eqdef \Psi(f_1) \Psi^{\diamond (n-1)}(f_2 \otimes \cdots \otimes f_n) - \\
	& \qquad -\sum_{i = 2}^n (-1)^{i} \omega_{F} \bigl(\Psi(f_1) \Psi(f_i)\bigr) \Psi^{\diamond(n-2)} \bigl(f_2 \otimes \cdots \otimes \widehat{f_i} \otimes \cdots \otimes f_n\bigr)    \; .
\end{equs}
Unfortunately, the boundedness property \eqref{eq:CAROpBnd} does not translate to higher Wick powers $\xi^{\diamond n}_{F}(G)$ and $\Psi^{\diamond n}(G)$  if one wishes to insert an arbitrary $G \in \mfH^{\wotimes_{\alpha} n}$, as will be necessary in order to solve fermionic SPDEs.

In \cite{CHP23}, the authors established an extension and localisation procedure of the algebra $\mfA_{F}(\mfH)$ that allows one to view the unbounded operators appearing as locally or ``pathwise'' bounded, with respect to a noncommutative notion of points.
We shall summarise the procedure and point to \cite{CHP23} for further details.

Let $\mfA_F(\mfH)$ be the free $*$-algebra generated by $\mfH$, with the generators denoted by $\balphas(f)$ for $f \in \mfH$ and $\balpha(f)$ for its adjoint.

Given a finite-dimensional subspace $b \subset \mfH$, we define the $*$-representation $\pi_b \colon \mfA_F(\mfH) \to \CA_F(b) \subset \CA_F(\mfH)$ by setting
\begin{equ}
	\pi_b( \balphas(f) ) \eqdef \alpha_{-1}^\dagger(P_b f)
\end{equ}
where $P_b$ is the orthogonal projection $\mfH \to b$. Let $\Gr(\mfH)$ (respectively $\Gr^U(\mfH)$) denote the set of all finite-dimensional subspaces of $\mfH$ such that $\kappa b \subset b$  (respectively all $b \in \Gr(\mfH)$ such that $\kappa U  b  \subset  b$). Here $\kappa$ is again either the identity or antilinear conjugation, depending on whether we are working with a real or complex vector space $\mfH$. 
We assume that we are given a sequence $(\Gamma^{(U)}_{n})_{n=0}^{\infty}$, with $\Gamma^{(U)}_{n} \subset \Gr^{(U)}(\mfH)$, such that the following conditions hold
\begin{enumerate}
	\item For any $n \geqslant m$, one has $\Gamma^{(U)}_{n} \supset \Gamma_{m}$. 
	\item For any $n \in \N$, $\sup_{b \in \Gamma^{(U)}_{n}} \dim(b) = n$. 
	\item For all $n \in \N$, $\sum_{b \in \Gamma^{(U)}_{n} \setminus \Gamma^{(U)}_{n-1}} b$  is dense in $\mfH$, with $\Gamma^{(U)}_{-1} = \emptyset$.
\end{enumerate}
We also set $\Gamma^{(U)}_{\infty} \eqdef \bigcup_{n \in \N} \Gamma^{(U)}_n$.

For $n \in \N \cup \{\infty\}$ we define the following $*$-algebra seminorms on $\mfA(\mfH)$
\begin{equ}\label{eq:ExtFermAlgUnbounded}
	\| a \|^{(U)}_n \eqdef \sup_{b \in \Gamma_n^{(U)}} \| \pi_b(a) \| \; .
\end{equ}
We define
\begin{equ}
	\cA^{(U)}_F(\mfH) \eqdef \overline{\mfA_F(\mfH) / \bigcap_{b \in \Gr^{(U)}(\mfH)} \ker \pi_b }^{\left(\|\bigcdot\|_n^{(U)}\right)_n}
\end{equ}
and
\begin{equ}
	\cA^{(U)}_\infty(\mfH) \eqdef \overline{\mfA_F(\mfH) / \bigcap_{b \in \Gr^{(U)}(\mfH)} \ker \pi_b }^{\|\bigcdot\|^{(U)}_\infty} \; .
\end{equ}
$\cA^{(U)}_F(\mfH)$ is a locally $C^*$-algebra and $\cA^{(U)}_\infty(\mfH)$ is a $C^*$-algebra extending $\CA_F(\mfH)$ in the sense that there exist unique surjective $C^*$-algebra morphism $\digamma^{(U)} \colon \cA^{(U)}_\infty(\mfH) \to \CA_F(\mfH)$. 
We now turn to defining our noise process in this setting. 

\begin{definition}
	We define extensions of our two fermionic noises by setting
	\begin{equ}
		\bxi_F(f) \eqdef \balpha^\dagger(f) + \balpha(\kappa f) \; , \qquad
		\bPsi(f) \eqdef \balpha^\dagger(f) + \balpha( \kappa U f ) \; .
	\end{equ}
	We define $\cA_F^{\Cl}(\mfH)$ and $\cG_F(\mfH)$ to be the closed algebras generated by the centre of $\cA_F(\mfH)$, respectively $\cA_F^{(U)}(\mfH)$, as well as $\bxi_F$ and $\bPsi$ respectively. 
\end{definition}
\begin{remark}
	The distinction between $\Gr^{U}(\mfH)$ and $\Gr(\mfH)$ was made to ensure that the operators $\bPsi(f)$ anticommute in $\cA^U_F(\mfH)$. We will drop the $U$ from the notation when it is clear from context which algebra we mean. 
\end{remark}
\begin{remark}
	Analogously to the CAR algebra case, we will be splitting $\bPsi$ into two components $(\bpsi, \bar\bpsi)$.
\end{remark}
Higher Wick powers of $\bxi_F$ and $\bPsi$ are then defined analogously to \eqref{eq:WickDef1} and \eqref{eq:WickDefDirac} by replacing instances of $\scal{f,g}_\mfH$ and $\omega_{F}(\Psi(f)\Psi(g))$ with
\begin{equ}
	\bigl[ \balpha(f), \balphas(g)  \bigr]_+\qquad\text{and}\qquad
	\bigl[ \balpha(\kappa U f), \balphas(g)  \bigr]_+
\end{equ}
respectively, both of which are no longer proportional to $\bone$ in $\cA^{(U)}_F(\mfH)$. For these extended Wick polynomials, we have the following hypercontractivity-type estimate.
\begin{proposition}
	\label{proposition:Wick_Product}
	For any $n \in \N\setminus\{0\}$ there exists a constant $C_n < \infty$, such that, for all $G \in \mfH^{\wotimes_\alpha n}$ and all $k \in \N$,
	\begin{equs}
		{}
		\bigl\| \bxi_F^{\diamond n}(G)\bigr\|_k  &\leqslant C_n (1+k)^{\frac{n-1}{2}} \|G\|_{\mathfrak {H}^{\wotimes_\alpha n}} ,  \\
		\bigl\| \bPsi^{\diamond n}(G)\bigr\|_k  &\leqslant C_n (1+k)^{\frac{n-1}{2}} \|G\|_{\mathfrak {H}^{\wotimes_\alpha n}} \; .
	\end{equs}
	In particular, $\bPsi^{\diamond n}$ extends to a continuous map $\mfH^{\wedge n} \to \cA_F(\mfH)$. Furthermore, one also has the inequality
	\begin{equ}
		\|G\|_{ \mfH^{\wedge n}} \leqslant \bigl\| \bxi_F^{\diamond n}(G)\bigr\|_\infty \; , \qquad
		\|G\|_{ \mfH^{\wedge n}} \leqslant \bigl\| \bPsi^{\diamond n}(G)\bigr\|_\infty \; .
	\end{equ}
\end{proposition}
\begin{proof}
For the statements regarding Wick powers of $\bPsi$, this is precisely \cite[Prop.~2.23]{CHP23}. 
The proof of the estimates involving $\bxi_F$ then follows since it corresponds to the 
special case $U=\bone$. 
\end{proof}

\begin{remark}
In this paper, when working with fermions, we will always pose our equations in spaces where the fermionic components belong to the extended algebra \eqref{eq:ExtFermAlgUnbounded} rather than the classical CAR algebra $\CA_{F}(\mfH)$. 
Again, this is because we cannot close singular equations in $\CA_{F}(\mfH)$ due to the lack of estimates for higher Wick powers.
%For that reason, we overload notation and simply write $\xi_F$ and $\Psi_F$ instead of $\bxi_F$ and $\bPsi_F$.
\end{remark}

\section{\TitleEquation{q}{q}-Mezdons and the Algebra \TitleEquation{\cA_{q}}{}% for \TitleEquation{|q| < 1}{|q|<1}
}
\label{sec:qmezdons}

In this section, we describe the pertinent analytic aspects of $q$-mezdon algebras. These $q$-mezdon algebras enjoy a stronger version of hypercontractivity, called ``ultracontractivity''. We review these properties, and introduce the new topology that we use to control the ``intertwined'' renormalised product estimates mentioned
in the discussion around \eqref{eq:intertwined}. These are necessary to solve singular PDEs with values in this algebra. We will be working with real Hilbert spaces $\mfH$ when dealing with $q$-mezdons and drop $\kappa$ from the notation.%\footnote{In particular we do not include factors of $\sqrt{n}$ and $\frac{1}{n!}$ commonly used for Bosons and fermions. This is done so that Wick ordering produces isometries.}.
%For the proofs of the claims in this introductory section, cf. \textit{loc. cit.} [???]

The main factor that differentiates the case $|q|<1$ is that the symmetrisation operator $P_q$ suppresses permutations with a high inversion number, and the norm of $P_q$ increases only exponentially with the number of particles instead of factorially.
To be more precise, we have the following proposition.
\begin{proposition}
\label{prop:SymOpEst}
	For $q \in (-1,1)$
	\begin{equ}
		\left\| P_q \big|_{\mfH^{\otimes n}} \right\|_{\CB(\CF(\mfH))} = \prod_{k = 0}^{n-1} \frac{1-|q|^{k}}{1-|q|} \leqslant D_{q}^{n}\;,
	\end{equ}
	with $D_q \eqdef (1-|q|)^{-1} $.
\end{proposition}
\begin{proof}
	See \cite{BS91}.
\end{proof}

\begin{corollary}
\label{cor:nPartEq}
	For $q \in (-1,1)$, and all $n \in \N$,  $\CF_{0,n}(\mfH) = \CF_{q,n}(\mfH)$ as Banach spaces, i.e.\ they are isomorphic as vector spaces and their norms are equivalent. However, the constants of equivalence are not uniform in $n$.
\end{corollary}
\begin{proof}
	This follows from the proof of \cite[Proposition~1]{BS91}, wherein it is shown that for every $q \in (-1,1)$ and $n \in \N$ there exists a constant $c(q,n) > 0$, s.t.\ $P_q \big|_{\CF_{0,n}} \geqslant c(q,n)$. Therefore, $\Braket{  \, \bigcdot \, ,  \, \bigcdot \, }_{\CF_0}$ and $\Braket{  \, \bigcdot \, ,  \, \bigcdot \, }_{\CF_q}$ induce equivalent norms for each fixed $n$-particle subspace.
\end{proof}

% \begin{remark}
% We also note that, for $q \in (-1,1)$, we have
% \begin{equ}
% 	\CF_{0}(\mfH) = \bigwoplus_{n \in \N} \mfH^{\wotimes n} \subset \CF_{q}(\mfH)\;,
% \end{equ}
% see \cite{??.}
% \end{remark}

In \cite[Lemma~4]{BS91} it was shown that $\alpha_q^\dagger(f)$ and $\alpha_q(f)$ extend to $\CF_{q}(\mfH)$ as bounded operators with norm
\begin{equ}\label{eq:qbound}
	\bigl\| \alpha^\dagger_q(f) \bigr\| = \| \alpha_q(f) \| = \begin{cases}
		\displaystyle\frac{\|f\|_{\mfH}}{\sqrt{1-q}} \; , \qquad & \text{if } q \in [0,1)\\
		\|f\|_{\mfH}  \; , \qquad & \text{if } q \in (-1,0]
	\end{cases} \; .
\end{equ}

% \begin{remark}
% For the cases $q=\pm 1$,  $\alpha_{\pm 1}(f)$ and $\alpha^{\dagger}_{\pm 1}(f)$ can be defined similarly and give the bosonic and fermionic creation and annihilation operators, which satisfy \eqref{eq:qrelations}.
% In this case the bottom estimate of \eqref{eq:qbound} holds for $q= -1$ but for $q=1$ we have that $\alpha_{1}(f)$ and $\alpha^{\dagger}_{1}(f)$ are unbounded for $f \not = 0$.
% \end{remark}

\begin{definition}
% 	The mezdonic white noise is given mapping $\mfH \ni f \mapsto \phi_{q}(f)$ where $\phi_{q}(f)$ is the bounded self-adjoint operator
% 	\begin{equ}
% 		\xi_q (f) \eqdef \alpha_q(f)^\dagger + \alpha_q(f)\;.
% 	\end{equ}
	We denote by $\CA_q(\mfH)$ the closure of $\mfA_q(\mfH)$ with respect to the operator norm in $\CB(\CF_q(\mfH))$; which is a $C^*$-algebra.
	% The vacuum state $\omega_q \colon \CA_q(\mfH) \to \R $ given by
 	% \begin{equ}
 	% 	\omega_q(A) \eqdef \scal{\1, A \1}_{\CF_q} \; .
 	% \end{equ}
\end{definition}

In \cite[Theorems~4.3 and~4.4]{BS94} it was shown that the extension $\omega_q$ of the state from $\mfA_q(\mfH)$ to $\CA_q(\mfH)$ is faithful\footnote{This means $\omega_{q}(a^{\dagger}a) = 0 \Rightarrow a = 0$} and tracial\footnote{This means $\omega_{q}(ab) = \omega_{q}(ba)$}.
However, neither of these facts are fundamental for our analysis here.

\begin{remark}
Note that $\omega_q(\, \bigcdot \, )$ is also a state on $\CB\big(\CF_q(\mfH) \big) \supset \mfA_q(\mfH)$ but for $f \not = 0$, $\omega_{q}(\alpha_{q}(f)\alpha_{q}(f)^{\dagger}) > 0$ while $\omega_{q}\big(\alpha_{q}(f)^{\dagger}\alpha_{q}(f)\big) = 0$, so $\omega_{q}$ is \textit{not} tracial on $\CB \big(\CF_{q}(\mfH) \big)$.
\end{remark}

\subsection{Mezdonic Chaos Estimates and Identities}

Fundamental to our analysis of mezdons will be the following \textit{ultracontractive} estimate which is proven in \cite[Proposition~2.1(b)]{Boz99}.

\begin{proposition}\label{}
For any  $k, \ell \in \N$, the Wick block map $W_q^{k,\ell}$ extends to a continuous maps $\CF_{q,k+\ell}(\mfH) \to \CB(\CF_q(\mfH))$.
In particular, for any $F \in \CF_{q,k+\ell}(\mfH)$ one has the estimate
\begin{equ}
	\label{eq:WklStrBnd}
	\| W^{k,\ell}_q (F) \| \leqslant C^{\frac{3}{2}}_{q} \| F \|_{\CF_q} \leqslant D_{q}^{\frac{k+\ell}{2}} C_{q}^{\frac{3}{2}} \| F \|_{\CF_0} \; ,
\end{equ}
where
\begin{equ}
	C_{q} \eqdef \prod_{n = 1}^\infty (1-|q|^{n})^{-1} \; .
\end{equ}
\end{proposition}
\begin{proof}
The first inequality of \eqref{eq:WklStrBnd} comes from \cite[Proposition~2.1(b)]{Boz99}. 
The second inequality follows from Proposition~\ref{prop:SymOpEst}.
\end{proof}

Note that \eqref{eq:WklStrBnd} implies that for all $F \in \CF_{q,n}(\mfH)$, 
\begin{equ}
	\label{eq:WklStrBnd2}
	\| \xi^{\diamond n}_q (F) \| \leqslant (n+1) D_{q}^{\frac{n}{2}} C_{q}^{\frac{3}{2}} \| F \|_{\CF_0} \; ,
\end{equ}
as $\xi_q^{\diamond n}$ is made up of $n+1$ Wick blocks.

\begin{remark}
The estimate \eqref{eq:WklStrBnd2} should be compared to the standard Gaussian \textit{hypercontractive} estimates where one bounds on the left hand side an $L^{p}$ norm for $p \in [2,\infty)$.

Ultracontractive estimates play a fundamental role in our analysis of mezdonic equations, in particular they allow us work with Banach algebras and avoid any sort of localisation procedure such as the one we introduced for fermions in Section~\ref{sec:extendedcar}
\end{remark}

% \begin{remark}
% 	Each operator $W_q^{k,\ell}$ annihilates first $\ell$ particles and then creates $k$ new one, i.e.\ for all $F \in \CF_{q,k+\ell}(\mfH)$ and all $n \geqslant \ell$
% 	\begin{equ}
% 		W_q^{k,\ell}(F) \colon \CF_{q, n} \longrightarrow \CF_{q, n - \ell}  \longrightarrow \CF_{q, n - \ell + k} \; .
% 	\end{equ}
% 	In particular, each term in the unimodular decomposition, \eqref{eq:UniModDecomp}, maps a fixed $n$-particle sector into mutually orthogonal particle sectors.
% \end{remark}

% We finally introduce the mezdonic number operator $N_q$ in the usual way, i.e.\ if $F \in  \CF_{q,n}(\mfH)$ we define
% \begin{equ}
% 	N_q F \eqdef n F
% \end{equ}
% which can be extended to a self-adjoint operator on $\CF_q(\mfH)$.

We shall also need a slightly different set of bounds than \eqref{eq:WklStrBnd} and~\eqref{eq:WklStrBnd2}, we state and prove these bounds below in Proposition~\ref{prop:OpBounds}.
We begin with the following combinatorial estimate, proven in \cite[Theorem~2.1]{Boz99}.
\begin{lemma}
\label{lemma:SymEst}
	For all $q \in (-1,1)$, and all $n,k \in \N$
	\begin{equ}
		\Bigl|\sum_{\sigma \in \mfS_{n,k}} q^{|\sigma|}\Bigr| \leqslant C_{q} \; .
	\end{equ}
\end{lemma}
Furthermore, we need the following explicit identity.
\begin{lemma}
\label{lemma:AnnProdExp}
	For all $n \leqslant m$, $f_1, \dots, f_n, g_1, \dots, g_m \in \mfH$
	\begin{equs}
		\label{eq:AnnOpId}
		& \alpha_q(f_1) \cdots \alpha_q(f_n) g_1 \otimes \cdots \otimes g_m = \\
		& \qquad = \hspace{-0.3cm} \sum_{\sigma \in \mfS_{m,n}} \hspace{-0.3cm} q^{|\sigma|} \Braket{f_1 \otimes \cdots \otimes f_n, g_{\sigma(n)} \otimes \cdots \otimes g_{\sigma(1)}}_{\CF_q} g_{\sigma(n+1)} \otimes \cdots \otimes g_{\sigma(m)} \; .
	\end{equs}
	In particular,
	\begin{equ}
		\alpha_q(f_1) \cdots \alpha_q(f_n) \big|_{\CF_{0,m}(\mfH)} = \hspace{-0.3cm} \sum_{\sigma \in \mfS_{m,n}} \hspace{-0.3cm} q^{|\sigma|} \alpha_0(f_1) \cdots \alpha_0(f_n) U_\sigma
	\end{equ}
	where $U_\sigma \colon \mfH^{\wotimes_\alpha n} \to \mfH^{\wotimes_\alpha n}$ is the unitary map that permutes the tensor factors according to $\sigma$.
\end{lemma}
\begin{proof}
	We will show this by duality. Let $h_1, \dots, h_{m-n} \in \mfH$. Then
	\begin{equs}
		\langle h_1 \otimes  & \cdots  \otimes h_{m-n}, \alpha_q(f_1) \cdots \alpha_q(f_n) g_1 \otimes \cdots \otimes g_m \rangle_{\CF_q} = \\
		&=  \Braket{f_n \otimes  \cdots \otimes f_1 \otimes h_1 \otimes \cdots \otimes h_{m-n}, g_1 \otimes \cdots \otimes g_m}_{\CF_q} = \\
		&= \hspace{-0.3cm} \sum_{\sigma \in \mfS_{m,n}} \hspace{-0.3cm} q^{|\sigma|} \Big\langle f_n \otimes \cdots \otimes f_1 \otimes h_1 \otimes \cdots \otimes h_{m-n} , \\
		& \qquad \qquad \qquad  P_{q} \left( g_{\sigma(1)} \otimes \cdots \otimes g_{\sigma(n)}\right) \otimes   P_{q} \left( g_{\sigma(n+1)} \otimes \cdots \otimes g_{\sigma(m)} \right)  \Big\rangle_{\CF_0} = \\
		& = \hspace{-0.3cm} \sum_{\sigma \in \mfS_{m,n}} \hspace{-0.3cm} q^{|\sigma|} \Big\langle f_n \otimes \cdots \otimes f_1 \otimes h_1 \otimes \cdots \otimes h_{m-n} , \\
		& \qquad \qquad \qquad P_{q} \left( g_{\sigma(1)} \otimes \cdots \otimes g_{\sigma(n)}\right) \otimes  P_{q} \left( g_{\sigma(n+1)} \otimes \cdots \otimes g_{\sigma(m)} \right)  \Big\rangle_{\CF_0} =  \\
		& = \hspace{-0.3cm} \sum_{\sigma \in \mfS_{m,n}} \hspace{-0.3cm} q^{|\sigma|} \Braket{f_1 \otimes \cdots \otimes f_n, g_{\sigma(n)} \otimes \cdots \otimes g_{\sigma(1)}}_{\CF_q} \\
		&\qquad \qquad \qquad  \Braket{ h_1 \otimes \cdots \otimes h_{m-n}, g_{\sigma(n+1)} \otimes \cdots \otimes g_{\sigma(m)} }_{\CF_q} \;,
	\end{equs}
as claimed.
\end{proof}
Lemma~\ref{lemma:AnnProdExp} allows us to write each Wick block $W_q^{k,\ell}$ for arbitrary $q \in (-1,1)$ in terms of $W_0^{k,\ell}$ Wick blocks; we record this in the following corollary.
\begin{corollary}
\label{cor:WickBlockExp}
	For all $q \in (-1,1)$ and $F \in \mfH^{\wotimes_\alpha (k+\ell)}$ and $m \geqslant \ell$ we can write
	\begin{equs}
		W^{k,\ell}_q(F) \big|_{\CF_{q,m}}  &= \sum_{\sigma \in \mfS_{m,\ell} }  q^{|\sigma|} W^{k,\ell}_0 \bigl( P^\dagger_q F \bigr) U_{\sigma} = \\
		&= \sum_{\sigma \in \mfS_{m,\ell} } \sum_{\rho \in \mfS_{k+\ell}} q^{|\sigma| + |\rho|} W^{k, \ell}_0 \left( U_{\rho}^{-1} F \right) U_{\sigma}
	\end{equs}
	where $P^\dagger_q$ is the $\CF_0$ adjoint of $P_q$.
\end{corollary}
\begin{proof}
	The first equality follows directly from the expression \eqref{eq:WickBlockExp}, the fact that $\alpha^\dagger_q \equiv \alpha^\dagger_0$ as operators on $\CF_0$, and Lemma~\ref{lemma:AnnProdExp}. The second expression follows from computing $P_q^\dagger$.
\end{proof}
Using the above corollary, we can now prove the following Wick block estimates.
\begin{proposition}
\label{prop:OpBounds}
	For all $k,\ell \in \N$, and $F \in \mfH^{\wotimes_\alpha (k+\ell)}$
	\begin{equ}
		\| W_0^{k,\ell}(F)\|_{\CB(\CF_0)} \leqslant \| F \|_{\CF_0} \; .
	\end{equ}
	Furthermore,
	\begin{equ}
		\| W_q^{k,\ell}(F)\|_{\CB(\CF_0)} \leqslant D_q^{k+\ell} C_{q} \| F \|_{\CF_0} \; ,
	\end{equ}
	and for $n = k + \ell$
	\begin{equ}
		\| \xi_q^{\diamond n} (F) \|_{\CB(\CF_0)} \leqslant (n+1) D^n_q C_{q} \| F \|_{\CF_0} \; .
	\end{equ}
\end{proposition}
\begin{proof}
The first estimate follows directly from \eqref{eq:WklStrBnd} for $q=0$, as $C_{0} = 1$.
The second estimate follows from the first combined with Corollary~\ref{cor:WickBlockExp}, Proposition~\ref{prop:SymOpEst}, and the fact that for any $\sigma \in \mfS_{m,\ell}$,  $\|U_{\sigma}\|_{\CB(\CF_{0,\ell})} = 1$.
The third estimate follows from the second as $\xi_q^{\diamond n}$ contains $n+1$ Wick blocks.
\end{proof}

We use Proposition~\ref{prop:OpBounds} to obtain an estimate on expansions for products of Wick-ordered operators.

% The following theorem is the main result of \cite{Boz99}, which will allow us to renormalise singular products of mezdonic operators analogously to the bosonic and fermionic cases.
% \begin{theorem}[Ultracontractivity, Bo{\.z}ejko '99]
% 	\label{thm:UltrContr}
% 	For all $A \in  \CF_{q,n}(\mfH)$
% 	\begin{equ}
% 		\|\phi_{q}^{\diamond n} (A) \|_{\CL^2} \leqslant \|\phi_{q}^{\diamond n} (A) \|_{\CL^\infty} \leqslant C^{\frac{3}{2}}_{|q|} (n+1) \|\phi_{q}^{\diamond n} (A) \|_{\CL^2} \; .
% 	\end{equ}
% 	This means that the operator norm topology and the $\CL^2$ topology are equivalent on the range of $\phi_{q}^{\diamond n}$.
% \end{theorem}
% The following corollary of this result will be crucial to control products of Wick-ordered operators.
\begin{corollary}
\label{cor:MultBnd}
	Fix $m,n \in \N$ and set
	\begin{equ}
		I \eqdef \left\{ n+m-2\ell \, \big| \, \ell \in \N :  n\wedge m -\ell \geqslant 0 \right\} \; .
	\end{equ}
	Let $H \in \mfH^{\otimes n}$, $G \in \mfH^{\otimes m}$.
	Let $(F_{k})_{k \in I}$ with $F_k \in \mfH^{\otimes k}$ such that
	\begin{equ}
	\label{eq:ProdToSum}
		\xi_q^{\diamond m}(G) \xi_q^{\diamond n}(H) = \sum_{k \in I} \xi_q^{\diamond k}(F_k) \; .
	\end{equ}
	Then
	\begin{equ}
		 \sum_{k \in I} (k+1) C_{q}^{\frac{3}{2}} D_q^k \left\| F_k \right\|_{\CF_0} \leqslant (n+1)(m+1)  C_{q}^{3} D_q^{n+m} \| G \|_{\CF_0} \| H \|_{\CF_0} \; .
	\end{equ}
\end{corollary}
\begin{proof}
Note that $F_k$ in \eqref{eq:ProdToSum} is the component in the $k$-particle sector of $\xi_q^{\diamond m}(G)H$,
in particular we can write $F_k = W_q^{m-\ell , \ell}(G) H$ where $\ell$ satisfies $m+n-2\ell = k$.
Since
	\begin{equ}
		\| W_q^{m-\ell, \ell}(G) \|_{\CB(\CF_0)} \leqslant 
C_q D_q^\ell \|G\|_{\CF_0}\;,
	\end{equ}
	we have, by Proposition~\ref{prop:OpBounds}, that
	\begin{equs}
		\left\| F_k \right\|_{\CF_0} & \leqslant \| W_q^{m-\ell,\ell}(G) \|_{\CB(\CF_0)} \| H \|_{\CF_0} \leqslant  
C_q D_q^m \|G\|_{\CF_0} \| H \|_{\CF_0} \leqslant \\
		& \leqslant 
C_q^\frac{3}{2} D_q^m \|G\|_{\CF_0} \|H\|_{\CF_0}  \; .
	\end{equs}
	Noting that
	\begin{equs}
		\sum_{k \in I} (k+1) = (n+1)(m+1)\;,
	\end{equs}
we have the estimate
	\begin{equs}
		\sum_{k \in I} (k+1) 
C_q^{\frac{3}{2}} D_q^k \left\| F_k \right\|_{\CF_0} & \leqslant 
C_q^3  \|G\|_{\CF_0} \|H\|_{\CF_0}  \sum_{k \in I} (k+1) D_q^{m+k} \leqslant  \\
		& \leqslant 
C_q^3 D_q^{n+m} \|G\|_{\CF_0} \|H\|_{\CF_0}  \sum_{k \in I} (k+1) \leqslant \\
		&= (m+1)(n+1)
C_q^3 D_q^{n+m} \|G\|_{\CF_0} \|H\|_{\CF_0} \;,
	\end{equs}
and the assertion follows at once.
\end{proof}

We now introduce the promised alternative topology for $q$-mezdons, which we use to estimate ``operator insertions into Wick-ordered products'' such as \eqref{eq:intertwined}.

\begin{definition}\label{def:newBanachAlg}
	Let $A \in \mfA_q(\mfH)$, then there exists a unique $n \in \N$ and unique $(F_k)_{k=0}^{n}$ with $F_{k} \in \mfH^{\otimes k}$ and $F_n \neq 0$ such that
	\begin{equ}\label{eq:Wick_expansion0}
		A = \sum_{k = 0}^n \xi_q^{\diamond k}(F_k) \; .
	\end{equ}
	We define
	\begin{equ}\label{eq:small-q-AltNorm}
		\vvvert A \vvvert \eqdef \sum_{k = 0}^n (k+1) 
C_q^{\frac{3}{2}} D_q^k \| F_k \|_{\CF_0} \; ,
	\end{equ}
	and let $\cA_q(\mfH)$ be the completion of $\mfA_q(\mfH)$ with respect to this norm.
\end{definition}

\begin{theorem}\label{thm:newBanachAlg}
	$\cA_q(\mfH)$ is a Banach algebra and the canonical algebra homomorphism $\iota \colon \cA_q(\mfH) \to \CA_q(\mfH)$ is continuous and injective.
\end{theorem}
\begin{proof}
	The fact that $\cA_q(\mfH)$ is a Banach algebra follows directly from Corollary~\ref{cor:MultBnd} and the fact that $\ell^1$ is a Banach algebra, as these imply that
	\begin{equ}
		\vvvert A B \vvvert \leqslant \vvvert A \vvvert \vvvert B \vvvert
	\end{equ}
	for all $A,B \in \mfA_q(\mfH)$ and by density for all $A,B \in \cA_q(\mfH)$.

	For the inclusion map we note that for all $A = \sum_{k  = 0}^n \xi_q^{\diamond k}(F_k) \in \mfA_q(\mfH)$
	\begin{equs}
		\| \iota(A) \| &= \| A \| \leqslant \sum_{k = 0}^n \| \xi_q^{\diamond k} ( F_k ) \| \leqslant  \sum_{k = 0}^n (k+1) 
C_q^{\frac{3}{2}} D_q^k  \| F_k  \|_{\CF_0} = \vvvert A \vvvert
	\end{equs}
	by \eqref{eq:WklStrBnd2}, which implies continuity by density. To show that $\iota$ is injective, suppose that $A \in \ker(\iota)$. Let $A = \sum_{n \in \N} \xi_q^{\diamond n}(F_n)$, then
	\begin{equs}
		0 &=\left\| \iota(A) \1 \right\|^2 = \sum_{n = 0}^\infty \| F_n \|_{\CF_q}^2 = \sum_{n = 0}^\infty \Braket{F_n , P_q F_n}_{\CF_0} \; .
	\end{equs}
	Since $P_q > 0$, it follows that $F_n = 0$ for all $n \in \N$ and thus $A = 0$.
\end{proof}

\begin{remark}
We will overload notation and view $\xi^{\diamond n}_q$ as a map with codomain $\cA_q(\mfH)$ rather than $\CA_q(\mfH)$.
Note that every element $A \in \cA_{q}(\mfH)$ admits a unique, absolutely convergent (in $\cA_{q}(\mfH)$) expansion
\begin{equ}\label{eq:Wick_expansion}
A = \sum_{k = 0}^{\infty} \xi_q^{\diamond k}(F_k)\;,
\end{equ}
which we call the Wick expansion of $A$.
\end{remark}

\begin{definition}[Homogeneous Chaoses]
	For $n \in \N$, we denote by $\cA_q^{(n)}(\mfH)$ the (closed) image of $\xi_q^{\diamond n}$ in $\cA_q(\mfH)$. 
\end{definition}
\begin{remark}
For each $n \in \N$, $\cA_q^{(n)}(\mfH)$ and $\mfH^{\wotimes_\alpha n}$ are isomorphic as Banach spaces by Corollary~\ref{cor:nPartEq} and the estimate \eqref{eq:WklStrBnd}.
\end{remark}

\subsection{Defining Poly-Wick Products}
In this section, we introduce notation to describe generic insertions of operators from $\cA_{q}$ into Wick products.
In the next section we will prove our estimates, with the key result being Corollary~\ref{cor:RenMultMap}, which shows that we can estimate such insertions using the Banach algebra norm of Theorem~\ref{thm:newBanachAlg}.

Given $m > 0$, $I \in \mfI_{m}$, and $J \in \mfI_{m-1}$, we 
write $I \wr J \in \mfI_{(2m-1)}$ for the interleaving of $I$ and $J$ defined as
\begin{equ}
I \wr J = I_1 \sqcup J_1 \sqcup I_2 \sqcup \ldots \sqcup J_{m-1}\sqcup I_m\;.
\end{equ}
The ordered set $I \wr J$ naturally has the structure of a $(2m-1)$-partitioned set.

When controlling the poly-Wick product,
we perform a Wick expansion on the operators from $\cA_{q}$ being inserted.
This reduces us to estimating a particular subset of contractions that appear in the expansion of a product of Wick-ordered products as given in \eqref{eq:WickProd}.
We make this last point precise in the following definition.

\begin{definition}\label{def:PartialWick0}
	Let $n \in \N_{>0}$, $I \in \mfI_n$, and $J \in \mfI_{n-1}$.
	We define a map 
	\begin{equ}
		\xi^{R ;  I ,  J}_q \colon  \mfH^{\otimes I} \otimes \mfH^{J} \longrightarrow \cA_q(\mfH)
	\end{equ}\martinp{Change $\mfH^J$ to $\mfH^{\otimes J}$?}
	as follows. 
	
	For $(f_{i})_{i \in I} \in \mfH^{I}$ and $(g_j)_{ j \in J} \in \mfH^J$, we set 
	\begin{equs}
	\label{eq:PartialWickDef}
		{}
		&\xi^{R ;I , J }_q
		\big( f_{\otimes I} ; G_1 ,  \dots , G_{n-1} \big) \eqdef  \sum_{\bpi \in \CP^R_{ I , J }} \xi_q^{\diamond \bpi}
		 \big( (f \wr g)_{\otimes (I \wr J)} \big) \eqdef \\
		&  \quad \eqdef \sum_{\bpi \in \CP^R_{ I, J }} \xi_q^{\diamond \bpi}
		 \big( F_{1} \otimes G_{1} \otimes F_{2} \otimes G_{2} \otimes \cdots \otimes F_{n-1} \otimes G_{n-1} \otimes F_{n} \big) \,,
	\end{equs}
	\martinp{Maybe change notation to $f_{\otimes I } \wr g_{\otimes J}$}
where we use the shorthand $F_i = f_{\otimes I_i}$ and $G_j = g_{\otimes J_j}$ and define $(f \wr g) \in \mfH^{I \wr J}$ in the natural way.
Moreover, $\CP^R_{I, J}$ is defined as those $\bpi \in  \widehat\CP_{ I \wr J}$ such that there is \textbf{no} $(s,t) \in \bpi$ with both $s \in I$ and $t \in I$.
\end{definition}

\begin{remark}
In \eqref{eq:PartialWickDef}, the
argument $f_{\otimes I}$ should be thought of as representing a ``purely stochastic term'' generated by the Picard iteration of our equation.
When taken as an argument for $\xi^{R ;I , J }_q$, we are assuming $f_{\otimes I}$ has been Wick-ordered  (i.e.\ all of its self-contractions have been subtracted),
the partition divisions in $I$ specify the insertion locations for $n-1$ different operators.
Each $G_{j}$ is a particular term in the Wick expansion \eqref{eq:Wick_expansion} of the $j^{\text{th}}$ operator. 
The Wick ordering of $f_{\otimes I}$ is reflected in the fact that $\bpi$ is not allowed to contain any contractions within $I$.  

In particular, by \eqref{eq:WickProd}, we have 
\begin{equs}
{}
\prod_{K \lhd (I \wr J)} &\xi_q^{\diamond K}( (f \wr g)_{\otimes K} )
= \\
 &= 
\xi_{q}^{\diamond I_{1}}(F_1)
\xi_{q}^{\diamond J_{1}}(G_1)
\xi_{q}^{\diamond I_{2}}(F_2)
\cdots
\xi_{q}^{\diamond J_{n-1}}(G_{n-1})
\xi_{q}^{\diamond I_{n}}(F_{n}) = \\
{}
&=
\sum_{\bpi \in  \hat\CP_{I \wr  J }}
 \xi_q^{\diamond \bpi}
		 \big( F_{1} \otimes G_{1} \otimes F_{2} \otimes G_{2} \otimes \cdots \otimes F_{n-1} \otimes G_{n-1} \otimes F_{n} \big) \,.
\end{equs}
By renormalising our model, we will be able to cancel diverging contractions $\bpi \in \CP_{ I \wr J }$ involving pairings $(s,t)$ with $s ,t \in I$.
However, contractions $\bpi \in \CP^R_{\boldsymbol{k}, \boldsymbol{\ell}}$ should not (and will not) be cancelled through our renormalisation of the model.
	This is why we separate them from the other contractions and try to estimate these contributions.  
\end{remark}

Note that contractions within $I$ above are implemented\slash removed by renormalisation of the model (which induces the renormalisation of the equation) so it will be useful to have a notation that allows us to make contractions within $I$ explicit using \eqref{eq:Wickcontraction}.
The next bit of notation does exactly this, and will be useful when we want to derive the explicit form of our renormalised equation. 

\begin{definition}\label{def:PartialWickIntertwined}
	Let $n \in \N_{>0}$, $I \in \mfI_{n}$, and $J \in \mfI_{n-1}$. be as in Definition~\ref{def:PartialWick0}.
	Let $\bpi \in \CP_{I}$ and let $L \eqdef I \setminus \bpi \in \mfI_{n}$ be the $n$-partitioned set obtained by removing the elements of $\bpi$ from $I$.\footnote{Some of the blocks of $L$ may be empty.} 
	\begin{equ}
		\xi^{R ; I, J; \bpi}_q \colon \mfH^{\otimes L} \otimes \mfH^{\otimes J} \longrightarrow \cA_q(\mfH)
	\end{equ}
	by setting, for $f \in \mfH^{L}$ and $g \in \mfH^{J}$, and using the shorthand $G_j = g_{\otimes J_j}$, 
	\begin{equ}
	\label{eq:PartialWickDef2}
		\xi^{R ; I, J ; \bpi }_q
		\big( f_{\otimes L};   G_{1}, \dots,  G_{n-1} \big) 
		  \eqdef  \sum_{\bsigma \in \CP^R_{L , J  }} q^{\crb (\bpi,\bsigma)} \xi_q^{\diamond \bsigma}
		\big( (f \wr g)_{\otimes (L \wr J)} \big)  \, .
	\end{equ}
\end{definition}

\begin{remark}
	In the above definition, one should think of $f_{\otimes L}$ as being ``what is left over'' after performing the partial contraction $\bpi$ on some larger object $f_{\otimes I}$. 
	The factor $q^{\crb(\bpi,\bsigma)}$ is the $q$-weight coming from the contractions $\bpi$ and $\bsigma$, 
	and its presence in our definition is why we still keep track of $I$ as an argument in $\xi^{R ; I, J; \bpi}_q$ rather than just replacing $I$ with $L$.
\end{remark}

\begin{remark}
	As mentioned earlier, Definition~\ref{def:PartialWickIntertwined} is used to write Wick products in terms of regular products.
	For this purpose, we only need to make explicit contractions inside the explicit stochastic object $F$. 
Since $\bpi \in \CP_{I} \supset \widehat\CP_{I}$, the partition $\bpi$ can contain contractions between different blocks $I_i$ of $I$ and contractions within a single $I_i$. 
	The partition structure on $I$ in terms of $n$ blocks does not mean $I$ is a product of $n$ Wick ordered products, but rather that we want to insert operators into $n-1$ different locations.

	The $G_j$'s being inserted are not explicit objects generated by our Picard iteration, but are inexplicit bounded operators that we decompose into Wick expansions. 
	Their Wick ordering internal to single $G_j$ factor is not directly related to renormalisation, but rather the fact that Wick operators are the basis this expansion. Moreover, we will not do any renormalization of contractions between different $G_{j}$'s.
	Therefore, we never make explicit contractions between different $G_{j}$'s nor contractions within a single $G_{j}$. 
\end{remark}
The following theorem generalises Proposition~\ref{prop:wickprop} and states the key combinatorial identity that motivates the definitions of the maps $\xi^{R ; I, J; \bpi}_q$.
\begin{theorem}
\label{thm:RenDisent0}
Let $n \in \N_{>0}$, $I \in \mfI_{n}$, and $J \in \mfI_{n-1}$. 
Let $f \in \mfH^{I}$, $g \in \mfH^{J}$, and write $F_i = f_{\otimes I_i}$ along with $G_j = g_{\otimes J_j}$.

We then have 
\begin{equs}[eq:RenDisent0]
	{}&
\xi_{q}^{ I_{1}}(F_1) \xi^{\diamond J_1}(G_1) 
	\xi_{q}^{I_{2}}(F_2) \xi^{\diamond J_2}(G_2)  \cdots 
	\xi_{q}^{I_{n-1}}(F_{n-1})
	\xi^{\diamond J_{n-1}}(G_{n-1}) \xi_{q}^{ I_{n}}(F_{n}) = \\
	{}& \enskip =
	\sum_{\bpi \in \CP_{I}}
	\Big(
	\prod_{(i,j) \in \bpi}  \Braket{\kappa f_i,f_j}
	\Big)
	\xi^{R ; I, J ; \bpi }_q
	(f_{\otimes (I \setminus \bpi)} ; G_1,\dots,G_{n-1}) \; .
	\end{equs}
\end{theorem} 

\begin{proof}
Let $\bar{I} \in \mfI_{|I|}$ be the trivial partition of $I$ into singletons and let $\bar{J} \in \mfI_{|I|-1}$ be obtained by inserting extra empty blocks into $J$ while keeping the existing blocks of $J$ so that $\bar{I} \wr \bar{J} = I \wr J$ at the level of totally ordered sets\footnote{In general they will not be the same as partitioned sets.}.

Note that there is a one-to-one correspondance $\widehat\CP_{\bar{I} \wr \bar{J}} \ni \bsigma \leftrightarrow (\theta(\bsigma),\rho(\bsigma)) \in \CP_{I} \times \CP^{R}_{I,J} $ such that $\bsigma = \theta(\bsigma) \sqcup \rho(\bsigma)$. 
Now, by Proposition~\ref{prop:WickProd}, we have that the left hand side of \eqref{eq:RenDisent0} is equal to
\begin{equs}
	\sum_{\bsigma \in \widehat\CP_{\bar{I} \wr \bar{J}}} \xi_q^{\diamond \bsigma} \big( (f \wr g)_{\otimes (I \wr J)} \big) 
	= \sum_{\bpi \in \CP_{I}} \sum_{\bsigma \in \CP^{R}_{I,J}} \xi_q^{\diamond (\bpi \sqcup \bsigma)} \big( (f \wr g)_{\otimes (I \wr J)} \big). 
\end{equs}
The conclusion follows from observing that, for any $\bpi \in \CP_{I}$ and $\bsigma \in \CP^{R}_{I,J}$,
\begin{equs}
	\xi_q^{\diamond (\bpi \sqcup \bsigma)} \big( (f \wr g)_{\otimes (I \wr J)} \big) 
	= q^{\crb(\bpi,\bsigma)} \Big( \prod_{(i,j) \in \bpi}  \Braket{\kappa f_i,f_j} \Big)
	\xi_q^{\diamond \bsigma} \big( (f \wr g)_{\otimes ((I \setminus \bpi) \wr J)} \big) \;.
\end{equs}
\end{proof}

Below, we identify $\mfA_q(\mfH) \subset \cA_{q}(\mfH)$ and note that the Wick expansion of any element of $\mfA_{q}(\mfH)$ is of the form \eqref{eq:Wick_expansion0}.
We will now define the following multiplication map, which sums over the Wick expansion for operator insertions.
When summing over terms in Wick expansions, it will be useful to associate partitioned sets to multi-indices. 
Given a multi-index $\boldsymbol{k} \in \N^n$, we overload notation and write 
$\boldsymbol{k} = [k_1] \sqcup \ldots\sqcup [k_n] \in \mfI_{n}$
\begin{definition}
\label{def:RenMultMap}
	Let $n \in \N_{>0}$, $I \in \mfI_{n}$, and fix $\bpi \in \CP_{I}$.

We then define the renormalised multiplication map
	\begin{equs}
		\Delta_q^{R; I, \bpi} \colon \mfH^{\otimes (I \setminus \bpi)} \times \big( \mfA_q(\mfH)  \big)^{n+1}
		 \longrightarrow \cA_q(\mfH)
	\end{equs}
	by setting, for $F \in \mfH^{\otimes (I \setminus \bpi)}$ and $ A_0, \dots, A_{n} \in \mfA_q(\mfH)$,
	\begin{equ}[eq:RenProdDef]
		\Delta_q^{R; I, \bpi}(F; A_0, \dots, A_{n}) \eqdef
		A_{0} \Big( \sum_{\boldsymbol{\ell} \in \N^{n-1}} \xi_q^{R; I, \boldsymbol{\ell} ; \bpi} \big( F; A_1^{\ell_1}, \dots, A_{n-1}^{\ell_{n-1}} \big)
		\Big)A_{n} \; ,
	\end{equ} 
	where the $A^{k}_j \in \mfH^{\otimes k}$ are the Wick expansion coefficients, that is 
	\[
	A_j = \sum_{k \in \N} \xi^{\diamond k}_q (A^k_j) \in \mfA_q(\mfH)\;.
	\]

	We will also use the shorthand notation
	\begin{equ}
		\CM^{I}_{q}(F;A_1, \dots, A_{n-1})  \eqdef  \Delta_q^{R ; I, \emptyset}(F ;\one, A_1, \dots, A_{n-1},\one) \; .
	\end{equ}
\end{definition}

\begin{remark}
	The outermost operator insertions $A_0$ and $A_n$ are included in the definition \eqref{eq:RenProdDef} so that $\Delta_q^{R ; I, \bpi}$ is an $\cA_q(\mfH)$-bimodule morphism from $\cA_q(\mfH)^{\otimes(n+1)} \to \cA_q(\mfH)$, equipped with their natural bimodule structures. This is a convenient structural assumption we need when using these maps in conjunction with regularity structures, cf.\ Section~\ref{sec:ConvBPHZ}.
\end{remark}

\begin{remark}
We sometimes abuse notation switch between taking Fock spaces vectors and operators in our maps \dash instead of writing $\Delta_q^{R; I, \bpi}(F; \cdots)$ for $F \in \mfH^{\otimes I}$ we instead write $\Delta_q^{R; I, \bpi}(H; \cdots)$ for $H = \xi_q^{\diamond I}(F) \in \mfA_q(\mfH)$.

We perform the same abuse of notation with $\CM^{I}_{q}$, and also perform a similar abuse from Section~\ref{sec:PolyWickEstimates} onwards, where we show that we can in fact take $F \in \mfH^{\widehat{\otimes}_{\alpha} I}$. 
\end{remark}

Below we give an example of writing a ``non-renormalised'' intertwined product of operators in terms of the corresponding poly-Wick product $\CM^{I}_{q}$, with the ``renormalisation counterterms'' given in terms of operators $\Delta_q^{R; I, \bpi}$ with $\bpi \not = \emptyset$.

\begin{example}\label{ex:RenProdExample1}
Let $I = \{1\} \sqcup \{2\} \sqcup \{3\}$ and $f_1,f_2,f_3 \in \mfH$
Then, for any $A_1,A_2 \in \mfA_q(\mfH)$, we have, using the shorthand   $\xi_i \eqdef \xi_q(f_i)$,  
\begin{equs}\label{eq:RenProdExample1}
		\xi_1 A_1 \xi_2 A_2 \xi_3
		%&= \CM_F(A_1, A_2) + \omega_q(\xi_0 \xi_1) \Delta_q^{R ; \boldsymbol{k} , (0,1)}( \xi_2 , A_1, A_2)+ \\
		%& \qquad + \omega_q(\xi_0 \xi_2) \Delta_q^{R ; \boldsymbol{k} , (0,2)}( \xi_1 , A_1, A_2)+ \\
		%& \qquad +\omega_q(\xi_1 \xi_2) \Delta_q^{R ; \boldsymbol{k} , (1,2)}( \xi_0 , A_1, A_2) = \\
		&= \CM_F^{I}(A_1, A_2) + \Braket{\kappa f_1,f_2} \Delta_q^{R ; I , (1,2)}( \xi_3; \one, A_1, A_2, \one)+ \\
		& \qquad + \Braket{\kappa f_1,f_3}  \Delta_q^{R ; I , (1,3)}( \xi_2; \one, A_1, A_2, \one)+ \\
		& \qquad +\Braket{\kappa f_2,f_3}  \Delta_q^{R ; I , (2,3)}( \xi_1; \one, A_1, A_2, \one) \; .
	\end{equs}
\end{example}
Note that the coefficient in front of the $\Delta^{R; I , \bpi}_q$ in an expansion like in the example above will always be of the form
$		\prod_{(i,j) \in \bpi} \Braket{\kappa f_{i}, f_{j}}$,
	with any $q$-weights (and therefore all $q$-dependence) incorporated in the action of $\Delta^{R; I, \bpi}_q$.

\begin{remark}\label{rem:notation_polywick}
Note that the arguments appearing in various maps on the right-hand side of \eqref{eq:RenProdExample1} depend on how we separated factors. 
In particular, we could have chosen $f'_1 = f_1/2$, $f'_2 =  2f_2$ and $f'_3 = 2f_3$ along with $A_1' =  A_1/2$, $A_2' = A_2$, then, writing $\xi'_i = \xi_q(f'_i)$, the left hand side of \eqref{eq:RenProdExample1} can also be written as $\xi'_1 A'_1 \xi'_2 A'_2 \xi'_3$ and the argument of the relevant maps in the individual terms on the right hand side would change - however this redundancy does not create any problems for well-definedness of the formulae of this section. 

Later, we will see maps like $\Delta^{R;  I , \bpi}$ appear in our description of renormalised equations. 
In this context there will be no ambiguity in definitions since there will be rigidity in how we write the products and tensor products that appear, which come from the structure of the model \slash regularity structure.  
\end{remark}

%In particular, it will be enough to establish that a negative tree that needs to be renormalised corresponds to a well-defined distribution with values in
%	\begin{equ}
%		\bigoplus_{n = 0}^N \mfH^{\wotimes_\alpha n} \; ,
%	\end{equ}
%	for $N$ large enough, after renormalisation. In particular, this is independent of $q \in (-1,1)$, and in fact also $q = \pm 1$.

\begin{example}
Continuing from Example~\ref{eq:RenProdExample1} we now look at the behaviour of $\Delta^{R; I , \bpi}_q$ above by specialising to specific cases of $A_{1}$ and $A_{2}$.
	\begin{enumerate}
		\item Let $g_1, g_2 \in \mfH$ and $A_1 \eqdef \xi_q(g_1)$, $A_2 \eqdef \xi_q(g_2)$, then
		\begin{equs}
			\Delta_q^{R; I, (1,3)}(\xi_2; \one, A_1, A_2,\one) & = q^3 \xi^{\diamond 3}_q \left( g_1 \otimes f_2 \otimes g_2  \right) + q^2 \Braket{g_1,g_2} \xi_q(f_2)  + \\
			&  \qquad \qquad +  q\Braket{\kappa f_2, g_2}  \xi_q(g_1)+ q  \Braket{g_1, f_2} \xi_q(g_2)  \; , \\
			\Delta_q^{R; I, (1,2)}(\xi_3; \bone, A_1, A_2,\bone) & = q \xi^{\diamond 3}_q \left( g_1 \otimes g_2 \otimes f_3 \right) + q  \Braket{g_1, g_2} \xi_q(f_3) + \\
			&  \qquad \qquad + q \Braket{g_2,f_3} A_1 +  q^2 \Braket{g_1, f_3} A_2  \; .
		\end{equs}
		\item Let $A_1 = \xi_{q}(g_1) \xi_{q}(g_2)$, $A_2 = \bone$, then we have
		\begin{equs}
			\Delta_q^{R; I, (1,3)}(\xi_2;
		\bone, A_1, A_2, \one) & = q^3 \xi^{\diamond 3}_q \left( g_1  \otimes g_2 \otimes f_2 \right) + q  \Braket{g_1, g_2} \xi_q(f_2) + \\
			&  \qquad \qquad +  q \Braket{g_2, f_2}  \xi_q(g_1) + q^2 \Braket{g_1,f_2} \xi_q(g_2) \; , \\
			\Delta_q^{R; I, (1,2)}(\xi_3; \one, A_1, A_2, \one) & = q^2 \xi^{\diamond 3}_q \left( g_1 \otimes g_2 \otimes f_3  \right) +   \Braket{g_1, g_2} \xi_q(f_3) + \\
			&  \qquad \qquad + q^2 \Braket{g_2,f_3} \xi_q(g_1) +  q^3 \Braket{g_1, f_3}  \xi_q(g_2) \; .
		\end{equs}
	\end{enumerate}
\end{example}

We rewrite Theorem~\ref{thm:RenDisent0} in terms of the map $\Delta^{R; I, \bpi}_q$.
\begin{theorem}
\label{thm:RenDisent}
Let $n \in \N_{>0}$, $I \in \mfI_{n}$. 
Fix $A_1, \dots, A_{n-1} \in \mfA_q(\mfH)$ and $f \in \mfH^{I}$. 
We then have 

\begin{equs}
{}& 
A_0 \xi_{q}^{ I_{1}}(f_{\otimes I_1}) A_1 
\xi_{q}^{I_{2}}(f_{ I_2}) A_2  \cdots
A_{n-1} \xi_{q}^{I_{n}}(f_{\otimes I_{n}})
A_n \\
{}& \enskip =
\sum_{ 
\bpi \in \CP_{I}}
\Big(
\prod_{(i,j) \in \bpi}  \Braket{\kappa f_i,f_j}
\Big)
\Delta^{R ; I , \bpi }_q
(f_{\otimes (I \setminus \bpi)}; A_0, A_1, \dots, A_{n-1}, A_n ) \; .
\end{equs}
\end{theorem}

\begin{proof}
This proof follows from using Theorem~\ref{thm:RenDisent0}, recalling the definition of the operators $\Delta^{R;  I , \bpi}$, and performing Wick expansions on the factors $A_1,\dots,A_{n-1}$ to
obtain the relevant $G_j$'s. 
\end{proof}

\subsection{Estimates on Poly-Wick Products}\label{sec:PolyWickEstimates}

In this section, we perform estimates on the poly-Wick products that allow us to extend their definition to the case where the $F$ belongs to a Hilbert space tensor product of $\mfH$ rather than only the algebraic tensor product.

We begin this section with an analytic estimate on contractions on tensor products. 
\begin{definition}
	Let  $I \in \mfI$, and $\bpi \in \hat \CP_{I}$, and let 
	\begin{equ}
		C_{\bpi} \colon \bigotimes_{I_i \lhd I} \mfH^{\otimes I_i} \longrightarrow \bigwotimesp_{L_i \lhd (I \setminus \bpi)} \mfH^{\wotimes_\alpha L_i}
	\end{equ}
	be the linear map that maps 
	\begin{equ}
		f_{\otimes I} \longmapsto
		\Big(\prod_{(i,j) \in \bpi} \Braket{\kappa f_i, f_j} 
		\Big)
		f_{\otimes I \setminus \bpi } \; .
	\end{equ}
\end{definition}

\begin{lemma}
\label{lem:ContEst}
	In the above setting, $C_{\bpi}$ extends to a continuous map
	\begin{equ}
		\bigwotimesp_{I_i \lhd I} \mfH^{\wotimes_\alpha I_i} \longrightarrow \bigwotimesp_{L_i \lhd (I \setminus \bpi)} \mfH^{\wotimes_\alpha L_i} \;
	\end{equ}
	with $\| C_{\bpi} \| \leqslant 1$.
\end{lemma}
\begin{proof}
	By induction and the associativity of the projective tensor product, cf.\ \cite[\textsection{1}, no.~4]{Gro55}, it is enough to show that the bilinear map $C \colon \mfH^{\wotimes_\alpha k} \wotimes_\pi \mfH^{\wotimes_\alpha l} \to \mfH^{\wotimes_\alpha (k-1)} \wotimes_\pi \mfH^{\wotimes_\alpha (l-1)}$
	\begin{equ}
		C \colon (v_1 \otimes \cdots \otimes v_k, w_1 \otimes \cdots \otimes w_\ell) \longmapsto \scal{v_1, w_1} \left( v_2 \otimes \cdots \otimes v_k\right) \otimes \left( w_2 \otimes \cdots w_l\right)
	\end{equ}
	is well-defined, continuous, and has operator norm bounded by $1$. The result then follows by composing such contraction maps as well as permutation maps.
	Up to unitary isomorphism, we may assume that $\mfH = \ell^{2}(\N)$. Then $\mfH^{\wotimes_\alpha k} = \ell^2(\N^k)$, and similarly for $\mfH^{\wotimes_\alpha l}$. For $\lambda =(\lambda_{i_1 \dots i_k})_{i} \in \ell^2(\N^k)$ and $\mu = (\mu_{j_1 \dots j_l})_{j} \in \ell^2(\N^l)$
	\begin{equ}
		C(\lambda, \mu)_{i_2, \dots, i_k, j_2, \dots, j_l} = \sum_{r = 0}^\infty \lambda_{r i_2 \cdots i_k} \otimes \mu_{r j_2 \dots j_l} \; .
	\end{equ}
	By the triangle and Cauchy--Schwarz inequalities
	\begin{equs}
		\left\| C(\lambda, \mu) \right\| & \leqslant \sum_{r = 0}^\infty \biggl( \sum_{i_2, \dots, i_k = 0}^\infty |\lambda_{r i_2 \cdots i_k}|^2 \biggr)^{\frac{1}{2}} \biggl( \sum_{j_2, \dots, j_l = 0}^\infty |\mu_{r j_2 \cdots j_l}|^2 \biggr)^{\frac{1}{2}} \leqslant \\
		& \leqslant \biggl( \sum_{r, i_2, \dots, i_k = 0}^\infty |\lambda_{r i_2 \cdots i_k}|^2 \biggr)^{\frac{1}{2}} \biggl( \sum_{s, j_2, \dots, j_l = 0}^\infty |\mu_{s j_2 \cdots j_l}|^2 \biggr)^{\frac{1}{2}} =\\
		&= \| \lambda \|_{\ell^2} \| \mu \|_{\ell^2}\;,
	\end{equs}
	which proves the assertion.
\end{proof}

Using this, we can reinterpret the action of the Wick block $W_{0}^{k, \ell}$ as performing $\ell$ contractions in the spirit of Proposition~\ref{prop:WickProd}. In particular, letting $k_1 = k + \ell$, $k_2 \geqslant \ell$, $\boldsymbol{k} = (k_1, k_2)$, then per definitionem we have for $F \in \mfH^{\wotimes_\alpha k_1}$ and $G \in \mfH^{\wotimes_\alpha k_2}$
\begin{equ}
	W_0^{k, \ell}(F)G = C_{\bpi}(F \otimes G)
\end{equ}
with the ``nested'' contraction $\bpi \in \widehat\CP_{[k_1] \sqcup [k_2]}$ given by
\begin{equ}
	\bpi = \left\{ (k_1, k_1+1), (k_1-1,k_1+2), \dots, (k+1, k_1+ \ell) \right\} \; .
\end{equ}

By comparing coefficients in Proposition~\ref{prop:WickProd} and Corollary~\ref{cor:WickBlockExp} we see that for an arbitrary $\bpi \in \CP_{\boldsymbol{k}}$
\begin{equ}
	\crb(\bpi) = |\sigma(\bpi)| + |\rho(\bpi)|
\end{equ}
where $\sigma(\bpi) \in \mfS_{k_2, \ell}$ and $\rho(\bpi) \in \mfS_{k_1}$ are the permutations such that, for all $F \in \mfH^{\wotimes_{\alpha} k_1}$ and $G \in \mfH^{\wotimes_\alpha k_2}$,  $W_{0}^{k,\ell}(U_{\rho^{-1}}F) U_{\sigma} G$ corresponds to the contraction $\bpi$ in $\xi^{\diamond k_1}_q(F) \xi^{\diamond k_2}_q(G_2)$. This correspondence allows us to establish the following combinatorial lemma concerning the $q$-weighted number of contractions $\bpi$.
\begin{lemma}
\label{lem:ContNum}
	For any $q \in (-1,1)$, $L \in \mfI$, we have the bound 
	\begin{equ}
		\sum_{\bpi \in \widehat\CP_{L}} \big( |L \setminus \bpi| +1 \big) D_{q}^{|L \setminus \bpi |} |q|^{\crb(\bpi)} \leqslant  D_{q}^{|L|} \prod_{ L_\ell \lhd L} (|L_\ell| + 1)  \; .
	\end{equ}
\end{lemma}
\begin{proof}
	%We consider the case $\CP_{\boldsymbol{r}} = \CP^R_{\boldsymbol{0} \sqcup \boldsymbol{r}}$, where $\boldsymbol{r}$ is the intertwined concatenation of $\boldsymbol{k}$ and $\boldsymbol{\ell}$, since this per definitionem admits more contractions $\bpi$.
 
	Without loss of generality, we assume that $q \in [0,1)$.
	
	The claim is trivial for $\|L\| < 2$.
	For $\|L\|=2$, we write $L = [(k_1,k_2)]$ and fix $f \in \mfH$ of norm $1$.
	We then have, by Proposition~\ref{prop:WickProd}, that
	\begin{equ}
	\xi_q^{\diamond k_1} \bigl( f^{\otimes k_1}\bigr) \xi_q^{\diamond k_2} \bigl( f^{\otimes k_2} \bigr) = \sum_{\bpi \in \widehat\CP_{[(k_1,k_2)]}} |q|^{\crb(\bpi)} \xi_q^{\diamond (k_1 + k_2 - 2|\bpi|)} \bigl( f^{\otimes (k_1 + k_2 - 2|\bpi|)} \bigr) \; ,
	\end{equ}
If we take, in the statement of Corollary~\ref{cor:MultBnd}, $G = f^{\otimes k_1}$ and $H = f^{\otimes k_2}$, we then have, for $k \in I$,  
\[
F_{k} = \Bigg[\sum_{\substack{ \bpi \in \widehat\CP_{L} \\ k_1 + k_2 - 2|\bpi| = k}} |q|^{\crb(\bpi)} \Bigg] f^{\otimes k}  \;.
\]
Since $\|f^{\otimes j}\|_{\CF_0} = 1$ for all $j \in \N$, it follows that we have the estimate
\begin{equ}
	\sum_{\bpi \in \widehat\CP_{L}} (|L \setminus \bpi| + 1) C_q^{\frac{3}{2}} D_{q}^{|L \setminus \bpi|} |q|^{\crb(\bpi)} 
	\leqslant (|L_1| + 1) (|L_2| + 1) C_q^{3} D_q^{|L|} \;,
\end{equ}
and we conclude the case $\|L\|=2$ since $C_{q} \geqslant 1$. 

For the general case of $\|L\| > 2$, one applies the above argument iteratively. \ajay{add more details here.}
\end{proof}

Combining Lemmata~\ref{lem:ContEst}~\&~\ref{lem:ContNum} yields the following theorem, which is fundamental for estimating intertwined renormalised products.

\begin{theorem}\label{thm:PolyWickestimate}
	For any $n \in \N_{>0}$, $I \in \mfI_{n}$, and $J \in \mfI_{n-1}$, we have that $\xi^{R; I, J; \bpi}_q$ extends to a continuous map
	\begin{equ}
		\hspace*{-5pt}
		\Big( \bigwotimesa_{I_i \lhd (I \setminus \bpi)} \mfH^{\wotimes_\alpha I_i} 
		\Big) 
		\wotimes_\pi \hspace*{-3pt} 
		\Big(  \hspace*{-3pt} \bigwotimesp_{J_{j} \lhd J} \mfH^{\wotimes_\alpha J_j} 
		\Big)
		\longrightarrow \cA_q(\mfH)\
	\end{equ}
	with  
	\begin{equ}
		\left\| \xi^{R; I, J; \bpi }_q \right\| \leqslant 
C_q^{\frac{3}{2}} D_{q}^{|(I \setminus \bpi)|+|J|} \prod_{K \lhd (I \setminus \pi) \wr J}(|K|+1) \; . 
	\end{equ}
\end{theorem}
\begin{proof}
	Since $|q| \leqslant 1$ and it only enters into the estimate via the factor $q^{\crb(\bpi,\bsigma)}$,  the argument for estimating $\| \xi_q^{R; I, J; \bpi}\|$
	will follow in a straightforward way from our argument for $\| \xi_q^{R; (I \setminus \bpi), J}\|$ as our estimate proceeds term-by-term in the sum over $\CP^{R}_{(I \setminus \bpi),J}$.
	We can therefore assume $\bpi = \emptyset$\footnote{We use $\bpi$ as a dummy variable in the proof below but it has no relation to the rest of the $\bpi$ appearing in the statement of this theorem.} in the statement of the theorem.

	Let $F_i \in \mfH^{\otimes I_i}$ and $G_j \in \mfH^{\otimes J_j}$. 
	We then have the estimate 
	\begin{equs}
	{}&\vvvert \xi_q^{R ; I, J}(\otimes_{i=1}^n F_i; G_1, \dots, G_{n-1}, F_n)  \vvvert \leqslant \\
		%&= \sum_{\bpi \in \CP^R_{\boldsymbol{k}, \boldsymbol{\ell}}} \xi_q^{\diamond \bpi} \Bigl( F_1, G_1, \dots, G_{n-1}, F_n  \Bigr) =  \\
		{}& \qquad  
		\leqslant \sum_{\bsigma \in \CP^R_{I, J}} |q|^{\crb(\bsigma)} \Big\vvvert \xi^{\diamond (I \wr J) \setminus \bsigma } \Big[ C_{\bpi} \left( F_1 \otimes G_1 \otimes \cdots \otimes G_{n-1} \otimes F_n \right) \Big] \Big\vvvert \; . 
	\end{equs}\martinp{Did you not write $\otimes_{i = 1}^n F_i$ not as $F_{\otimes I}$ on purpose? }
	Since $\CP^R_{I,J}$ is a subset of the contractions of $\CP_{I \wr J}$, it follows from
	 Lemma~\ref{lem:ContNum} and the definition of the $\vvvert \bigcdot \vvvert$-norm in \eqref{eq:small-q-AltNorm}, that we have the estimate 
	\begin{equ}
		\left\| \xi^{R; I, J  }_q \right\| \leqslant 
C_q^{\frac{3}{2}} D_q^{|I|+|J|} \prod_{K \lhd (I \wr J)} (|K| + 1) \sup_{\bpi \in \CP^R_{I, J}} \left\| C_{\bpi} \right\| \; .
		\end{equ} 
		
	Finally, for fixed $\bpi \in \CP^R_{I, J}$, none of the pairings within $\bpi$ occur between the factors $\mfH^{\wotimes_\alpha I_i}$ and $\mfH^{\wotimes_\alpha J_j}$, so modulo a permutation of factors, estimating $\|C_{\bpi}\|$ can be reduced to estimating $\|C_{\bpi'}\|$ for $\bpi' \in \widehat\CP_{\tilde{I} \sqcup J}$ where $\tilde{I} \in \mfI_{1}$ has the same underlying totally ordered set 
	as $I$ but only a single block, and $C_{\bpi'}$ acts on 
	\begin{equ}
	\bigwotimesp_{K \lhd \tilde{I} \sqcup J}\mfH^{\wotimes_\alpha K} =	\Big( \mfH^{\wotimes_\alpha I} \Big) \wotimes_\pi  \Bigg( \hspace*{-5pt} \bigwotimesp_{J_j \lhd J} \mfH^{\wotimes_\alpha J_j} \Bigg) \;.
	\end{equ}
	We then conclude by observing that, by Lemma~\ref{lem:ContEst}, we have $\|C_{\bpi'}\| \leqslant 1$. 
\end{proof}

We immediately have the following corollary, an instructive application of which is control 
of the object $\<2_1Rprime>$ in the proof of Theorem~\ref{thm:Phi42Ren} below.
\begin{corollary}
\label{cor:RenMultMap}
	Let  $I \in \mfI$ with $\|I\| > 0$,\ and $\bpi \in \CP_{I}$. 
	The renormalised multiplication map extends to a separately continuous multilinear map
	\begin{equs}
		\Delta_q^{R; I, \bpi} \colon \mfH^{\wotimes_\alpha (I \setminus \bpi)} \times \left( \cA_q(\mfH) \right)^{\|I\|+1} \longrightarrow \cA_q(\mfH)
	\end{equs}
	with norm
	\begin{equ}
		\left\| \Delta_q^{R; I, \bpi} \right\| \leqslant 
C_q^{\frac{3}{2}} D_q^{|I \setminus \bpi|}  \prod_{L_i \lhd (I \setminus \bpi)} (|L_i| + 1)
	\end{equ}
	In particular, $\Delta^{R; I, \bpi}_q$ extends to a continuous bilinear map 
\begin{equ}
\Delta^{R; I, \bpi}_q \colon \mfH^{\wotimes_\alpha (I \setminus \bpi)} \times \cA_q(\mfH)^{\wotimes_\pi (\|I\|+1)} \to \cA_q(\mfH)\;.
\end{equ}
\end{corollary}

\begin{remark}
	Note that control of $ \Delta_q^{R; I, \bpi}$ in the specific case $I$ is partitioned into singletons would actually be enough for our purposes, 
	since one can reduce any operator insertion to this case by inserting identity operators.
	This, however, would yield a slightly weaker estimate since $\prod_{L_{i} \lhd (I \setminus \bpi)} (|L_i| + 1)$ would be replaced by $2^{|I \setminus \bpi|}$.
\end{remark}

\section{Besov Space Estimates}\label{sec:Besovestimates}

Throughout this section, we fix a base field $\mathbb{K} \in \{\R, \C\}$ for all vector spaces unless otherwise stated, and $E$ shall denote a complete Hausdorff locally convex vector space over $\mathbb K$. That is, $E$ is equipped with a set of continuous seminorms $\mfP$ defining the topology on $E$, such that a net $(v_\lambda)_\lambda$ converges to $v$ if and only if for all $\mfp \in \mfP$,
\begin{equ}
	\lim_\lambda \mfp(v_\lambda - v) = 0 \; .
\end{equ}
For $\mfp \in \mfP$, define the Banach space $E_{\mfp}$ by setting
\begin{equ}
	\label{eq:LocSpaceDef}
	E_{\mfp} \eqdef \overline{E/\mfp^{-1}(0)}\;.
\end{equ}

We also fix a scaling $\s = (\s_i)_{i=1}^{d} \in \R_{\geqslant 1}^{d}$. For a monomial $X^k$ for $k \in \N^d$ we define the $\s$-scaled degree $\deg_\s X^k \eqdef |k|_\s$, and for a general polynomial $P = \sum_{k \in \N^d} c_k X^k$
\begin{equ}
	\deg_\s P \eqdef \max_{\substack{k \in \N^d\\ c_k \neq 0}} \deg_\s X^k \; .
\end{equ}

\subsection{Vector-Valued H\"older--Besov Spaces}

\begin{definition}[Test Functions \TitleEquation{\CB^r_{\s,x}}{B^r_{s,x}}]
	Let $r \in \N$ and $x \in \R^d$. Define
	\begin{equ}
		\CB^r_{\s,x} \eqdef \left\{ \eta \in \cD(\R^d) \, \middle| \, \supp \eta \subset B_{\s}(x,1), \; \|\eta\|_{\cC^r} \leqslant 1 \right\}
	\end{equ}
	where $\|\bigcdot\|_{\cC^r}$ denotes the usual $\cC^r$-norm.

	Furthermore, for $\alpha \in \R$, let $\CB^{r,\alpha}_{\s, x}$ be the subset of $\CB^r_{\s, x}$ such that, for all polynomials $P$ of degree $\deg_\s P \leqslant \alpha$ and all $\eta \in \CB^{r,\alpha}_{\s,x}$,
$			\int \eta(x) P(x) \d x	= 0$ .
\end{definition}

\begin{definition}[Vector-Valued H\"older Space]
	Let $\alpha \in \R$, let $r$ be the smallest natural number such that $r+\alpha > 0$, and define for $\xi \in \cD' \left( \R^d ;  E \right) $, $\mfp \in \mfP$, $\K \Subset \R^d$, the seminorm
	\begin{equ}
			\|\xi\|_{\alpha;\K,\mfp} \eqdef \sup_{x \in \K} \sup_{\eta \in \CB^{r,\alpha}_{\s,0}} \sup_{\lambda \leqslant 1} \lambda^{-\alpha} \mfp \left( \xi \left( \CS^\lambda_{\s,x} \eta \right) \right) +  \sup_{x \in \K} \sup_{\eta \in \CB^{r}_{\s,0}}  \mfp \left( \xi \left( \CS^1_{\s,x} \eta \right) \right)  \; .
	\end{equ}
	Then the space of $E$-valued $\alpha$-H\"older (generalised) functions is given by
	\begin{equ}
			\CC^\alpha_{\s} \bigl( \R^d; E\bigr) \eqdef \overline{\left\{ \xi \in  \cD \left( \R^d ; E\right)  \, \middle| \, \forall \K \Subset \R^d \, \forall \mfp \in \mfP : \|\xi\|_{\alpha; \K, \mfp}  < \infty \right\} }
	\end{equ}
	where the closure is taken with respect to the topology induced by the seminorms\\ $\left\{ \| \bigcdot\|_{\alpha;\K,\mfp} \right\}_{\K \Subset \R^d, \mfp \in \mfP}$ in $\cD'(\R^d; E)$. Furthermore, let
	\begin{equ}
		\|\xi\|_{\alpha;\mfp} \eqdef \sup_{x \in \R^d} \sup_{\eta \in \CB^{r,\alpha}_{\s, 0}} \sup_{\lambda \leqslant 1} \lambda^{-\alpha} \mfp\left( \xi \left( \CS^{\lambda}_{\s, x} \eta \right) \right) + \sup_{x \in \R^d} \sup_{\eta \in \CB^{r}_{\s, 0}}  \mfp\left( \xi \left( \CS^{1}_{\s, x} \eta \right) \right) \; .
	\end{equ}
\end{definition}
% \begin{definition}[Vector-Valued H\"older Space]
% 	Let $U \subset \R^d$ open, let $\alpha \in \R$, let $r$ be the smallest natural number such that $r+\alpha > 0$, define for $\xi \in \cD' \left( U; E \right) $, $\mfp \in \mfP$, $\K \Subset U$ and $\lambda_{\K} \eqdef 1 \wedge \frac{d_{\s}(\K, \partial U)}{2}$, the seminorm
% 	\begin{equ}
% 			\|\xi\|_{\alpha;\K,\mfp} \eqdef \sup_{x \in \K} \sup_{\eta \in \CB^{r,\alpha}_{\s,0}} \sup_{\lambda \leqslant \lambda_{\K}} \lambda^{-\alpha} \mfp \left( \xi \left( \CS^\lambda_{\s,x} \eta \right) \right) +  \sup_{x \in \K} \sup_{\eta \in \CB^{r}_{\s,0}}  \mfp \left( \xi \left( \CS^{\lambda_{\K}}_{\s,x} \eta \right) \right)  \; .
% 	\end{equ}
% 	Then the space of $E$-valued $\alpha$-H\"older (generalised) functions is given by
% 	\begin{equ}
% 			\CC^\alpha_{\s} \left(  U; E\right) \eqdef \overline{\left\{ \xi \in  \cD' \left( U ; E \right)  \, \middle| \, \forall \K \Subset U \, \forall \mfp \in \mfP : \|\xi\|_{\alpha; \K, \mfp}  < \infty \right\} }
% 	\end{equ}
% 	where the closure is taken with respect to the topology induced by the seminorms\\ $\left\{ \| \bigcdot\|_{\alpha;\K,\mfp} \right\}_{\K \Subset U, \mfp \in \mfP }$. Furthermore, let
% 	\begin{equ}
% 		\|\xi\|_{\alpha;\mfp} \eqdef \sup_{x \in \R^n} \sup_{\eta \in \CB^{r,\alpha}_{\s, 0}} \sup_{\lambda \leqslant 1} \lambda^{-\alpha} \mfp\left( \xi \left( \CS^{\lambda}_{\s, x} \eta \right) \right) + \sup_{x \in \R^n} \sup_{\eta \in \CB^{r}_{\s, 0}}  \mfp\left( \xi \left( \CS^{\lambda(x)}_{\s, x} \eta \right) \right) \; .
% 	\end{equ}
% \end{definition}
\begin{remark}
	The space $\CC^\alpha_\s (\R^d ; E)$ is locally $t$-convex over $E$, with the $\mfp$-seminorms being exactly $\left( \| \bigcdot \|_{\alpha ; \K , \mfp}\right)_{\mfK \Subset \R^d}$. Furthermore, when $E =\CA$ is a locally $m$-convex algebra, $\CC^\alpha_{\s}(\R^d ; \CA)$ is a locally $t$-convex bimodule over $\CA$.
\end{remark}

\begin{definition}[Vector-Valued H\"older Space with Singularities]
\label{def:SingHolderSp}
	With the above notation, fix $\alpha, \eta \in \R$, $\bar d \in [d]$, the hyperplane $H \eqdef \{ z\in\R^d  \, \big| \, \forall i \in [\bar d] : z_i = 0\}$, and $\bar{\s} \eqdef (\s_1, \dots, \s_{\bar d})$. We accordingly denote elements in $\R^d$ by $z = (x,y) \in \R^{\bar d} \times \R^{d-\bar d}$.

	We denote by  $\CC_{\s,H}^{\alpha,\eta}(\R^d; E)$ the subset of $\xi \in \CC^{\eta \wedge 0}_{\s}(\R^d; E)$, such that  for every $\K \Subset \R^d$ and every seminorm $\mfp \in \mfP$ there exists a constant $C_{\mfp} > 0$, s.t.\ for all $z = (x,y)\in  \R^{d} \setminus H \cap \K$, all $\psi \in \CB^{r,\alpha}_{\s ,0}$ and all $\lambda \in (0, 1]$ satisfying  $2 \lambda \leqslant |x|_{\bar\s}$,
	\begin{equ}
		\mfp \left( \xi \left( \CS^{\lambda}_{\s,z} \psi \right)  \right) \leqslant C_\mfp \left(|x|_{\bar\s} \wedge 1\right)^{\eta-\alpha} \lambda^{\alpha} \; ,
	\end{equ}
	and for all $\phi \in \CB^r_{\s, 0}$
	\begin{equ}
		\mfp \left( \xi \left( \CS_{\s,z}^{\lambda_x} \phi \right)  \right) \leqslant C_\mfp \left(|x|_{\bar\s}\wedge 1\right)^{\eta \wedge 0  } \; ,
	\end{equ}
	where $\lambda_x \eqdef \frac{|x|_{\bar\s}}{2}\wedge 1$. The smallest possible constant for the above inequalities will be denoted by $\| \xi \|_{\alpha,\eta;\K,\mfp}$.
\end{definition}

The following theorem is proven in Appendix~\ref{subsec:PrimTensProdproof}.
\begin{theorem}
\label{thm:PrimTensProd}
	For all $\alpha \in \R$ we have the isomorphism
	\begin{equ}
		\CC^\alpha_{\s}(\R^d ; E) \cong \CC^\alpha_{\s}(\R^d) \wotimes_\eps E \; ,
	\end{equ}
	where $\wotimes_{\eps}$ denotes the completion of the algebraic tensor product with respect to the injective topology.
\end{theorem}

\begin{remark}
	Analogously, one can also define H\"older spaces on open or closed subsets $U \subset \R^n$, as well as H\"older spaces with local blow-up properties, cf.\ \cite{CHP23}. For $U$ open, the proof in Appendix~\ref{subsec:PrimTensProdproof} holds verbatim, whereas for the $U$ closed it holds since it is constructed as a projective limit over $V$ open, s.t.\ $U \subset V$, and projective limits commute with the injective tensor product, cf.\ \cite[Corollary~16.3.2]{Jar81}. Analogously, for the blow-up spaces which we define below.
\end{remark}

Theorem~\ref{thm:PrimTensProd} allows us to make the following important meta-statement.
\begin{metatheorem}
	\label{MetaThm}
	Any statement about linear operations can be lifted from $\CC^\alpha_{\s}(\R^d)$ to $\CC^\alpha_{\s}(\R^d ; E)$ for any complete, Hausdorff, locally convex vector space $E$.
\end{metatheorem}

An example of this is the following theorem.
\begin{theorem}
	\label{thm:ExtLinOp}
	Any continuous linear map $A \colon \CC^\alpha_{\s}(\R^d) \to \CC^\beta_{\s}(\R^d)$ for $\alpha,\beta \in \R$, can be extended to a $t$-continuous linear map $A \wotimes_\eps \bone_{E} \colon \CC^\alpha_{\s}(\R^d; E) \to \CC^\beta_{\s}(\R^d ; E)$.
\end{theorem}
\begin{proof}
	This follows directly from Proposition~43.6 and Definition~43.6 in \cite{Trev67}.
\end{proof}

Another direct consequence of the Metatheorem~\ref{MetaThm} is that we can always ``change coefficients''.
\begin{theorem}
\label{thm:ChngCoef}
	Let $E,F$ be two complete, Hausdorff, LCTVSs. Any continuous linear map $A \colon E \to F$ extends to a continuous linear map $\bone_{\CC^\alpha_{\s} (\R^d)}  \wotimes_\eps A \colon \CC^\alpha_{\s}(\R^d; E) \to \CC^\alpha_{\s}(\R^d; F)$.
\end{theorem}

\begin{remark}\label{rem:bilinear_difficulty}
Theorem~\ref{thm:ExtLinOp} allows us to extend scalar Schauder estimates to vector-valued H\"older--Besov spaces.
Extending scalar Young product estimates, or other bilinear\slash multilinear estimates, to ($m$-convex) algebra-valued H\"older--Besov spaces takes more care.
The main issue that arises is that the seminorm that lies at the heart of the definition of the H\"older--Besov spaces,
\begin{equ}
	\lambda^{-\alpha} \mfp \left(\xi \left( \CS_{\s,x}^\lambda \eta \right) \right) \; ,
\end{equ}
does not admit any submultiplicative equivalents for $\alpha < 0$, whereas for $\alpha > 0$, $ \max_{|k| \leqslant \lfloor \alpha \rfloor} \mfp( \partial^k \xi(x))$ and
\begin{equ}
	\sup_{\substack{x,y \in \R^d \\ x\neq y}}\frac{ \mfp \left( \partial^{\ell} \xi(x) - \partial^{\ell} \xi(y) \right)  }{|x- y|_{\s}^{\{\alpha\}}}
\end{equ}
for $|\ell| = \lfloor \alpha \rfloor$ and $\{\alpha\} = \alpha - \lfloor \alpha \rfloor$, are submultiplicative up to a constant.
See Theorem~\ref{thm:YoungMult} below.
However, the issue above is elegantly solved by regularity structures,
which allow to describe distributions using modelled distributions. Since these behave like 
regular functions, they do admit submultiplicative seminorms.
We give the statement of this Young product estimate below,
but defer its proof using regularity structures to Section~\ref{subsec:multiplication}.
\end{remark}

\begin{theorem}[Young Multiplication]
	\label{thm:YoungMult}
	Let $\alpha,\beta \in \R$ and $\CA$ be an $m$-convex algebra.
	Then there exists a continuous bilinear map $\CC^\alpha_\s(\R^d; \CA) \times \CC^\beta_\s(\R^d ; \CA) \to \CC^{\alpha \wedge \beta}_\s(\R^d ; \CA) $ extending pointwise multiplication if and only if $\alpha + \beta > 0$.
\end{theorem}

\begin{remark}
	Since we have chosen a definition for the H\"older--Besov spaces such that $\CC^0_{\s}(\R^d;\CA) \neq \cC(\R^d;\CA)$, we do not need the extra assumption that $\alpha \notin \N$ in the above statement.
\end{remark}

\subsection{Solving the \TitleEquation{\Phi^4_2}{d}-Equation for \TitleEquation{q}{q}-Mezdons}
\label{sec:Phi42}

As an instructive example, we now show how to apply the Da Prato--Debussche argument to prove local well-posedness for the mezdonic $\Phi^4_2$ dynamic.
Given $q \in (-1,1)$, the $q$-mezdonic $\Phi^4_2$-equation on $\R_+ \times \T^2$ is given by
\begin{equs}
	\label{eq:Phi42}
	\partial_t \phi = \Delta \phi - m^2 \phi - \lambda \phi^3 + \xi_q \; ,
\end{equs}
where $m^2 \geqslant 0$ and $\xi_q$ is the $2+1$-dimensional $q$-white noise, i.e.\ it is the mezdonic field for the algebra $\CA_q(L^2(\spacetime))$ with  $\spacetime \eqdef \R \times \T^2$. In particular, for all $(t,x), (s,y) \in \spacetime$
\begin{equ}
	\label{eq:Cov2D}
	\omega_q \left( \xi_q (t,x) \xi_q(s,y) \right) = \delta(t-s) \delta(x-y) \; .
\end{equ}
Using \eqref{eq:Cov2D} and a standard scaling argument, one can show that
\begin{equ}
	\xi_q \in \CC^{-2}_{\s}\left(\R^3; \cA_q\left(L^2(\spacetime)\right)\right) \subset \CC^{-2}_{\s}\left(\R^3; \CA_q\left(L^2(\spacetime)\right)\right)
\end{equ} 
where we have chosen the parabolic scaling $\s = (2,1,1)$. 
Using the Schauder estimate for %\footnote{Here $\CI$ always denotes the solution to the heat equation on $\spacetime$ with $0$ initial condition.} 
$\CI \eqdef (\partial_t - \Delta+m^2)^{-1}$, it follows that a potential solution $\phi$ can have at most the regularity of $\<1> \eqdef \CI(\xi_q) \in \CC_\s^0$ which is not sufficient to define the product $\phi^3$. However, if we expand $\phi = \upsilon + \<1>$, then \eqref{eq:Phi42} takes the form 
\begin{equ}
	(\partial_t - \Delta+m^2) (\upsilon + \<1>) = -\lambda \left( \upsilon^3 + \upsilon^2 \<1> + \upsilon\<1>\upsilon +  \<1>\upsilon^2+\upsilon \<1>^2 + \<1> \upsilon \<1> + \<1>^2 \upsilon + \<1>^3 \right) + \xi_q \; ,
\end{equ}
i.e.\
\begin{equ}[eq:RenEqDPD]
	(\partial_t - \Delta+m^2) \upsilon  = -\lambda \left( \upsilon^3 + \upsilon^2 \<1> + \upsilon\<1>\upsilon +  \<1>\upsilon^2+\upsilon \<1>^2 + \<1> \upsilon \<1> + \<1>^2 \upsilon + \<1>^3 \right) \; .
\end{equ}
If we were able to define $\cA_q \ni A \mapsto \<1> A \<1> $, $\<1>^2$, and $\<1>^3$ as (reasonable) 
versions of the usual product with regularities $0$, then $\upsilon$ would have at worst regularity $\CC^2_{\s}$, leaving the remaining products well-defined.

The usual way this is done for $\<1>^2$ and $\<1>^3$ is using Wick ordering introduced in Section~\ref{sec:WickRenorm}. Let $K$ be the integration kernel of $\CI$ and $K_z(y) \eqdef K(y-z)$. We define the two $0$-Fock-space-valued distributions
\begin{equs}
	\mfZ[\<2>](f) &\eqdef \int\limits_{\spacetime} f(z) K_z \otimes K_z \,  \d z \in \cD'\Bigl( \spacetime ;  L^2(\spacetime)^{\wotimes_\alpha 2} \Bigr)\;, \\
	\mfZ[\<3>](f) &\eqdef \int\limits_{\spacetime} f(z) K_z \otimes K_z \otimes K_z \,  \d z \in \cD'\Bigl( \spacetime ;  L^2(\spacetime)^{\wotimes_\alpha 3} \Bigr)\;.
\end{equs}
By the scaling properties of $K$, it is straightforward to check, e.g.\ \cite[Proposition~A.11]{CHP23}, that %\ajay{Missing citation}
\begin{equs}
	\mfZ[\<2>] \in \CC_\s^0 \Bigl( \spacetime ;  L^2(\spacetime)^{\wotimes_\alpha 2} \Bigr) \; \text{ and } \; \mfZ[\<3>] \in \CC_\s^0 \Bigl( \spacetime ;  L^2(\spacetime)^{\wotimes_\alpha 3} \Bigr) \; ,
\end{equs}
and thus we can define
\begin{equs}
	\<2> &\eqdef \xi_q^{\diamond 2} \circ \mfZ[\<2>] \in \CC^0_{\mfs} ( \spacetime ; \cA_q(\spacetime)) \; , \\
	\<3> &\eqdef \xi_q^{\diamond 3} \circ \mfZ[\<3>] \in \CC^0_{\mfs} ( \spacetime ; \cA_q(\spacetime)) \; .
\end{equs}
These are the renormalisations of the products $\<1>^2$ and $\<1>^3$. In particular, if we introduce a suitable smooth mollification $\<1>_\eps$ of $\<1>$, then one can show that, in $\CC^0_\s$,
\begin{equs}
	\<2> = \lim_{\eps \downarrow 0} \left( \<1>^2_\eps - \omega_q \left( \<1>^2_\eps \right) \right) \; \text{ and } \; \<3> = \lim_{\eps \downarrow 0} \left( \<1>^3_\eps - (2+q) \omega_q \left( \<1>^2_\eps \right) \<1>_\eps \right) \; . 
\end{equs} %\ajay{Just a note - I think in the current set up here, $\omega_q(\<1>^2_\eps)$ depends on time since the stochastic objects are not stationary in time. }
\begin{remark}
Note that for fixed $\eps > 0$, $\lim_{\eps \downarrow 0} \omega_q(\<1>^2_\eps)$ is independent of $q$ for $q \in (-1,1]$, in particular it matches its value in the bosonic case. 
\end{remark}

This takes care of defining the terms $\<1>^3$, $\<1>^2 \upsilon$, and $\upsilon\<1>^2$ in \eqref{eq:RenEqDPD}, by replacing them with $\<3>$, $\<2>\upsilon$, and $\upsilon\<2>$ respectively. 
However, for the last term $\<1>\upsilon\<1>$, we need the renormalised product map defined in \eqref{eq:RenProdDef},
which yields a map $\CM^{(1,1)}_q \left( {\mfZ[\<2>]} ; \, \bigcdot \, \right) \colon \CC_\s^2(\spacetime ; \cA_q) \to \CC_\s^0(\spacetime ; \cA_q)$ by Lemma~\ref{lemma:2SYoung} below.

To be more precise, given $\upsilon \in \cD(\spacetime) \otimes \cA_q(L^2(\spacetime))$ of the form
$
	\upsilon = \sum_{i = 1}^n f_i \otimes a_i
$,
with $f_i \in \cD(\spacetime)$ and $a_i \in \cA_q$, let
\begin{equs}
	\<2_1Rprime>(x) \eqdef \sum_{i = 1}^n f_i(x) \CM^{(1,1)}_q  \left( \mfZ[\<2>](x) ;  a_i \right) \; .
\end{equs}
Since the maps $ L^2(\spacetime)^{\wotimes_\alpha 2}  \ni F \mapsto \CM^{(1,1)}_{q}(F; a_i)$ for $i \in [n]$ are continuous, it follows that $\CM^{(1,1)}_q \left( \mfZ[\<2>](x) ; a_i \right) \in \CC^0_\s(\spacetime ; \cA_q)$.
Thus\footnote{Any commutative $C^{\star}$-algebra $\mathcal{A}$ can be realised as an $C^{\star}$ algebra of complex valued functions on a compact Hausdorff space, the latter algebra is often called the Gelfand representation of $\mathcal{A}$. In our case our algebra is not a Banach algebra, but this doesn't matter since we can construct the representation explicitly.}, the pointwise in $x$ multiplication with $f_i$ above can be interpreted using the classical Young product estimate for scalar-valued functions and $\upsilon \mapsto \<2_1Rprime>$ is well-defined and in $\CC^0_\s(\spacetime; \cA_q)$.
In particular, we have the following lemma.
\begin{lemma}
\label{lemma:2SYoung}
	For every $\alpha > 0$ the map $\upsilon \mapsto \<2_1Rprime>$ extends to a continuous map
	\begin{equ}
		\CC_\s^\alpha \left(\spacetime ; \cA_q(L^2(\spacetime)) \right) \longrightarrow \CC_\s^0 \left(\spacetime ; \cA_q(L^2(\spacetime)) \right)\;.
	\end{equ}
\end{lemma}
\begin{remark}
	There are several ways of proving Lemma~\ref{lemma:2SYoung}, for example by extending the usual paraproduct inequalities \cite{GIP15} to the algebra-valued setting or by proving that this map is continuous with respect to the injective topology on the tensor product.
	We present a proof in Appendix~\ref{sec:NonComYng} based on the regularity structures proof of the Young multiplication theorem, using the material from Sections~\ref{sec:Reconstruction}~\&~\ref{sec:RS}.
\end{remark}

\begin{remark}
It will be clear from our proof that Lemma~\ref{lemma:2SYoung} continues to hold if we replace 
$\mfZ[\<2>]$ by an arbitrary distribution $F \in \CC^\beta_\s \bigl( \R^d  ; L^2(\R^d)^{\wotimes_\alpha 2}\bigr)$ and $\upsilon \in \CC^\alpha_\s \bigl( \R^d  ; \cA_q\bigl(L^2(\R^d) \bigr) \bigr)$ as long as $\alpha + \beta > 0$.
\end{remark}

The final renormalised equation for $\upsilon$ reads
\begin{equ}[eq:RenEqDPD2]
	(\partial_t - \Delta+m^2) \upsilon  = -\lambda \left( \upsilon^3 + \upsilon^2 \<1> + \upsilon\<1>\upsilon +  \<1>\upsilon^2+\upsilon \<2> + \<2_1Rprime> + \<2> \upsilon + \<3> \right) \; .
\end{equ}
One possible way to formulate the solution theory for this equation is the following. For a more detailed exposition of the same argument in a similar style, cf.\ \cite[Section~5]{CHP23}.

We fix a sufficiently small $\kappa \in (0,\frac{1}{2})$. We consider a ``generic'' set of driving noises $\<1>,\<2>,\<3>$ for the equation, i.e.\ an arbitrary element $\Xi = \bigl(\<1>,\<2>,\<3>\bigr) \in \CN$
with
\begin{equ}
	\CN \eqdef \CC^{-\kappa,-\kappa}_\s (\spacetime ; \cA_q^{(1)}) \times \CC^{-2\kappa,-2\kappa}_\s (\spacetime ; \cA_q^{(2)}) \times \CC^{-3\kappa,-3\kappa}_\s (\spacetime ; \cA_q^{(3)}) \; .
\end{equ} 
\begin{remark}
	The choice of space $\CC_\s^{\alpha,\alpha}$, rather than $\CC^{\alpha}_\s$, was made because we need to multiply the driving noises by the temporal cut-off $\bone_{\{t \geqslant 0\}} \in \CC_{\s}^{\infty,0}(\spacetime)$ with their product belonging to $\CC^{\alpha,\alpha}_{\s}$.
\end{remark}

\begin{theorem}\label{thm:Phi42Ren}% \ajay{Disentangle statement into two statements: local existence statement and global existence for small coupling/initial data.}
	Let $\lambda, m \in \R$ and $\Xi \in \CN$. For any initial condition
	\begin{equ}
		\upsilon_0 \in \reminit \eqdef \CC^{-\kappa}(\T^2 ; \CA_q)
	\end{equ}
	the equation
	\begin{equs}[eq:Phi42Ren]
		\begin{cases}
			(\partial_t - \Delta+m^2) \upsilon  = -\lambda \left( \upsilon^3 + \upsilon^2 \<1> + \upsilon\<1>\upsilon +  \<1>\upsilon^2+\upsilon \<2> + \<2_1Rprime> + \<2> \upsilon + \<3> \right) \\
			\upsilon(0 , \, \bigcdot \, ) = \upsilon_0
		\end{cases}
	\end{equs}
	has a mild solution in $\CC_T \eqdef \CC^{2-3\kappa,-\kappa}_{\s}([0,T]\times \T^2 ; \cA_q (L^2(\spacetime)))$ for $0 < T = T(\upsilon_0 , \Xi, m , \lambda) \leqslant 1$. Furthermore, the solution is unique on the temporal interval $[0,T]$ and for $\delta_1, \delta_2 > 0$ small enough, the solution map is jointly continuous on the balls of radii $\delta_1$ and $\delta_2$ around $\upsilon_0$ and $\Xi$, respectively
	\begin{equ}
		\begin{aligned}
		\CS^T \colon \reminit \times \CN \supset B_{\delta_1}(\upsilon_0) \times B_{\delta_2}(\Xi) & \longrightarrow \CC_{T}\\
		(\upsilon_0', \Xi') & \longmapsto u' \;.
		\end{aligned}
	\end{equ}
\end{theorem}
%\martinp{Add a sketch of proof, check exact exponents and initial condition}

\begin{remark}
	With respect to a suitable mollification mentioned above, one can also show that
	$
		\<2_1Rprime> \eqdef \lim_{\eps \downarrow 0} \bigl( \<1>_\eps \upsilon \<1>_\eps - \omega_q(\<1>_\eps^2) \Delta_q(\upsilon) \bigr)
	$,
	with convergence taking place in the $\CC^0_\s$-topology. Therefore, one can also view $\phi = \<1> + \upsilon$, where $\upsilon$ is the solution to \eqref{eq:Phi42Ren}, as solving the equation
	\begin{equ}
		(\partial_t - \Delta + m^2) \phi  = - \lambda \phi^3 + \omega_q \left( \<1>_\eps^2 \right) \lambda (2+\Delta_q) \phi  + \xi_\eps
	\end{equ}
	in the limit as $\eps \downarrow 0$. (Recall that $\Delta_q$ was defined in \eqref{e:Deltaqdef}.) 
\end{remark}

\begin{proof}[Sketch of Proof]
	 We start by rewriting the equation in its mild form, i.e.\ 
	 \begin{equ}
		\upsilon = \CF(\upsilon) = G \upsilon_0 - \lambda \CI \left( \upsilon^3 + \upsilon^2 \<1> + \upsilon\<1>\upsilon +  \<1>\upsilon^2+\upsilon \<2> + \<2_1Rprime> + \<2> \upsilon + \<3> \right) \; . 
	 \end{equ}
	 Here $G$ denotes the action of the heat kernel on the initial condition mapping $\CC_\s^{-\kappa} (\T^2 \; \CA_q) \to \CC_\s^{\infty, -\kappa} ([0,1] \times \T^2 \; \CA_q)$. $\CI$ denotes multiplication with $\1_{\R_+}$ and then applying $(\partial_t  - \Delta + m^2)^{-1}$. In particular, for every $T \in (0,1]$ 
	 \begin{equs}
		\CI \colon \CC^{-3\kappa, -3\kappa }_{\s} ([0,T] \times \T^2 ; \cA_q) &\longrightarrow \CC^{2-3\kappa, 2-3\kappa }_{\s} ([0,T] \times \T^2 ; \cA_q) \hooklongrightarrow \\
		& \hooklongrightarrow \CC^{2-3\kappa, -\kappa }_{\s} ([0,T] \times \T^2 ; \cA_q)
	 \end{equs}
	 with norm $\propto T^{1-\kappa}$ which provides the contractive factor for $T$ small enough. Thus, we only need to show that 
	 \begin{equ}
		F_{\lambda}(\upsilon) = \lambda \left( \upsilon^3 + \upsilon^2 \<1> + \upsilon\<1>\upsilon +  \<1>\upsilon^2+\upsilon \<2> + \<2_1Rprime> + \<2> \upsilon + \<3> \right)
	 \end{equ}
	 is a well-defined map $\CC^{2-3\kappa, -\kappa}_\s \to \CC^{-3\kappa, -3\kappa}_\s$. One can easily check that all the multiplications are well-defined by Theorem~\ref{thm:YoungMult} and Lemma~\ref{lemma:2SYoung}, since the sum of all the regularities is positive. Here, we have used the extensions of those theorems to spaces with blow-ups
	 \begin{equ}
		\CC_{\s}^{\alpha, \eta} \times \CC_{\s}^{\beta, \eta'} \longrightarrow \CC_{\s}^{\alpha \wedge \beta, \eta+\eta'}
	 \end{equ}
	 for $\eta \leqslant \alpha$, $\eta' \leqslant \beta$ and $\eta+\eta' > - 2$, cf.\ \cite[Theorem~A.20]{CHP23}. Furthermore, the sums of the blow-up regularities are also always $\geqslant -3\kappa$. Thus, $\CF$ maps $\CC_T$ into itself for all $T \in (0,1]$.

	 The proof now proceeds by the usual arguments. Using the polynomial structure of $F_{\lambda}$, there exists a constant $C$, s.t.\ for all  $R>1$, $F_{\lambda}$ maps the ball of radius $R$ around $0$ in $\CC^{2-3\kappa, -\kappa}_\s$ to the ball of radius $CR^3$ around $0$ in $\CC^{-3\kappa, -3\kappa}_\s$ and is Lipschitz continuous with constant $\leqslant CR^2$ for all $T \in (0,1]$. Then, restricting to $[0,T]$ for $T = T(R)$ sufficiently small, $\CI \circ F_{\lambda}$ maps $B_R(0)$ into $B_{\frac{1}{2}}(0)$. Thus, choosing $R = \| G \upsilon_0 \|_{\CC_T} + 1$, $\CF$ maps $B_R(0)$ into itself and is a contraction satisfying
	 \begin{equ}
		\left\| \CF(\upsilon) - \CF(\upsilon')\right\|_{\CC_T} \leqslant \frac{1}{2}\left\| \upsilon - \upsilon'\right\|_{\CC_T} \; .
	 \end{equ}
	 The Banach fixed-point theorem now finishes the proof of the assertion.
\end{proof}

In fact, since our solution is necessarily uniformly bounded in $\cA_q$, we can show that, for $\lambda$ and $\| \upsilon_0\|_{\CC^{-\kappa}}$ small enough, it exists for all times. 
\begin{theorem}\label{thm:Phi42RenGl} %\ajay{Disentangle statement into two statements: local existence statement and global existence for small coupling/initial data.}
	Fix $m^2 > 0$ and $C>0$. There exists a constant $\lambda_0(m^2, C) > 0$, such that, for all $\lambda \in \R$ with $|\lambda| < \lambda_0(m^2, C)$ and $\upsilon_0 \in \CC^{-\kappa}_\s ( \T^2 ; \cA_q(L^2(\spacetime)))$ with $\|\upsilon_0 \|_{\CC^{-\kappa}} <  C$, equation~\eqref{eq:Phi42Ren} has a mild solution in $\CC^{2-3\kappa,-\kappa}_{\s}(\R_+ \times \T^2 ; \cA_q (L^2(\spacetime)))$.
\end{theorem}
\begin{proof}[Sketch of Proof]
	The main technical difference to the previous theorem is that the norm of $(\partial_t - \Delta +m^2)^{-1}$ as an operator $\CC^{\alpha,\eta}(\R_+ \times \T^2) \to \CC^{\alpha+2,\eta+2}(\R_+ \times \T^2)$ is $\lesssim m^{-2}$ whereas $(\partial_t - \Delta)^{-1}$ is unbounded. This means that we can view $\CI$ as an operator
	\begin{equs}
		\CI \colon \CC^{-3\kappa, -3\kappa }_{\s} (\R_+ \times \T^2 ; \cA_q) &\longrightarrow \CC^{2-3\kappa, -\kappa }_{\s} ( \R_+ \times \T^2 ; \cA_q)
	 \end{equs}
	 of fixed norm. Thus, for $\lambda$ small enough, rather than $T$ small enough, $\CI \circ F_{\lambda}$ maps $B_{R}(0)$ into $B_{\frac{1}{2}}(0)$. The theorem now follows by setting $R = C \|G\| + 1$ and choosing $\lambda_0(m^2,C) > 0$ accordingly.
\end{proof}

\section{Reconstruction}
\label{sec:Reconstruction}

In this section, we use the approach of \cite{CZ20} \dash which is well-adapted to a tensor product argument \dash to give a very general version of the reconstruction theorem that holds for arbitrary Fr\'echet spaces.

\begin{definition}[\TitleEquation{(\alpha,\gamma)}{(\alpha,\gamma)}-Smooth Section]
	Let $E$ be a Fr\'echet space. Consider the (trivial) vector bundle $B = \R^d \times \cD'(\R^d; E) \to \R^d$. Fix $\alpha,\gamma \in \R$, such that $\alpha \leqslant 0 \wedge \gamma$ and $r = - \lfloor \alpha \rfloor$. For any $\mfp \in \mfP$ and $\mfK \Subset \R^d$, we define the following seminorm on measurable sections $(F_x)_x \in \Gamma(B)$: 
	\begin{equs}
		\| F \|_{\Gamma^{\alpha,\gamma} ; \mfK , \mfp } &\eqdef \sup_{\eta \in \CB^r_{\s,0}} \sup_{\lambda\in (0,1]} \sup_{x \in \K} \frac{\mfp \left( F_x  \left( \CS^\lambda_{\s,x} \eta\right) \right) }{ \lambda^{ \alpha}  } +\\
		& \qquad \qquad + \sup_{\eta \in \CB^r_{\s,0}} \sup_{\lambda\in (0,1]} \sup_{\substack{x,y \in \K \\ x \neq y}} \frac{\mfp \left( \left(F_y - F_x\right) \left( \CS^\lambda_{\s,x} \eta\right) \right) }{ \lambda^{ \alpha} \left( |y-x|_{\s}  + \lambda\right)^{\gamma - \alpha} }\;.
	\end{equs}
	We define the space of $(\alpha,\gamma)$-smooth sections of $B$ to be
	\begin{equ}
		\Gamma^{\alpha,\gamma}_{\s}(\R^d; E) \eqdef \left\{ F \in \Gamma(B) \, \middle| \, \forall \mfp \in \mfP, \; \forall \mfK \Subset \R^d,\; \| F \|_{\Gamma^{\alpha,\gamma}; \K , \mfp} < \infty\right\}
	\end{equ}
	equipped with the topology defined by these seminorms.
\end{definition}

\begin{remark}
	The requirement that $E$ be a Fr\'echet space is a technicality to ensure that the space $\Gamma^{\alpha,\gamma}_{\s}(\R^d;E)$ is complete. It comes about from this Regularity Structures-free formulation of $(\alpha,\gamma)$-smooth sections, as we must allow sections that are merely measurable and not necessarily continuous. If we instead used a definition based on regularity structures, the sections would have to be continuous, and we could drop the restriction that $E$ be a Fr\'echet space.

	However, we note that whenever we consider random models, e.g.\ when we are trying to solve a mixed commutative-noncommutative SPDE, we have to restrict ourselves to a Fr\'echet space anyway, for the same technical reasons concerning measurability as here.
\end{remark}

\begin{remark}
	The space $\Gamma^{\alpha, \gamma}_\s (\R^d ; E)$ is locally $t$-convex over $E$, with the 
	collection of seminorms indexed by $\mfp$ being $ \left( \| \bigcdot \|_{\Gamma^{\alpha,\gamma} ; \K , \mfp}\right)_{\K \Subset \R^d}$.
\end{remark}

We now state that $\Gamma_{\s}^{\alpha,\gamma}(\R^d ; E)$ is a tensor product of the scalar-valued function space and the target space in the same way that $\CC^\alpha_\s(\R^d ; E)$ was stated to be in Theorem~\ref{thm:PrimTensProd},
this will allow us to extend the standard reconstruction theorem out setting.
The proof of the following proposition can be found in Appendix~\ref{subsec:PrimTensProdproof}.
\begin{proposition}
	\label{prop:TensProdPropSec}
	For all Fr\'echet spaces $E$, $\gamma \in \R$ and $\alpha \leqslant 0 \wedge \gamma$, the following spaces are isomorphic
	\begin{equ}
		\Gamma^{\alpha,\gamma}_{\s}(\R^d; E) \cong \Gamma^{\alpha,\gamma}_{\s}(\R^d; \mathbb{K}) \wotimes_\eps E \; .
	\end{equ}
\end{proposition}

\begin{theorem}[Reconstruction]
	There exists a $t$-continuous linear map \[\CR \colon \Gamma^{\alpha,\gamma}_{\s}(\R^d; E) \to \CC^\alpha_\s(\R^d; E)\] such that, for all $\eta \in \CB^{r}_{\s, 0}$, all $x \in \K \Subset \R^d$, and all $\lambda \in (0,1]$,
	\begin{equ}
		\mfp \Bigl( \bigl(\CR F- F_x\bigr) \bigl( \CS^\lambda_{\s, x} \eta \bigr) \Bigr) \lesssim \| F \|_{\Gamma^{\alpha,\gamma} ; \overline{\K} , \mfp} \begin{cases}
			\lambda^\gamma \; , \qquad &\text{if } \gamma \neq 0\;,\\
			1 - \log(\lambda) \; , & \text{if } \gamma = 0\;,
		\end{cases} 
	\end{equ}
	where $\overline{\K} = \left\{ x \in \R^d \, \big| \, d_{\s}(x, \K) \leqslant 4 \right\}$.
	If furthermore $\gamma > 0$, then this map is unique.
\end{theorem}
\begin{proof}
	This again follows from the scalar reconstruction theorem \cite{Hai14,CZ20} and  Metatheorem~\ref{MetaThm}, more specifically \cite[Proposition~43.6]{Trev67}.
\end{proof}

\section{\TitleEquation{\CA}{A}-Regularity Structures}
\label{sec:RS}

Most of this section is concerned with adapting the original analytic theory of $\mathbb{K}$-valued regularity structures to our $\CA$-valued setting.
We introduce definitions extending those in \cite{Hai14} to our new setting; these are designed so that the results of \cite[Sections~2 to 7]{Hai14} generalise straightforwardly.

In particular, we will only state the most important theorems explicitly and provide a list of all other results that continue to hold in the $\CA$-valued setting.
Whenever a proof of a result in the $\CA$-valued setting follows by a simple substitution of the norm of $\R$ with one of the norms $\mfp \in \mfP$, we will not provide a proof.

In this section, we work with an abstract notion of an $\CA$-regularity structure (see the definition below).
In Section~\ref{sec:TreeRegStruc} we will specialise to the much more concrete case of $\mathcal{A}$-regularity structures generated by trees. This case will be the most relevant one when it comes to solving SPDEs.

\begin{definition}\label{def:reg_struct}
	An $\CA$-regularity structure $\cT = (A,T,G)$ is a triple consisting of
	\begin{itemize}
		\item an index set $A \subset \R$ bounded from below, containing $0$, which is locally finite,
		\item a model space $T$, which is a graded $\CA$-bimodule
		\begin{equ}
			T = \bigoplus_{\alpha \in A} T_\alpha
		\end{equ}
		where each $T_\alpha$ is a complete, locally $m$-convex $\CA$-bimodule with seminorms $ \left( \mfp_\alpha \right)_{\mfp \in \mfP} $, see Definition~\ref{def:tmconvex}. Further, we suppose that $T_0 \cong \CA$ with its unit denoted by $\1$ and its dual vector denoted by $\1^*$,
		\item a structure group $G$ of $t$-continuous, invertible, $\CA$-bimodule morphisms $T  \to T$, such that, for every $\Gamma \in G$, every $\alpha \in A$, and every $a \in T_\alpha$, one has
		\begin{equ}
			\Gamma a - a   \in \bigoplus_{\beta < \alpha} T_\beta \; .
		\end{equ}
		Furthermore, $ \Gamma \big|_{T_0} \equiv \bone_{T_0}$.
	\end{itemize}

	If $\CA$ is equipped with a grading $I$, then we assume that each space $T_\alpha$ is also $I$-graded and that the homomorphisms in $G$ are of degree $0$.

	For $\alpha \in A$, let $\CQ_\alpha \colon T \to T_\alpha$ be the projection onto $T_\alpha$.
	We further define, for $\alpha \in \R$, the subspaces
	\begin{equ}
		T_\alpha^-  \eqdef\bigoplus_{\gamma < \alpha} T_\gamma, \qquad T_\alpha^+ \eqdef \bigoplus_{\gamma \geqslant \alpha} T_\gamma \; .
	\end{equ}
\end{definition}

\begin{remark}
	By abuse of notation we shall denote $\mfp_\alpha(\CQ_\alpha a)$ by $\mfp_\alpha(a)$.
\end{remark}

\begin{definition}
	Given an $\CA$-regularity structure $\cT = \left( A, T ,G \right)$ and $\alpha \leqslant 0$, a sector $V$ of regularity $\alpha$ is a graded subspace $V = \bigoplus_{\beta \in A} V_\beta$ with $V_\beta \subset T_\beta$,  such that:
	\begin{itemize}
		\item for every $\beta \in A$, $V_\beta$ is closed and complemented\footnote{This means that there exists a closed subspace $W_\beta \subset T_\beta$, such that $V_\beta \oplus W_\beta = T_\beta$} in $T_\beta$.
		\item for $\beta < \alpha$, $V_\beta = \{0\}$,
		\item $G\cdot  V  \subset V$,
	\end{itemize}
	A sector of regularity $0$ is called function-like.
\end{definition}

\begin{definition}\label{def:model}
	A model $Z = (\Pi,\Gamma)$ for an $\CA$-regularity structure $\cT = (A,T,G)$ on $\R^d$ with scaling $\s$ is given by:
	\begin{itemize}
		\item a map $\Gamma \colon \R^d \times \R^d \rightarrow G$, the structure group of $\cT$,  such that $\Gamma_{xx} = \bone_T$, the identity operator, and for all $x,y,z \in \R^d$
		\begin{equ}
			\Gamma_{xy} \Gamma_{yz} = \Gamma_{xz} \; ,
		\end{equ}
		\item a collection $\left( \Pi_x \right)_{x \in \R^d}$ of $t$-continuous bimodule morphisms 
		\begin{equ}
		\Pi_x \colon T \rightarrow \cD'(\R^d;\CA)
		\end{equ} 
		such that, for all $x,y \in \R^d$,
		\begin{equ}
			\Pi_y = \Pi_x  \circ \Gamma_{xy} \; .
		\end{equ}
		\item If $\CA$ is $I$-graded, then we also require that the maps $\Pi_x$ are of degree $0$.
	\end{itemize}

	Furthermore, we require the following analytic bounds: For every seminorm $\mfp \in \mfP$, every $\gamma > 0$, and every compact set $\K \Subset \R^d$, there exist some smallest constants $\|\Pi\|_{\gamma; \K, \mfp}, \|\Gamma\|_{\gamma; \K,\mfp}$ such that, for all $x,y \in \K$, all $\ell \in A$ with $\ell < \gamma$, and all $a \in T_\ell$,
	\begin{equ}
	\label{ModDefRel}
		\mfp\Bigl( \bigl( \Pi_x a \bigr) \bigl( \CS^\lambda_{\s,x} \eta \bigr) \Bigr) \leqslant \|\Pi\|_{\gamma; \KK,\mfp} \, \mfp_{\ell }(a) \lambda^\ell, \quad \mfp_m\left( \Gamma_{xy} a \right)
		\leqslant \|\Gamma\|_{\gamma; \KK,\mfp} \, \mfp_{\ell }(a) |x-y|^{\ell-m}_{\s}
	\end{equ}
	for all $\lambda \in (0,1]$, all $\eta \in \CB^r_{\s, 0}$, all $m < \ell$, where $r$ is the smallest natural number such that $r > -\min A$.

	Finally, we introduce the following ``seminorms'' on the space of models $Z = (\Pi, \Gamma)$
	\begin{equs}
		\vvvert Z \vvvert_{\gamma ; \K , \mfp} &\eqdef \| \Pi \|_{\gamma ; \K , \mfp} + \| \Gamma \|_{\gamma ; \K , \mfp} \; ,\\
		\vvvert Z ; \overline{Z} \vvvert_{\gamma ; \K , \mfp} &\eqdef \| \Pi - \overline{\Pi} \|_{\gamma ; \K , \mfp} + \| \Gamma  - \overline{\Gamma} \|_{\gamma ; \K , \mfp} \;,
	\end{equs}
	which are restrictions of genuine seminorms to the nonlinear space of models.
\end{definition}

\begin{remark}
	We will also use the following notations. If $V$ is a sector of $\cT$, we write $\| \Pi \|_{V, \gamma ; \K, \mfp}$ and $\| \Gamma \|_{V, \gamma ; \K, \mfp}$ for the restrictions of the two norms to $V$. If $V_\alpha \neq 0$ only for a finite subset of $A$, we denote by $\| \Pi \|_{V ; \K, \mfp}$ and $\| \Gamma \|_{V ; \K, \mfp}$ the supremum of the norms over all $\gamma \in \R$.

	For $\alpha \in A$, we also set
	\begin{equ}
		\| \Pi \|_{T_\alpha ; \K , \mfp} \eqdef \sup_{x \in \K} \sup_{\lambda \in (0,1]} \sup_{\eta \in \CB^r_{\s, 0}} \sup_{\substack{a \in T_\alpha \\ \mfp_\ell(a) \leqslant 1 }} \mfp \left( \bigl(\Pi_x a\bigr) \bigl( \CS^\lambda_{\s, x} \eta \bigr) \right) \; .
	\end{equ}
	We note here that $\| \Pi \|_{T_\alpha  ; \K, \mfp}$ is \textit{a priori} completely independent of $\| \Pi \|_{\alpha ; \K, \mfp}$.
\end{remark}

\begin{remark}
	Most of the analytic heavy lifting is done by the assumption that $\Pi_x$ is continuous! For example, in the regularity structure one can build to solve the mezdonic $\Phi^4_2$-equation, continuity means that the three operator insertions in $\<2>$ must be continuous, which is a non-trivial constraint, as we saw throughout Section~\ref{sec:qmezdons}. 
\end{remark}

Before we turn to the next example, we introduce the notation $\spn_{\CA} B$, for $B$ an arbitrary
set, as a shorthand for $\bigoplus_{x\in B}\CA$, endowed with the $\CA$-bimodule structure given 
by pointwise multiplication. Elements of $\spn_{\CA} B$ will typically we written as 
finite $\CA$-linear combinations of elements of $B$ so that, for $x \in B$ and
$a,b,c\in \CA$, one has $a\cdot (bx)\cdot c = (abc)x$. 

\begin{example}[Polynomial Regularity Structure]
	\label{ex:PolMod}
	Given a set $X_1, \dots, X_d$ of abstract commuting variables, the canonical regularity structure $\cT_{d,\s}$ is given by
	\begin{itemize}
		\item $A \eqdef \N$,
		\item $T_n \eqdef \spn_{\CA} \left\{ X^k \,\middle|\, k \in \N^d :  |k|_{\s} = n \right\}$ and $T_{d,\s} \eqdef \bigoplus_{n \in \N} T_n$,
		\item $G \eqdef \R^d$ with $h \in G$ acting via $X  \mapsto (X+h)$ extended multiplicatively and $\CA$-linearly.
	\end{itemize}
	If $\CA$ is $I$-graded then $T$ naturally inherits an $I$-grading by postulating that all $X_i$ have degree $0$.

	The standard model for the polynomial regularity structure $\cT_{d,\s}$ is given by the 
	$\CA$-bimodule morphism such that
	\begin{equ}
		\bigl( \Pi_x X^k \bigr)(y) = (y-x)^k \1_{\CA}\;, \qquad \Gamma_{xy} P(X) = P(X+y-x)\;,
	\end{equ}
	for any abstract polynomial $P(X) \in \cT_{d,\s}$.
\end{example}

Other definitions from \cite[Section~2]{Hai14} generalise in the obvious way. In particular, these include Definitions~2.22~\&~2.28.

\subsection{Modelled Distributions}

\begin{definition}[Modelled Distribution]
	Fix an $\CA$-regularity structure $\cT$, a $\gamma \in \R$, and let $Z = (\Pi,\Gamma)$ be a model. We define the space of modelled distribution $\CD^\gamma(\Gamma)$ with respect to to $Z$, to be the space of functions $f \colon \R^d \to T^-_\gamma$, such that, for every $\K \Subset \R^d$ and $\mfp \in \mfP$,
	\begin{equ}
		\vvvert f \vvvert_{\gamma; \K , \mfp} \eqdef \sup_{x \in \K} \sup_{\beta< \gamma} \mfp_{\beta}\left( f(x) \right) + \sup_{\substack{x,y \in \K \\ |x-y|_\s \leqslant 1}} \sup_{\beta < \gamma} \frac{\mfp_{\beta} \left(f(x) - \Gamma_{xy}f(y)\right)  }{|x-y|_\s^{\gamma-\beta}} < \infty \; ,
	\end{equ}
	where here and in the following we take $\mfp_\beta(f(x))$ to mean $\mfp_\beta(\CQ_\beta f(x))$ by abuse of notation.

	For a second model $\overline{Z} = \bigl( \overline{\Pi}, \overline{\Gamma}  \bigr)$ and $\bar f \in \CD^\gamma(\overline{\Gamma})$ we define the distance function
	\begin{equs}
		\vvvert f ; f' \vvvert_{\gamma; \K, \mfp} &  \eqdef \sup_{x \in \K} \sup_{\beta< \gamma} \mfp_{\beta}\left( f(x) - \bar f(x) \right) + \\
		& \qquad + \sup_{\substack{x,y \in \K \\ |x-y|_\s \leqslant 1}} \sup_{\beta < \gamma} \frac{\mfp_{\beta} \left(f(x) - \bar f(x) - \Gamma_{xy}f(y) + \overline{\Gamma}_{xy} \bar f(y) \right)  }{|x-y|_\s^{\gamma-\beta}}  \; .
	\end{equs}
	If $f$ takes values in a sector $V$ of regularity $\alpha$, we will sometimes denote this by 
	writing $f \in \CD^\gamma_\alpha(V)$ or simply $f \in \CD^\gamma_\alpha$.
\end{definition}

\begin{proposition}
	Fix an $\CA$-regularity structure $\cT = (A, T ,G)$, let $\alpha = \min A \leqslant 0$, and $\gamma \geqslant \alpha$. Let $Z = (\Pi,\Gamma)$ be a model and $f \in \CD^\gamma(\Gamma)$. Then $\left( \Pi_xf(x) \right)_{x \in \R^d}$ is an $(\alpha,\gamma)$-coherent section of $\Gamma^{\alpha,\gamma}_{\s}(\R^d ; \CA)$ and the assignment $f \mapsto \left( \Pi_x f(x)\right)_{x \in \R^d}$ is $t$-continuous, i.e.\
	\begin{equ}
		\left\| \left( \Pi_x f(x)\right)_{x \in \R^d} \right\|_{\Gamma^{\alpha,\gamma} ; \K, \mfp} \lesssim_{\cT, Z, \gamma} \vvvert f\vvvert_{\gamma;\mfK,\mfp}  \; .
	\end{equ}
	Furthermore, for a second model $\overline{Z} = \bigl(\overline{\Pi},\overline{\Gamma} \bigr)$ and $\bar f \in \CD^\gamma \bigl( \overline{\Gamma} \bigr)$
	\begin{equ}
		\left\| \left( \Pi_x f(x) - \overline{\Pi}_x \bar{f}(x)\right)_{x \in \R^d} \right\|_{\Gamma^{\alpha,\gamma}; \K , \mfp} \lesssim_{\cT, Z, \overline{Z},\gamma} \vvvert f ; \bar{f} \vvvert_{\gamma;\mfK,\mfp} + \vvvert Z ; \overline{Z} \vvvert_{\gamma ; \K, \mfp} \; .
	\end{equ}
\end{proposition}
\begin{proof}
	Fix $\K \Subset \R^d$ and $\mfp \in \mfP$. Then for all $x \in \mfK$, $\lambda \in (0,1]$, $\eta \in \CB^r_{\s, 0}$
	\begin{equs}
		\mfp \left( \left(\Pi_x f(x)\right) \left( \CS^\lambda_{\s, x} \eta\right) \right) &\leqslant \left\| \Pi \right\|_{\gamma ; \K , \mfp } \sum_{\beta < \gamma} \mfp_{\beta}\left( \Pi_x f(x) \right) \lambda^\beta  \leqslant \\
		& \leqslant  \left\| \Pi \right\|_{\gamma ; \K , \mfp } \vvvert f\vvvert_{\gamma; \mfK, \mfp} \sum_{\beta < \gamma} \lambda^\beta  \lesssim  \\
		& \lesssim_{A,\gamma}   \left\| \Pi \right\|_{\gamma ; \K , \mfp } \vvvert f\vvvert_{\gamma; \mfK, \mfp}  \lambda^\alpha\;,
	\end{equs}
	and, for all $y \in \K$, $y \neq x$,
	\begin{equs}
		\mfp \left( \left( \Pi_y f(y) - \Pi_x f(x)\right) \left( \CS^\lambda_{\s,x} \eta\right) \right)
		& = \mfp \left( \left( \Pi_x  \left( \Gamma_{xy} f(y) - \Pi_x f(x)\right) \right) \left( \CS^\lambda_{\s,x} \eta\right) \right) \leqslant  \\
		& \leqslant \| \Pi\|_{\gamma; \K, \mfp}  \sum_{\beta < \gamma} \mfp_{\beta} \left(f(x) - \Gamma_{xy} f(y)  \right) \lambda^\beta \leqslant  \\
		&\leqslant \| \Pi\|_{\gamma; \K, \mfp}  \|\Gamma\|_{\gamma; \K, \mfp} \sum_{\beta < \gamma}  |x-y|^{\gamma-\beta}_{\s} \lambda^\beta  \lesssim \\
		& \lesssim_{A,\gamma} \| \Pi\|_{\gamma; \K, \mfp}  \|\Gamma\|_{\gamma; \K, \mfp} \lambda^\alpha ( |x-y|_\s + \lambda )^{\gamma-\alpha}\;.
	\end{equs}
	The second bound follows analogously.
\end{proof}

\begin{corollary}[Reconstruction]
	Fix an $\CA$-regularity structure $\cT = (A, T ,G)$. Let $\alpha = \min A \leqslant 0$, and $\gamma \geqslant \alpha$. For every model $Z = (\Pi,\Gamma)$ there exists a $t$-continuous $\CA$-bimodule morphism $\CR \colon \CD^\gamma(\Gamma) \to \CC^\alpha_\s(\R^d ; \CA)$ satisfying the reconstruction bound
	\begin{equ}
		\mfp \left( \left(  \CR f - \Pi_x f(x) \right) \left( \CS^\lambda_{\s,x} \eta \right) \right) \lesssim \vvvert f \vvvert_{\gamma ; \overline{\K} , \mfp} \begin{cases}
			\lambda^\gamma \; , \qquad &\text{if } \gamma \neq 0\;,\\
			1 - \log(\lambda) \; , & \text{if } \gamma = 0\;,
		\end{cases}
	\end{equ}
	for all $x \in \K$, $\eta \in \CB^r_{\s,0}$, and $\lambda \in (0,1]$, where $\overline{\K} \eqdef \left\{ x \in \R^d \, \big| \, d_{\s}(x, \K) \leqslant 4 \right\}$.
If $\gamma > 0$, this map is unique.
\end{corollary}

We will also make direct use of the following estimate, which is an analogue of \cite[Proposition~3.31]{Hai14}
\begin{proposition}
\label{prop:PosTreeBnd}
	Let $\alpha \in A$ with $ \alpha > 0$, and let $(\Pi, \Gamma)$ be a model for $\cT$. Then $\Pi \big|_{T_\alpha}$ is determined by $\Gamma\big|_{T_\alpha}$ and $\Pi\big|_{T_\alpha^-}$. In particular, we have the following bounds
	\begin{equ}
		\| \Pi \|_{T_\alpha ; \K , \mfp} \lesssim \| \Pi \|_{\alpha ; \overline\K , \mfp} \| \Gamma \|_{\alpha ; \overline{\K}, \mfp}
	\end{equ}
	and for a second model $(\overline{\Pi}, \overline{\Gamma})$
	\begin{equ}
		\| \Pi - \overline{\Pi} \|_{T_\alpha ; \K , \mfp} \lesssim \left( \vvvert Z \vvvert_{\alpha ; \overline{\K}, \mfp} + \vvvert \overline Z \vvvert_{\alpha ; \overline{\K}, \mfp}  \right) \vvvert Z - \overline{Z} \vvvert_{\alpha ; \overline{\K}, \mfp} \; ,
	\end{equ}
	with $\overline\K$ being the $1$-fattening of $\K$.
\end{proposition}

Beyond these, all related results of \cite[Section~3]{Hai14} continue to hold. In particular, these include Propositions~3.28,~3.29,~and~3.32. See also \cite[Appendix~B]{HS24} for results  similar to Proposition~3.32 without the use of wavelets.

\begin{definition}
	Let $H\eqdef \left\{ x \in \R^d \, \middle| \, \forall i \in [\bar d ] : x_i = 0 \right\}$, where $\bar d \leqslant d$. Further we denote by $\mfm$ the codimension of $H$ given by $\mfm \eqdef \sum_{i = 1}^{\bar d} \s_i$.

	For $x, y \in \R^d$ let
	\begin{equ}
		|x|_H \eqdef 1 \wedge d_{\s}(x,H), \qquad |x,y|_H \eqdef  |x|_H \wedge |y|_H
	\end{equ}
	and for $\K \subset \R^d$ let
	\begin{equ}
		\K_H \eqdef \left\{ (x,y) \in \left( \K \setminus H \right)^2 \, \middle| \, (x \neq y)  \land \left(  |x-y|_{\s} \leqslant | x,y |_H  \right) \right\} \; .
	\end{equ}
\end{definition}

\begin{definition}
	Fix an $\CA$-regularity structure $\cT$, a model $(\Pi,\Gamma)$, a hyperplane $H$, $\gamma > 0$, and $\eta \in \R$. For a function $f \colon \R^d \setminus H \to T^-_\gamma $ we set
	\begin{equ}
		\|f\|_{\gamma,\eta; \K , \mfp } \eqdef \sup_{x \in \K \setminus H} \sup_{\substack{\ell \in A \\ \ell < \gamma}} \frac{\mfp_\ell \left( f(x)\right)}{| x|_H^{(\eta- \ell)\wedge 0}} %, \qquad \n{f}_{\gamma,\eta;\K, p}\eqdef \sup_{x \in \K \setminus H} \sup_{\substack{\ell \in A \\ \ell < \gamma}} \frac{\|f(x)\|_{\ell \otimes p}}{\|x\|_H^{\eta-\ell}}
	\end{equ}
	The space $\CD^{\gamma,\eta}_{H}(V)$ then consists of functions $f : \R^d \setminus H \to T^-_\gamma $ taking values in the sector $V$ such that, for every  $\K \Subset \R^d$ and every $\mfp \in \mfP$, one has
	\begin{equ}
		\vvvert f\vvvert_{\gamma,\eta;\K,\mfp} \eqdef \|f\|_{\gamma,\eta; \K,\mfp} + \sup_{(x,y) \in \K_H} \sup_{\substack{\ell \in A \\ \ell < \gamma}} \frac{\mfp_\ell\left( f(x)-\Gamma_{xy} f(y)\right)}{|x-y|_{\s}^{\gamma-\ell} |x,y|_H^{\eta-\gamma}} < \infty \; .
	\end{equ}
	We also have for $\bar f$ associated with a second model $\left( \overline{\Pi}, \overline{\Gamma} \right)$
	\begin{equs}
		\left\vvvert f; \bar f \right\vvvert_{\gamma,\eta;\K,\mfp} & \eqdef \left\| f-\bar f \right\|_{\gamma,\eta;\K,\mfp} +  \\
		& \qquad + \sup_{(x,y) \in \K_H} \sup_{\substack{\ell \in A \\ \ell < \gamma}} \frac{\mfp_{\ell} \left( f(x)-\bar f(x)-\Gamma_{xy} f(y) + \overline{\Gamma}_{xy}\bar f(y)\right)}{|x-y|_{\s}^{\gamma-\ell} |x,y|_H^{\eta-\gamma}} \; .
	\end{equs}
\end{definition}

With this definition at hand, the results of \cite[Section~6]{Hai14} extend to the setting of $\CA$-Regularity Structures, these include Lemmata~6.5,~6.6,~6.7, and in particular Proposition~6.9.

\begin{corollary}[The Singular Case]
	Let $\eta \leqslant \gamma$. If $\eta \wedge \alpha > - \mfm$, then the reconstruction operator $\CR$ extends to a continuous map $\CR \colon \CD^{\gamma,\eta} \to \CC_{\s, H}^{\alpha, \alpha \wedge \eta}(\R^d ; \CA)$.
\end{corollary}

\begin{proof}
	For a proof of the result for $\CA = \R$, see \cite[Theorem~A.16]{CHP23} which is an extension of \cite[Proposition~6.9]{Hai14}. The usual arguments then extend the result to $\CA$-regularity structures.
\end{proof}

\begin{corollary}[The Function-Like Case]
	Let $V$ be a sector that admits a decomposition of the form $V = T_{d, \s} \oplus T^+_\alpha$ for some $\alpha > 0$. Then, for $f \in \CD^{\gamma , \eta}$ with $\gamma > \alpha$ and $\eta \leqslant 0$, we have $\CR f \in \CC^{\alpha, \eta}_{\s , H}(R^d ; \CA)$ and $\CR f = f_0$, where $f_0$ is the component of $f$ in $T_0$.
\end{corollary}
\begin{proof}
	The argument of \cite[Proposition~3.28]{Hai14} shows that $\CR f(x)  = f_0(x)$ for $x \in \R^d \setminus H$ and that $f_0 \in \CC^\alpha_{\s}(R^d \setminus H)$. The same argument as in \cite[Theorem~A.16]{CHP23} finishes the proof.
\end{proof}

\subsection{Abstract Integration}
\label{sec:AbstInt}
We shall need the following set of operators:
\begin{definition}
	Let $L_{T}^\beta$ be the set of $t$-continuous $\mathbb{K}$-linear operators $T \to T$ such that, for $M \in L_{T}^\beta$,
	\begin{equ}
		M \big|_{T_{\alpha}} \colon T_\alpha \longrightarrow T_{\alpha+\beta}^-  \; .
	\end{equ}
\end{definition}

\begin{assumption}
\label{assump:PolyT}
	In the following given an $\CA$-regularity structure $\cT = (A,T,G)$, we shall assume that $\cT_{d,\s} \subset \cT$ where $\cT_{d,\s}$ is the canonical polynomial regularity structure, the vector space of which we denote by $\overline{T}$, and that for $n \in \N$, $T_n = \overline{T}_n$. Further, we assume that any model $\Pi$ acts on $\overline{T}$ as in Example~\ref{ex:PolMod}.
\end{assumption}

\begin{definition}[\TitleEquation{\beta}{beta}-Regularising Kernel]
	\label{def:SingKern}
	Fix $\beta > 0$.
Denoting the diagonal of $\R^{d} \times \R^{d}$ by  $\Delta \eqdef \left\{ (x,x) \in \R^d \times \R^d \, \middle| \, x \in \R^d \right\}$, we say
\[
\overline K \in \cC^\infty \left( \left(\R^d \times \R^d \right) \setminus \Delta ; \CA \right)
\]
is a $\beta$-regularising integral kernel if we have a  decomposition 
\[
\overline K = K + R\;,
\] 
where $R \in  \cC^\infty \left( \R^d \times \R^d ; \CA \right) $ and $K$ is a function, such that, for every $\mfp \in \mfP$, there exists a constant $C_{\mfp}>0$ such that, for all $k,\ell \in \N^d$ and $x,y \in \R^d$
	\begin{equ}
		\mfp\left( D^k_1 D^\ell_2 K(x,y) \right) \leqslant C_{\mfp} |x-y|_{\s}^{\beta- |\s| -|k|_{\s}-|\ell|_{\s}}
	\end{equ}
	Further, we suppose that $K$ can be decomposed as
	\begin{equ}
		K(x,y) = \sum_{n \in \N} K_n(x,y)
	\end{equ}
	with the functions $K_n(x,y)$ satisfying:
	\begin{itemize}
		\item for all $n \in  \N$ we have $\supp K_n \subset \left\{ (x,y) \in \R^d \times \R^d \, \middle| \, |x-y|_{\s} \leqslant 2^{-n} \right\}$,
		\item for all $k,\ell \in \N^d$ and all $\mfp\in \mfP$, there exists a constant $C_{\mfp} \geqslant 0$ such that, for all $x,y \in \R^d$ and $n \in \N$,
		\begin{equ}
			\mfp\left( D^k_1 D^\ell_2 K_n(x,y) \right) \leqslant C_{\mfp} 2^{ \left( |\s| - \beta + |k|_{\s} + |\ell|_{\s} \right) n } \; ,
		\end{equ}
		\item for any $k, \ell \in \N^d$ and all $\mfp \in \mfP$, there exists a constant $C_{\mfp} \geqslant 0$ such that, for all $x,y \in \R^d$ and all $n \in \N$,
		\begin{equs}
			\mfp\biggl( \int\limits_{\R^d} (x-y)^\ell D_2^k K_n(x,y) \,\d x \biggr)&\leqslant C_{\mfp} 2^{-\beta n}\\
			\mfp\biggl( \int\limits_{\R^d} (x-y)^\ell D_1^k K_n(x,y) \,\d y \biggr)&\leqslant C_{\mfp} 2^{-\beta n}
		\end{equs}
		\item there exists an $r > 0$ such that, for all $n \in \N$ and all polynomials $P$ with $\deg P \leqslant r$,
		\begin{equ}
			\int\limits_{\R^d} K_n(x,y) P(y) \,\d y = 0 \; .
		\end{equ}
	\end{itemize}

	We  call $K$ (resp. $R$) the singular (resp. smooth) part of $\bar{K}$. 
\end{definition}

\begin{remark}
	We note that the assumptions on $K$ in particular imply that convolution with $K$ maps $\CC^\alpha_{\s}(\R^d ; \CA) \to \CC^{\alpha+\beta}_{\s}(\R^d ; \CA)$ for all $\alpha \in \R$, by the scalar Schauder estimates.

	Moreover, we are technically restricting ourselves here to integration kernels (and thus differential operators) that are acting from the left. Similarly, one could also define kernels that act from both sides, i.e.\ $K = \sum_{i = 1}^n K^1_i \otimes K^2_i \in \cC^\infty \left( \left(\R^d \times \R^d \right) \setminus \Delta ; \CA \right)^{ \otimes 2} $ with
	\begin{equ}
		K \ast \phi( x ) \eqdef \sum_{i = 1}^n \int\limits_{\R^d} K^1_i(x,y) \phi(y) K^2_i(y, x) \d y
	\end{equ}
	satisfying analogous conditions to the ones above. However, we will only consider the case of left-acting integration kernels for the sake of notational simplicity.
\end{remark}

\begin{definition}[Abstract Integration]
	Given a sector $V$, a $\mathbb{K}$-linear $t$-continuous map $\CI \colon V \to T $ is an abstract integration map of order $\beta$ if it satisfies:
	\begin{itemize}
		\item $\CI \colon V_\alpha \to T_{\alpha + \beta} $ for all $\alpha \in A$,
		\item $\CI a = 0$ for all $a \in V \cap \overline{T}$,
		\item $\CI \Gamma a - \Gamma \CI a \in \overline{T}$, for all $\Gamma \in G$.
	\end{itemize}
	 If $T$ carries a grading, then we assume that the map $\CI$ is of grade 0.
\end{definition}

\begin{remark}
	The reason we assume that $\CI$ is only $\mathbb{K}$-linear and not $\CA$-bilinear is that convolution with $K$ is not necessarily $\CA$-bilinear, as $K$ may take values in $\CA$ proper and not just in $\mathbb{K}$.
\end{remark}

\begin{definition}[Compatibility Kernels]
	Let $\CJ \colon \R^d \to L^\beta_{T}$ be defined for $a \in T_\alpha$, $\alpha \in A$
	\begin{equ}
		\CJ(x) a \eqdef \sum_{|k|_{\s} < \alpha + \beta} \frac{X^k}{k!} \int\limits_{\R^d} D^k_1 K(x,z) \left( \Pi_x a \right)(z) \d z.
	\end{equ}
	Furthermore, for a sector $V$, $\gamma \in \R$, and $f\in \CD^\gamma(V)$ define $\CN_\gamma \colon \CD^\gamma(V) \to \CD^{\gamma+\beta}(\overline{T})$
	\begin{equ}
		\left( \CN_\gamma f \right)(x) \eqdef \sum_{|k|_{\s} < \gamma + \beta} \frac{X^k}{k!} \int\limits_{\R^d} D^k_1 K(x,y) \left(  \CR f - \Pi_x f(x) \right) (y) \d y \; .
	\end{equ}

\end{definition}

\begin{definition}
\label{def:CompKern}
	Given a sector $V$ and an abstract integration operator $\CI$ on $V$, a model $(\Pi,\Gamma)$ realises (the singular part of) a $\beta$-regularising kernel $K$ for $\CI$ if, for every $\alpha \in A$, every $a \in V_\alpha$, and $x \in \R^d$ we have
	\begin{equ}
		\Pi_x \CI a = \int\limits_{\R^d} K(\,\bigcdot\, , z) \left( \Pi_x a \right) (z) \d z - \Pi_x \CJ(x)a \; .
	\end{equ}
	For such a sector $V$, model $\Pi$, and $\gamma > 0$, we define the operator $K_\gamma \colon \CD^\gamma(V) \to \CD^{\gamma+\beta}$ for $f \in \CD^\gamma(V)$
	\begin{equ}
		(\CK_\gamma f) (x) \eqdef \CI f(x) +  \CJ(x)f (x)+ \left( \CN_\gamma f \right)(x) \; .
	\end{equ}
\end{definition}

Given these definitions all results of \cite[Section~5]{Hai14} extend straightforwardly to the case of $\CA$-regularity structures, in particular this is true for Theorems~5.12~\&~5.14, Lemmata~5.16,~5.18,~5.19,~5.21 and Propositions~5.23~\&~6.16.

\begin{theorem}
	\label{SingKernThm}
		Let $\cT = (A,T,G)$ be an $\CA$-regularity structure and $(\Pi,\Gamma)$ a model for $\cT$ satisfying Assumption~\ref{assump:PolyT}. Let $\gamma > 0$ and let $K$ be (the singular part of) a $\beta$-regularising kernel for $\beta > 0$ satisfying the polynomial annihilation condition with $r = \beta + \gamma$, let $\CI$ be an abstract integration map of order $\beta$ acting on some sector $V$, let $\Pi$ realise $K$ for $\CI$, and $\CK_\gamma$ as above.

		Then, for $\gamma+\beta \notin \N$, $\CK_\gamma$ is well-defined, maps $\CD^\gamma(V)$ into $\CD^{\gamma+\beta}$, and the identity
		\begin{equ}
			\CR \CK_\gamma f = K \ast \CR f
		\end{equ}
		holds for every $f \in \CD^\gamma(V)$. If $\left( \overline{\Pi},\overline{\Gamma} \right)$ is a second model realising $K$ and $\bar f \in \CD^\gamma \left( V ; \overline{\Gamma} \right)$, then one has the bound
		\begin{equ}
			\left\vvvert \CK_\gamma f ; \overline{\CK}_\gamma \bar f \right\vvvert_{\gamma+\beta;\K,\mfp} \lesssim \left\vvvert f ; \bar f \right\vvvert_{\gamma; \overline{\K},\mfp} + \left\vvvert Z ;  \overline{Z} \right\vvvert_{\gamma;\overline{\K}, \mfp} 
		\end{equ}
		for every $\K \Subset \R^d$ and $\overline{\K}$ being its 1-fattening. The proportionality constant in the bound depends only on the norms $\vvvert f \vvvert_{\gamma; \overline{\K},\mfp}, \vvvert  \bar f\vvvert_{\gamma; \overline{\K},\mfp}$ as well as similar bounds on the two models.
	\end{theorem}
	\begin{proof}
		The proof follows as in the proof of \cite[Theorem~5.12]{Hai14} and its supporting lemmas by exchanging the norm on $\R$ with a seminorm $\mfp$ of $\CA$.
	\end{proof}

\begin{theorem}[Extension Theorem]
Suppose we are given a $\CA$-regularity structure $\cT = (T,A,G)$ 
 satisfying Assumption~\ref{assump:PolyT}. Let $V \subset T$ be a sector of order $\overline{\gamma}$ with the property that for every $\alpha \not\in \N$ with $V_\alpha \neq \{0\}$, one has $\alpha + \beta \not\in \N$. Furthermore, let $W \subset V$ be a possibly empty subsector of $V$, let $K$ be a $\beta$-regularising kernel on $\R^d$ satisfying the polynomial annihilation condition for  $\overline{\gamma}$. Let $(\Pi,\Gamma)$ be a model for $\cT$, let $\CI \colon W \to T$ be an abstract integration map of order $\beta$ such that $\Pi$ realises $K$ for $\CI$.

	Then, there exists a regularity structure $\widehat{\cT}$ containing $\cT$, a model $\big( \widehat{\Pi},\widehat{\Gamma} \big)$ for $\widehat{\cT}$ extending $(\Pi,\Gamma)$, and an abstract integration map $\widehat{\CI}$ of order $\beta$ acting on $\widehat{V} = \iota V$ such that
	\begin{itemize}
		\item the model $\widehat{\Pi}$ realises $K$ for $\widehat{\CI}$,
		\item the map $\widehat{\CI}$ extends $\CI$ in the sense that $\widehat{\CI} \iota a = \iota \CI a$ for every $a \in W$.
	\end{itemize}
	Furthermore, the map $\left( \Pi, \Gamma \right) \mapsto \big( \widehat{\Pi},\widehat{\Gamma}\big)$ is locally bounded and Lipschitz continuous in the sense that if $\left( \Pi,\Gamma \right)$ and $\left( \overline{\Pi},\overline{\Gamma} \right)$ are two models for $\cT$ and $\big( \widehat{\Pi},\widehat{\Gamma} \big)$ and $\big( \widehat{\overline{\Pi}}, \widehat{\overline{\Gamma}} \big)$ are their respective extensions, then one has the bounds
	\begin{equs}
	\label{ExtThmEst}
		\big\| \widehat{\Pi} \big\|_{\widehat{V};\K, \mfp}+\big\| \widehat{\Gamma} \big\|_{\widehat{V};\K, \mfp} &\lesssim \left\| \Pi \right\|_{V;\overline{\K}, \mfp}  \left( 1+\left\| \Gamma \right\|_{V;\overline{\K}, \mfp}  \right)\\
		\big\| \widehat{\Pi} - \widehat{\overline{\Pi}} \big\|_{\widehat{V};\K, \mfp}+\big\| \widehat{\Gamma}  - \widehat{\overline{\Gamma}}\big\|_{\widehat{V};\K, \mfp} &\lesssim \left\| \Pi - \overline{\Pi} \right\|_{V;\overline{\K}, \mfp}  \left( 1+\left\| \Gamma \right\|_{V;\overline{\K}, \mfp}  \right) + \\
		& \qquad \qquad \qquad  + \left\| \overline \Pi \right\|_{V;\overline{\K}, \mfp}  \left\| \Gamma - \overline{\Gamma} \right\|_{V;\overline{\K}, \mfp}
		\end{equs}
	for any compact $\K \subset \R^d$ and its 2-fattening $\overline{\K}$.
\end{theorem}
\begin{proof}
	The proof follows as in the proof of \cite[Theorem~5.14]{Hai14}. In particular, the required complement $\overline W$ exists by our definition of a sector of an $\CA$-regularity structure. When defining the new structure space $\hat{T}_\alpha$ of the regularity structure, we set
	\begin{equs}
		\hat{T}_\alpha \eqdef T_\alpha \oplus \left( \CA \wotimes_\pi \overline{W} \wotimes_\pi \CA \right)
	\end{equs}
	when $\alpha = \alpha_n + \beta$. This is equipped with the natural bimodule structure. Even though the maps $M^\alpha_{\overline W}$ are only $\K$-linear, the action of the structure group is still an $\CA$-bimodule morphism. This is because $M^\alpha_{\overline{W}}$ only acts on the inner $\overline{W}$ factor of $\CA \wotimes_\pi \overline{W} \wotimes_\pi \CA $ whereas the $\CA$-bimodule structure rests on the outer $\CA$ factors.
\end{proof}

\begin{remark}
	The choice of the $t$-projective tensor product here was made solely to construct a specific extension of the regularity structure; see Definition~\ref{def:tProjTens} for its definition. Other choices are in principle possible.
\end{remark}

\subsection{Multiplication and Miscellaneous Constructions}\label{subsec:multiplication}

\begin{definition}
	Let $\cT = (T,A,G)$ be an $\CA$-regularity structure. We call a continuous bilinear map $(r,s)\mapsto r \star s$ on $T$ a product if
	\begin{itemize}
		\item $r \star s \in T_{\alpha+\beta}$, for all $r \in T_\alpha$, $s \in T_\beta$, and for all $\mfp \in \mfP$
		\begin{equ}
			\mfp_{\alpha+\beta}(r \star s) \lesssim \mfp_\alpha(r) \mfp_\beta(s) \; ,
		\end{equ}
		\item for all $a \in \CA$, for all $r,s \in T$
		\begin{equs}
			(ar) \star s  = a(r \star s) \; , \qquad
			(ra)\star s  = r \star (as) \; , \qquad
			r\star (sa)  = (r \star s) a \; ,
		\end{equs}
		\item for every $a \in T$ one has $\1  \star a = a \star \1 = a$.
	\end{itemize}
\end{definition}

\begin{definition}
	Let $\cT$ be an $\CA$-regularity structure, $V,W$ two sectors of $\cT$, and $\star$ a product on $\cT$. The pair $(V,W)$ is said to be $\gamma$-regular if for all $\Gamma \in G$, and all $a \in V_\alpha$, $b\in W_\beta $ with $\alpha + \beta < \gamma$
	\begin{equ}
		(\Gamma a) \star (\Gamma b) = \Gamma(a \star b) \; .
	\end{equ}
	A product is regular for $(V,W)$ if it is regular for all $\gamma \in A$, and if $V= W$, we shall simply say that the product is $\gamma$-regular on $V$.
\end{definition}

\begin{theorem}
\label{thm:RSMult}
	Let $V,W$ be of regularities $\alpha_1$ and $\alpha_2$ respectively, and let $(\Pi,\Gamma)$ be a model. Let $f_1 \in \CD^{\gamma_1}(V)$, $f_2 \in \CD^{\gamma_2}(W)$ and define $\gamma := \left( \gamma_1+\alpha_2 \right) \wedge \left( \gamma_2+\alpha_1 \right)$. Then if $(V,W)$ is $\gamma$-regular, one has $f_1 \star f_2 \in \CD^\gamma$ and for all $\K \Subset \R^d$, all $\mfp \in \mfP$
	\begin{equ}
		\vvvert f_1 \star f_2 \vvvert_{\gamma; \K,\mfp} \lesssim \vvvert f_1\vvvert_{\gamma_1 ; \K,\mfp}\vvvert f_2\vvvert_{\gamma_2 ; \K,\mfp} \left( 1+ \|\Gamma\|_{\gamma_1+\gamma_2; \K,\mfp} \right)^2
	\end{equ}
	with the proportionality constant only depending on $\cT$.

	Let $\big( \overline{\Pi},\overline{\Gamma}\big)$ be a second model and $g_1 \in \CD^{\gamma_1}\big(V;\overline{\Gamma}\big)$, $g_2 \in \CD^{\gamma_2}\big(W; \overline{\Gamma}\big)$. Then for every $C > 0$
	\begin{equ}
		\big\vvvert f_1 \star f_2 ; g_1 \star g_2 \big\vvvert_{\gamma; \K,\mfp} \lesssim \vvvert f_1 ; g_1 \vvvert_{\gamma_1;\K,\mfp}	+\vvvert f_2 ; g_2 \vvvert_{\gamma_2;\K,\mfp} + \big\| \Gamma-\overline{\Gamma} \big\|_{\gamma_1+\gamma_2; \K,\mfp} \; ,
	\end{equ}
	uniformly for all $\vvvert f_i \vvvert_{\gamma_i; \K,\mfp} +\vvvert g_i \vvvert_{\gamma_i; \K,\mfp} \leqslant C$ and $\|\Gamma\|_{\gamma_1+\gamma_2;\K,\mfp} + \|\overline\Gamma\|_{\gamma_1+\gamma_2;\K,\mfp} \leqslant C$. The proportionality constant only depends on $C$.

	Analogously, the same holds for singular modelled distributions $f_1 \in \CD^{\gamma_1, \eta_1}_H (V)$ and $f_2 \in \CD^{\gamma_2, \eta_2}_H (V)$ with $f_1 \star f_2 \in \CD^{\gamma, \eta}_H$ for $\gamma$ as above and $\eta \eqdef (\eta_1 + \alpha_2) \wedge (\eta_2 + \alpha_1) \wedge (\eta_1 + \eta_2)$.
\end{theorem}
\begin{proof}
	See the proofs of Theorem 4.7 and Proposition 4.10 \& 6.12 in \cite{Hai14}. We note that by assumption we have for any $a,b \in \CA$ and $\mfp \in \mfP$, the estimate
	\begin{equ}
		\mfp(ab) \leqslant \mfp(a)\mfp(b)
	\end{equ}
	which is the only additional property of the standard norm on $\R$, beyond the usual axioms for seminorms, that is required.
\end{proof}

\begin{proposition}
	Let $\cT$ be an $\CA$-regularity structure, $V,W$ sectors of $\cT$, $\gamma \in \R$ and $(\Pi,\Gamma)$ a model for $\cT$. One can always define an extension $\iota \colon \cT \hookrightarrow \overline{\cT}$, an extension $(\overline\Pi,\overline\Gamma)$ of $(\Pi,\Gamma)$ as well as a product $\star$ on $\overline{\cT}$, such that $(\iota V, \iota W)$ are $\gamma$-regular.
\end{proposition}
\begin{proof}
	One can proceed as in the proof of \cite[Proposition~4.11]{Hai14}. The condition that the subspace $T_\alpha$ of $T$ be finite-dimensional is unnecessary, as long as we equip $V_\alpha \otimes_{\CA} W_\beta$ with the topology generated by the seminorms $(\mfp_\alpha \otimes_\CA \mfp_\beta)_{\mfp \in \mfP}$ and take the completion with respect to this topology. Here $\mfp_\alpha \otimes_\CA \mfp_\beta$ is defined, slightly differently from $\mfp_\alpha \otimes \mfp_\beta$ via
	\begin{equs}
		(\mfp_\alpha \otimes_\CA \mfp_\beta)(A) &\eqdef \inf \biggl\{ \sum_{k = 1}^n \mfp_\alpha(a_k b_k)  \mfp_\beta(c_k), \sum_{k = 1}^n \mfp_\alpha(a_k)  \mfp_\beta(b_k c_k) \, \bigg| \, a_k \in V_\alpha, \\
		& \qquad  c_k \in W_\beta, b_k \in \CA :  A = \sum_{k = 1}^n (a_k b_k) \otimes_\CA c_k = \sum_{k = 1}^n a_k \otimes_\CA ( b_k c_k )
		\biggr\}
	\end{equs}
\end{proof}

In particular, we can now extend the proof of the \cite[Proposition~4.14]{Hai14} for the scalar Young multiplication theorem so that it holds for $\CA$-valued Besov spaces.

\begin{proof}[Proof of Theorem~\ref{thm:YoungMult}]
	We shall describe the regularity structure built for the sufficiency condition, for further details, cf.\ \cite[Proposition~4.14]{Hai14}.

	W.l.o.g.\ we may assume that $\alpha \leqslant 0$ and that $\alpha \notin \N$. Let $\xi \in \CC_\s^\alpha(\R^d ; \CA)$. We set $A = \N \cup (\N + \alpha)$, $T = T_{d,\s} \oplus T'$, where $T' = \bigoplus_{n \in \N}  T_{n+\alpha}$ and $T_{n+\alpha} \eqdef \CA \wotimes_\pi \CA$ for all $n \in \N$. Here one should think of $T_{n+\alpha}$ as the $\CA$-bimodule generated by the $\CA$-basis element $\Xi X^k = X^k \Xi$, where $\Xi$ is an abstract symbol representing $\xi$ and thus does not commute with elements of $\CA$. This is why $T_{n+\alpha}$ must be isomorphic to a tensor product of two copies of $\CA$.

	Finally, the structure group is again isomorphic to $\R^d$ with
	\begin{equ}
		\Gamma_h \left( \Xi X^k \right) = \Xi \left( \Gamma_h X^k \right) \; ,
	\end{equ}
	and the model is given by
	\begin{equ}
		\left(\Pi_x \Xi X^k \right)(y) = \xi(y) (y-x)^k \; ;
	\end{equ}
	both extended $\CA$-bilinearly to $T$.	The product is given by $\star \colon T' \times T_{d,\s} \to T' \subset T$ with
	\begin{equ}
		\Xi X^k \star X^\ell = \Xi X^{k + \ell} \; .
	\end{equ}
	Because $X^\ell$ commutes with elements of $\CA$, this is well-defined.

	With these definitions in hand, the rest of the proof follows verbatim as in \textit{loc.\ cit.}
\end{proof}

The last construction that needs to be adjusted is lifting the composition with functions $F \colon \mathbb{K}\to \mathbb{K}$ to the regularity structure. Unfortunately, we need to restrict ourselves to functions that are $\CA$-analytic, cf.\ Definition~\ref{def:AAnaFct}, to be able to extend them to the $\CA$-bimodule $T$ by formally substituting an element of $T$ into the variable slot instead of an element of $\CA$.

\begin{remark}
	One of the major reasons we have to restrict ourselves to this subset is that we generally assume $\CA$ to be infinite-dimensional over $\mathbb{K}$. In the finite-dimensional case we always have the isomorphism $\CB^k(\mathbb{K}^n ; \mathbb{K}) \cong \CB( \mathbb{K}^{n} ; \mathbb{K})^{\otimes k}$. Thus, we can extend elements of $\CB^k(\mathbb{K}^n ; \mathbb{K})$ to maps $\CB^k(\CA^n ; \CA)$ for any $\mathbb{K}$-algebra $\CA$ via the following chain of morphisms
	\begin{equ}
		\CB^k(\mathbb{K}^n ; \mathbb{K}) \otimes \CA^{\otimes k} \cong \left( \CB(\mathbb{K}^n ; \mathbb{K}) \otimes \CA \right)^{\otimes k} \cong \left(\CB( \CA^n ; \CA)\right)^{\otimes k} \hooklongrightarrow \CB^k(\CA^n ; \CA)
	\end{equ}
	and by tensoring $\phi \in \CB^k(\mathbb{K}^n ; \mathbb{K})$ with the identity map $\1_{\CA}^{\otimes k} \colon \CA^{\otimes k} \to \CA^{\otimes k}$. In particular, one can do this for finite-dimensional commutative regularity structures. However, these isomorphisms fail for general $\CA$-regularity structures.
\end{remark}
%Let $F \in \cC^\omega(\CA^n ; \CA)$%\footnote{We remind the reader that this means that for every seminorm $\mfp$ of $\CA$, $\mfp(F(a) - F(b))$ is controlled by $\mfp(a-b)$ with the same holding true for all its derivatives.}
Let $V$ be a function-like sector $V$ of $\cT$ with product $\star$. Given a monomial $ \left\llbracket A_1, \dots,  A_{m+1} ; X^k \right\rrbracket $ for $A_i \in \CA$ we can extend it to a function $\CA^{m+1} \times V^n \to V$ using the formula
\begin{equ}
	\left(A_1 X_{k_1}\right) \star \left( A_2 X_{k_2} \right) \star \cdots \star \left( A_{m} X_{k_m} A_{m+1} \right) \; .
\end{equ}
Due to the $\CA$-bimodule structure of $\cT$ and the compatibility of $\star$ with this bimodule structure, this extension is independent of our specific choice of placement of brackets. This straightforwardly extends to a continuous function $\cA^{\wotimes_{\pi}(m+1)} \times V^n \to V$.

Now we define the lift $\widehat{F}$ of a general $F \in \cC^\omega(\CA^n ; \CA)$ to $\cT$ by extending $F$ term by term.
% \begin{equ}
% 	\widehat{F}(a) \eqdef \sum_{\substack{k \in \N^d \\ |k|_\s \leqslant \kappa}} \frac{D^k F(\bar a)}{k!} [\tilde a, \dots, \tilde a]
% \end{equ}
% where we use the unique decomposition $a = \bar a \1 + \tilde{a}$ with $\bar a \in \CA$ and define
The proof of the following theorem follows verbatim as in \cite[Theorem~4.16]{Hai14}. %\martinp{Add explanation to proof}
\begin{theorem}
\label{thm:FncLift}
	Given a function-like sector $V$ with least non-zero regularity $\zeta$ that is $\gamma$-regular for $\gamma > 0$, let $F \in \cC^\omega(\CA^n;\CA)$. Then, the map $\widehat{F}_\gamma \colon \CD^\gamma(V) \to \CD^\gamma(V)$ given by
	\begin{equ}
		\widehat{F}_\gamma(f)(x) = \CQ_\gamma^- \widehat{F}(f(x))
	\end{equ}
	is well-defined and locally Lipschitz $t$-continuous w.r.t.\ $\vvvert \bigcdot \vvvert_{\gamma ; \K , \mfp}$.
\end{theorem}

Analogously, one can straighforwardly extend the definition of abstract differentiation on regularity structures \cite[Definitions~5.25~\&~5.26]{Hai14} and the implementation of symmetries \cite[Definition~3.33]{Hai14}. In particular, Propositions 3.38, 5.23, 5.24, 5.28, 6.15 \& 6.16 continue to hold.

\subsection{Fixed-Point Map}

We are now ready to tackle the short-time existence of $\CA$-valued SPDEs in terms of a fixed-point problem of modelled distributions in an $\CA$-regularity. In particular, we will focus on equations of the form 
\begin{equ}
	\cL u = F(u, \Xi)
\end{equ}
where $\cL$ is a differential operator, s.t.\ the kernel $\overline{K} = K + R$ of $\cL^{-1}$ is $\beta$-regularising, cf.\ Definition~\ref{def:SingKern}. Here $K$ denotes the compactly supported singular part of $\overline{K}$ and $R$ the smooth part. 

For this section we fix $H = \{ (0,x) \in \R \times \R^{d-1} \, \big| \, x \in \R^{d-1}\}$ to be the $t = 0$ hyperplane, and $\bR^+ \in \CD_H^{\infty, \infty}$ the modelled distribution $\bR^+ = \bone_{\{t > 0\}} \1$. We assume that $K$ can be chosen in such a way that it is non-anticipative w.r.t.\ $H$, i.e.\ 
\begin{equ}
	K((t,x),(s,y)) = 0
\end{equ}
if $t < s$, and analogously for $R$. We denote by $R_{\gamma} \colon \CC^\alpha_\s \to \CD^\gamma$ the canonical lift of the convolution with $R$ to the polynomial sector of the regularity structure $T_{<\gamma}$.

\begin{remark}
	If we were interested in boundary-value problems, we could have instead assumed that $\cL$ is elliptic, then the following theorem would instead hold for a sufficiently small coupling constant. This would be particularly meaningful in the case when $\CA$ is a Banach algebra, as we would not have to restrict ourselves to \textit{a priori} seminorm dependent, i.e.\ ``random'', coupling constants. 
\end{remark}

Furthermore, we assume that we are given a crystallographic group $\cS$ and a representation $\tilde T \colon \cS \to \GL(\R^{d-1})$ with compact fundamental domain $\K \Subset \R^{d-1}$, i.e.\ $\bigcup_{g \in \cS} \tilde T_g \K = \R^{d-1}$. We assume that for all $g \in \cS$, $T_g$ commutes with the convolution with $K$ and $G$, where $T_g$ is the extended representation $T_g \eqdef \bone_{\R} \oplus \tilde T_g$. \cite[Lemma~5.24]{Hai14} guarantees that this is possible if $T_g$ commutes with $\overline{K}$.

For $\cS$-invariant modelled distributions, it is enough to take the supremum over sets of the form $(s,t) \times \K$ or, equivalently, $(s,t) \times \R^{d-1}$. We thus set $O \eqdef [-1,2] \times \R^{d-1}$ and $O_T \eqdef (-\infty,T]\times \R^{d-1}$ with corresponding modelled distribution seminorms $\vvvert \bigcdot \vvvert_{\gamma, \eta ; O, \mfp}$ and $\vvvert \bigcdot \vvvert_{\gamma, \eta ; T, \mfp}$. Finally, for each $\mfp \in \mfP$, we define the Banach spaces of modelled distributions $\left(\CD^{\gamma,\eta}_H\right)_{\mfp}$ obtained by quotienting all modelled distributions $f \in \CD^{\gamma,\eta}_H$, such that $\vvvert f \vvvert_{\gamma,eta; \K , \mfp} = 0$ for all $\K \Subset \R^d$.

With these definitions at hand, we can state the main abstract fixed-point theorem. 
\begin{theorem}
\label{thm:AbsFPM}
	Let $V, \overline{V}$ be two sectors of a regularity structure $\mathscr T$ of regularities $\zeta,\overline{\zeta} \in \R$ respectively, such that $\zeta\leqslant\overline{\zeta}+\beta$. Let $\gamma \geqslant \overline{\gamma} >0$, $\eta,\overline{\eta} \in \R$ as above, let $F \colon \R^d \times V_\gamma  \to \overline{V}_{\overline{\gamma}}$ be a smooth function such that, for $f \in \CD^{\gamma,\eta}_H$ symmetric with respect to $\cS$,  we have $F(f) \in \CD^{\overline{\gamma},\overline{\eta}}_H$ and that $F(f)$ is also symmetric with respect to $\cS$.
	Furthermore, suppose we have an abstract integration map $\CI$ such that $\CQ^-_\gamma \CI \overline{V}_{\overline{\gamma}} \subset V_\gamma$.

	If $\eta < (\overline{\eta}\wedge\overline{\zeta}) +\beta$, $\gamma < \overline{\gamma} +\beta$, $(\overline{\eta}\wedge\overline{\zeta})>-\beta$, and $F$ is locally Lipschitz, then for every $v \in \CD^{\gamma,\eta}_H$ which is symmetric with respect to $\cS$, for every symmetric model $Z = (\Pi,\Gamma)$ of the regularity structure $\cT$ such that $\CI$ is adapted to the kernel $K$, and every $\mfp\in \mfP$, there exists a time $\tau_\mfp> 0$ such that the equation
	\begin{equ}
		\label{eq:LiftEq}
		u = \left( \CK_{\overline{\gamma}}+R_{\gamma} \CR \right) \bR^+ F(u)+v
	\end{equ}
	admits a unique solution $u^{\mfp} \in \left( \CD^{\gamma,\eta}_{H} \right)_{\mfp} $ on $(0,\tau_{\mfp})$. The solution map $\CS^{\mfp}_\tau \colon (v,Z) \mapsto u^\mfp$ is jointly continuous in a neighbourhood of $(v,Z)$, i.e.\ for every fixed $v,Z$ as well as $\eps > 0$ there exists a $\delta > 0$ such that a second initial condition $\bar v$ and second model $\overline{Z}$ we have
	\begin{equ}
		\bigl\vvvert Z; \overline{Z} \bigr\vvvert_{\gamma; O,\mfp} + \vvvert v; \bar v \vvvert_{\gamma,\eta;\tau,\mfp} \leqslant \delta \implies \vvvert u^\mfp, \bar u^\mfp\vvvert_{\gamma,\eta;\tau,\mfp} \leqslant \eps
	\end{equ}
	where $\bar u^{\mfp} = \CS^{\mfp}_\tau (\bar v, \overline{Z})$.

	If $F$ is strongly locally Lipschitz, then the map $\CS^{\mfp}_\tau$ is jointly locally Lipschitz continuous in the regular sense of the word.
\end{theorem}

	We note here that by $t$-continuity we for any sequence $u_n$ that converges on $(0,\tau_\mfp)$ to $u^\mfp$ in $\left( \CD^{\gamma,\eta}_H\right)_{\mfp}$, we have that $\pi_\mfp \circ \CR u_n$ must also be a Cauchy sequence in $\CC^{\alpha \wedge \eta}_{\s}(O_{\tau_\mfp} ; \CA_\mfp)$, as 
	\begin{equ}
		\| \CR u_n \|_{\alpha \wedge \eta ; \tau_{\mfp} , \mfp} \lesssim \vvvert u_n \vvvert_{\gamma,\eta ; O_T , \mfp} \; . 
	\end{equ}
	Here $\CA_\mfp \eqdef \overline{\CA / \mfp^{-1}(0)}^{\mfp}$ and $\pi_\mfp$ is the canonical projection. In particular, this means that there exists a canonical reconstruction of $u^\mfp$ in $\CC^{\alpha \wedge \eta}_{\s}(O_{\tau_\mfp} ; \CA_\mfp)$.

We can, in fact, push this line of reasoning further. Because each $T_\ell$ for $\ell \in A$, and $T_{< \gamma}$ as a whole, is a locally $m$-convex $\CA$-bimodule, $\CA_\mfp$ acts naturally on $T^\mfp_\ell \eqdef \overline{T_\ell / \mfp_\ell^{-1}(0)}^\mfp$. For if $a \in \mfp^{-1}(0)$, then for all $m \in T_\ell$ 
\begin{equ}
	\| m a \|_{\mfp_\ell} \leqslant \| m \|_{\mfp_\ell} \| a \|_\mfp = 0 \; ,
\end{equ}
i.e.\ $m a$ is equivalent to $0$ in $T^\mfp_\ell$, and analogously for the left action. In particular, due to the $t$-continuity that we have imposed, $T^\mfp \eqdef \bigoplus_{a \in A} T^\mfp_a$ is itself the structure space of an $\CA_\mfp$-regularity structure, $\cT^\mfp$. Each of the solutions $u^\mfp$ can then be viewed as $\cT^\mfp$-modelled distributions.

W.l.o.g.\ we may assume that the set of seminorms $\mfP$ of $\CA$ is directed, i.e.\ that for any two seminorms $\mfp, \mfq$ there exists a seminorm $\mfr$, s.t.\ for all $a \in \CA$
\begin{equ}
	\mfp(a) \vee \mfq(a) \leqslant \mfr (a) \; . \footnote{E.g.\ by adding the maxima over all finite collections of seminorms in $\mfP$ to the set of seminorms.}
\end{equ} 
Then, in addition to the canonical projections $\pi_\mfp \colon \CA \to \CA_\mfp$, there exist canonical projections $\pi_{\mfp\mfq} \colon \CA_{\mfq} \to \CA_{\mfp}$ such that $\pi_{\mfp} = \pi_{\mfp \mfq} \circ \pi_{\mfq}$ when $\mfp \leqslant \mfq$. Then $\CA$ is the projective limit of the system $\left( \left( \CA_\mfp\right)_{\mfp}, \left( \pi_{\mfp\mfq}\right)_{\mfp \leqslant \mfq}\right)$ inside of $\prod_{\mfp \in \mfP} \CA_{\mfp}$. The same holds for $T_{<\gamma}$ being a projective limit inside of $\prod_{\mfp \in \mfP} T^\mfp_{< \gamma}$.

\begin{corollary}
	There exists a stopped solution $u$ of \eqref{eq:LiftEq} with values in the regularity  structure $\prod_{\mfp \in \mfP }\cT^{\mfp}$, given by
	\begin{equ}
		u = \left( \bone_{ (0,\tau_\mfp) } u^\mfp \right)_{\mfp \in \mfP} \; . 
	\end{equ}
	If $\tau \eqdef \inf_{\mfp \in \mfP} \tau_\mfp  > 0$, then $u$ takes value in $\cT$ on the temporal interval $(0,\tau)$.  
\end{corollary}

Beyond this, Theorem~7.1, Lemmata~7.3~\&~7.5, Proposition~7.11, and Corollary~7.12 of \cite[Section~7]{Hai14} continue to hold in the setting of $\CA$-regularity structures. These allow one to extend local-in-time solutions up until their local maximal interval of existence $[0,T^\mfp_{\max})$.

\section{\TitleEquation{\CA}{A}-Regularity Structures Generated by Trees}\label{sec:TreeRegStruc}
In this section, in analogy with \cite[Sec.~5.1]{BHZ19}, we describe a more concrete class of $\CA$-regularity structures generated by trees which will be the most relevant for applications to SPDE.

\subsection{Construction of \TitleEquation{\CH}{H}}

Let $\mfL^- \eqdef \{\Xi_1, \dots, \Xi_k\}$ be a finite set of noise types and $\mfL^+ \eqdef \{\CI_1, \dots, \CI_\ell\}$ be a finite set of integration kernel types. Let $\mfL \eqdef (\mfL^- \sqcup \mfL^+) \times \N^d$.
We assume that $\mfL^- \sqcup \mfL^+$ is equipped with an $\R$-valued grading $|\,\bigcdot\,|_{\s}$, such that $|\mft|_{\s} > 0$ for all $\mft \in \mfL^+$. We extend this to $\mfL$ by setting $|(\mfl, k)|_{\s} \eqdef |\mfl|_{\s} - |k|_{\s}$. Furthermore, by slight abuse of notation, let $\Z(\mfL) \eqdef \Z^d \oplus \Z(\mfL^- \sqcup \mfL^+)$, where $\Z(\mfL^- \sqcup \mfL^+)$ is the free Abelian group generated by the set $\mfL^- \sqcup \mfL^+$. We extend the grading to $\Z(\mfL)$ by letting
\begin{equ}
	\left|\left( \textstyle \sum_i \ell_i \mft_i , k \right)\right|_{\s} \eqdef  \sum_i \ell_i |\mft_i|_{\s} - |k|_{\s} \; .
\end{equ}
%\martinp{I think for a non-translation invariant case it is enough add one extra symbol $\mfc$ to $\mfL$ for each counterm where $|\mfc|$ equals the degree of the negative tree. This would also make the extended decoration unnecessary, as the counterm edges would already contain the same information. $(\Pi^{(\eps)}_x \mfc )(y) = c_\eps(y)$ equals the value of counterterm at that point and I think $\Gamma_{xy}$ should act trivially. }

For a rooted tree $T$, we shall denote its set of vertices by $N_T$ and its edges by $E_T$.
Its root shall be called $\rho(T)$ or $\rho_T$.
A leaf of $T$ is a vertex that has degree $1$ and which is distinct from the root $\rho_T$.
We shall call an edge $e$ connected to vertex $v$ outgoing if it does not lie on the unique shortest path connecting $v$ to $\rho_T$. The set of outgoing edges of a vertex $v$ will be denoted by $\delta(v)$. We call a tree ordered if for all $v \in N_T$ each of the sets $\delta(v)$ is totally ordered, which induces a total order on $E_T$ by taking the lexicographic order. We will assume that the trees we are using are ordered.

An ordered disjoint union of rooted trees will be called a forest. A morphism of forests is an order-preserving injective map between forests such that vertices are mapped to vertices, nodes to nodes, and the connectivity structure is preserved. We say that a $(\mff,\iota)$ is a subforest of a tree $T$, if $\mff$ is a forest and $\iota \colon \mff \to T$ is a morphism of forests. We will typically suppress the map $\iota$ from the notation.

\begin{definition}
	\label{def:DecTrees}
 	Let $\CA$ be a locally $m$-convex algebra. Let $\mfT = \mfT(\CA)$ denote the set of quintuples $\tau = (T,\mfe, \mfn,\mfo, \mfa)$ consisting of
    \begin{itemize}
        \item the underlying tree for $\tau$, $T = (N_T, E_T)$ is an ordered, rooted tree.
        \item the edge decoration, $\mfe \colon E_T \to \mfL$. We say $e$ is noised valued if $\mfe(e) \in \mfL_{-} \times \N^{d}$, and enforce that $e$ noise valued $\Rightarrow$ $e$ is connected to a leaf.
        \item the polynomial decoration, $\mfn \colon N_T \to \N^d$, associating a term $X^k$ to each node and satisfying the constraint that $n(v) = 0$ for all $v \in N_{T,+}$ where $N_{T,+} \subset N_{T}$ are those nodes that are not leaves connected to a noise type edge.
        \item the negative tree decoration, $\mfo = \big( (\mff , \mfe_\mff, \mfn_\mff) , \bpi\big)$, where $\mff$ is a subforest of $T$, $\mfe_\mff$ and $\mfn_\mff$ are edge and node decorations as above, and $\mfe_\mff = \mfe \big|_{\mff}$, and $\bpi$ is a partition of the set of negative leaves of $\mff$.
		For every $\pi \in \bpi$, the leaves $v \in \pi$ have to belong to a single tree of $\mff$.
        \item the algebra decoration, $\mfa \in  \CA^{\wotimes N_{T,+} \sqcup E_{T} }$. We also view $\CA^{\wotimes N_{T,+} \sqcup E_{T} } \subset  \CA^{N_{T} \sqcup E_{T}}$  by adding factors of $\one_{\CA}$.
    \end{itemize}

For those familiar with the construction of regularity structures in \cite{BHZ19}, the biggest changes in our setting are replacing the extended decoration of \cite{BHZ19} with the ``richer'' negative tree decoration, and the algebra decoration.

\begin{example}
\label{ex:ExTrees1}

We use as an example the quintic tree from the $\Phi^4_3$ stochastic quantisation equation.
We first discuss the negative tree decoration.

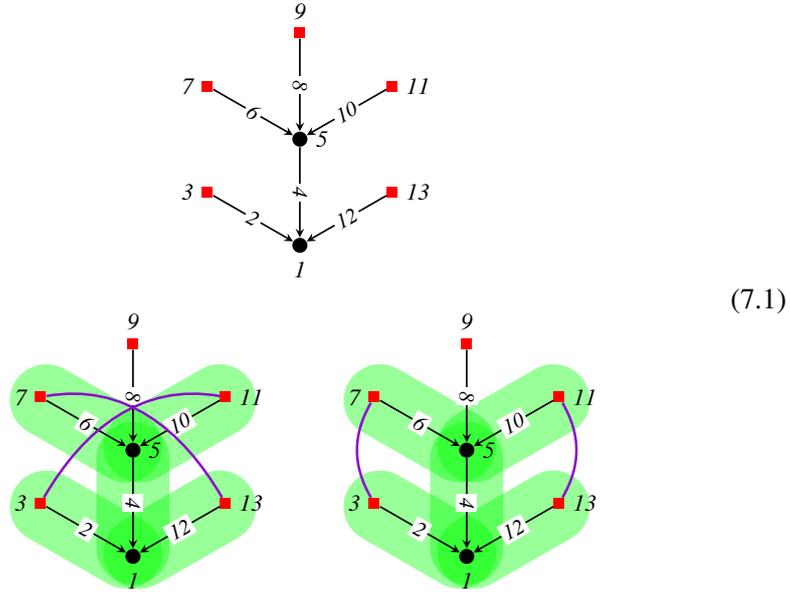
\begin{equ}[eq:ExTrees1]
\begin{gathered}
\scalebox{0.8}{
\begin{tikzpicture}[
    % --- Styles --- (These are unchanged)
    red_square/.style={rectangle, fill=red, inner sep=2.5pt},
    black_dot/.style={circle, fill=black, inner sep=2.5pt},
    arrow/.style={->, >=stealth, thick},
    every edge quotes/.style={font=\small, sloped, fill=white, inner sep=1.5pt}, 
	scale = 0.8
]
% --- 1. Define Visible Node positions ---
% The coordinates have been completely recalculated to make all edge lengths equal,
% resulting in a more compact and symmetric arrangement.
\node[black_dot, label=below:1]   (1) at (0,0) {};
\node[red_square, label=left:3]     (3) at (-1.9, 1.1) {};
\node[black_dot, label=right:5]    (5) at (0, 2.2) {};
\node[red_square, label=left:7]     (7) at (-1.9, 3.3) {};
\node[red_square, label=above:9]    (9) at (0, 4.4) {};
\node[red_square, label=right:11]    (11) at (1.9, 3.3) {};
\node[red_square, label=right:13]    (13) at (1.9, 1.1) {};

% --- 2. Draw Edges ---
% This section is unchanged. The new node coordinates automatically adjust the edge lengths.
\draw[arrow] (3) edge ["2"] (1);
\draw[arrow] (13) edge ["12"] (1);
\draw[arrow] (7) edge ["6"] (5);
\draw[arrow] (9) edge ["8"] (5);
\draw[arrow] (11) edge ["10"] (5);
\draw[arrow] (5) edge ["4"] (1);
\end{tikzpicture}
}
\\
\scalebox{0.8}{
	\begin{tikzpicture}[
    % --- Styles ---
    red_square/.style={rectangle, fill=red, inner sep=2.5pt},
    black_dot/.style={circle, fill=black, inner sep=2.5pt},
    arrow/.style={->, >=stealth, thick},
    every edge quotes/.style={font=\small, sloped, fill=white, inner sep=1.5pt},
    highlight/.style={green, opacity=0.4, line cap=round, line join=round, line width=1.2cm},
    % The 'spring' style is replaced with a simple 'purple_line' style.
    purple_line/.style={draw=purple, very thick}, 
	scale = 0.8
]
% --- 1. Define Node positions ---
% (This section is unchanged)
\node[black_dot, label=below:1]   (1) at (0,0) {};
\node[red_square, label=left:3]     (3) at (-1.9, 1.1) {};
\node[black_dot, label=right:5]    (5) at (0, 2.2) {};
\node[red_square, label=left:7]     (7) at (-1.9, 3.3) {};
\node[red_square, label=above:9]    (9) at (0, 4.4) {};
\node[red_square, label=right:11]    (11) at (1.9, 3.3) {};
\node[red_square, label=right:13]    (13) at (1.9, 1.1) {};

% --- 2. Draw Highlights on the Background Layer ---
% (This section is unchanged)
\begin{pgfonlayer}{background}
    \draw[highlight] (7) -- (5);
    \draw[highlight] (11) -- (5);
    \draw[highlight] (3) -- (1);
    \draw[highlight] (13) -- (1);
    \draw[highlight] (5) -- (1);
\end{pgfonlayer}

% --- 3. Draw Original Edges ---
% (This section is unchanged)
\draw[arrow] (3) edge ["2"] (1);
\draw[arrow] (13) edge ["12"] (1);
\draw[arrow] (7) edge ["6"] (5);
\draw[arrow] (9) edge ["8"] (5);
\draw[arrow] (11) edge ["10"] (5);
\draw[arrow] (5) edge ["4"] (1);

% --- 4. Add the new purple line connections ---
% The line connecting 3 and 11 is now bent upwards.
\draw[purple_line] (3) to[out=60, in=170] (11);
\draw[purple_line] (7) to[out=10, in=120] (13);

\end{tikzpicture}
}
\qquad
\scalebox{0.8}{
	\begin{tikzpicture}[
    % --- Styles ---
    red_square/.style={rectangle, fill=red, inner sep=2.5pt},
    black_dot/.style={circle, fill=black, inner sep=2.5pt},
    arrow/.style={->, >=stealth, thick},
    every edge quotes/.style={font=\small, sloped, fill=white, inner sep=1.5pt},
    highlight/.style={green, opacity=0.4, line cap=round, line join=round, line width=1.2cm},
    % New style for the purple polygon line
    purple_line/.style={draw=purple, very thick}, 
	scale = 0.8
]
% --- 1. Define Node positions ---
% (This section is unchanged from the base version)
\node[black_dot, label=below:1]   (1) at (0,0) {};
\node[red_square, label=left:3]     (3) at (-1.9, 1.1) {};
\node[black_dot, label=right:5]    (5) at (0, 2.2) {};
\node[red_square, label=left:7]     (7) at (-1.9, 3.3) {};
\node[red_square, label=above:9]    (9) at (0, 4.4) {};
\node[red_square, label=right:11]    (11) at (1.9, 3.3) {};
\node[red_square, label=right:13]    (13) at (1.9, 1.1) {};

% --- 2. Draw Highlights on the Background Layer ---
% (This section is unchanged from the base version)
\begin{pgfonlayer}{background}
    \draw[highlight] (7) -- (5);
    \draw[highlight] (11) -- (5);
    \draw[highlight] (3) -- (1);
    \draw[highlight] (13) -- (1);
    \draw[highlight] (5) -- (1);
\end{pgfonlayer}

% --- 3. Draw Original Edges ---
% (This section is unchanged from the base version)
\draw[arrow] (3) edge ["2"] (1);
\draw[arrow] (13) edge ["12"] (1);
\draw[arrow] (7) edge ["6"] (5);
\draw[arrow] (9) edge ["8"] (5);
\draw[arrow] (11) edge ["10"] (5);
\draw[arrow] (5) edge ["4"] (1);

% --- 4. Add new purple lines connecting the two pairs of nodes ---
\draw[purple_line] (3) to[out=120, in=-120] (7);
\draw[purple_line] (11) to[out=-60, in=60] (13);

\end{tikzpicture}
}
\end{gathered}
\end{equ}
In the picture above we work with one edge of type $(\mft_{+},0) \in \mfL_{+} \times \N^{d}$ drawn with black arrows pointing towards the root, and one ``edge'' of type   $(\mft_{-},0) \in \mfL_{-} \times \N^{d}$.
As is common, when we draw them, we collapse noise-type edges (and the two nodes they are incident to) into vertices, indicated by red squares.
Above, the node decoration is taken to vanish, and we also take a unit, or ``empty'' algebra decoration.

In the top picture, we have an empty negative tree decoration, while in the two pictures below, we give two examples of the same tree but with two different negative tree decorations.
In both of these pictures, $(\mff, \mfe_{\mff})$ are the same and indicated by the green highlighting.
However, the choice of $\bpi$ differs between these two trees --  each $\bpi$ is indicated by a perfect pairing drawn with purple lines.

Analogously to the extended decoration of \cite{BHZ19}, the negative tree decoration encodes data about substructures that we have renormalised.
However, in our current noncommutative setting, the rule for ``extracting'' the divergent substructure can differ for different Wick contractions on the substructure since the extraction procedure may involve algebra decorations or non-contracted noises (like $9$ in the above picture) ``crossing'' Wick contractions.

We now turn to the algebra decoration, which is used to give our regularity structure $\CA$-bimodule structure.

Recall that our trees are completely ordered \dash each node comes with a total order on edges incident to that node that connect to a node sitting further away from the root.
In this context, the factor of the algebra decoration $\mfa$ corresponding to an edge $e \in E_{T}$ corresponds to an algebra element sitting after this edge, after any descendant edges, but before any edges later in the total order that emanate from the same vertex that the edge $e$ is born from.
The algebra decoration associated with a node corresponds to an algebra element before its first edge.

We again use the quintic tree from $\Phi^4_3$, displayed at the top of the above figure, as an example.
We write
\begin{equ}
	N_{T,+} = \{1,3,5,7,9,11,13\}\;, \quad
	E_{T} = \{2,\widehat{3},4,6,\widehat{7},8,\widehat{9},10,\widehat{11},12,\widehat{13}\}\;.
\end{equ}
Note that we have introduced $\widehat{3},\widehat{7},\widehat{9},\widehat{11},\widehat{13}$ since each of the red triangles represents the collapsing of a noise edge and its corresponding nodes, so we need additional labels.

The quintic tree with an empty algebra decoration would be written as
\begin{equ}
	\CI(\Xi) \CI \big( \CI(\Xi)^3 \big) \CI(\Xi)\;.
\end{equ}
However, given a simple tensor algebra decoration $ \mfa = \otimes_{w \in N_{T,+} \sqcup E_{T}} a_{w}$ the corresponding tree would be
\begin{equ}
	a_{1} \CI(a_{3} \Xi a_{\widehat{3}}) a_{2}
	\CI \Big(
	a_{5}
	\CI(a_{7}\Xi a_{\widehat{7}} )
	a_{6}
	\CI(a_{9} \Xi a_{\widehat{9}})
	a_{8}
	\CI(a_{11} \Xi a_{\widehat{11}})
	a_{10}
	\Big)
	a_{4}
	\CI(a_{13} \Xi a_{\widehat{13}}) a_{12}\;.
\end{equ}
%We call the total order
%\begin{equs}
%1 <  3&  < \widehat{3} < 2 < 5 < 7 < \widehat{7} < 6 < 9 \\
%< &\ \widehat{9} < 8 < 11 < \widehat{11} < 10 < 4 < 13 < %\widehat{13} < 12
%\end{equs}
%on  $N_{T,+} \sqcup E_{T}$ the total order on ``attachment points''.
%\ajay{Perhaps either adapt the rule for assigning algebra elements or change how the tree is numbered....}
\end{example}

We shall generally write $\tau$ for the full quintuple $(T,\mfe, \mfn, \mfo, \mfa)$. \ajay{Can we use $\tau$ for the full quintuples, and $(T,\mfe,\mfn,\mfo)$ for classical trees?}
If we remove the decoration $\mfa$ from a tree $T$, we will call this its scalar skeleton. Trees lacking the $\mfa$ decoration or have trivial decoration $\bone_{\CA}^{|T|}$ will be called scalar. We will denote the set of scalar trees by $\mathring{\mfT}$. A tree is bare if it is scalar and $\mfn \equiv 0$. %\martinp{Change classical to scalar}

We shall denote the set vertices $\iota(N_\mff)$ by $N_\mfo$ and $\iota(E_\mff)$ by $E_\mfo$. We shall denote by $\mfo_\rho$ the (possibly empty) tree of $\mff$ that contains the root $\rho$ of $T$. If $\mfo_\rho \neq \emptyset$, we say that the tree is root-renormalised. Finally, a tree $\tau$ is called fully renormalised if $\mff = T$ and $\mfn \equiv 0$, i.e.\ if $\mfo$ fully covers $\tau$. %The subset of trees that are not root-renormalised will be denoted by $\mfT_\rho$, and the subset of $\mft$
\end{definition}

\begin{remark}
	We remark here again that one should think of the separate $\CA$ factors lying in between the outgoing branches of the tree at a vertex, as well as one copy of $\CA$ sitting to the left and one to the right of all outgoing edges. This is also reflected in the way one can inductively define a tree by multiplying it from the left and right with algebra elements, as well as adding edges to the root and multiplying it with other trees.
\end{remark}

\begin{definition}
	For a tree, we shall call the position on the tree associated with the $k^{\text{th}}$ algebra copy the $k^{\text{th}}$ attachment spot.

	These attachment spots are exactly the same as the possible locations where one can graft an ordered tree onto another ordered tree, cf.\ Example~\ref{ex:Grafting}. With this in mind, we let $|T|$ denote the ordered set of such attachment spots. 
\end{definition}

\begin{remark}
	$\mathring{\mfT}$ will be used to index the subspaces of the regularity structure. While in the scalar case these are all one-dimensional, in our case, they are generically going to be each infinite-dimensional.

	In the scalar case, one usually assumes that the regularity structure is finite-dimensional, which in this case translates into assuming that we only use a finite subset of $\mathring{\mfT}$.
\end{remark}

This set naturally allows for multiplication of each element $T \in \mfT$ from the left and the right by elements of $\CA$ by using the $\CA$-bimodule structure of $\CA^{\wotimes |T|}$.
% : For all $a,b \in \CA$, let $aTb$ be the tree $(T,\mft, \mfn, \bar\mfa)$ where
% \begin{equ}
%   \bar\mfa(\rho_T) \eqdef \left( a \bigl(\mfa(\rho_T)\bigr)^0 , \bigl(\mfa(\rho_T)\bigr)^1, \dots, \bigl(\mfa(\rho_T)\bigr)^{|\pi(v)|}, \bigl(\mfa(\rho_T)\bigr)^{|\pi(v)|} b\right)\;.
% \end{equ}
Note that this implies that, for any single vertex tree $(\bullet ,\emptyset, \mfn, \emptyset, \bone_\CA) = \bullet^{\mfn,\bone_\CA}$ and all $a \in \CA$, we have $a \bullet^{\mfn,\bone_\CA} = \bullet^{\mfn,\bone_\CA} a$. We also set $\1 \eqdef \bullet^{0, \bone_\CA}$.

\begin{definition}

    For $\mfk \in \mfL$, let $\CI_\mfk$ denote the map
    \begin{equ}
        \CI_{\mfk} \colon \mfT \longrightarrow \mfT
    \end{equ}
    that adds a new vertex $\rho_{\CI T}$ to a tree $T$, which is the root of $\CI T$, and adds a new edge $e = (\rho_T, \rho_{\CI T})$ connecting the old and new roots. $\mfe$ is extended to the new tree by setting $\mfe(e) = \mfk$, $\mfn$ is extended by setting $\mfn (\rho_{\CI T}) \eqdef 0$, $\mfo$ remains unchanged, and the new algebra element is $\bone_{\CA} \otimes \mfa \otimes \bone_{\CA}$. If $\mfk \in \mfL^- \times \N^d$, the integration map may only be applied to $\1$.

    For $k \in \N^d$, let $X^{k}$ denote the map
    \begin{equ}
        X^{k} \colon   {\mfT} \longrightarrow \mfT\;,
    \end{equ}
    which for a given tree $T$ only replaces the decoration of the root $\rho_T$ by setting the new decoration $\widetilde{\mfn}$
    \begin{equ}
        \widetilde\mfn(\rho_T) \eqdef \mfn(\rho_T) + k   \; .
    \end{equ}
	By abuse of notation we will also denote $\bullet^{k, \bone_{\CA}} = X^k \1$ by $X^k$ and $\CI_{\mfk}( \1 )  $ by $\Xi_{\mfk}$ for $\mfk \in \mfL^- \times \N^d$.

	% Analogously for $\alpha \in \Z(\mfL)$ with $|\alpha| < 0$, let $\cR_\alpha$ denote the map
	% \begin{equ}
	% 	\cR_\alpha \colon \mfT \longrightarrow \mfT \; ,
	% \end{equ}
	% that only changes the extended decoration $\mfo$ at the root $\rho_T$ setting $\widetilde{\mfo}$
	% \begin{equ}
	% 	\widetilde{\mfo} \eqdef \mfo (\rho_T) +\alpha.
	% \end{equ}
\end{definition}

\begin{remark}
  We note here that per definitionem $X^{k}$ commutes with multiplication by elements of $\CA$, i.e.\
  \begin{equ}
    a \bigl( X^k \tau \bigr) b = X^k \bigl( a\tau b \bigr) \; .
  \end{equ}
  Furthermore, $X^k$ does not modify the partial contraction $\mfo$.
  %and also with $\CI^j_{\mfl}$ if $j \neq c$.

  At this stage, we have not specified how multiplication by elements of $\CA$ might commute with $\CI_\mfk$.
\end{remark}

% \begin{remark}
% 	The operation $\cR_\alpha$ will start playing a visible role once we start explicitly describing the renormalisation process, where each contracted tree corresponds
% \end{remark}

Finally, we can also multiply two trees.

\begin{definition}
	For two trees $(T_1, \mfe_1, \mfn_1, \mfo_1 , \mfa_1), (T_2,  \mfe_2, \mfn_2, \mfo_2 ,  \mfa_2) \in \mfT$, with at most one of them being root-renormalised, we define their product to be%\martinp{Check whether root-renormalised is necessary.}
  	\begin{equ}
    	(T_1, \mfe_1, \mfn_1, \mfo_1, \mfa_1) (T_2,  \mfe_2, \mfn_2 \mfo_2,  \mfa_2) \eqdef (T_1 T_2, \mfe, \mfn, \mfo, \mfa)
  	\end{equ}
  	where $T_1T_2$ is the usual root joining product of ordered trees, $\mft$ is defined on $E_{T_1} \sqcup E_{T_2}$ in the obvious way, for $v \neq \rho_{T_1T_2}$, $\mfn(v)$ is defined in the obvious way, whilst
	\begin{equ}
    	\mfn (\rho_{T_1T_2}) \eqdef  \mfn(\rho_{T_1})+\mfn(\rho_{T_2}) \; ,
	\end{equ}
	$\mfo \eqdef \mfo_1 \cup \mfo_2$, i.e.\ you take the union of the forests $\mff_1$ and $\mff_2$ as well as $\bpi_1$ and $\bpi_2$, and the new algebra element is given by $\mfa \eqdef   \mu_{n,1}(\mfa_1, \mfa_2)$ where
	\begin{equ}
		\mu_{n,1} \colon \CA^{\wotimes n} \times \CA^{\wotimes m} \longrightarrow \CA^{\wotimes (n+m-1)}
	\end{equ}
	is the natural multiplication map that multiplies the last term in $\CA^{\wotimes n}$ with the first one in $\CA^{\wotimes m}$.
%     \mfa (\rho_{T_1T_2}) & \eqdef \Bigg( \bigl(\mfa(\rho_{T_1})\bigr)^0, \dots, \bigl(\mfa(\rho_{T_1})\bigr)^{|\pi(\rho_{T_1})|-1},\\
%     {}& \qquad \enskip \bigl(\mfa(\rho_{T_1})\bigr)^{|\pi(\rho_{T_1})|} \bigl(\mfa(\rho_{T_2})\bigr)^{0},\\
%     {}& \qquad  \enskip \bigl(\mfa(\rho_{T_2})\bigr)^{1}, \dots,  \bigl(\mfa(\rho_{T_2})\bigr)^{|\pi(\rho_{T_2})|} \Bigg)\;.
%   \end{equs}
  	Note that  $\mu_{n,1}(\mfa_1,\mfa_2) \in \CA^{\wotimes|T_1T_2|}$ as required. This multiplication is associative and compatible with the multiplication by elements of $\CA$. By Proposition~\ref{prop:ProjTensExt}, it also extends to the $t$-projective tensor product.
\end{definition}

% \begin{remark}
% 	If both trees were root-renormalised, then $\mfo_1 \cap \mfo_2 \neq \emptyset$ (as subsets of $T_1T_2$) and we would have to join the root components of $\mfo_1$ and $\mfo_2$, a situation we wish to avoid as the counterterm belong to this joined negative subforest would be different from the product of the separate counterterms. \martinp{Such products do not appear in practice when solving equation because planted. }
% \end{remark}

\begin{definition}
  Let $\CH$ denote the $\CA$-bimodule generated by $\mfT$ using the above ``scalar'' multiplication, i.e.\
  \begin{equ}
	\CH \eqdef \bigoplus_{(T, \mfe , \mfn, \mfo) \in \mathring{\mfT}} \CA^{\wotimes |T|} \; ,
  \end{equ}
  equipped with the product topology. We will denote the subspace of $\CH$ isomorphic to $\CA^{\wotimes |T|}$ that corresponds to $(T, \mfe , \mfn, \mfo)$ by $\CH[(T, \mfe , \mfn, \mfo)]$.  By abuse of notation we say that a submodule $\CK$ of $\CH$ is finite-dimensional if there exists a finite subset $\mfQ \subset \mathring{\mfT}$, s.t.\
  \begin{equ}
	\CK \subset \bigoplus_{(T, \mfe , \mfn, \mfo) \in \mfQ}  \CH[(T, \mfe , \mfn, \mfo)] \; .
  \end{equ}
  We will also call $\bigoplus_{\mfQ}  \CH[(T, \mfe , \mfn, \mfo)]$ the subspace spanned by $\mfQ$.

  $\CH$ is a locally convex algebra and locally $m$-convex over $\CA$ but not complete nor itself $m$-convex. However, any finite-dimensional submodule is complete.

  Let $\CH_\rho$ denote the submodule of $\CH$ generated by trees $\tau$ that are not root-renormalised. Note that $\CH_{\rho}$ forms an algebra under the above tree multiplication. 
  
  Finally, let $\CH_0$ denote the subalgebra of $\CH_\rho$ of trees which have $0$ polynomial decoration at the root, i.e.\ $\mfn(\rho) = 0$.
\end{definition}

\begin{proposition}
	The bimodule operations, $\CI_{\mfk}$, and the multiplication of trees extend to continuous maps on $\CH$.
\end{proposition}

Finally, we will need to be able to inductively represent a decorated tree $\tau$ as a bare tree $\tau_\rho$ with $\iota(\mfo_\rho) = T_\rho$ and with a set of trees $\tau_i \in \CH_\rho$ attached to it.

%we will need to introduce a grafting operation, i.e.\ attaching a tree $T$ at its root to an arbitrary vertex $v$ of another tree $T'$. By this we mean that we identify the root of $T$ with $v$. However, when considering ordered trees, one has $|\delta(v)|+1$ different ways to graft $T$ onto $T'$, choosing between which edges in $\delta(v)$ on puts the new tree $T'$. Since we consider trees with an algebra decoration sitting in-between the edges of $T'$, the choices in fact double, as one can place the new tree $T'$ to the left or right of the algebra.

\begin{definition}[Grafting]

	Let $\tau = (T, \mfe , \mfn' , \mfo, \bone_\CA^{\otimes |T|} )$ be a bare tree with $\iota(\mff) = T$ and let $n \eqdef |T|$, let $\tau_1, \dots, \tau_n$ be decorated trees which have zero polynomial decoration at their roots, and let $\mfn \in \left( \N^d \right)^{N_T}$.

	We denote by
	\begin{equ}\label{eq:grafting1}
		(\tau_1, \dots, \tau_n; \mfn )  \curvearrowright \tau
	\end{equ}
	the tree obtained by attaching the tree $\tau_k$ at the $k^{\text{th}}$ attachment spot and setting the polynomial decoration $v \in T$ to $\mfn(v) + \mfn'(v)$. The new extended decoration is the disjoint union of the ones of $\tau$, $\tau_1$, \dots, $\tau_n$.

	The definition above can be interpreted as defining a $\mathbb{K}$-multilinear continuous map
	\begin{equ}
		\bigcdot \curvearrowright \tau \colon \CH_\rho^{\wotimes_\pi n} \otimes \mathbb{K}\left[ \N^d \right]^{\otimes N_T} \longrightarrow \CH \; ,
	\end{equ}
	where we have equipped $\mathbb{K}\left[ \N^d \right]^{\otimes N_T}$ with the discrete topology.

	\begin{remark}
		Here one should interpret $\mfn \in (\N^d)^{N_T}$, as the collection $(X^{\mfn(v)})_{v \in N_T}$, and the corresponding element in $\mathbb{K}[\N^d]^{\otimes N_T}$ as $\bigotimes_{v \in N_T} X^{\mfn(v)}$.
	\end{remark}

	\begin{remark}
		It will also be convenient to write \eqref{eq:grafting1} as 
		\begin{equ}
			\Bigl( \bigotimes_{v \in N_T} \left( \tau_{v,1} \otimes \cdots \otimes \tau_{v,n_v} \otimes  X^{\mfn(v)} \right) \curvearrowright_v \Bigr) \tau
		\end{equ}
		where we have grouped together all trees being attached to the same node $v$, and $\curvearrowright_v$ denotes grafting only onto the vertex $v$. Here, we use the fact that the set of vertices induces a partition on the set of attachment spots. 
	\end{remark}

	% For two decorated Trees $\tau_1 ,  \tau_2$, $v \in N_{T_1}$, and $n \in [2 |\delta(v)| + 2]$, let $\tau_2 \curvearrowright_{v,n} \tau_1$ be the tree defined as follows. Let $k  = \lfloor \frac{n}{2} \rfloor$, let $T$ be the tree obtained by identifying $\rho_{T_2}$ with $v$ and extending the order on $T_1$ to $T$ by that $e_1^k < e_2^1 < \cdots < e_2^\ell < e_1^{k+1}$, where $e_1^k$ and $e_1^{k+1}$ are the $k^{\text{th}}$ and $(k+1)^{\text{st}}$ edges of $\delta(v)$ and $e_2^1, \dots, e_2^\ell$ are the edges of $\delta(\rho_{T_2})$ with their inherited order.

	% We extend $\mfe$ and $\mfo$ in the obvious way, we extend $\mfn$ by $\mfn(\{\rho_{T_2}, v\}) = \mfn(v) + \mfn(\rho_{T_2})$ and in the ovious way for all other vertices.

	% Let $k' \in \left[ |T_1| \right]$ be in the index corresponding to the algebra sitting at the location where $\tau_2$ is being grafted onto $\tau_1$. For $\mfa_1 = a_1 \otimes \cdots \otimes a_{|T_1|}$ and $\mfa_2 = b_1 \otimes \cdots \otimes b_{|T_2|}$ we set
	% \begin{equs}
	% 	\mfa \eqdef \begin{cases}
	% 		a_1 \otimes \cdots \otimes a_{k'-1} \otimes b_1 \otimes \cdots \otimes b_{|T_2|} a_{k'} \otimes \cdots \otimes a_{|T_1|} \; , & \text{if } n \in 2\Z \\
	% 		a_1 \otimes \cdots \otimes a_{k'-1} \otimes a_{k'} b_1 \otimes \cdots \otimes b_{|T_2|}  \otimes \cdots \otimes  a_{|T_1|} \; ,   & \text{if } n \in 2\Z+1
	% 	\end{cases}
	% \end{equs}
\end{definition}

\begin{example}\label{ex:Grafting}
The following tree has four vertices and three edges, and thus exactly $7$ locations where a new rooted tree $T$ can be grafted onto it.
	\begin{equ}
		\scalebox{0.8}
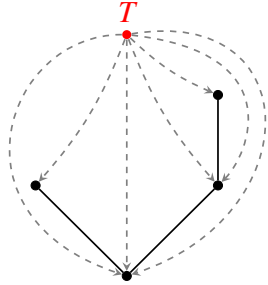
{\begin{tikzpicture}[
			node style/.style={circle, draw, fill=black, inner sep=1.5pt},
			T node style/.style={circle, fill=red, inner sep=1.5pt},
			T label style/.style={font=\Large\color{red}},
			line style/.style={thick},
			dashed line style/.style={dashed, gray, thick, ->}
			]

			% Nodes
			\node[T node style, label={[T label style]above:T}] (T) at (0,2.5) {};
			\node[node style] (n1) at (-1.5, 0) {};
			\node[node style] (n2) at (1.5, 0) {};
			\node[node style] (n3) at (0, -1.5) {};
			\node[node style] (n4) at (1.5, 1.5) {};

			% Solid lines
			\draw[line style] (n1) -- (n3);
			\draw[line style] (n2) -- (n3);
			\draw[line style] (n2) -- (n4);

			% Original dashed lines
			\draw[dashed line style] (T) to[bend left=10] (n1);
			\draw[dashed line style] (T) to[bend right=10] (n2);
			\draw[dashed line style] (T) -- (n3);
			\draw[dashed line style] (T) to[bend left=-10] (n4);

			% Smoothed outer dashed lines using two control points
			\draw[dashed line style] (T) .. controls (-2,2.5) and (-3,-0.5) .. (n3); % left around
			\draw[dashed line style] (T) .. controls (3,3) and (3,-0.5) .. (n3);    % right around
			\draw[dashed line style] (T) .. controls (2.7,2.2) and (2,0.5) .. (n2);     % outside right

		\end{tikzpicture}
		}
	\end{equ}
	We note that most of the grafting spots are only different due to the ordered nature of the tree. 
\end{example}

One can easily prove the following proposition.

\begin{proposition}
\label{prop:GraftRep}
	Any decorated tree $\tau$ may uniquely be written as
	\begin{equ}\label{eq:RepGraft}
		\tau = \left(\CF \otimes X^{\mfn}\right) \curvearrowright \tau_\rho
	\end{equ}
	where $T_\rho  \eqdef  \iota(\mfo_\rho)$ and $\tau_\rho = \left(T_\rho, \mfe\big|_{T_\rho}, 0 , \mfo_\rho, \bone_\CA^{\otimes |T_\rho|}\right)$ is fully renormalised,$ \CF \in \CH_0^{\wotimes |T_\rho|}$, and $\mfn(v) = \mfn_{\tau} (v)$ for all $v \in T_\rho$.

	Furthermore, for $\tau \in \CH[(T, \mfe , \mfn, \mfo)]$  the assignment $\tau \mapsto \CF$ is a homeomorphism onto its image, as, at the level of the algebra decoration $\CA^{\wotimes |T|}$,  it is just the identity map. 
	We call $\CF$ the non-root-renormalised part of $\tau$ 
\end{proposition} %\martinp{Add remark that we are writing $\tau_\rho$ without polynomial decorations.}
\begin{remark}
	Note here that $\tau_\rho$ has $0$ polynomial decoration.
\end{remark}

\begin{remark}
	The reason we chose to restrict the domain of $\curvearrowright$ to $\CH_0$, rather than $\CH_\rho$, was to ensure that the inductive representation is unique. 
	Furthermore, the reason we allow $\CF \in \CH_0^{\wotimes |T_\rho|}$ is that the algebra decoration need not be a simple tensor. 
	However, writing $\check{\tau}$ for the scalar form of $\tau$, we can find scalar trees $\check{\tau}_i \in \CH_0$, s.t.\ 
	\begin{equ}\label{eq:scalarGraftRep}
	\check{\tau} = 
		(\check{\tau}_1, \dots, \check{\tau}_n; \mfn )  \curvearrowright \tau_\rho \; .
	\end{equ}
	In fact, this always holds for the scalar skeleton of $\tau$.
\end{remark}

\begin{definition}
	We define the following two gradings on $\mfT$
	\begin{equs}
		|\tau|_\s &\eqdef \sum_{n \in N_T} |\mfn(n)|_{\s} + \sum_{e \in E_T} |\mfe(e)|_{\s} - \sum_{e \in E_\mff} |\mfe(e)|_{\s} \; ,\\
		|\tau|_+ &\eqdef \sum_{n \in N_T} |\mfn(n)|_{\s} + \sum_{e \in E_T} |\mfe(e)|_{\s} + \sum_{n \in N_\mff} |\mfn(n)|_{\s} \; .
	\end{equs}
	Let $\CH^{+}$ be the subalgebra of $\CH_\rho$ generated by trees $\tau$ that satisfy \ $|\tau|_+ \geqslant 0$ and $N_{\mfo} \cap \rho_T = \emptyset$. Furthermore, let $P \colon \CH \to \CH^{+}$ be the canonical projection. Let $\CJ_{\mfk} \eqdef P \circ \CI_{\mfk}$ for $\mfk \in \mfL^+\times \N^d$.

	Let $\cJ$ be the intersection of $\CH^+$ with the ideal $\CH$ generated by trees $\tau$, such that $|\tau|_+ \leqslant 0$.
\end{definition}

\begin{remark}
	The scalings $|\,\bigcdot\,|_\s$ and $|\,\bigcdot\,|_+$ record the regularity of the distribution represented by the tree, with $|\,\bigcdot\,|_\s$ disregarding the ``renormalised'' subforest $\mfo$.

	Furthermore, $\CH^+$ is the subalgebra that transforms non-trivially when recentred, and the projection $P$ exists as one can canonically write $\CH$ as a direct sum $\CH^+ \oplus \CH'$.

	Finally, $\cJ$ is the linear span of trees that, although overall positive, contain a negative factor with respect to the tree product.
\end{remark}

% \begin{definition}
%   The single vertex tree $\bullet$ with decorations
%   \begin{equs}
%     \mfn(\bullet) = 0, \qquad \mfa(\bullet) = (\bone_{\CA})
%   \end{equs}
%   will denote the unit of the bialgebra $\CH$, which we shall also denote by $\1$
% \end{definition}

\begin{definition}[Unit \& Co-Unit]
  We define the inclusion map $\eta \colon \CA \to \CH$ with $\eta(a) = a \1$. For a tree $\tau$, we define $\1^* \colon \CH \to \CA$ to be $\1^*(\tau) = a$ if $\tau = a\bone$, otherwise $\1^*(\tau) = 0$, and extend this by $\CA$-bilinearity.

\end{definition}

\begin{remark}  
$\1^*$ and $\eta$ satisfy the identity $  \1^* \circ \eta = \bone_{\CA}$. 
Moreover, $\1^*$ and $\eta$ are both multiplicative, i.e.\ for $\tau,\tau' \in \CH$ , $\1^*(\tau\tau') = \1^*(\tau) \1^*(\tau')$ and $\eta(\tau\tau') = \eta(\tau) \eta(\tau')$. 
\end{remark}

\subsection{Inductive Construction of Continuous Maps on Trees}
\label{sec:MetaIndExp}

Throughout this section, we shall inductively define multiplicative, $t$-continuous, $\CA$-bimodule morphism $\phi$ on $\CH$ by first defining them on ``primitive'' trees, i.e.\ $X_\mu$ or $\CI_{\mfk} \sigma$ for some tree $\sigma$.

In particular, the definition of $\CI_{\mfk} \sigma$ will only depend on $\mfk \in \mfL$ and the previously defined action of the map on $\sigma$, which is the main inductive step.

We then extend $\phi$ multiplicatively to arbitrary trees in $\CH_0$. Trees in $\CH_0$ can be non-uniquely represented as elements of $\left( \bigwotimes_\pi \right) _{i = 1}^n \CI_{\mfk_i} \CH$ and $\phi^{\otimes n}$ uniquely extends to a $t$-continuous map on this space. The non-uniqueness is not an issue as we require $\phi$ to be an $\CA$-bimodule morphism, so any two representations yield the same result. Analogously, we also extend $\phi$ to products of $X_\mu$, although we do not need to worry about continuity in this case.

Finally, for root-renormalised trees, we use the representation of the tree from Proposition~\ref{prop:GraftRep}. If a map $\phi$ is a $t$-continuous map that maps $\CH_0 \to \CH_\rho$ and $\mathbb K[\N^d] \to \mathbb K[\N^d]$, then it extends uniquely to a continuous map $\phi^{\wotimes_\pi n } \colon \CH_0^{\wotimes_\pi n} \to \CH_\rho^{\wotimes_\pi n}$. Therefore, for any scalar tree $(T, \mfe , \mfn, \mfo)$ the map $\phi \colon \CH[(T, \mfe , \mfn, \mfo)]  \to \CH$
\begin{equs}
	\tau & \longmapsto \left( \phi^{\wotimes_\pi |T| }(\CF) \otimes \phi^{\otimes |N_T|}(X^{\mfn}) \right) \curvearrowright \tau_\rho
\end{equs}
is continuous and well-defined, using Proposition~\ref{prop:GraftRep}.

%Let $\CH_P \supset \CH_0$ be the closed $\CA$-bilinear span of primitive trees in $\CH$. Any tree $\tau \in \CH$ with $n$ edges at the root and $k = \left|\mfn\left(\rho_\tau\right)\right|$

\subsection{Characters}

To define the character group, as well as the structure group of the regularity structure and the canonical model, we employ the recursive approach of \cite{Br18}. This approach seems easier to adapt to the noncommutative structure of $\CA$-regularity structures when compared to the coproduct \slash Hopf algebra\footnote{See  \cite{BellingeriGilliers22} where a coproduct \slash Hopf algebra approach is used in the context of rough paths.
A related approach would likely apply in the setting of regularity structures but it is not clear to us if it would overall be a significant improvement on the recursive approach we take here.} approach of \cite{BHZ19}.

\begin{definition}
  A character of $\CH^{+}$ is a multiplicative, $t$-continuous, $\CA$-bimodule morphism $g \colon \CH^{+} \to \CA$ that vanishes on the ideal $\cJ$. We shall denote the collection of characters of $\CH^{+}$ by $\CG(\CH^{+})$.
\end{definition}
\begin{remark}
    Since $X^k$ is in the centre of $\CH^+$ it follows that $g(X^k)$ must be in the centre of $\CA$ for all $g \in \CG(\CH^+)$.
\end{remark}
For each element $g \in \CG(\CH^+)$ we define a multiplicative, $t$-continuous, $\CA$-bimodule morphism $\Gamma_g \colon \CH \to \CH$ inductively as explained in Section~\ref{sec:MetaIndExp}. If a trees $\sigma$ is primitive, then it must be of the form $\1$, $\Xi_\mfl$, $X_\mu$, or $\CI_\mfk \tau$, with $\mfk \in \mfL^+ \times \N^d$ and $\mfl \in \mfL^- \times \N^d$. For these, we set
\begin{equs}[eq:GamDef]
    \Gamma_g \1 &\eqdef \1 \; , \\
    \Gamma_g \Xi_\mfl &\eqdef \Xi_{\mfl} \; , \\
    \Gamma_g X_\mu &\eqdef X_\mu + g(X_\mu) \; , \\
	%\Gamma_g \cR_\alpha \tau & = \cR_\alpha \Gamma_g \tau\\
    \Gamma_g (\CI_\mfk \tau) &\eqdef \CI_{\mfk}(\Gamma_g \tau) + \sum_{\ell} \frac{X^\ell}{\ell!} g \left( \CJ_{\mfk+\ell} \tau\right) \; .
\end{equs}
Since $g$ is $t$-continuous by assumption, so is $\Gamma_g$ on the given trees. For products of trees $\tau$, $\bar\tau$ that are not root-renormalised, we extend $\Gamma_g$ as described in Section~\ref{sec:MetaIndExp}, i.e.\ by setting
\begin{equ}
	\Gamma_g (\tau\bar\tau) = \Gamma_g(\tau)\Gamma_g(\bar\tau)
\end{equ}
By Proposition~\ref{prop:GraftRep}, if $\tau \in \CH$ is root-renormalised, we can write $\tau$ uniquely as
\begin{equ}
	\tau = \Bigl(\CF \otimes \bigotimes_{v \in N_T} X^{\mfn(v)}\Bigr) \curvearrowright \tau_\rho \; ,
\end{equ}
with $\CF \in \CH_0^{\wotimes_\pi |\tau_\rho|}$ and $\tau_\rho$ fully renormalised. We set
\begin{equ}
	\Gamma_g \tau \eqdef \Bigl(\Gamma_g \left(\CF \right) \otimes \bigotimes_{v \in N_T} \Gamma_g \left( X^{\mfn(v)} \right) \Bigr) \curvearrowright \tau_\rho \; .
\end{equ}

% \begin{remark}\label{rem:rootrenormproducts}
% 	Note that if $\tau = \tau_1 \tau_2$ then at most one of the two trees may be root-renormalised. 
% 	Thus, $\Gamma_g (\tau_1 \tau_2) = \Gamma_g (\tau_1) \Gamma_g (\tau_2)$ holds for all reducible trees in $\CH$. \martinp{Redo this without the constraint.}
% \end{remark}

\begin{remark}
	Note that the sum in the definition of the model of $\CI_\mfk \tau$ is necessarily finite as $\CJ_{\mfk + \ell } \tau $ can be non-zero for only finitely many $\ell$.

	%Furthermore, the extended decoration at the root is essentially ignored, however, it can impact the model further up along the tree as it can reduce the $|\,\bigcdot\,|_+$-degree of $\CJ_{\mfk+\ell}\tau$, and thus also truncate the sum even further. \martinp{Move this to definition of structure group}
\end{remark}

\begin{remark}
	A straightforward induction shows that the $|\,\bigcdot\,|_+$-degree of all summands of $\Gamma \tau$ does not exceed the $|\,\bigcdot\,|_+$-degree of $\tau$ for all trees $\tau$.
\end{remark}

% $\mfo = \mfo_1 \sqcup \mfo_2$ where $N_{\mfo_2} \cap \rho_{T} =  \emptyset$ and the root of every tree of $\mfo_1$ is $\rho_T$. Then we can write $\tau$ as
% \begin{equ}
% 	\tau =  \prod_{i = 1}^\ell \tau_{i} \curvearrowright_{v_i, n_i} \tau_\rho
% \end{equ}
% where $\tau_\rho = \left(\iota( \mff_1 ), \mfe\big|_{\iota( \mff_1 )} , 0, \mfo_1, \bone_\CA^{\otimes |\iota(\mff_1)|} \right)$. We set
% \begin{equ}
% 	\Gamma_g \tau = \prod_{i = 1}^\ell \Gamma_g\left(\tau_{i}\right) \curvearrowright_{v_i, n_i} \tau_\rho
% \end{equ}
% Because of the multiplicativity and $\CA$-bilinearity of $\Gamma_g$, the definition of $\Gamma_g$ does not depend on the specific choice of $\tau_i$.

\begin{lemma}
\label{lemma:CharHomo}
    $\CG(\CH^{+})$ forms a group under the operation
    \begin{equs}
    \CG(\CH^+) \times \CG(\CH^+) & \longmapsto \CG(\CH^+)\\
	(g, \bar{g}) &\longmapsto g \cdot \bar{g}
	\end{equs}
 	inductively defined via
    \begin{equs}
        (g \cdot \bar g)(X_\mu) &= g(X_\mu) +  \bar g(X_\mu)\;, \\
        (g\cdot \bar g)(\tau \bar\tau) &= (g\cdot \bar g)(\tau)  (g\cdot \bar g)(\bar\tau)  \;,\\
        (g\cdot \bar g)( \CJ_{\mfk} \tau ) &= g\left( \CJ_\mfk \left( \Gamma_{\bar g} \tau \right) \right) +  \sum_{\ell} \frac{g(X^\ell)}{\ell!} \bar g\left( \CJ_{\mfk + \ell} \tau \right) \;,
    \end{equs}
    unit $\1^*$, and inverse $g \mapsto g^{-1}$ inductively defined via
    \begin{equs}
        g^{-1}(X_\mu) &= - g(X_\mu) \;,\\
        g^{-1}(\tau\bar\tau) &= g^{-1}(\tau)g^{-1}(\bar\tau) \;,\\
        g^{-1}(\CJ_\mfk \tau) &= - \sum_{\ell} \frac{(-g(X))^\ell}{\ell!} g \left( \CJ_{\mfk + \ell} \left( \Gamma_{g^{-1}} \tau \right)\right)\;.
    \end{equs}
    Furthermore, $\Gamma$ is a group action with respect to this multiplication, i.e.\
    \begin{equs}
        \Gamma_{g \cdot \bar g} = \Gamma_{g} \circ \Gamma_{\bar g}\;, \quad
		\Gamma_{\1^*} = \bone\;, \quad \text{and} \quad
		\Gamma_{g^{-1}} = \left(\Gamma_g\right)^{-1}\;.
    \end{equs}
\end{lemma}
We leave the proof to Appendix~\ref{app:alg} as these are mostly standard computations. We note that the noncommutative nature of the regularity structures does not feature in these computations since one only needs to commute products of $g(X_\mu)$ past other elements, and we have assumed that $g(X_\mu)$ is always central. 
\begin{remark}
	There is no specification for trees with $\mfo \cap \rho_T \neq \emptyset$ as we assume that the trees are in $\CH^+$.
\end{remark}

\subsection{Renormalising Functional}

To define renormalisation recursively, we split it into two steps: first, renormalising at the root, and then renormalising everything further up along the tree.
Particular choices of renormalisation will be encoded using preparation maps as introduced in \cite{Br18}. % however we will call these maps root renormalisation maps. \ajay{The heuristic of thinking of these maps as root renormalisation maps is useful and good to mention, but in general what a preparation map does to a tree is not necessarily local to the root, so maybe keeping the name preparation map makes sense?}

For a tree $\tau$, let $|\tau|_\Xi$ denote the number of edges $e \in E_\tau \setminus E_{\mfo}$ with $\mfe(e) \in \mfL^- \times \N^d$ and we extend this as well as $|\,\bigcdot\,|_\s$ to linear combinations $\tau = \sum_i \tau_i \in \CH$ of trees $\tau_i$, by setting
\begin{equ}
	|\tau|_\Xi \eqdef \max_{i} |\tau_i|_\Xi \; , \qquad |\tau|_\s \eqdef \min_{i} |\tau_i|_\s \; .
\end{equ}

\begin{remark}
	Note that any linear combination of trees is necessarily made up of only finitely many different bare trees. Thus, the $\max$ and $\min$ are actually achieved.
\end{remark}

Using these, we define the following partial order on $\CH$
\begin{equ}
	\tau <_{\cT} \tau' \eqvdef \left( |\tau|_\Xi < |\tau'|_\Xi \right) \lor \left( \left( |\tau|_\Xi = |\tau'|_\Xi \right) \land  \left( |\tau|_\s < |\tau'|_\s \right) \right) \; .
\end{equ}

% We say that a tree, that is not root-renormalised, is primitive if $\tau  = X_i$ or $\tau = \CI_\mfl \sigma$ for some tree $\sigma$ and $\mfl \in \mfL$. Every tree $\tau$ can be written as
% \begin{equ}
% 	\tau = X^\ell \sum_k a_k \prod_{i} \tau_i b_{ik}
% \end{equ}
% where $a_k,b_{ik} \in \CA$, and $\tau_i$ is a unique finite set of primitive trees.

% \begin{remark}
% 	The finiteness follows from the definition of $\tau_i$ and the uniqueness from t
% \end{remark}

\begin{definition}[Preparation Maps]
\label{def:RtRen}
	An  $t$-continuous $\CA$-bimodule morphism $R \in \CB(\CH)$ is called a preparation map if and only if
	\begin{enumerate}
		\item $R\tau = \tau$ for all trees $\tau$ that are primitive or root-renormalised,
		\item $R$ commutes with multiplication by $X^k$,
		\item $R$ commutes with the structure group $G$\footnote{This is the crucial assumption that allows us to conclude that a renormalised model is again a model.},
		\item for all $\tau \in \CH$\footnote{These assumptions guarantee that the counterterms we add to a negative tree are analytically better behaved and correspond to subtrees of the original tree.}
		\begin{equ}
			| R \tau - \tau |_\Xi < | \tau |_\Xi \; , \qquad |R\tau-\tau|_{\s} > |\tau|_{\s} \; ,
		\end{equ}
		\item for all trees $\tau \in \CH$, there exists a finite set of trees $\{\tau_i\}$, such that\footnote{Note that, since our trees include algebra decorations, any element of our regularity structure can be written as a sum of distinct trees.} $R \tau = \sum_i \tau_i$ with $\tau_i$ trees, $|\tau_i|_+ = |\tau|_+$ for all $i$.
		\item There exists a $\delta > 0$, such that, for all $\tau \in \CT_\alpha$, $R\tau \in \CH^-_{\alpha+\delta}$.
	\end{enumerate}
\end{definition}

\begin{definition}[Renormalisation Map]
	Given a preparation map $R$, we define two renormalisation maps $M_R \colon \CH \to \CH$ and $M^+_R \colon \CH^+ \to \CH^+$ as follows. Let $M_R = M^\circ_R R$ where $M^\circ_R$ is inductively defined for trees that are not root\-/renormalised via
	\begin{equ}
		M^\circ_R \bone = \bone \; , \quad
		M^\circ_R X_\mu = X_\mu \; , \quad
		%M^\circ_R( \tau \bar\tau ) &= M^\circ_R( \tau ) M^\circ_R( \bar\tau )  \; , \\
		%M^\circ_R( \cR_\alpha \tau ) & = M^\circ_R(  \tau )\\
		M^\circ_R( \CI_\mfk \tau )  =  \CI_{\mfk} \left( M_R \tau \right) \; ,
	\end{equ}
	all of which are $t$-continuous by assumption. Then we extend $M_R^{\circ}$ continuously to products and root-renormalised trees as explained in Section~\ref{sec:MetaIndExp}, i.e.\
	\begin{equs}
		M^\circ_R( \tau \bar\tau ) &= M^\circ_R( \tau ) M^\circ_R( \bar\tau )  \; , \\
		M_R^\circ \tau &= \Bigl( \left(M_R^\circ\right)^{\wotimes_\pi |T|} \left(\CF \right)  \otimes \bigotimes_{v \in N_T}  X^{\mfn(v)} \Bigr) \curvearrowright \tau_\rho \; ,
	\end{equs}
	where we used here that $M_R^\circ \left( X^k \right) = X^k$.

	Furthermore, we define $M_R^+ \eqdef M_R^{\circ}\big|_{\CH^+}$, i.e.\ it renormalises elements of $\CH^+$ above the root. $M_R^+$ is multiplicative.
\end{definition}

% \begin{proposition}
% 	$M_R$, $M_R^+$ and $M_R^\circ$ are continuous. Furthermore, for every $\beta \in \R$ there exists a $\gamma > 0 $, that depends on the rule, s.t. for all $\beta \leqslant \gamma$.
% 	\begin{equs}
% 		\max\{ \| M_R \circ \CQ_{\alpha} \| , \| M_R^+ \circ \CQ_{\alpha} \|, \| M_R^\circ \circ \CQ_{\alpha} \| \} \lesssim \|R \circ \CQ_{<{\alpha + \gamma}} \|
% 	\end{equs}
% \end{proposition}
% \begin{proof}
% 	We can write $\tau$ as
% 	\begin{equ}
% 		\tau = \sum_{i} a_i \prod_{k = 1}^n \CI_{\mfk_k} \left( \tau_k \right) b_{ki}
% 	\end{equ}
% 	and thus
% 	\begin{equs}
% 		\mfp_{\CT}(M_R^\circ\tau) &= \mfp_{\CT} \Bigl( \sum_{i} a_i \prod_{k = 1}^n \CI_{\mfk_k} \left( M_R \tau_k \right) b_{ki}   \Bigr) \leqslant \\
% 		& \leqslant \Bigl(\sum_{i} \mfp(a_i) \prod_{k = 1}^n \mfp(b_{ki}) \Bigr) \prod_{k = 1}^n \mfp_{\CH} \left( \CI_{\mfk_k} \left(M_R \tau_k \right)  \right) =\\
% 		& = \Bigl(\sum_{i} \mfp(a_i) \prod_{k = 1}^n \mfp(b_{ki}) \Bigr) \prod_{k = 1}^n \mfp_{\CH} \left(M_R \tau_k \right)  \lesssim\\
% 		& \lesssim_R \Bigl(\sum_{i} \mfp(a_i) \prod_{k = 1}^n \mfp(b_{ki}) \Bigr) \prod_{k = 1}^n \mfp_{\CH} \left(\tau_k \right) =  \\
% 		& = \Bigl(\sum_{i} \mfp(a_i) \prod_{k = 1}^n \mfp(b_{ki}) \Bigr) \prod_{k = 1}^n \mfp_{\CH} \left( \CI_{\mfk_k} \tau_k \right) . \\
% 	\end{equs}
% 	Taking the infimum over all such representations yields the result for the projective crossnorm and therefore for any other seminorm that is admissible, as the projective crossnorm is the strongest crossnorm.
% \end{proof}

In order to define renormalisation at the level of models rather than just at the level of trees, we will need to describe an ``action'' of the preparation maps on the space of characters $\CG(\CH^+)$.
\begin{definition}[Renormalisation Action on Characters]
	\label{def:RenChar}
Given a preparation map $R$ and a character $g \in \CG(\CH^+)$, we define
	$g^R \eqdef g \cdot R  \eqdef g \circ M_R^+ \in \CG(\CH^+) $.
\end{definition}
\begin{remark}
Note that  $g \circ M_R^+$ above is indeed a character as $M_R^+$ preserves the $|\,\bigcdot\,|_+$-grading and is multiplicative. 
Furthermore, $\Gamma_{g^R}$ maps trees with empty extended decoration to linear combinations of such trees.
\end{remark}

The key property that ensures the action by preparation maps on the canonical model again yields a model that satisfies the necessary homogeneity estimates is the following ``cointeraction'' property.
\begin{proposition}[Cointeraction]
For any $g \in \CG(\CH^+)$ and preparation map $R$, the following holds
	\label{prop:CoInt}
	\begin{equ}
		M_R^{\circ} \circ \Gamma_{g^R} = \Gamma_g \circ M_R^\circ
	\end{equ}
\end{proposition}
\begin{proof}
	We proceed as usual by induction. Because of the defining properties of $M_R^\circ$, the only non-trivial case to check is $\CI_\mfk \tau$.
	\begin{equs}
		\Gamma_g M_R^\circ \CI_\mfk \tau &= \Gamma_g \CI_\mfk M_R \tau = \CI_\mfk \Gamma_g M_R \tau + \sum_{\ell} \frac{X^\ell}{\ell!} g \left( \CJ_{\mfk + \ell} M_R \tau \right) = \\
		& = \CI_\mfk M_R^\circ \Gamma_{g^R} R \tau + \sum_{\ell} \frac{X^\ell}{\ell!} g \left( M_R^+ \CJ_{\mfk + \ell} \tau \right) =  \\
		& = \CI_\mfk M_R^\circ R \Gamma_{g^R} \tau + \sum_{\ell} \frac{X^\ell}{\ell!} g^R \left(\CJ_{\mfk + \ell} \tau \right) = \\
		& = M_R^\circ \CI_\mfk  \Gamma_{g^R} \tau + \sum_{\ell} \frac{X^\ell}{\ell!} g^R \left(\CJ_{\mfk + \ell} \tau \right) = M_R^\circ \Gamma_{g^R} \CI_\mfk \tau
	\end{equs}
	where we repeatedly used the fact that $M_R^\circ \CI_\mfk = \CI_\mfk  M_R = \CI_\mfk  M_R^\circ R $.
\end{proof}
\begin{corollary}
	For any preparation map $R$ and character $g \in \CG(\CH^+)$,
	\begin{equ}
		M_R \circ \Gamma_{g^R} = \Gamma_{g} \circ M_R\;.
	\end{equ}
\end{corollary}
\begin{proof}
	Using the previous proposition, as well as  we immediately see that
	\begin{equs}
		M_R \circ \Gamma_{g^R} & = M_R^\circ \circ R \circ \Gamma_{g^R} = M_R^\circ  \circ \Gamma_{g^R}  \circ R = \Gamma_{g} \circ M_R^\circ   \circ R=   \Gamma_{g}  \circ M_R \;,
	\end{equs}
	as claimed.
\end{proof}

\subsection{Extraction-Contraction Renormalisation at the Root}
\label{sec:ExtContRen}

% \begin{assumption}[Canonical Renormalisation]
% 	The maps $R$ can be written as a composition of a contraction $C_{\tau}$ of a (bare) subtree $\tau$ and (?) linear map $\Delta_\tau$ acting only on the algebra decoration.
% \end{assumption}

For this section, we fix a finite set $\CT^-$ of scalar trees that all have a strictly negative $|\,\bigcdot\,|_{\s}$ degree and empty extended decoration. Once we introduce rules that choose a subspace of $\CH$ to be our $\CA$-regularity structure, the choice of $\CT^-$ will become more explicit. Let $\mfQ$ be the set of all pairs $ \mfd  = (\sigma, \bpi)$ where $\sigma \in \CT^-$ and $\bpi$ is a partition of the negative leaves of $\sigma$.

For a tree $T$ and a subtree $S$, let $\partial_T S$ be the set of edges $e$ in $E_T$, such that exactly one vertex of $e$ lies in $S$. Within this section we say that $(\sigma, \iota)$, with $\sigma = (S, \mfe_S, \mfn_S)$, is a subtree of $\tau = (T, \mfe, \mfn, \mfo, \mfa)$ if $(S, \iota)$ is a subtree of $T$, $\mfn_S(v) \leqslant \mfn(\iota(v))$ for all $v \in N_S$, and $\mfe_S(e) = \mfe(\iota(e))$ for all $e \in E_S$.

Let $\CG(\CT^-)$ be the set of maps $\mfQ \to \mathbb{K}$ which naturally carries the structure of a $\mathbb{K}$-vector space. A choice of $\ell \in \CG(\CT^-)$ defines a preparation map that renormalises a tree at the root in the following way. 

Let $\tau = (T, \mfe, \mfn, \mfo, \mfa) \in \CH$ be a tree, let $\mfd = \bigl( (S, \mfe_S , \mfn_S), \bpi \bigr) \in \mfQ$, and let $(\tau: \mfd )$ be the set of all decorated trees $\beta = (T, \mfe_B, \mfn_B, \mfo_B, \mfa_B)$ satisfying the following conditions:
\begin{itemize}
%	\item If $S$ is not a subtree of $T$ with the same root as $T$ or $S$ is not disjoint from $\mff_\mfo$ or $\mfe\big|_T \not\equiv \mfe_S$, then $(\tau: \sigma) = \emptyset$\footnote{In particular, if $\tau$ is root-renormalised, then $(\tau:\sigma) = \emptyset$.}.
	\item There exists an embedding $\iota_\beta : S \hookrightarrow T$ mapping the root of $S$ onto the root of $T$ and such that
	$\iota_\beta(S) \cap \mff_\mfo = \emptyset$.
	\item One has $\mfa_B = \mfa$, $\mfo_B = \mfo \sqcup \mfd  $. %, where $\boldsymbol{\pi}_\sigma$ is a partition of the negative leaves of $\sigma$, identifying $\sigma$ with $\iota_\beta(\sigma)$.
	\item One has $\mfe_B \big|_{ T \setminus \partial_T S} \equiv \mfe \big|_{ T \setminus \partial_T S}$, $(\mfe(e))_1  = (\mfe_B(e))_1 \in \mfL^- \sqcup \mfL^+$ for all $e \in \partial_T S$, and for all $v \in S$
	\begin{equs}
		\sum_{e \in \delta(v)} \left( (\mfe_B(e))_2 - (\mfe(e))_2 \right) &= \mfn_S(v) + \mfn_B(v)  - \mfn(v) \geqslant 0 \; ,
	\end{equs}
	identifying $S$ with $\iota_\beta(S)$.
\end{itemize}
\begin{remark}
	Here one should think of each $\beta \in \bigl(\tau : \mfd \bigr)$ as $\tau$ with an instance of $\mfd = (\sigma, \bpi)$, embedded in $\tau$ with $\iota_\beta$,  added to the extended decoration of $\tau$.
	In particular, the first condition guarantees that if $\tau$ is root-renormalised, then $\bigl(\tau:\mfd \bigr) = \emptyset$.

	Furthermore, we note that the trees contained $\bigl(\tau : \mfd \bigr)$ are exactly the same trees that would be extracted by the ``negative'' coproduct $\Delta_1$ of \cite{BHZ19}, in particular see Definition~3.3 therein. In particular, the extracted polynomial decorations are exactly the ones that are contained in the extended decoration of this paper. The polynomial decorations that are attached to the regular tree are the decorations that have been left behind. 
	%\martinp{Add explanation of $X_i$ extraction. Say that this is $\mfA_1$ from \cite{BHZ19}. Add example with $\<3>$.}

\end{remark}

For $\beta \in \bigl(\tau : \mfd \bigr)$ we define the combinatorial coefficient
\begin{equ}
	\binom{\mfn_\beta}{\mfn_\mfd} = \prod_{v \in N_{S} } \binom{\mfn(\iota_\beta(v))}{\mfn_S(v)} \; ,
\end{equ}
and
\begin{equ}
	\eps^\beta_{\mfd} =  \sum_{v \in S} \sum_{e \in \delta(v)} \left( (\mfe_B(e))_2 - (\mfe(e))_2 \right)  \in \N^d \; .
\end{equ}

\begin{definition}
	Let $\ell \in \CG(\CT^-)$. Let $R_\ell \colon \CH \to \CH$ be the map defined for a decorated tree $\tau \in \CH$
	\begin{equ}
		R_\ell (\tau) \eqdef \tau + \sum_{\mfd \in \mfQ} \ell(\mfd) \sum_{\beta \in (\tau : \mfd)} \frac{1}{\eps^\beta_{\mfd}!} \binom{\mfn_\beta}{\mfn_\mfd} \beta \; ,
	\end{equ}
	extended $\CA$-bilinearly to all of $\CH$. By abuse of terminology, we say that the tree $\sigma$ has been contracted in $\beta \in (\tau : \mfd)$. We will refer to this subset of preparation maps as root renormalisation maps.
\end{definition}

\begin{remark}
	We note here that $R_\ell$ maps scalar trees to scalar trees.
\end{remark}

\begin{example}
	We consider the tree $\<3>$ of the $\Phi^4_3$-regularity structure, cf.\ Section~\ref{sec:Phi43Mez}. W.r.t.\ $\mfd = \left( \<2> , (1,2)\right)$, the set $(\<3>: \mfd)$ is made up of the three distinct ways one can embed the trees $\<2>$ and $X_\mu \<2>$ in $\<3>$. The former tree simply changes the extended decoration of $\<3>$, whereas the latter also changes the edge decoration of the remaining branch by $e_\mu$.
\end{example}

\begin{proposition}
\label{prop:RenHomo}
	The map $R_{\bigcdot}$ is a homomorphism from the (free) Abelian group $\left(\CG(\CT^-), + \right)$ to $( \GL(\CH) , \circ )$. 
\end{proposition}
\begin{remark}
	Once we have specified a subset of $\CH$ that is a regularity structure, s.t.\ $\CT^-$ is the set of its negative trees, $\left(\CG(\CT^-), + \right)$ or more exactly its image under $R_{\bigcdot}$ will be the Renormalisation Group of the regularity structure. 
\end{remark}
\begin{proof}
	The homomorphism property follows directly from the strict upper triangular structure of $R_\ell$. In particular, we have for $\ell, \ell' \in \CG(\CT^-)$
	\begin{equs}
		R_{\ell'} \left( R_{\ell}(\tau)\right) &= R_{\ell'} \left(\tau\right) + \sum_{\mfd \in \mfQ} \ell(\delta) \sum_{\beta \in (\tau : \mfd)}  \frac{1}{\eps^\beta_{\mfd}!} \binom{\mfn_\beta}{\mfn_\mfd} R_{\ell'} (\beta ) = \\
		&= R_{\ell'} \left(\tau\right) + \sum_{\mfd \in \mfQ} \ell(\mfd) \sum_{\beta \in (\tau : \mfd)}  \frac{1}{\eps^\beta_{\mfd}!}\binom{\mfn_\beta}{\mfn_\mfd} \beta  = \\
		&= \tau + \sum_{\mfd' \in \mfQ} \ell'(\mfd') \sum_{\beta' \in (\tau : \mfd')}  \frac{1}{\eps^{\beta'}_{\mfd'}!} \binom{\mfn_{\beta'}}{\mfn_{\mfd'}} \beta'  + \\
		& \qquad \qquad  + \sum_{\mfd \in \mfQ} \ell(\mfd) \sum_{\beta \in (\tau : \mfd)}  \frac{1}{\eps^{\beta}_{\mfd}!} \binom{\mfn_\beta}{\mfn_\mfd} \beta  = \\
		& = \tau + \sum_{\mfd \in \mfQ} (\ell+\ell')(\mfd) \sum_{\beta \in (\tau : \mfd)}  \frac{1}{\eps^\beta_{\mfd}!}\binom{\mfn_\beta}{\mfn_\mfd} \beta  = R_{\ell + \ell'}(\tau)
	\end{equs}
	the second equalities follows because $\beta$ is necessarily root-renormalised and thus $R_{\ell'}$ acts trivially on it per definitionem. 
	%That $R_{\bigcdot}$ is injectivity now follows from 
\end{proof}

We will now provide an alternative description of the root renormalisation map $R_\ell$, which will be more convenient for showing that $R_\ell$ is indeed a preparation map.

For $i \in [d]$ we also define a set of root derivative operators $D_i \colon \CH \to \CH$ by setting, for $j \in [d]$ and $\mfk \in \mfL$,
\begin{equs}
	D_i X_j & \eqdef \delta_{ij} \\
	D_i \CI_{\mfk} \tau & \eqdef \CI_{\mfk + e_i} \tau
\end{equs}
and extending $D_i$ to arbitrary trees in $\CH$ via the (noncommutative) Leibniz rule. The root derivatives commute with the structure group, as the following proposition shows.

\begin{remark}
	These root derivatives are in fact abstract gradients $\cD_i$ as defined in \cite[Definition~5.25]{Hai14} as the following proposition shows.
\end{remark}
\begin{proposition}
\label{prop:RootDerCom}
	For all $\tau \in \CH$, all $k \in \N^d$, and all $g \in \CG(\CH^+)$
	\begin{equ}
		\Gamma_g D^k \tau = D^k \Gamma_g \tau \; .
	\end{equ}
\end{proposition}
We leave the proof of this proposition to Appendix~\ref{app:alg}.

Let $[\tau \colon \sigma]$ denote the set of (distinct) embeddings $\iota \colon \sigma \hookrightarrow \tau$, such that $\iota  = \iota_\beta$ for some $\beta \in (\tau \colon \mfd)$. Given $\iota$, $v \in N_{S}$, let $\tau^{\iota}_{v, 1}, \dots, \tau^{\iota}_{v, n_v}$, s.t.\
\begin{equ}[eq:TauDecomp]
	\tau = \Bigl( \bigotimes_{v \in S} \left( \tau^{\iota}_{v, 1} \otimes \cdots \otimes \tau^{\iota}_{v, n_v} \otimes X^{\mfn(\iota(v))} \right) \curvearrowright_v \Bigr) \left( \iota(S) , \mfe \big|_{\iota(S)} , \mfn \big|_{\iota(S)} , \emptyset , \bone^{\otimes |S|}_{\CA}  \right)
\end{equ}
Then we set
\begin{equs}
	\mfR_{(\iota,\bpi)} \tau &\eqdef  \Bigl( \bigotimes_{v \in S} \frac{1}{\mfn(v)!} D^{\mfn(v)} \left( \tau^{\iota}_{v, 1} \otimes \cdots \otimes \tau^{\iota}_{v, n_v} \otimes X^{\mfn(\iota(v))} \right) \curvearrowright_v \Bigr) \\
	 & \hspace{4cm}  \left( \iota(S) , \mfe \big|_{\iota(S)} , \mfn \big|_{\iota(S)} , (\sigma, \bpi) , \bone^{\otimes |S|}_{\CA}  \right) \; .
\end{equs}

Here, the sum is taken over all contractions of $\sigma$ and the root derivative is extended to the tensor product via the usual Leibniz rule. One should think of $\mfR_{(\iota,\bpi)} \tau$ as applying the root derivative $D^{\mfn(v)}$ to the combined tree being attached to the vertex $v$.

A straightforward calculation using multinomial coefficients now shows that
\begin{equ}
	R_{\ell}(\tau) = \tau + \sum_{\mfd \in \mfQ} \ell(\mfd) \sum_{\iota \in [\tau : \sigma]} \mfR_{(\iota, \bpi)} \tau \; ,
\end{equ}
where $\mfd = (\sigma, \bpi)$.

\begin{theorem}
	$R_\ell$ is a preparation map.
\end{theorem}
\begin{proof}
	We need to verify the axioms laid out in Definition~\ref{def:RtRen}.
	\begin{enumerate}
		\item Per definitionem, $(X_i : \mfd) = (\CI_{\mfk} \tau : \mfd ) = (\bar\tau : \mfd) = \emptyset$ for all $\mfk \in \mfL$ and all root-renormalised trees $\bar\tau$.
		\item It follows straightforwardly from the definition that $\beta \in (\tau : \mfd)$ if and only if $X^k \beta \in (X^k \tau :\mfd)$.
		\item W.l.o.g.\ we shall assume that $\ell$ is non-zero only on a single pair $\mfd = (\sigma, \bpi) \in \mfQ$, as the general case then follows by Proposition~\ref{prop:RenHomo}. %Furthermore, we may also assume that the algebra decoration at the root is trivial as it remains unaffected by the action of both the structure group and renormalisation group.

		Now let $g \in \CG(\CH^+)$, and assume that $(\tau: \mfd)$ is non-empty, for otherwise the statement is trivial. In order to disentangle the combinatorics, we shall assume that we have labeled every edge appearing in $\tau$ with a unique label $k \in [|E_T|]$. Furthermore, furthermore a given $\iota \in [\tau : \sigma]$ the edges $S$ get mapped unto some subset $I_\iota \subset [|E_T|]$.

		Let $\tau^g_i$ be the set of trees appearing in the sum $\Gamma_g \tau$, i.e.\ $\Gamma_g \tau = \sum_i \tau^g_i$. Per definitionem, each tree $\tau^g_i$ is a subtree of $\tau$, and thus we can label the edges of $\tau^g_i$ uniquely with a corresponding subset of the labels of $\tau$. In particular,  each $\iota' \in [\tau^g_i : \sigma]$ can be identified with a unique $\iota \in [\tau : \sigma]$, i.e.\ $[\tau^g_i : \sigma] \subset [\tau : \sigma]$. We can thus extend $\mfR_{(\iota,\bpi)}$ to a linear map on linear combinations of $\tau^g_i$, by setting $\mfR_{(\iota,\bpi)} \tau^g_i  = 0$ if $\iota \notin [\tau^g_i : \sigma]$. Thus, $\mfR_{(\iota,\bpi)} \Gamma_g \tau$ is defined as well, and in such a way that
		\begin{equ}
			R_\ell(\Gamma_g \tau) = \Gamma_g \tau + \ell(\mfd) \sum_{\iota \in [\tau : \sigma]} \mfR_{(\iota,\bpi)} \Gamma_g \tau.
		\end{equ}
		We are therefore only left with showing that $\mfR_{(\iota,\bpi)} \Gamma_g \tau = \Gamma_g \mfR_{(\iota,\bpi)} \tau$.

		The important thing to notice now is that for each $\iota \in [\tau: \sigma]$ we can split $\Gamma_g \tau$ into
		\begin{equ}
			\Gamma_g \tau = \sum_{i} \tau^g_i  = \sum_{i : \iota \in [\tau^g_i : \sigma] } \tau^g_i +  \sum_{i : \iota \notin [\tau^g_i : \sigma] } \tau^g_i \eqqcolon \tau^{g, \iota}_0 +  \tau^{g,\iota}_1 ,
		\end{equ}
		where per definitionem $\mfR_{(\iota,\bpi)} \tau_1^{g,\iota}  = 0$ and
		\begin{equs}
			\tau_0^{g,\iota} &= \Bigl( \bigotimes_{v \in S} \left( \Gamma_g \tau^{\iota}_{v, 1} \otimes \cdots \otimes \Gamma_g  \tau^{\iota}_{v, n_v} \otimes \Gamma_g X^{\mfn(\iota(v))} \right) \curvearrowright_v \Bigr) \left( \iota(S) , \mfe \big|_{\iota(S)} , \mfn \big|_{\iota(S)} , \mfd , \bone^{\otimes |S|}_{\CA}  \right) = \\
			& =  \Bigl( \bigotimes_{v \in S}  \Gamma_g \left( \tau^{\iota}_{v, 1} \otimes \cdots \otimes  \tau^{\iota}_{v, n_v} \otimes  X^{\mfn(\iota(v))} \right) \curvearrowright_v \Bigr) \left( \iota(S) , \mfe \big|_{\iota(S)} , \mfn \big|_{\iota(S)} , \mfd , \bone^{\otimes |S|}_{\CA}  \right)
		\end{equs}
		where we used the decomposition \eqref{eq:TauDecomp} of $\tau$. This formula follows directly from the inductive description where $\Gamma_g$ acts at a vertex of $\tau$ multiplicatively on each edge emanating from the vertex, either cutting or moving up along it to the next vertex. One obtains $\tau^{g,\iota}_0$ by not performing any cuts until one reaches the first edges that do not belong to $\iota(S)$ and then stopping the procedure. The reason for not performing those cuts is that they remove an edge of $\iota(S)$ and the resulting trees thus lie in the kernel of $\mfR_{(\iota,\bpi)}$, i.e.\ they belong to $\tau_1^{g,\iota}$.

		Since
		\begin{equs}
			\Gamma_g \mfR_{(\iota,\bpi)} \tau &\eqdef \sum_{\bpi_\sigma} \Bigl( \bigotimes_{v \in S} \frac{1}{\mfn(v)!} \Gamma_g D^{\mfn(v)} \left( \tau^{\iota}_{v, 1} \otimes \cdots \otimes \tau^{\iota}_{v, n_v} \otimes X^{\mfn(\iota(v))} \right) \curvearrowright_v \Bigr) \\
			 & \hspace{4cm}  \left( \iota(S) , \mfe \big|_{\iota(S)} , \mfn \big|_{\iota(S)} , \mfd , \bone^{\otimes |S|}_{\CA}  \right) \;
		\end{equs}
		we only need to prove that $\Gamma_g D^{k} = D^k \Gamma_g$ for all $k \in \N^d$. However, we essentially already showed this in Proposition~\ref{prop:RootDerCom}, as that result directly generalises because we extended the action of $D_i$ to the tensor product via the Leibniz rule.

	\end{enumerate}
\end{proof}
We also define the map $R'_\ell \colon \CT^- \to \CH$
\begin{equ}
	R'_\ell(\tau) \eqdef \tau + \sum_{\substack{(\sigma,\bpi) \in \mfQ \\ \sigma  \neq \tau }} \ell(\sigma,\bpi) \sum_{\iota \in [\tau : \sigma]} \mfR_{(\iota,\bpi)} \tau
\end{equ}
which we will use to define the BPHZ prescription.

\subsection{Regularity Structures Generated by Rules}

We remind the reader of the definitions used to select a subspace of trees in $\CH$ suitable to construct a regularity structure for a given equation, cf.\ \cite{BHZ19}.
\begin{definition}
	Let $\CN$ denote the set of finite multisets with elements chosen from $\mfL$ and let $\CP\CN$ denote its power set. A rule is a function $\cR \colon \mfL \to \CP\CN \setminus \emptyset$, such that $\cR(\mfl) = \{ \emptyset \}$ for all $\mfl \in \mfL^- \times \N^d$. A rule is normal if $N \in \cR(\mfk)$ and $M \subset N$ implies that $M \in \cR(\mfk)$ for all $\mfk \in \mfL$.
\end{definition}

In the following, we will always assume that $\cR$ is normal. We can always do this w.l.o.g. as we can add all subsets of $N \in \cR(\mfk)$ to $\cR(\mfk)$ and thereby canonically create a normal rule from an arbitrary rule.

\begin{definition}
	Given a rule $\cR$, we define the subset $\mathring{\CT}$ of $\mathring{\mfT}$ as follows.
	\begin{itemize}
		\item $\bullet \in \mathring{\CT}$.
		\item If $\tau \in \mathring{\CT}$, then $X^k \tau \in \mathring{\CT}$, for all $k \in \N^d$.
		\item If $\CI_{\mfk_1} \tau_1, \dots, \CI_{\mfk_n} \tau_n \in \mathring{\CT}$, and there exists $\mfl \in \mfL$, such that $\{ \mfk_1, \dots, \mfk_n \} \in \cR(\mfl)$, then $\tau = \CI_{\mfk_1}\tau_1 \cdots \CI_{\mfk_n}\tau_n \in \mathring{\CT}$ and $\CI_\mfl \tau \in \mathring{\CT}$.
	\end{itemize}
\end{definition}
We note here that none of the elements of $\mathring{\CT}$ have been renormalised, i.e.\ they have a trivial extended decoration.
\begin{definition}
	Let $\CT^- \eqdef \left\{ \tau \in \mathring{\CT} \, \middle| \, |\tau|_{\s} < 0 \, , \, \tau \text{ not planted} \right\}$.
\end{definition}
As we wish for this set to be finite, we need the well-known notion of subcriticality.
\begin{definition}
	Let $\reg \colon \mfL^+ \sqcup \mfL^- \to \R$ be a map. This map is canonically extended to $\mfk = (\mft , k) \in \mfL$ by setting $\reg(\mfk) \eqdef \reg(\mft) - |k|_\s$ and to $\CN$ be setting
	\begin{equ}
		\reg(N) \eqdef \sum_{\mfk \in N} \reg(\mfk) \; .
	\end{equ}
	A rule $R$ is subcritical if and only if there exists a map $\reg$ as above, such that, for all $\mft \in \mfL^+ \sqcup \mfL^-$,
	\begin{equ}
		\reg(\mft) < |\mft|_\s + \inf_{N \in \cR(\mft)} \reg(N) \; .
	\end{equ}
\end{definition}

\begin{proposition}
	If $\cR$ is subcritical, then $\CT^-$ is finite.
\end{proposition}

\begin{definition}[Regularity Structure Generated by $\cR$]
	Given a subcritical rule $\cR$, let $\overline{\CT}$ be the subset $\mathring{\mfT}$ defined as follows.
	\begin{itemize}
		\item $\bullet \in \overline{\CT}$.
		\item If $\tau \in \overline{\CT}$, then $X^k \tau \in \overline{\CT}$, for all $k \in \N^d$.
		\item If $\CI_{\mfk_1} \tau_1, \dots, \CI_{\mfk_n} \tau_n \in \overline{\CT}$, and there exists $\mfl \in \mfL$, such that $\{ \mfk_1, \dots, \mfk_n \} \in \cR(\mfl)$, then $\tau = \CI_{\mfk_1}\tau_1 \cdots \CI_{\mfk_n}\tau_n \in \overline{\CT}$ and $\CI_\mfl \tau \in \overline{\CT}$.
		\item If $\tau \in \overline{\CT}$, then all the trees appearing in the sum $R_\ell \tau$ for some $\ell \in \CG(\CT^-)$ are in $\overline{\CT}$.
	\end{itemize}
	Let the structure space $\CT$ of the regularity structure be
	\begin{equ}
		\CT(\cR) \eqdef \bigoplus_{(T, \mfn , \mfe, \mfo) \in \overline{\CT}} \CH[(T, \mfn , \mfe, \mfo)] \; .
	\end{equ}
	Furthermore, let
	\begin{equs}
		G(\cR) &\eqdef \left\{ \Gamma_g \big|_{\CT(\cR) } \, \middle| \, g \in \CG(\CH^+) \right\} \; , \\
		A(\cR) &\eqdef \left\{ |(T, \mfn, \mfe,\mfo)|_{\s} \, \middle| \, (T, \mfn, \mfe,\mfo) \in \overline\CT \right\} \; .
	\end{equs}
	We define the regularity structure associated with $\cR$ to be 
	\begin{equ}
		\cT(\cR) \eqdef \left( \CT(\cR), G(\cR), A(\cR) \right) \; .
	\end{equ}
	For $\gamma > 0$, let $\CT_\gamma(\cR)$ be the sector of $\cT(\cR)$ generated by unrenormalised trees $\tau$ with $|\tau|_+  < \gamma$, and let $G_{\gamma}(\cR)$ and $A_\gamma(\cR) \eqdef A(\cR) \cap (-\infty, \gamma)$ be restrictions of $G(\cR)$ and $A(\cR)$ to $\CT_\gamma(\cR)$. We define the cut-off regularity structure to be
	\begin{equ}
		\cT_\gamma(\cR) \eqdef \left( \CT_\gamma(\cR), G_\gamma(\cR), A_\gamma(\cR) \right) \; .
	\end{equ}

\end{definition}

\begin{remark}
	Three remarks are in order to check that $G(\cR)$ is well-defined. First, the projection $P$ restricts to a map $\CT \to \CT$ as it is defined componentwise in terms of the decomposition along scalar trees, thus we only really need to define the action of $g \in \CG(\CH^+)$ on the subset $\CT^+ \eqdef P \CT$.

	More importantly, $\Ran \left( \Gamma_g\big|_{\CT(\cR)} \right) \subset \CT(\cR)$, as at a given vertex $v$ of a tree $\tau \in \CT$, $\Gamma_g$ can only remove an outgoing edge and replace it by a polynomial decoration. Thus, the new tree still conforms to the rule as we have assumed $\cR$ to be normal.

	Finally, we need to check that $A$ is a discrete set. This essentially follows from the finiteness of $\CT^-$, which means that $R_\ell$ can only ever generate finitely many new subspaces, which are all of better regularity than the original ones. To be more concrete, we have the following proposition.
\end{remark}

\begin{remark}
	We note that $\CT_{\gamma}(\cR)$ is indeed a sector as per definitionem all $\Gamma \in G(\cR)$ do not increase the $|\,\bigcdot\,|_+$-degree. Furthermore, by assumption there exists a $\delta > 0$, s.t.\ $A_\gamma(\cR) \subset (-\infty, \gamma+\delta)$. Therefore, by the next proposition, $\CT_\gamma$ is finite-dimensional.
\end{remark}

\begin{proposition}
	For any $\alpha \in \R$, $\CT^-_\alpha$ is finite-dimensional.
\end{proposition}
\begin{proof}
	This follows directly from the combinatorial arguments in \cite[Proposition~5.15~\&~5.21]{BHZ19}.
\end{proof}

\subsection{Canonical Model}
\label{sec:CanMod}

We now introduce the canonical model for an $\CA$-regularity structure generated by trees. 
The canonical model will be multiplicative on trees that are not root-renormalised. 
Carefully defining how the canonical models act on root-renormalised trees, where it is not necessarily multiplicative, will be required in order to be able to describe how root-renormalised maps determine renormalised models.  
Defining this action for root-renormalised trees requires extra care. Noncommutativity means that extracting at the root can cause algebraic operations to occur elsewhere in the tree.

We now introduce a set $\mfO$ which is the set of possible components of extended decorations. 

\begin{definition}
	Let $\mfO$ be the set of pairs $(\tau, \bpi)$, where $\tau \in \CT^-$ and $\bpi$ is a partition of the set of its negative leaves. %$\CT^-$ is by assumption necessarily finite. Let $\CF^-$ be the commutative $\mathbb{K}$-algebra generated by $\CT^-$ equipped with the forest product.  	
\end{definition}

We encode how every renormalisation $\mfo \in \mfO$ induces an algebraic operation on the rest of the tree using an object we call a contraction structure. 
\begin{definition}
	A contraction structure is a map $\Delta \colon \mfO \to \bigsqcup_{n \in \N} \CB^n(\CA ; \CA)$  with the property that for each $\mfo \in \mfO$, $\Delta(\mfo) \in \CB^{n}(\CA; \CA)$, where $n = |\tau_\rho|$ and $\tau_\rho$ is the negative tree in $\mfo$. We assume that each map $\Delta(\mfo)$ is $\CA$-bimodule morphism when $\CA^{\wotimes_\pi |\tau_\rho|}$ is equipped with the natural bimodule structure. We do this to ensure that the maps $\Pi_x$ are bimodule morphisms.  
\end{definition}

\begin{remark}
	Alternatively, we will also view each $\Delta(\mfo)$ above as a continuous map $\CA^{\wotimes_\pi (n+1)} \to \CA$. 
\end{remark}
%\martinp{Check whether $t$-continuity with itself and multilinearly is defined.}

%\ajay{Discuss notation here a bit, better to use kernel + noise assignments + contraction structure rather than $\eps$ as a parameter indexing the canonical model.}
Given a contraction structure $\Delta$, a choice of noises $\xi_{\mfl} \in \cC^\infty( \R^d ; \CA)$ for $\mfl \in \mfL^-$, and a set of $|\mft|_\s$-regularising  integration kernels $K_{\mft}$ for $\mft \in \mfL^+$, we proceed to define the corresponding canonical model $\Pi \colon \R^d \times \CT \to \cC^\infty(\R^d; \CA)$, which is linear and multiplicative in $\CT$, inductively. We set
\begin{equs}[eq:CanMod]
  \left(\Pi_{x} \Xi_{(\mfl, k) } \right)(y) & =  D^k \xi_{\mfl}(y)\\
  \left(\Pi_{x} X^k \right)(y)& = (y-x)^k\bone_{\CA}\\
  %\left(\Pi_{x}^{(\eps)}  \left(\cR_\alpha \tau  \right) \right)(y) & = \left(\Pi_{x}^{(\eps)} \tau  \right)(y)\\
  \left(\Pi_{x} \left( \CI_\mfk \tau \right)  \right)(y)&= \int D^{\mfk_2}_1 K_{\mfk_1} (y,z) \left(\Pi_{x} \tau \right)(z) \d z - \sum_{\ell} \frac{(y-x)^\ell}{\ell!} f_{x}\left( \CJ_{\mfk + \ell} \tau \right)
\end{equs}
for trees $\tau, \bar\tau$ that are not root-renormalised we define
\begin{equ}
	\left(\Pi_{x}  \left(\tau \bar\tau \right) \right)(y) = \left(\Pi_{x} \tau  \right)(y)\left(\Pi_{x} \bar\tau \right)(y) \;, 
\end{equ}
and for a root-renormalised tree
\begin{equs}
	\Pi_x \tau &= \Pi_x \Bigl( \Bigl(\CF \otimes \bigotimes_{v \in N_T}  X^{\mfn(v)} \Bigr)  \curvearrowright \tau_\rho \Bigr) \eqdef \\
	&\eqdef \Delta(\mfo_\rho) \Bigl( \left(\Pi_x\right)^{\wotimes_\pi n} \left(\CF \right)  \Bigr) \prod_{v \in N_T} \Pi_x X^{\mfn(v)} \; .
\end{equs}
$f_x \in \CG(\CT^{+})$, above, is inductively defined for non-negative, non-root-renormalised trees
\begin{equs}[eq:CanStGr]
	f_{x}(\1) &= \bone_{\CA}\\
	f_{x}(X_{\mu}) &=  x_\mu \bone_{\CA} \\
	f_{x}(\tau\bar\tau) &=f_{x}(\tau)f_{x}(\bar\tau)\\
	f_{x} \left( \CJ_{\mfk} \tau \right) &=
		  \int D^{\mfk_2}_1 K_{\mfk_1} (x,z) \left( \Pi_{x} \tau  \right)(z) \d z \; ,
\end{equs}
and then $\CA$-bilinearly extended to $\CT^{+}$. We also need a second set of inductively defined characters which correspond to partial inverses of $f_x$, compare with the positive twisted antipode in \cite{BHZ19}. We define them to be 
\begin{equs}[eq:CanStGr2]
	g_{x}(\1) &= \bone_{\CA}\\
	g_{x}(X_{\mu}) &=  -x_\mu \bone_{\CA} \\
	g_{x}(\tau\bar\tau) &=g_{x}(\tau)g_{x}(\bar\tau)\\
	g_{x} \left( \CJ_{\mfk} \tau \right) &= - \sum_{\ell} \frac{(-x)^\ell}{\ell!} f_x \left( \CJ_{\mfk +\ell} \tau\right) ,
\end{equs}
extended continuously in the usual way. We use these to define the structure group of the model 
\begin{equ}
	\Gamma_{xy} \eqdef \Gamma_{g_x}^{-1} \circ \Gamma_{g_y} \; .
\end{equ}
We denote the overall model by $Z^{\can}$.

In addition to Equations~\eqref{eq:GamDef}, we have the following more explicit formulae for the structure group. 

\begin{lemma}
\label{lemma:StGrForm}
	For all $x,y \in \R^d$, and trees $\tau$
	\begin{equs}
		\Gamma_{xy} X_\mu &= X_\mu + (x_\mu - y_\mu) \1 
	\end{equs}
	and for any tree $\tau$ and $\mfk \in \mfL^+ \times \N$
	\begin{equ}
		\Gamma_{xy} \CI_{\mfk} \tau = \CI_{\mfk} \Gamma_{xy} \tau + \sum_{\ell} \frac{X^\ell}{\ell!} \int K_{\mfk + \ell, |\tau|_{\s}} (x,y; z) \left( \Pi_y \tau\right)(z) \d z \; . 
 	\end{equ}
	Here $K_{\mfk, |\tau|_{\s}} \equiv 0$ if $|\mfk|_{\s}  + |\tau|_{\s} < 0$ and otherwise it is the truncated Taylor expansion of $D^{\mfk_2} K_{\mfk_1}$ around $y$, explicitly 
	\begin{equ}
		K_{\mfk, |\tau|_{\s}}(x,y ; z) \eqdef D_1^{\mfk_2 } K_{\mfk_1} (x,z) - \sum_{\substack{k \in \N^d \\ |k|_{\s} < |\mfk|_{\s} + |\tau|_{\s} }} \frac{(x-y)^k}{k!} D_1^{\mfk_2 + k } K_{\mfk_1} (y,z) \; . 
	\end{equ}
\end{lemma}
The proof of this lemma may be found in Appendix~\ref{app:alg}.

\begin{remark}
	The maps are continuous by the same general arguments as in Section~\ref{sec:MetaIndExp}. The only extra thing to note is that by the usual arguments the ($t$-continuous) multiplication map $\cC^\infty(\R^d ; \CA) \times \cC^\infty(\R^d ; \CA) \to \cC^\infty(\R^d ; \CA)$ extends to a $t$-continuous map $\cC^\infty(\R^d ; \CA)^{\wotimes_\pi 2} \to \cC^\infty(\R^d ; \CA)$. This is used in the definition of the action of $\Pi_x$ on products as well as its action on root-renormalised trees. Finally, we also use Theorem~\ref{thm:ChngCoef} to conclude that applying $\Delta(\mfo_\rho)$ preserves continuity.
\end{remark}

\begin{remark}
	Here one should think of $\xi_\mfl$ as mollifications of the true distributional noises which are in $\CC^{|\mfl|_\s - \kappa }(\R^d ; \CA)$. 
	In the case of $q$-mezdons and fermions, one can take $\kappa = 0$ since we have ``$L^{\infty}$'', uniformly bounded regularity estimates on the noise, whereas for bosons we need to take some $\kappa > 0$ arbitrarily small since we need to use a Kolmogorov argument to turn moment estimates into an almost surely finite regularity estimate.
	See, for instance, \cite[Lemma~4.2]{CHP23} for a regularity estimate in the fermionic case. The same estimates also apply to the mezdonic case.% \martinp{Say that the same applies for mezdons} 
\end{remark}

\begin{theorem}
	For any admissible choice of noises $\xi_\mfl$ and $|\mft|_\s$-regularising kernels $K_{\mft}$, $Z$ is a model.
\end{theorem}
\begin{proof}
	The algebraic constraint follows directly from the fact that $\Gamma_{xy} = \Gamma_{g_x}^{-1} \circ \Gamma_{g_y}$ and that we can write $\Pi_x$ as $\PPi \circ \Gamma_{g_x}$ for a ``stationary'' model $\PPi$ defined below, see \eqref{eq:StatMod} and Lemma~\ref{lemma:StatMod}.

	For the analytic bounds, we proceed, as always, by induction. Since we have assumed that the model is smooth, the bounds for $\Xi_{\mfk}$, $\mfk \in \mfL^- \times \N^d$, as well as $X^k$ follow trivially. This is true for product $\tau\bar\tau$ as both $\Pi_x$ and $\Gamma_{xy}$ are multiplicative per definitionem. 

	For $\Pi_x \CI_{\mfk} \tau$, this follows essentially from the definition. For negative regularities, the scaling estimates are a consequence of the smoothness of the model, whereas for positive regularities, it follows as we are subtracting the truncated Taylor expansion at $x$, see the proof of \cite[Lemma~5.19]{Hai14} and replace the (implicit) norm of $\R$ with a seminorm $\mfp \in \mfP$. Concerning the bound for $\Gamma_{xy} \CI_{\mfk} \tau$ we have
	\begin{equ}
		\Gamma_{xy} \CI_{\mfk} \tau - \CI_{\mfk} \tau = \CI_{\mfk} \left( \Gamma_{xy} \tau - \tau \right) + \sum_{\ell} \frac{X^\ell}{\ell!} \int K_{\mfk + \ell, |\tau|_{|\s|}} (x,y;z) (\Pi_y \tau)(z) \d z \; . 
	\end{equ}
	The bound on the first term follows from the induction, whilst for each term in the sum it follows from the fact that $K_{\mfk + \ell, |\tau|_{|\s|}}$ is a truncated Taylor expansion. For the details, see \cite[Lemma~5.21]{Hai14}.
\end{proof}

% \begin{remark}
% 	Note that the sum in the definition of the model of $\CI_\mfk \tau$ is necessarily finite as $\CJ_{\mfk + \ell } \tau $ can be non-zero for only finitely many $\ell$.

	%Furthermore, the extended decoration at the root is essentially ignored, however, it can impact the model further up along the tree as it can reduce the $|\,\bigcdot\,|_+$-degree of $\CJ_{\mfk+\ell}\tau$, and thus also truncate the sum even further. \martinp{Move this to definition of structure group}
% \end{remark}

As we mentioned above, one can also define a ``stationary'' model $\PPi \colon \CT \to \cC^\infty(\R^d; \CA)$ given by
\begin{equs}[eq:StatMod]
	\left(\PPi \Xi_{\mfl} \right)(y) & = \xi^{(\eps)}_{\mfl}(y) \; , \\
	\left(\PPi X^k \right)(y)& = y^k \bone_{\CA} \; , \\
	%\left(\PPi^{(\eps)}  \left(\cR_{\alpha}\tau \right) \right)(y) & = \left(\PPi^{(\eps)} \tau  \right)(y)\\
	\left(\PPi \left( \CI_\mfk \tau \right)  \right)(y)&= \int D^{\mfk_2}_1 K_{\mfk_1} (y,z) \left(\PPi \tau \right)(z) \d z \; , 
\end{equs}
for $\tau, \bar\tau$ not root-renormalised 
we set 
\begin{equ}
	\left(\PPi  \left(\tau \bar\tau \right) \right)(y)  = \left(\PPi \tau  \right)(y)\left(\PPi \bar\tau \right)(y) 
\end{equ}
and for a root-renormalised tree $\tau$
\begin{equs}[eq:StatMod2]
	\PPi \tau &= \PPi \Bigl( \Bigl(\CF \otimes \bigotimes_{v \in N_T}  X^{\mfn(v)} \Bigr)  \curvearrowright \tau_\rho \Bigr) \eqdef \\
	&\eqdef \Delta(\mfo_\rho) \Bigl( \left(\PPi\right)^{\wotimes_\pi n} \left(\CF \right)  \Bigr) \prod_{v \in N_T} \PPi X^{\mfn(v)} \; .
\end{equs}
all of which are extended $\CA$-bilinearly to $\CT$.

\begin{lemma}
\label{lemma:StatMod}
	For all $x \in \R^d$, $\Pi_x = \PPi \circ \Gamma_{g_x}$. 	
\end{lemma}
We leave the proof to Appendix~\ref{app:alg}.

%We note that $\Pi_x^{(\eps)} (\tau\bar\tau)$ is indeed an element of $\cC^\infty(\R^d ; \CA)$ since we assume that the multiplication map $\CA^{\wotimes n} \to \CA$ is continuous since we can view

\begin{definition}[Renormalised Model]
Given a preparation map $R$ and a smooth canonical model $Z = \big(\Pi,\Gamma\big)$ we define a corresponding renormalised model $Z^R \eqdef (\Pi^{R}, \Gamma^{R})$ by setting  
	\begin{equs}\label{eq:renormalised_model}
		\Pi^{R}_x &\eqdef \Pi_x \circ M_R \; ,\\
		\Gamma^{R}_{xy} &\eqdef \Gamma^{-1}_{g^{R}_x} \circ \Gamma_{g^{R}_y} \; .
	\end{equs}
\end{definition}
% \begin{remark}
% 	In fact whenever a model is given by a sufficiently smooth, character-valued map $(g_x)_{x \in \R^d}$ and a single map $\PPi \colon \CT \to \cD'(\R^d; \CA)$, such that $\Pi_x \eqdef \PPi \circ \Gamma_{g_x}^{-1}$ and $\Gamma_{xy} \eqdef \Gamma_{g_x} \circ \Gamma_{g_y}^{-1}$, then one can analogously define a renormalised model $(\Pi^R, \Gamma^R)$.
% \end{remark}

\begin{proposition}
	$Z^R$ as given in \eqref{eq:renormalised_model} is indeed a model for $\cT$ in the sense of Definition~\ref{def:model}.
\end{proposition}
\begin{proof}
	Algebraically, this is indeed a model because of Proposition~\ref{prop:CoInt} as
	\begin{equs}
		\Pi_{x}^{R} \circ \Gamma^{R}_{xy} &= \Pi_x \circ M_R \circ \Gamma^{-1}_{g^{R}_x} \circ \Gamma_{g^{R}_y} = \Pi_x  \circ \Gamma^{-1}_{g_x} \circ \Gamma_{g_y} \circ M_R = \\
		&=\Pi_{y}^{R} \; .
	\end{equs}
	The analytic bounds are proven by induction with respect to the preorder $<_\cT$. Let $\tau$ be a tree. We can first rewrite $M_R\tau$ as  $M_R\tau = M^\circ_R ( R\tau - \tau ) + M_R^\circ \tau$. The first term satisfies per definitionem $R\tau - \tau <_{\cT} \tau$ and $\alpha \eqdef |\tau|_\s<|R\tau - \tau|_\s  $ thus by induction
	\begin{equs}
		\left\| \Pi^{(\eps)}_x M^\circ_R ( R\tau - \tau ) \right\|_{\alpha; \K, \mfp} &\leqslant \left\| \Pi_x M^\circ_R ( R\tau - \tau ) \right\|_{|R\tau - \tau|_\s; \K, \mfp} \lesssim \\ & \lesssim_{\Pi,R} \| R\tau- \tau\|_{< \alpha+\gamma;\mfp}  \lesssim_{R} \| \tau \|_{\alpha; \mfp} \; .
	\end{equs}
	The bound for the second term follows straightforwardly from the multiplicativity of $M_R^\circ$ and $\Pi_x$.
\end{proof}

\begin{remark}
	The renormalised characters $f_x^{R}$ satisfy the following inductive relations
	\begin{equs}
	f_{x}^{R}(\1) &= \bone_{\CA} \; , \\
	f_{x}^{R}(X_{\mu}) &=  x_\mu \bone_{\CA} \; ,  \\
	f_{x}^{R}(\tau\bar\tau) &=f_{x}^{R}(\tau)f_{x}^{R}(\bar\tau) \; , \\
	f_{x}^{R} \left( \CJ_{\mfk} \tau \right) &=
		  \int D^{\mfk_2}_1 K_{\mfk_1} (x,z) \left( \Pi^{R}_{x} \tau  \right)(z) \d z \; ,
\end{equs} 
and analogously for $g^R_x$, cf.\ \cite[Section~3.2]{Br18}.

%\ajay{give reference to where this is proved in the commutative case, I suppose in Yvain's preparation map paper.}

\end{remark}

\begin{proposition}
\label{prop:NegModBnd}
	Let $\gamma \in \R$, $R,R'$ two preparation maps, and $Z,\overline{Z}$ two canonical models with the same integration kernels but different noise assignments. There exists a $k \in \R$ and compactum $\overline{\K} \supset \K$, s.t.\
	\begin{equs}
		\left\vvvert Z^R \right\vvvert_{\CT_\gamma ; \mfK, \mfp} \lesssim \left(1+\| \Pi^{R} \|_{\CT^- ; \overline\mfK , \mfp}\right)^k \; ,
	\end{equs}
	and
	\begin{equ}
		\left\vvvert Z^R ; \overline{Z}^{R'} \right\vvvert_{\CT_\gamma ; \mfK, \mfp} \lesssim \left(1+\| \Pi^{R}  \|_{\CT^- ; \overline\mfK , \mfp} \wedge  \| \overline{\Pi}^{R'} \|_{\CT^- ; \overline\mfK , \mfp}\right)^{k} \| \Pi^{R} -\overline{\Pi}^{R'}  \|_{\CT^- ; \overline\mfK , \mfp} \;.
	\end{equ}
	Here, how large $k$ and $\overline{\K}$ have to be chosen only depends on $\gamma$ and the subcritical rule $\cR$. 
	Note that the implicit constants above can depend on the seminorm $\mfp$.
\end{proposition}
\begin{proof}
	Let $V$ be the smallest subspace of $\cT$ containing $\CT_\gamma$ and every subtree of every tree in $\CT_\gamma$. $V$ is a sector as for a tree $\tau$, as all summands in $\Gamma\tau$ are subtrees of $\tau$. Let $\gamma'$ be larger than the maximal regularity of $V$.

	Let $\CF = \left\{ \tau_0, \dots, \tau_n \right\}$ be the finite number of scalar trees spanning $V$. Let $\prec$ be a total order on $\CF$, s.t.\ $\tau \prec \tau'$  if $|T| < |T'|$  or $|T|=|T'|$ and $|\tau|_\s < |\tau'|_\s$. W.l.o.g.\ we assume that $\tau_i \prec \tau_j$ if and only if $i < j$. Note that there necessarily  is a $l \leqslant d$, s.t.\ the initial segment $\{ \tau_0 , \dots, \tau_k \}$ consists of $\1$ and $X_\mu$. Furthermore, $\tau <_{\cT} \tau'$ implies that $\tau \prec \tau'$.

	Let $V^{(i)}$ be the subspace of $V$ spanned by $\{\tau_0, \dots, \tau_i\}$. Since any summands $\tau'$ appearing in the sum of $\Gamma \tau$ satisfies $\tau' \preceq \tau$, $V^{(i)}$ are subsectors of $V$. We proceed by induction on these sectors $V^{(i)}$. %Furthermore, we only need to check that the maps $\Gamma_{xy}^{(\eps), R}$ satisfy the bound.

	For $i \leqslant l$, we have per definitionem $\| \Gamma^{R} \|_{V^{(i)} ; \K , \mfp}  = 1$ and $\| \Gamma^{R} - \overline{\Gamma}^{R'} \|_{V^{(i)} ; \K , \mfp}  = 0$, and analogously for $\Pi^{R}$. For the inductive step, suppose that
	\begin{equ}
		\vvvert Z \vvvert_{V^{(\ell)} ; \K , \mfp} \lesssim \left( 1 + \| \Pi^{R} \|_{ \CT^-; \K(\ell), \mfp} \right)^{k(\ell)}
	\end{equ}
	as well as the difference bound hold for $\ell \in [n]$ for some $k(\ell) \in \N$ and $\K(\ell) \Subset \R^d$. We will first consider the $\Gamma$-bounds.

	If $\tau_{\ell + 1}  = \Xi_{\mfk}$, then we have $\Gamma_{xy}^{R} \Xi_{\mfk} - \Xi_{\mfk} = 0$ for all $x,y \in \R^d$. Therefore,  $\| \Gamma^{R} \|_{V^{(\ell+1)}; \K , \mfp} = \| \Gamma^{R} \|_{V^{(\ell)}; \K , \mfp} $. The difference bound follows analogously, proving the assertion in this case with $k(\ell + 1) = k(\ell)$ and $\K(\ell+1) = \K(\ell)$.

	If $\tau_{\ell + 1} = \tau \tau'$, then $\tau,\tau' \in V^{(\ell)}$. Thus,  for $r \leqslant |\tau|_{\s}$, $s \leqslant |\tau'|_{\s}$ we have
	\begin{equs}
		\mfp_{r+s}\left( \Gamma_{xy}^{R} \tau_{\ell + 1}  \right) &\leqslant \mfp_{r+s}\left( \Gamma_{xy}^{R} \tau   \Gamma_{xy}^{R} \tau' \right) \leqslant \\
		& \leqslant \left\| \Gamma^{R}_{xy} \right\|_{V^{(\ell)} ; \K(\ell), \mfp}^2 \mfp_{|\tau|_\s+|\tau'|_\s}(\tau \tau') |x-y|_{\s}^{|\tau|_\s+|\tau'|_\s - r - s} \lesssim \\
		& \lesssim \left( 1 + \| \Pi^{R} \|_{ \CT^-; \K(\ell), \mfp} \right)^{2k(\ell)} \mfp_{|\tau_{\ell+1}|_\s}(\tau_{\ell+1}) |x-y|_{\s}^{|\tau_{\ell+1}|_\s -r -s} \; ,
	\end{equs}
	thus
	\begin{equ}
		\left\| \Gamma^{R}\right\|_{V^{(\ell+1)} ; \K , \mfp} \lesssim \left( 1 + \| \Pi^{R} \|_{ \CT^-; \K(\ell) , \mfp} \right)^{k(\ell+1)}
	\end{equ}
	with $k(\ell+1) = 2k(\ell)$. Here we used that $\Gamma_{xy}^{R} \otimes \Gamma_{xy}^{R}$ extends to a $t$-continuous linear map on the $t$-projective tensor product in the second inequality. The difference bound is proved analogously.

	Finally, suppose that $\tau_{\ell + 1} = \CI_{\mfk} \tau$ with $\tau \in V^{(\ell)}$. To prove the $\Gamma$-bound it suffices to show that we can bound $\left\| f_x^{R} \big|_{V^{(\ell + 1)}} \right\|_{\mfp}$ in terms of $\left( 1 + \| \Pi^{R} \|_{\CT^- ; \K(\ell + 1), \mfp} \right)^{k(\ell)+1}$
	as
	\begin{equ}
		\Gamma_g \left( \CI_\mfk \tau \right) = \CI_\mfk \left(\Gamma_g \tau\right) + \sum_{m \in \N^d} \frac{X^m}{m!} g\left( \CJ_{\mfk + m} \tau \right) \; ,
	\end{equ}
	and the first term is directly bounded by the induction assumption. We use another induction to prove the $f_x^{R}$-bound. The cases other than $\CI_{\mfk}\tau$ follow as above. For the latter, we use that
	\begin{equ}
		f_{x}^{R} \left( \CJ_{\mfk} \tau \right) =  \int D^{\mfk_2}_1 K_{\mfk_1} (x,z) \left( \Pi^{R}_{x} \tau  \right)(z) \d z
	\end{equ}
	which means that
	\begin{equ}
		\sup_{x \in \K(\ell)} \left\| f_{x}^{R} \left( \CJ_{\mfk} \tau \right) \right\| \lesssim \left\| \Pi^{R} \right\|_{V^{(\ell)} ; \K(\ell) , \mfp } \mfp_{|\tau|_\s}( \tau )
	\end{equ}
	by \cite[Equation~5.40]{Hai14}. The difference bound follows analogously.%Here $\K(\ell+1)$ is the $1$-fattening of $\K(\ell)$.

	Now for the $\Pi$-bounds, if $\tau_{\ell+1} \in \CT^-$, then seminorm of $\Pi_x^{R} \tau_{\ell+1}$ is already controlled by $\left\| \Pi^{R} \right\|_{\CT^- ; \K(\ell), \mfp}$. On the other hand, if $|\tau_{\ell+1}|_{\s} > 0$, then by Proposition~\ref{prop:PosTreeBnd}, the seminorm of $\Pi_x^{R} \tau_{\ell+1}$ is controlled by
	\begin{equ}
		\| \Gamma^{R} \|_{V^{(\ell)} ; \K(\ell+1), \mfp}  \| \Pi^{R} \|_{V^{(\ell)} ; \K(\ell+1), \mfp} \lesssim \left( 1+ \| \Pi^{R} \|_{\CT^- ; \K(\ell+1), \mfp} \right)^{2 k(\ell)} \; ,
	\end{equ}
	where $\K(\ell+1)$ is the $1$-fattening of $\K(\ell)$. One bounds the difference analogously. Thus, the assertion holds for $Z^R$ with $k(\ell+1) = 2 k (\ell)$ and $\K(\ell+1)$ as above. The difference bound follows from the same proposition.

	Overall, this means that at worst
	\begin{equ}
		\vvvert Z^R \vvvert_{\CT_\gamma ; \K , \mfp} \lesssim \left(1+\| \Pi^{R} \|_{\CT^- ; \overline{\mfK} , \mfp}\right)^{2^n}
	\end{equ}
	with $\overline{\K}$ being the $n$-fattening of $\K$. The same holds for the difference bound.

\end{proof}

Most importantly, the above results hold for the extraction-contraction renormalisation described in Section~\ref{sec:ExtContRen}.

\begin{remark}
\label{rem:MixAlg1}
	When we are dealing with systems describing mixtures of bosons, with algebra of random variables $\CA_B = L^0(\Omega, \mu)$, and a set of different particles described by a locally $m$-convex algebra $\CA$, the above results continue to hold. However, instead of viewing the resulting object as an $\CA_B \wotimes \CA$-regularity structure, we shall describe it as a fixed $\CA$-regularity structure with a random model. Proposition~\ref{prop:NegModBnd} then continues to hold pathwise. 
\end{remark}

\section{Stochastic Estimates}\label{sec:StochEst}
In this section, we investigate the convergence of the BPHZ renormalised lift of a mollified  $q$-$\mfH$-white noise in the rough limit where the mollification is removed. 
Here $\mfH$ is a Sobolev space of potentially mixed regularity in spatial and temporal variables. 

The core of our strategy involves obtaining $q$-Gaussian stochastic estimates by applying\footnote{For the purposes of our argument, the diagrammatic stochastic estimates of \cite{CH16} would also serve our purposes.} \cite{HS24} to carefully designed commutative regularity structures with $q=1$ \slash bosonic Gaussian noise. 

The key observation underlying this strategy is that -- thanks to the isomorphism between $\CL^{2}(\CA_{q}(\mfH),\omega_{q})$ and $\mathcal{F}_{q}(\mfH)$ for any $q \in [-1,1]$ -- estimates on second moments of the model on a given bare tree reduce to certain linear combinations of $q$-symmetrised $L^{2}$ estimates on deterministic kernels of multiple space-time arguments. 
The particular deterministic kernels that appear in this argument are independent of $q$. However, since the Fock spaces $\mathcal{F}_{q}(\mfH)$ are different for each $q$, how these kernels determine the desired second moments does vary with $q$. 
For instance, the symmetrisation of kernels is $q$-dependent, and the algebraic rule for how a product of noises decomposes into homogeneous chaoses is also $q$-dependent.

To carry out the strategy mentioned at the beginning of the section, it will be useful to introduce an intermediate algebra $\mfA_N(\mfH)$ which we call the (truncated) ``Fock Space algebra''.  
This algebra is not suitable for solving equations, as it is a truncated algebra with nilpotent elements; however, it is useful as a bookkeeping tool for the aforementioned deterministic kernels.
Estimates on models on this $\mfA_N(\mfH)$-regularity structure efficiently store estimates on these deterministic kernels in a way that they can be automatically turned into estimates in $\CL^{2}(\CA_{q}(\mfH),\omega_{q})$ for the corresponding $q$-Gaussian models. For $q \in [-1,1)$ this also gives estimates in $\CA_{q}(\mfH)$.  
See Corollary~\ref{cor:FocktoQbound} for a precise statement of this. 

As alluded to above, we prove estimates for models on the $\mfA_N(\mfH)$-regularity structure by repeatedly appealing to $q=1$ stochastic estimates from \cite{HS24} for BPHZ models on carefully designed trees \dash see Theorem~\ref{thm:BPHZ}. 

In this section, we shall give a simple criterion that lets us conclude the existence of renormalised models when built on top of $q$-$\mfH$-white noise when $\mfH$ is a Sobolev space of potentially mixed regularity.

In Section~\ref{sec:MixedRegSob}, we define mixed regularity Sobolev spaces and define the simple regularity parameters that we can use as  a criterion for our BPHZ theorem. 
Readers who are primarily concerned with a $q$-space-time white noise, that is, when $\mfH$ is precisely the $L^{2}$ over space-time, can skip Section~\ref{sec:MixedRegSob}.

In Section~\ref{subsec:FockSpaceAlg} we introduce the  Fock space algebras $\mfA_N(\mfH)$ described above and introduce the map $\iota_{q}$ which continuously maps the Fock space algebra into $\cA_{q}(\mfH)$ for $q \in [-1,1)$.
In Section~\ref{subsec:FockSpaceAlgtoRegStruc} we introduce the relevant models on $\mfA_N(\mfH)$-regularity structures.
Finally, in Section~\ref{sec:ConvBPHZ} we show how to obtain estimates on $\mfA_N(\mfH)$-regularity structures model using \cite{HS24}. We then show they yield the desired bounds on the corresponding $q$-Gaussian models. 

\subsection{Mixed Regularity Sobolev Spaces}
\label{sec:MixedRegSob}

For the rest of the section we fix  $n,m \in \N$ and two sequences of dimension $\mfn \in \N^n$, $\mfm \in \N^m$ and corresponding scalings  $\s_1 \in \R_{\geqslant 1}^n$ and $\s_2 \in \R_{\geqslant 1}^m$. Furthermore, let $d_1 \eqdef |\mfn|$, $d_2 \eqdef |\mfm|$, $d \eqdef d_1 + d_2$, and $\s \in \R^d_{\geqslant 1}$ be obtained by concatenating the entries of $\s_1$ and $\s_2$. 

Now, let $s_1 \in \R^n$ and $s_2 \in \R^m$ be two sequences of regularities. Let
\begin{equ}
	H^{s_1, s_2}\left( \R^\mfn \times \T^\mfm\right) \eqdef \bigwotimesa_{i = 1}^n H^{s_1(i)} \left(\R^{\mfn(i)}\right) \wotimes_\alpha \bigwotimesa_{j = 1}^m H^{s_2(j)} \left( \T^{\mfm(j)} \right) \; .
\end{equ}
A straightforward scaling argument then shows that the $q$-$H^{s_1, s_2}\left( \R^\mfn \times \T^\mfm\right)$-white noise for $q \in [-1,1)$ is an element of (resp.\ when $q = 1$ a.s.\ an element of) $\CC_{\s}^{\alpha}(\R^{d_1} \times \T^{d_2})$ (resp.\ of $\CC_{\s}^{\alpha-}(\R^{d_1} \times \T^{d_2})$), where
\begin{equ}
	\alpha(s_1,s_2) \eqdef \sum_{i = 1}^n \mfs_1(i)\left( - s_1(i) - \frac{\mfn(i)}{2}\right) \wedge 0 + \sum_{j = 1}^m \mfs_2(j)\left( - s_2(j) - \frac{\mfm(j)}{2}\right) \wedge 0
\end{equ}
\begin{remark}
	In fact, if all $-s_1(i) - \frac{\mfn(i)}{2} > 0$ and $-s_2(j) - \frac{\mfm(j)}{2} > 0$, then $\xi_q \in \CC^{\alpha}_{\s}(\R^{d_1} \times \T^{d_2})$ with
	\begin{equ}
		\alpha \eqdef \min_{i \in [n]} \mfs_1(i)\left( - s_1(i) - \frac{\mfn(i)}{2}\right)  \wedge \min_{j \in [m]} \mfs_2(j)\left( - s_2(j) - \frac{\mfm(j)}{2}\right) \; .
	\end{equ}
	However, for the sake of convenience, we shall only consider negative noises.% as otherwise there is no need for renormalisation.
\end{remark}

\subsection{The \TitleEquation{\cA_q}{Aq}-Regularity Structures and Canonical Models}

Let $\spacetime \eqdef \R^{d_1} \times \T^{d_2}$. We now fix a finite set of positive and negative symbols $\mfL^{\pm}$. For each $\mfl \in \mfL^-$ we choose a mixed regularity Sobolev space $\mfH_\mfl \eqdef H^{s^\mfl_1, s^\mfl_2}\left(\R^{\mfn} \times \T^{\mfm}\right)$, s.t.\ $\alpha_\mfl \eqdef \alpha(s^\mfl_1, s^\mfl_2) < 0$, and set $\xi_q^\mfl$ to be the $q$-$\mfH_\mfl$-white noise. If $q = -1$, $\bXi_{\mfl}$ denotes the components of the extended $-1$-$\mfH_\mfl$-white noise and $\bPsi_{\mfl}$ those of the canonical $\mfH$-Dirac noise with covariance $U$. We assume that $U$ commutes with the group of translations on $\R^{d_1} \times \T^{d_2}$ and that $\kappa$ is the extension of the usual complex conjugation.

We also define $\mfH \eqdef \bigoplus_{\mfl \in \mfL^-} \mfH_\mfl$ and for $q \in [-1,1)$ let $\cA_q \eqdef \cA_q(\mfH)$.

\begin{remark}
	For $f_\mfl \in \mfH_\mfl$, $\xi_q^{\mfl}(f_\mfl)$ can naturally be viewed as an element of $\cA_q$, by identifying it with $\xi_q^{\mfH}(0, \dots, 0, f_{\mfl}, 0, \dots, 0)$, where $\xi_q^{\mfH}$ is the $q$-$\mfH$-white noise. Analogously, we can thus also view $\xi_q^\mfl$ as an element of $\CC^{\alpha_\mfl}_{\s}(\spacetime)$. 
\end{remark} 

We now make an important structural assumption. 
\begin{assumption}
\label{as:ScalarV}
    For each $\mft \in \mfL^+$, we choose a $\beta_{\mft}$-regularising, \textbf{scalar-valued}, translationally invariant kernel $K_\mft \colon \R^d \setminus \{0\} \to \mathbb{K}$ with $\beta_{\mft} > 0$.    
\end{assumption}
\begin{remark}
    We need to assume this, since we are relying on the scalar BPHZ theorem and our argument will be agnostic to the specific $q \in (-1,1)$. A more general result would need techniques tailored more concretely for the noncommutative case.
\end{remark}
Finally, we also choose $\kappa > 0$ sufficiently small. We set
\begin{equs}
	\forall \mfl \in \mfL^- &: |\mfl|_{\s} \eqdef \alpha_\mfl - \kappa \; ,  \\
	\forall \mft \in \mfL^+ &: |\mft|_{\s} \eqdef \beta_{\mft} \; .
\end{equs}
We assume that $\kappa$ is such that 
\begin{equ}
    \sum_{\mfl \in \mfL^-} n_{\mfl} |\mfl|_{\s} + \sum_{\mft \in \mfL^+} m_{\mft} |\mft|_{\s} \notin \Z 
\end{equ}
for all $n_{\mfl}, m_{\mft} \in \N$.

Let $\cR$ be a subcritical rule on $\mfL$. For $q \in (-1,1)$ let $\cT_q$ be the $\cA_q$-regularity structure generated by $\cR$, for $q = -1$, let $\cT_{\Cl}$ be the $\cA_{F}^{\Cl}$-regularity structure and $\cT_F$ the $\cG_F$-regularity structure. The respective canonical models $Z^q_{(\eps)} = \left( \Pi^{q, (\eps)}, \Gamma^{q, (\eps)} \right)$, $Z_{(\eps)}^{\Cl}$, and $Z_{(\eps)}^{F}$ with mollification parameter $\eps \in (0,1]$ are then specified as follows. For $\mfl \in \mfL^-$
\begin{equs}
	\left(\Pi^{q; (\eps)}_x \Xi_{\mfl}\right)(y) &\eqdef \xi^{\mfl , (\eps)}_q ( y ) \eqdef \xi^\mfl_q \left( \CS^\eps_{\s, y} \rho \right) \\
	\left(\Pi^{\Cl; (\eps)}_x \Xi_{\mfl}\right)(y) &\eqdef \bxi_F^{\mfl, (\eps)} ( y ) \eqdef \bxi_F^\mfl \left( \CS^\eps_{\s, y} \rho \right)\\
	\left(\Pi^{F; (\eps)}_x \Xi_{\mfl}\right)(y) &\eqdef \bPsi^{(\eps)}_\mfl ( y ) \eqdef \bPsi_\mfl \left( \CS^\eps_{\s, y} \rho \right)
\end{equs}
where is the mollifier $\rho$ defined in the introduction, Section~\ref{subsec:notation}. $\CI_\mft$ for $\mft \in \mfL^+$ implements abstract integration against $K (x,y) = K_\mft(x-y)$ as outlined in Section~\ref{sec:AbstInt}.

Finally, we need to specify the action $\Delta(\mfo_\rho)$. For a root-renormalised tree $\tau$, let $\tau_\rho$ be the scalar tree, s.t.\
\begin{equ}
	\tau = \Bigl(\CF \otimes \bigotimes_{v \in N_T} X^{\mfn(v)}\Bigr) \curvearrowright \tau_\rho \; , 
\end{equ}
see \eqref{eq:scalarGraftRep}. We remind the reader that $\tau_\rho$ is uniquely determined by $\mfo_\rho$, the connected component of the root of $\mfo$. In particular, $\tau_\rho  = (\mff_\rho, \mfe_{\mff_\rho}, \mfn_{\tau_\rho}, \mfo_\rho, \bone^{\otimes \mff_\rho} )$ for $\mfo_{\rho} = ((\mff_\rho, \mfe_{\mff_\rho}, \mfn_{\tau_\rho}), \bpi_\rho)$.

We now define  $\Delta(\mfo_{\rho}) \in \mathcal{B}^{\ell}(\mathcal{A} ; \mathcal{A})$ where $\ell = |\tau_\rho|$ the number of attachment spots. We also label from left to right the negative leaves of $\tau_\rho$ by the set $[n]$. 
We set $\Delta(\mfo_\rho) \equiv 0$ unless the partition of the negative leaves $\bpi_\rho$ of $\mfo_{\rho}$ is made up only of pairs, i.e.\ two-element subsets.

Let $B_1, \dots, B_k  \in \cA_q$. For $i \in [n]$ we set $C_i$ to be the product of all the $B_j$ such that the attachment spot $j$ lies between the $(i-1)^{\text{st}}$ and $i^{\text{th}}$ negative leaf of $\tau_\rho$.\footnote{There is always an attachment spot between two leaves.}
%\footnote{All trees $\tau_j$ lie exactly between two negative trees per our grafting definition of these trees, unless $\tau_1$ lies to the left of the rest of the tree $\tau$ or $\tau_k$ lies to the right of all of the tree $\tau$.} 
Analogously, $C_0$ is the ordered product of all those $B_j$ such that the attachment spot $j$ sits to the left of the first negative leaf of $\tau_{\rho}$, and $C_{n+1} $ is the ordered product of all those $B_j$ such that $j$ sits to the right of the $n^{\text{th}}$ negative leaf of $\tau_{\rho}$. Writing $\boldsymbol{k} = (1,\dots, 1) \in \N^n$, we then define
\begin{equ}
	\Delta(\mfo_\rho) \left( B_1,\dots,B_k \right) \eqdef \Delta_{q(\#)}^{R ; \boldsymbol{k}, \bpi_{\rho}}( C_0,C_1, \dots, C_{n+1} )   \; ,
\end{equ} 
cf.\ Definition~\ref{def:RenMultMap} and Corollary~\ref{cor:RenMultMap}. Here $\# \in (-1,1) \sqcup \{\Cl, F\}$ with $q(r) = r$ for $r \in (-1,1)$, $q(\Cl) = q(F) = -1$.

% \begin{remark}
% 	We note that for $\cT_q$, the norm on each subspace generated by a tree with $n$ vertices and edges is just the projective tensor product of $n$ copies of $\vvvert \bigcdot \vvvert$. 

% \end{remark}

\begin{example}
	We give an example of the definition of $\Delta(\mfo_\rho)$ by continuing Example~\ref{ex:ExTrees1}. In particular, we consider the tree $\tau$ given by
	\begin{equ}
		\scalebox{0.8}
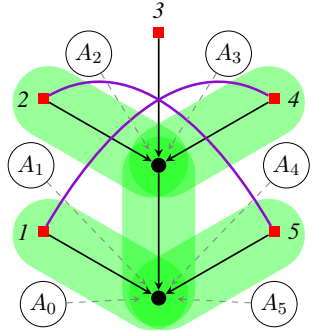
{
		\begin{tikzpicture}[
				% --- Styles ---
			red_square/.style={rectangle, fill=red, inner sep=2.5pt},
			black_dot/.style={circle, fill=black, inner sep=2.5pt},
			arrow/.style={->, >=stealth, thick},
			every edge quotes/.style={font=\small, sloped, fill=white, inner sep=1.5pt},
			highlight/.style={green, opacity=0.4, line cap=round, line join=round, line width=1.2cm},
			% The 'spring' style is replaced with a simple 'purple_line' style.
			purple_line/.style={draw=purple, very thick}
		]
			% --- 1. Define Node positions ---
			% (This section is unchanged)
			\node[black_dot, label=below: ]   (1) at (0,0) {};
			\node[red_square, label=left: 1]     (3) at (-1.9, 1.1) {};
			\node[black_dot, label=right: ]    (5) at (0, 2.2) {};
			\node[red_square, label=left: 2]     (7) at (-1.9, 3.3) {};
			\node[red_square, label=above: 3]    (9) at (0, 4.4) {};
			\node[red_square, label=right: 4]    (11) at (1.9, 3.3) {};
			\node[red_square, label=right: 5]    (13) at (1.9, 1.1) {};

			% --- 2. Draw Highlights on the Background Layer ---
			% (This section is unchanged)
			\begin{pgfonlayer}{background}
				\draw[highlight] (7) -- (5);
				\draw[highlight] (11) -- (5);
				\draw[highlight] (3) -- (1);
				\draw[highlight] (13) -- (1);
				\draw[highlight] (5) -- (1);
			\end{pgfonlayer}

			% --- 3. Draw Original Edges ---
			% (This section is unchanged)
			\draw[arrow] (3) edge  (1);
			\draw[arrow] (13) edge (1);
			\draw[arrow] (7) edge (5);
			\draw[arrow] (9) edge (5);
			\draw[arrow] (11) edge (5);
			\draw[arrow] (5) edge (1);

			% --- 4. Add the new purple line connections ---
			% The line connecting 3 and 11 is now bent upwards.
			\draw[purple_line] (3) to[out=60, in=150] (11);
			\draw[purple_line] (7) to[out=30, in=120] (13);
			%\draw[purple_line] (3) .. controls (0,6) .. (11);
			%\draw[purple_line] (7) .. controls (0,6) .. (13);

			\node[circle, draw, inner sep=2pt] (A0) at ($(3) + (0,-1.2)$) {$A_{0}$};
			\node[circle, draw, inner sep=2pt] (A1) at ($(7)!0.5!(3) + (-0.2,0)$) {$A_{1}$};
			\node[circle, draw, inner sep=2pt] (A2) at ($(7)!0.5!(9) + (-0.2,0.2)$) {$A_{2}$};
			\node[circle, draw, inner sep=2pt] (A3) at ($(11)!0.5!(9) + (0.2,0.2)$) {$A_{3}$};
			\node[circle, draw, inner sep=2pt] (A4) at ($(11)!0.5!(13) + (0.2,0)$) {$A_{4}$};
			\node[circle, draw, inner sep=2pt] (A5) at ($(13) + (0,-1.2)$) {$A_{5}$};

			\draw[->, dashed, gray, shorten >=5pt] (A0) -- (1);
			\draw[->, dashed, gray, shorten >=5pt] (A1) -- (1);
			\draw[->, dashed, gray, shorten >=5pt] (A4) -- (1);
			\draw[->, dashed, gray, shorten >=5pt] (A5) -- (1);

			\draw[->, dashed, gray, shorten >=5pt] (A2) -- (5);
			\draw[->, dashed, gray, shorten >=5pt] (A3) -- (5);

		\end{tikzpicture}
		}
	\end{equ}
	where we used the same conventions as in \eqref{eq:ExTrees1}. For the sake of simplicity \ajay{Rather than just being an asusmption for simplicity, isn't this an assumption in force for the whole BPHZ section. Moreover, the definition of $\Delta(\mfo)$ isn't really reasonable without this assumption. Also, when this assumption is in force, isn't there some redundancy in the regularity structure since decorations can move around?}\martinp{Added in Assumption~\ref{as:ScalarV}}, we assumed here that integration commutes with multiplication by fixed algebra elements and $A_i \in \cA_q$ for $i \in {0,\dots, 5}$ denote the algebra decorations that remain if one multiplies all the algebra elements that are not separated by a noise. We may do this, since we have assumed that all the integration kernels are scalar-valued. We have drawn grey dashed arrows connecting each $A_i$ to its corresponding attachment spot. 

	Here $\mfo_\rho$ is made up of the tree highlighted in green and $\pi_\rho = \{(1,4),(2,5)\}$. Thus, we have the grafting decomposition
	\begin{equ}
		(A_0 \1 \otimes A_1 \1 \otimes A_2 \CI \left( \Xi \right)  A_3 \otimes A_4 \1 \otimes A_5 \1 \otimes X^{0} ) \curvearrowright \tau_\rho \;  .
	\end{equ}
	Here we are slightly abusing the notation as we are not exactly specifying where we are attaching $A_i \1$. Applying the definition, we get 
	\begin{equ}
		\Pi_x \tau = \Delta_{q}^{R; \boldsymbol{k} , \bpi_\rho} \left( A_0 , A_1, A_2 \Pi_x \CI(\Xi) A_3 , A_4, A_5\right)
	\end{equ}
	where $\boldsymbol{k} = (1,1,1,1)$.

\end{example}

\begin{remark}
In all of the equations we consider in this paper, when we write the modelled distribution for the solution to the equation all the algebra decorations that appear come from the polynomial part of the modelled distribution describing the solution and this constrains where and how such decorations can appear. 
However, this particular structure is an artifact of the fact that all the equations we present in this paper have non-linearities given by real functions lifted to the relevant algebra. 
Our machinery is more general than this however, for instance we could in principle replace the nonlinearity $\phi^3$ by $b_0 \phi b_1 \phi b_2 \phi b_3$ for fixed $b_i \in \cA_q$ which would make the trees in the expansion of the solution become ``more decorated''. 
\end{remark}

\begin{remark}
\label{rem:RanRS}
	In the case $q=1$, we would be working with $\CA_B(\mfH) \supsetneq \cA_B(\mfH)$. Furthermore, the above construction should be viewed as constructing an $\R$-regularity structure with a random model rather than a true $\CA_B(\mfH)$-regularity structure, as the latter is not a true locally $m$-convex algebra. 

	Furthermore, in the mixed case when we are dealing with a random algebra $\CA_B(\mfH_1) \otimes \cA_q(\mfH_2)$, we will split up the negative symbols $\mfL^-$ into a bosonic subset $\mfL^-_B$, giving $\mfH_1$, and a noncommutative subset $\mfL^-_A$, giving $\mfH_2$. Then, for $\mfl \in \mfL^-_B$, $\xi_{\mfl}^{(\eps)} \in L^{\infty-}\left( \Omega , \mu ; \cC^\infty(\spacetime)\right)$ whereas for $\mfl \in \mfL^-_A$, $\xi_{\mfl}^{(\eps)} \in \cC^\infty(\spacetime ; \cA_q(\mfH_2))$. Interpreting the driving noises in this way, any tree in the $\cA_B(\mfH_1)$-random $\cA_q(\mfH_2)$-regularity structure gets mapped to an element of 
	\begin{equ}
		L^{\infty-}(\Omega , \mu ; \cC^\infty(\spacetime ; \cA_q(\mfH_2))) \; . 
	\end{equ}
	All further results concerning random $\cA_q(\mfH)$ algebras should be interpreted in this way. 

	We note that we have drastically restricted the randomness in this way. Allowing for true $\CA_B(\mfH)$-regularity structures in some form would be analogous to defining stochastic integration w.r.t.\ random functions of the integrator, cf.\ \cite{Kern24,BK25}.

\end{remark}

\subsection{Fock Space Algebra Model%~\&~The BPHZ Character
}\label{subsec:FockSpaceAlg}

We are interested in the convergence of the canonical model renormalised using the BPHZ character when the mollification is removed.
In particular, we shall give in Theorems~\ref{thm:BPHZ}~\&~\ref{thm:BPHZFerm} sufficiency conditions for the $q$-Gaussian and fermionic cases similar to \cite[Theorem~10.7]{Hai14}. 
We then prove that the stochastic estimates appearing in this condition can be obtained in an automatic fashion by applying \cite{CH16,HS24} to carefully designed commutative regularity structures with $q=1$ \slash bosonic Gaussian noise.

The key observation in this strategy is that, thanks to the isomorphism between $\CL^{2}(\CA_{q}(\mfH),\omega_{q})$ and $\mathcal{F}_{q}(\mfH)$ for any $q \in [-1,1]$, estimates on second moments of the model on a given bare tree reduce to certain linear combinations of $L^{2}$-type estimates on deterministic kernels of multiple space-time arguments.
The particular deterministic kernels that appear in this argument are independent of $q$, but since the Fock spaces $\mathcal{F}_{q}(\mfH)$ are different for different $q$, how these kernels determine the desired second moment does vary with $q$.

To carry out the strategy mentioned earlier, it will be useful to introduce intermediate algebras $\mfA_N(\mfH)$ which we call the ``Fock Space algebras''.
This algebra is not good for solving equations since it is a truncated algebra with nilpotent elements, but it is useful as a bookkeeping tool \dash stochastic estimates in a $\mfA_N(\mfH)$-regularity structure will store estimates on the deterministic kernels mentioned above in a way where they can easily be used as input to prove estimates in $L^{2}(\CA_{q}(\mfH),\omega_{q})$. 
Again, the proof of the stochastic estimates in a $\mfA_N(\mfH)$-regularity structure is obtained by repeatedly appealing to $q=1$ stochastic estimates for BPHZ models on carefully designed trees.

We now prepare to introduce the promised auxiliary Fock space algebras $\mfA_N(\mfH)$.

\begin{definition}
	Let
	\begin{equs}
		\mfA(\mfH) \eqdef \bigoplus_{n \in \N} \bigoplus_{\bpi \in \CP_{[n]}} \mfH^{\wotimes_\alpha [n] \setminus \bpi } \; ,
	\end{equs}
	where $\bigoplus$ denotes the algebraic direct sum, i.e.\ only finitely many entries of each vector are non-zero.
	We will denote the subspace corresponding to a contraction $\bpi$ in the above definition by
	\begin{equ}
		\mfA(\mfH)[n,\bpi] \; ,
	\end{equ}
	and
	\begin{equ}
		\mfA(\mfH)[n] \eqdef \bigoplus_{\bpi \in \CP_{[n]}} \mfA[n, \bpi] \; .
	\end{equ}
	We also the define the inhomogeneous $N$-particle sector of $\mfA(\mfH)$, with $N \in \N$, to be
	\begin{equs}
		\mfA_N(\mfH) \eqdef \bigoplus_{n = 0}^N \mfA[n] =\bigoplus_{n = 0}^N \bigoplus_{\bpi \in \CP_{[n]}} \mfH^{\wotimes_\alpha [n] \setminus \bpi } \; .
	\end{equs}
	Let $\| \bigcdot \|_N$ be the norm on $\mfA^{N}(\mfH)$ obtained by taking the $\ell^1$-sum of the norms of all subspaces $\big\{ \mfA(\mfH)[n,\bpi] \, \big| \, n \leqslant N,\ \bpi \in \CP_{[n]} \big\}$.
	We also equip $\mfA(\mfH)$ with the set of seminorms $\left( \| \bigcdot \|_N \right)_{N \in \N}$.
	
\end{definition}
We define a product in $\mfA_N(\mfH)$. Let $k, l, n \in \N$ and $I_f$, $I_g$ be finite sets. We denote by $\CP^2_{I_g,I_f ; n}$ the collection of all subsets of\footnote{Here $I_g$ and $I_f$ should be thought as disjoint.} $I_{g} \times I_{f}$ of cardinality $n$. Furthermore, given $\bpi \in \CP^2_{I_g,I_f ; n}$, we set 
\begin{equs}
	{}
	I_g \setminus_{1} \bpi &\eqdef I_g \setminus \left\{ r \in I_g \, \big| \, \exists (i,j) \in \bpi : r = i \right\} \; , \\
	{} 
	I_f \setminus_{2} \bpi &\eqdef I_f \setminus \left\{ r \in I_f \, \big| \, \exists (i,j) \in \bpi : r = j \right\} \; . 
\end{equs}

\begin{definition}
	Let $\bpi_g \in \CP_{[n],k}$, $\bpi_f \in \CP_{m,\ell}$, $I_g \eqdef [n] \setminus  \bpi_g$, $I_f \eqdef [m] \setminus \bpi_f$, $g_i, f_j \in \mfH$ for $i \in I_g$, $j \in I_f$, and let
	\begin{equs}
		G \eqdef  \bigotimes_{i \in I_g} g_i  \in \mfA(\mfH)[n, \bpi_g] \;, \quad \text{and} \quad F \eqdef  \bigotimes_{j \in I_f} f_j \in  \mfA(\mfH)[m, \bpi_f] \; .
	\end{equs}
	We define the product $GF$ to have the components
	\begin{equ}[eq:FSMult]
		\prod_{(i,j) \in \bsigma} \Braket{g_i,f_j} \bigotimes_{r \in [k] \setminus_1 \bpi} g_r \otimes \bigotimes_{s \in [l] \setminus_2 \bpi} f_s \;
	\end{equ}
	in $\mfA(\mfH)[n+m, \bpi_f \cup \bpi_g \cup \bsigma]$ for $s \leqslant n-2k \wedge m-2\ell$ and $\bsigma \in \CP^{n-2k,m-2\ell, s}$; and $0$ otherwise. Here we view $\bpi_g \cup \bpi_f \cup \bsigma$ as a contraction of the set $[n+m]$ with $\bpi_g$ contracting elements of $[n]$, $\bpi_f$ elements of $[m]+n$, and $\bsigma$ between these two sets.

	We extend this by continuity and distributivity to all $\mfA(\mfH)$, cf.\ Lemma~\ref{lem:ContEst}.
\end{definition}

\begin{remark}
Henceforth we denote the component of $F \in \mfA(\mfH)$ in the subspace $\mfA(\mfH)[n, \bpi]$ by $F[n,\bpi]$.
\end{remark}

\begin{proposition}
	$\mfA(\mfH)$ is a locally $m$-convex algebra.
\end{proposition}
\begin{proof}
	First, we note that the complement $\mfA_{>N}(\mfH)$ of $\mfA_N(\mfH)$ in $\mfA(\mfH)$ is a two-sided ideal, as follows from a direct inspection of the definition of the multiplication in $\mfA(\mfH)$. Furthermore, for $F,G \in \mfA_N(\mfH) \subset \mfA(\mfH)$. Thus, for $G, F  \in \mfA(\mfH)$ the component of $GF$ in $\mfA_N(\mfH)$ can only depend on the components of $G$ and $F$ in $\mfA_N(\mfH)$. In particular,
	\begin{equ}
		\| G F \|_{N} \leqslant C(N) \| G \|_N \| F \|_N
	\end{equ}
	where $C(N)$ is a combinatorial constant depending on the number of possible pairings. In particular, restricting the multiplication to $\mfA_N(\mfH)$ by composing it with the canonical projection $\mfA(\mfH) \to \mfA_N(\mfH)$, turns $\mfA_N(\mfH)$ into a unital Banach algebra. Thus, by \cite[Proposition~1.3]{Tak02}, for every $N \in \N$ there exists a norm $\vvvert \bigcdot \vvvert_N$ equivalent to $\| \bigcdot \|_N$ that is submultiplicative.
\end{proof}

In order to cover the Dirac fermion case, we need to slightly modify the multiplication operation of $\mfA(\mfH)$.
\begin{definition}
	We denote by $\mfA^F(\mfH)$ and $\mfA^F_N(\mfH)$ the spaces $\mfA(\mfH)$ and $\mfA_N(\mfH)$ respectively equipped with the same topology as before and multiplication defined as in \eqref{eq:FSMult} except that the factors $\Braket{g_i, f_j}$ are replaced by $\Braket{\kappa g_i , Uf_j}$. 
\end{definition}

\begin{definition}
	We will call the Banach algebra $\left(\mfA_N(\mfH), \vvvert \bigcdot \vvvert \right)$, the truncated Fock space algebra. $\mfA_N(\mfH)$ will be viewed as a subspace of $\mfA_M(\mfH)$ for $M \geqslant N$.
\end{definition}

The multiplication in $\mfA(\mfH)$ was defined in such a way that it be compatible with the multiplication operations in $\cA_q(\mfH)$ for all $q \in (-1,1]$, and for $q = -1$ it be compatible with the one defined in $\CA_F(\mfH)$. In particular, we define linear map $\iota_q \colon \mfA(\mfH) \to \cA_q(\mfH)$ for $q \in (-1,1]$, $\iota_{-1}\colon \mfA(\mfH) \to \CL^2(\CA^{\Cl}_F(\mfH), \omega_F )$, and $\iota_{F}\colon \mfA^F(\mfH) \to \CL^2(\CA_F(\mfH), \omega_\lambda )$ for $F \in \mfA(\mfH)[n]$ with $\bpi \in \CP_{[n]}$ by setting
\begin{equ}
	\iota_q\big|_{\mfA(\mfH)[n, \bpi]} (F) \eqdef q^{\crb(\bpi)} \xi_q^{\diamond(n-2|\bpi|)}(F)
\end{equ}
as well as for $F \in \mfA^F[n, \bpi]$ 
\begin{equ}
	\iota_F\big|_{\mfA^F(\mfH)[n, \bpi]} (F) \eqdef (-1)^{\crb(\bpi)} \Psi^{\diamond(n-2|\bpi|)}(F)
\end{equ}
and then extending it by linearity to all of $\mfA(\mfH)$ and $\mfA^F(\mfH)$. % Let $\mfA_P(\mfH)$ be the subalgebra generated by $\mfA(\mfH)[1,\emptyset]$. \ajay{Right now, $\mfA_P$ is not used again}
\begin{proposition}
\label{prop:AlgMorphCont}
	For all $q \in [-1,1]$ the maps $\iota_q$ as well as $\iota_F$ are continuous. Furthermore, they are also algebra morphisms.
	
	% In particular, for all $N \in \N$ we have 
	% \begin{equ}[eq:FermCont]
	% 	\| \iota_{-1}(a) \|_N \leqslant \| a \|_N \qquad \text{and} \qquad \| \iota_{F}(b) \|_N \leqslant \| b \|_N
	% \end{equ}
	% for all $a \in \mfA(\mfH)$ and $b \in \mfA^F(\mfH)$. Here the seminorms on left-hand side are in $\overline{\cA}_F(\mfH)$. Furthermore, the maps $\iota_{-1}$ and $\iota_F$ map into the subspace of elements $a$ of $\overline{\cA}_F(\mfH)$, s.t.\ $\omega_F(a^*a)< \infty$.
\end{proposition}

\begin{remark}
	We remind the reader here that $\cA_1(\mfH)$ is $L^{\infty -}(\Omega , \mu)$ for an appropriate probability space $(\Omega, \mu)$.
\end{remark}

\begin{remark}
\label{rem:NonHomo}
	The slight discrepancy in the codomains of $\iota_{-1}$ and $\iota_F$ is due to the difference in how we define the Wick products in the case of $\cA_F$ and $\cG_F$, where
	\begin{equ}
		\bxi_F(f) \bxi_F(g) = \bxi_F^{\diamond 2}(f \otimes g) + \left[ \balpha_F(f), \balpha_F(g)^\dagger \right]_+ \neq \bxi_F^{\diamond 2}(f \otimes g) + \Braket{f,g} \bone \; .
	\end{equ}
	We postpone this until Section~\ref{sec:BPHZFermi} in order to avoid an overabundance of case distinction. 
	There, we address this by modifying the multiplication in $\mfA(\mfH)$ in a way that respects localisation.

	Furthermore, when we actually need to use $\iota_F$ and $\iota_{-1}$ it will always be w.r.t.\ the noncommutative points, i.e.\ as maps $\mfA_N(b) \to \CA_F(b)$ for $b \in \Gr(\mfH)$. Since $\dim b < \infty$, $\CL^2(\CA^{\Cl}_F(b), \omega_F)$ and $\CL^2(\CA_F(b), \omega_\lambda)$ are equivalent to $\CA^{\Cl}_F(b)$ and $\CA_F(b)$ respectively as Banach spaces and we can replace the Hilbert spaces norms with the corresponding algebras norms. %they yield equivalent elements in the noncommutative $\CL^2$-space $\CL^2(\overline{\cA}_F(\mfH) , \omega_F)$. The same also holds for $\Psi_F$. We will come back to this below. 

\end{remark}

\begin{proof}
	Continuity follows directly from the continuity of $\xi_q^{\diamond n}$ for all $n \in \N$ and all $q \in [-1,1]$ as well as the continuity of $\Psi^{\diamond n}$.

	To check the algebra morphism property let $F \in \mfA[n, \bpi]$, $G \in \mfA[m, \bsigma]$, and let $\boldsymbol{\rho}$ be a contraction between $F$ and $G$. To show the morphism property we need to establish that $\crb(\bpi \cup \bsigma \cup \boldsymbol{\rho})$ is the same as $\crb(\bpi \cup \bsigma ) + \crb'(\boldsymbol{\rho})$. Here $\crb'$ denotes the intertwining number w.r.t.\ the concatenation of the sets $[n] \setminus\bpi$ and $[m] \setminus  \bsigma$.

	First we note that the contribution of $(i,j) \in \boldsymbol{\rho}$ to $\mathrm{sp}'(\boldsymbol{\rho})$ is the same as the contribution of $(i,j)$ to $\mathrm{sp}(\bpi \cup \bsigma \cup \boldsymbol{\rho})$ as in both cases the element of $[n+m]$ that have been contracted by $\bpi \cup \bsigma$ are not counted. Similarly, the contribution of internal crossings of $\boldsymbol{\rho}$ to $\mathrm{cr}(\bpi \cup \bsigma \cup \boldsymbol{\rho})$ is exactly $\mathrm{cr}'(\boldsymbol{\rho})$.

	Now although each $(i,j) \in \boldsymbol{\rho}$ can reduce $\mathrm{sp}(\bpi)$ by $1$ for each $(r,s) \in \bpi$ with $r < i < s$, however, $(i,j)$ then increases the crossing number by $1$. Therefore, the two sums remain the same.
	%follows from the explicit formula for the component of $\xi_q(f_1) \cdots \xi_q(f_n)$ in the $k^{\text{th}}$ chaos.
\end{proof}

% \begin{proposition}
% 	For all $N \in \N$ and all $a,b \in \mfA_N(\mfH)$ (resp.\ $a,b \in \mfA^F_N(\mfH)$)
% 	\begin{equ}
% 		\left[\iota_{-1}(ab)\right] = \left[\iota_{-1}(a) \iota_{-1}(b) \right] \in \CL^2(\cA_F(\mfH), \bomega_F) \; , 
% 	\end{equ}
% 	resp.\
% 	\begin{equ}
% 		\left[\iota_{F}(ab)\right] = \left[\iota_{F}(a) \iota_{F}(b)\right] \in \CL^2(\overline{\cA}_F(\mfH), \bomega_\lambda) \; . 
% 	\end{equ}
% 	%where we have extended $\digamma$ to $\CL^{2,\infty}(\cA_F(\mfH), \omega_F)$ (resp.\ $\CL^{2,\infty}(\overline{\cA}_F(\mfH), \omega_\lambda)$). 
% \end{proposition}
% \begin{proof}
% 	We will only prove the $\iota_{-1}$ case. The reason why 
% 	\begin{equ}
% 		\iota_{-1}(ab) \neq \iota_{-1}(a) \iota_{-1}(b)
% 	\end{equ}
% 	is that contractions in $\cA_F(\mfH)$ produce 
% 	\begin{equ}
% 		\left[ \balpha(\kappa f), \balpha(g)^\dagger \right]_+
% 	\end{equ}
% 	and not $\Braket{\kappa f, g} \bone $. However, 
% 	\begin{equ}
% 		\digamma\left( \left[ \balpha(\kappa f), \balpha(g)^\dagger \right]_+ \right) = \Braket{\kappa f, g} \bone \; 
% 	\end{equ}
% 	whence in follows that they produce equivalent elements of $\CL^2(\cA_F(\mfH), \bomega_F)$ since $\bomega_F \eqdef \omega_F \circ \digamma$. By continuity the result follows.  
% \end{proof}

Finally, in order to describe systems of bosons mixed with fermions or mezdons, we will also consider the tensor product algebra $\mfA(\mfH_1) \wotimes \mfA(\mfH_2)$ with multiplication defined via 
\begin{equ}
	(a \otimes c) (b \otimes d) = (ab) \otimes (cd) 
\end{equ}
and for $m,n \in \N$ and $\bpi \in \CP_m$, $\bsigma \in \CP_{[n]}$ we topologise and complete each subspace
\begin{equ}
	\mfA(\mfH_1)[m,\bpi] \wotimes_\alpha \mfA(\mfH_2)[n, \bsigma]
\end{equ}
with the Hilbert space tensor product, which they inherit from being isomorphic to $\mfH_1^{\wotimes_\alpha (m-2|\bpi|)} \wotimes_\alpha \mfH_2^{\wotimes_\alpha (n-2|\bsigma|)}$. Note that we do not complete along the direct sum over $m,n \in \N$, thus we are still restricting ourselves to finitely many non-zero components. This algebra is again locally $m$-convex with subspaces we denote $\mfA_N(\mfH_1) \wotimes \mfA_M(\mfH_2) $, which we define in the obvious way. 

\begin{proposition}
\label{prop:AlgMorphCont2}
For any $q \in (-1,1]$, the map 
\[
\iota_1 \otimes \iota_q \colon \mfA(\mfH_1) \wotimes \mfA(\mfH_2) \to \cA_B(\mfH_1) \wotimes \cA_q(\mfH_2)\;
\] 
is a continuous algebra morphism, and analogously for $\iota_1 \otimes \iota_F$ and $\iota_1 \otimes \iota_F$ mapping into $L^{\infty-}(\Omega, \mu ; \CL^2(\CA^{\Cl}_F(\mfH_2), \omega_F))$ and $L^{\infty-}(\Omega, \mu ; \CL^2(\CA_F(\mfH_2), \omega_\lambda))$ respectively.
\end{proposition}
\begin{proof}
	This follows analogously to the proof of \cite[Proposition~2.29]{CHP23}, which in turn is essentially a restatement of \cite[Theorem~2.1.9]{Hyt16}. It in particular states that for any map $T \colon L^{q_1}(\Omega, \mu) \to L^{q_2}(\Omega, \mu)$ for any $q_1, q_1 \in [1,\infty)$ and any Hilbert space $\mfh$, $T \otimes \bone_{\mfh}$ extends to a continuous map  $L^{q_1}(\Omega, \mu ; \mfh) \to L^{q_2}(\Omega, \mu ; \mfh)$.

	First, we note that for every $N,M \in \N$ up to an equivalence of norms, we may consider $\mfA_N(\mfH_1) \wotimes \mfA_M(\mfH_2)$ to be a Hilbert space. Furthermore, the codomain of $\iota_1$ is in fact the Hilbert space $L^2(\Omega, \mu)$ and $\iota_1$ is itself a bounded operator between Hilbert spaces. 
	
	Thus, $\iota_1 \otimes \bone_{\mfA_M(\mfH_2)}$ extends to a continuous map of Hilbert spaces. The hypercontractivity of $\xi_1$ and the result of \cite{Hyt16} imply that $\iota_1 \otimes \bone_{\mfA_M(\mfH_2)}$ in fact maps into $L^{\infty-}(\Omega , \mu ;  \mfA_M(\mfH_2))$. Because $\iota_q$ is a continuous linear map between locally convex spaces it follows that $\bone_{\cA_1(\mfH_1)} \otimes \iota_q$ extends to a continuous map 
	\begin{equ}
		L^{\infty-}(\Omega , \mu ;  \mfA_M(\mfH_2)) \longrightarrow L^{\infty-}(\Omega , \mu ;  \cA_q(\mfH_2)) = \cA_B(\mfH_1) \wotimes \cA_q(\mfH_2) \; .
	\end{equ}
	This proves the assertion for $q \in (-1,1]$. The proof in the Clifford and fermionic case is simple, as the target space is already a Hilbert space. 
\end{proof}

%\subsubsection{Fermionic Fock Space Algebra}

% \begin{definition}
% 	Let $\varpi \colon \mfA(\mfH) \to \mfA(\mfH)$ be the linear map that is $0$ on $\mfA(\mfH)[n]$ if $n$ is odd and projects $\mfA(\mfH)[n]$ to
% 	\begin{equ}
% 		\bigoplus_{\bpi \in \CP_{n,n/2}} \mfA(\mfH)[n, \bpi]
% 	\end{equ}
% 	when $n$ is even.

% 	For all $q \in [-1,1]$,
% 	\begin{equ}
% 		\omega_q \circ \iota_q \circ \varpi = \omega_q \circ \iota_q \; .
% 	\end{equ}
% \end{definition}

% \begin{proposition}
% 	The image of $\varpi$ lies in the centre of $\mfA(\mfH)$.
% \end{proposition}

\subsection{Fock Space Algebra $\mapsto$ Regularity Structure and BPHZ Model}  
\label{subsec:FockSpaceAlgtoRegStruc}

Let $N \in \N$ be sufficiently large. Let $\widehat{\cT}_N = (\widehat{\CT}_N , \widehat{G}_N, A)$ be the $\mfA_N(\mfH)$-regularity structure generated by the same set of symbols $\mfL$ and subcritical rule $\cR$ as $\cT_q$. Furthermore, in the fermionic case for $b \in \Gr^{(U)}(\mfH)$, let $\widehat{\cT}^{(F)}_{N;b}$ be the $\mfA^{(F)}_N(b)$-regularity structure  generated by the selfsame rule.

%Furthermore, we assume that both $\cT_q$ and $\widehat{\cT}_N$ have been truncated so that they do not contain trees with more than $N$ negative leaves.

We define the canonical model $\widehat{Z}^N_{(\eps)}$ for $\widehat{\cT}_N$ by setting
\begin{equs}
	\left(\widehat{\Pi}^{N; (\eps)}_x \Xi_{\mfl}\right)(y) \eqdef  \CS^\eps_{\s, 0} \rho(y)  \in \mfH_{\mfl} \subset \mfH = \mfA_N(\mfH)[1,\emptyset] \; ,
\end{equs}
and defining integration as above. The model on root-renormalised trees is defined as follows. First, if $\bpi$ does not consist of pairs, then $\Delta(\mfo_\rho) \equiv 0$. Let $A_i$ be defined as above with $\Pi_x^{(\eps)}$ replaced by $\widehat{\Pi}^{N; (\eps)}_x$. Let $I_i$ denote the index set of $A_i$, and let $I$ be the set obtained by concatenating $I_0$ through $I_{n+1}$ and inserting a further indexing element between every two $A_i$. Thus, $I = [n+m+1]$ where $m = \sum_{i = 0}^{n+1} |I_i|$. Given a contraction $\bsigma$ occurring in the product $A_0 \cdots A_{n+1}$, the corresponding component of the product will be an element of $\mfA_N[n+m+1, \bpi \cup \bsigma]$.  

Analogously, to define $\widehat{Z}^{(F), b, N}_{(\eps)}$ for $\widehat{\cT}_{N;b}^{(F)}$ we simply change 
\begin{equ}
	\left(\widehat{\Pi}^{(F) , b , N ; (\eps)}_{x} \Xi_\mfl \right)(y) \eqdef P_b \CS^\eps_{\mfs, 0} \rho(y) \; . 
\end{equ}

Per constructionem, there is a linear isomorphism $\digamma_q^N$ between the set of scalar trees in $\widehat{\cT}_N$ and $\cT_q$. This map can be extended to all of $\widehat{\cT}_N$ by mapping the algebra decoration $\mfa \in \left(\mfA_N\right)^{\wotimes_\pi n}$ to $\iota_q^{\wotimes n}(\mfa) \in \cA_q^{\wotimes_\pi n}$. The same holds for $\widehat{\cT}^{(F)}_{N; b}$ since $\dim b < \infty$ with the maps denoted by $\digamma^N_{b}$ and $\digamma^N_{F,b}$. 

%Noting this, the canonical model for $\widehat{\cT}_N$ ( $\widetilde{\cT}^{(F)}_N$, $\widehat{\cT}^{(F)}_{N; b}$ resp.) satisfies the following compatibility condition with the model for $\cT_q$ ().

\begin{proposition}
\label{prop:ComDiagAlg}
	Let $\tau \in \widehat{\cT}_N$ be a tree. Then for $M \geqslant N |T| + n$, with $n$ being the number of negative leaves of $\tau$, we have
	\begin{equ}
		\iota_q\Bigl( \widehat{\Pi}_x^{M; (\eps)} \tau \Bigr) = \Pi^{(\eps)}_x \left( \digamma^M_q (\tau) \right) 
	\end{equ}
	for all $q \in (-1,1]$. For $q = -1$, this holds for the models $\widehat{\Pi}^{b, N ; (\eps)}$ and $\widehat{\Pi}^{F, b, N ; (\eps)}$ for all $b \in \Gr(\mfH)$ and  $b \in \Gr^{U}(\mfH)$ respectively.
	% Furthermore, $M$ is

	% and let $n$ be the number of negative leaves of $\tau$. %Furthermore, assume that all algebra elements to the left of all negative leaves and to the right of all negative leaves are scalar multiples of the identity.

	% Let $\widetilde{\tau}$ be the classical version of $\tau$ and let $\widetilde{\mfa}$ be the element of $\cA_q(\mfH)^{\wotimes_\pi (n-1)}$ obtained by multiplying together all factors of $\mfa$ not separated by a negative leaf. Then using the notation of Theorem~\ref{thm:RenDisent}, for all $x \in \R^d$
	% \begin{equs}
	% 	\Pi_x \tau = \sum_{\bpi \in \CP_{[n]}}
	% \end{equs}
\end{proposition}

\begin{proof}
	This is again proven using induction over the tree structure. Most of the steps follow trivially per definitionem using the fact that $\iota_q$ is an algebra morphism when the cut-off $M$ is sufficiently large. We only check the compatibility of the two contraction operations, i.e.\
	\begin{equ}
		\Delta(\mfo_\rho) \circ \iota_q^{\otimes (k+1)} = \iota_q \circ \Delta(\mfo_\rho) \; .
	\end{equ}
	as maps $\mfA_N^{\wotimes_\pi (k+1)} \to \cA_q$. First note that $\mfo_\rho$ is indeed the same on both sides of the equality, since $\digamma^M$ acts as an isomorphism on the level of the tree structure. Now, we only need to compare the factors $q^r$ multiplying each summand. Unpacking the defining Equations~\eqref{eq:RltvContr} and \eqref{eq:AbsContr}, we see that for $A_i \in \mfA_N[n_i, \bsigma_i]$ for $i = 0,\dots, n$ and a contraction $\boldsymbol{\upsilon}$ between $\iota_q(A_1), \dots, \iota_q(A_{n-1})$. Then per definitionem
	\begin{equs}
		\Delta(\mfo_{\rho})\left( \iota_q(A_0), \dots, \iota_q(A_n) \right) &= q^{r} \iota_q(A_0)\Delta_q^{R; \boldsymbol{k}, \boldsymbol{\pi}}\left(A_{1} \otimes \cdots \otimes A_{n-1} \right) \iota_q(A_n) \; .
	\end{equs}
	where $r = \sum_{i = 1}^{n-1} (\crb(\bsigma_i))$ and $\boldsymbol{k} = (1, \dots, 1) \in \N^n$. Letting $\boldsymbol{\ell} = (n_1, \dots, n_{n-1})$ we can split the $\Delta_q$-term up into
	\begin{equs}
		q^r \Delta_q^{R; \boldsymbol{k}, \boldsymbol{\pi}}  & \left(A_{1} \otimes \cdots \otimes A_{n-1} \right) = \sum_{\boldsymbol{\sigma} \in \CP^R_{0, \boldsymbol{\ell}}} q^{\crb(\boldsymbol{\upsilon}, \boldsymbol{\sigma})} \xi_q^{|\ell|, \bsigma}( A_1 \otimes \cdots \otimes A_{n-1}) %= \\
		%& =  \sum_{\boldsymbol{\sigma} \in \CP^R_{0, \boldsymbol{\ell}}} q^{\mathrm{cr}(\boldsymbol{\upsilon}, \boldsymbol{\sigma})+\mathrm{sp}(\boldsymbol{\upsilon}, \boldsymbol{\sigma})} \xi_q^{|\ell|-2|\bsigma|}\Bigl( \bigotimes_{i \in [n-1]\setminus \bsigma}  A_i \Bigr) \prod_{(i,j) \in \bsigma }
	\end{equs}
	where $\boldsymbol{\upsilon} \eqdef \bigcup_{i = 1 }^{n-1} \boldsymbol{\sigma}_i \cup \boldsymbol{\pi} $. Now, each of these summands is exactly the image of the component of $A_1 \cdots A_{n-1}$ in $\mfA_N[n + \sum_i n_i , \boldsymbol{\upsilon} \cup \bsigma]$ under $\iota_q$. This component in turn is the image under $\Delta(\mfo_\rho)$ of the $\mfA_n[\sum_i n_i , \bigcup_{i} \bsigma_i \cup \bsigma]$-component of $A_1 \cdots A_{n-1}$, which proves the assertion.
\end{proof}

With these definitions at hand, we are finally able to define the BPHZ character for the Fock space algebra regularity structure, and thereby the BPHZ character for the mezdonic regularity structure.

\begin{definition}[BPHZ Character]
	The BPHZ character for $\widehat{\cT}_N$ w.r.t.\ $\widehat{Z}^N_{(\eps)}$ is defined inductively by setting
	\begin{equ}
		\ell^{N ; (\eps)}_{\BPHZ}(\tau, \bpi) \eqdef - \left( \widehat{\PPi}^{N;(\eps)} \left( M^\circ_{R_{\ell^{N; (\eps)}_{\BPHZ}}} R'_{\ell^{N; (\eps)}_{\BPHZ}} (\tau) \right)(0)\right)[n, \bpi] \in \mathbb{K}
	\end{equ}
	for $(\tau, \bpi) \in \mfQ$ if $\tau$ consists only of pairs and $\ell^{N ; (\eps)}_{\BPHZ}(\tau, \bpi) = 0$ otherwise. Here $F[n, \bpi]$ denote the component of $F \in \mfA$ in $\mfA_N[n,\bpi]$ and we have also canonically identified $\mfA_N[n,\bpi]$ with $\mathbb{K}$.

	The BPHZ character for $\cT_q$ with $q \in (-1,1]$, w.r.t.\ the canonical model $Z^q_{(\eps)}$ is given by
	\begin{equ}
		\ell^{q;(\eps)}_{\BPHZ}(\tau, \bpi) \eqdef \ell^{N ; (\eps)}_{\BPHZ}(\tau,\bpi)
  	\end{equ}
	where $N \in \N$ is chosen arbitrarily but larger than the number of negative leaves of $\tau$ and \dash since $\tau$ is a classic tree \dash we have identified it with the corresponding tree in $\widehat{\CT}_N^-$.

	For $q = -1$ and $b \in \Gr^{(U)}(\mfH)$ we define the characters $\ell_{\BPHZ}^{b , N ; (\eps)}$, $\ell_{\BPHZ}^{F, b , N ; (\eps)}$  and $\ell_{\BPHZ}^{b ; (\eps)}$, $\ell_{\BPHZ}^{F, b ; (\eps)}$ analogously w.r.t.\ $\widehat{\cT}_{N; b}$ and $\widehat{\cT}^{F}_{N; b}$ respectively.

\end{definition}

\begin{remark}
	In the definition of the extended decoration in Definition~\ref{def:DecTrees}, we allow $\bpi$ to be an arbitrary partition of the negative leaves of $\tau$ whereas only a partition into pairs makes sense in the context of $\mfA_N$.
\end{remark}

\begin{remark}
	We note here that $\ell^{q; (\eps)}_{\BPHZ}(\tau, \bpi)$ captures the numerical value of the counterterms that diverge as one removes the mollification. In particular, they are $q$-independent. The combinatorial factors, such as applying the correct $\Delta_q$-maps, only appear after reconstruction\slash applying the morphism $\digamma_q$. 
\end{remark}

\begin{remark}
	The induction starts with trees $\tau$, s.t.\ $[\tau: \sigma] = \emptyset$ for all $\sigma \in \CT^- \setminus\{\tau\}$, in which case
	\begin{equ}
		\ell^{N;(\eps)}_{\BPHZ}(\tau, \bpi ) = - \left(\widehat{\PPi}^{N;(\eps)} \tau (0) \right)[n , \bpi] \; .
	\end{equ}
\end{remark}

\begin{remark}
	Although it appears as if the BPHZ characters depend on the base point $0$, the fact that we have chosen both the inner products of the individual $\mfH_\mfl$ as well as the integration kernels to be translation invariant guarantees that it is independent.
\end{remark}

Since $\mfA(\mfH)$ is a locally $m$-convex algebra we could directly solve $\cA_q(\mfH)$-valued SPDEs directly in an $\overline{\mfA}(\mfH)$-regularity structure, once we have completed $\mfA(\mfH)$ to $\overline{\mfA}(\mfH)$. This would yield a hierarchy of solutions $(u_N)_{N \in \N}$ with $u_N \in \cC\left([0,T_N]  ; \cD'\left(\R^d ; \mfA_N(\mfH) \right) \right)$ for a monotonously decreasing sequence of stopping times $(T_N)_N \subset \R_{>0}$. These then fit together to a stopped solution
\begin{equ}
	u(t) =\bone_{\left[T_{1},T_0\right]}(t) u_0(t) + \sum_{N = 1}^\infty \bone_{\left[T_{N+1},T_N\right]} \left( u_{N}(t) - u_{N-1}(t)\right) \; . 
\end{equ} 
To transport this equation back to $\cA_q$, we would need to extend $\iota_q$ to the completion (or at least a sufficiently large subspace thereof).%Due to \eqref{eq:FermCont}, $\iota_{-1}$ and $\iota_F$ directly extend to the closures and we will in fact employ this strategy to solve fermionic equations. 

When $|q|<1$, this would be a strictly inferior strategy when dealing with local-in-time solutions, as this would at best give us the existence of an unbounded operator-valued solution. %On the other hand working directly in $\cA_q(\mfH) \subset \CA_q(\mfH)$ allows for much better control. However, for well-behaved nonlinearities that admit energy estimates working in $\overline{\mfA}(\mfH)$ rather than $\cA_q(\mfH)$ will be useful as it allows us to restart the equation from $\CL^2(\CA_q(\mfH), \omega_q)$-initial data rather than just $\cA_q(\mfH)$. We will explore this further in future work. 
In the case $q = 1$, we could in principle topologise $\CA_B(\mfH)$ appropriately so that $\iota_1$ extends to the closure, however, this would be more akin to solutions w.r.t.\ an extended white noise calculus rather than pathwise solutions one typically considers in the context of regularity structures and we will not pursue this further here.

%We now discuss the fermionic case.

% \begin{remark}
% 	Typically, $\bpi$ is a contraction of a set $I \subset [m]$ with $|I| = n$. However, for the purposes of the BPHZ-character $I$ will always be identified with $[n]$.
% \end{remark}

From now on, we will always just write $R$ instead of $R_{\ell_{\BPHZ}}$.

\subsection{Convergence of the BPHZ Model}\label{sec:ConvBPHZ}

We start this section out by noting the following direct consequence of Propositions~\ref{prop:ComDiagAlg}~\&~\ref{prop:NegModBnd}. In particular, the following corollary allows us to abstract away the question of the convergence of the model for each $q \in [-1,1]$ to the convergence of a single model based on the algebra $\mfA_N(\mfH)$.
\begin{corollary}\label{cor:FocktoQbound}
	There exists $N \in \N$ large enough that only depends on the subcritical rule $\cR$, s.t.\ for all $\gamma \in \R$ there exist $k \in \N$ and for all $\K \Subset \R^d$ there exists $\overline\K \Subset \R^d$, s.t.\ for all $q \in (-1,1)$% and $\mfp \in \mfP_q$
	\begin{equ}
		\vvvert Z^{q; R}_{(\eps)} \vvvert_{\gamma ; \mfK} \lesssim_{q} \left( 1+ \bigl\| \widehat{\Pi}^{N ; (\eps), R} \bigr\|_{\widehat{\CT}_N^- ; \overline\mfK} \right)^k
	\end{equ}
	uniformly in $ \eps \in (0,1]$. %Here $\mfP_q = \{\vvvert \bigcdot \vvvert\}$ for $q \in (-1,1)$ and $\mfP_{-1} = \{ \| \bigcdot \|_n \}_{n \in \N}$. 

	In the case $q = -1$, for all $b \in \Gr^{(U)}(\mfH)$ we have the bounds 
	\begin{equ}
		\vvvert Z^{q; R}_{(\eps) ; b } \vvvert_{\gamma ; \mfK} \lesssim_{q} \left( 1+ \bigl\| \widehat{\Pi}^{(F),b, N ; (\eps), R} \bigr\|_{\widehat{\CT}_N^- ; \overline\mfK} \right)^k
	\end{equ}
	
	When considering a random algebra $\cA_B(\mfH') \wotimes \cA_q(\mfH)$ with $q \in (-1,1)$ one has for every $p \in [1,\infty)$
	\begin{equ}
		\left\| \vvvert Z^{q; R}_{(\eps)} \vvvert_{\gamma ; \mfK}  \right\|_{L^p}\lesssim_{q, p} \left( 1+ \bigl\| \widehat{\Pi}^{N ; (\eps), R} \bigr\|_{\widehat{\CT}_N^- ; \overline\mfK} \right)^k \;  
	\end{equ}
	with the analogous bound for $q = -1$.

	Under the same conditions, we also have
	\begin{equs}
		\vvvert Z^{q; R}_{(\eps)} &; Z^{q; R}_{(\eps')} \vvvert_{\gamma ; \mfK} \lesssim_{q} \\
		&\lesssim\left( 1+ \bigl\| \widehat{\Pi}^{N ; (\eps), R} \bigr\|_{\widehat{\CT}_N^- ; \overline\mfK} \vee \bigl\| \widehat{\Pi}^{N ; (\eps'), R} \bigr\|_ {\widehat{\CT}_N^- ; \overline\mfK} \right)^k \bigl\| \widehat{\Pi}^{N ; (\eps), R} - \widehat{\Pi}^{N ; (\eps'), R} \bigr\|_ {\widehat{\CT}_N^- ; \overline\mfK}
	\end{equs}
	uniformly in $\eps, \eps' \in (0,1]$, and analogously for $q=-1$ and random algebras. 
\end{corollary}

\begin{remark}
\label{rem:MixConv}
	We note here that we are actually mixing two modes of convergence when solving SPDEs, while letting the renormalised model converge to its UV limit. On one hand, the approximation of the renormalised driving noises (and thus the model) happens in $\cA_B(\mfH)$ as the above proposition shows. On the other hand, we find solutions to SPDEs driven by fixed noises using pathwise arguments, i.e.\ using convergence almost everywhere. We are thus essentially working within a subspace of the $\CA_B(\mfH') \wotimes \cA_q(\mfH)$-regularity structure. This is why one can only establish continuity w.r.t.\ convergence in measure for the solution map provided by the abstract fixed-point theorem. 
\end{remark}

\begin{proof}
	Let $N$ be large enough such that all trees in $\CT_q^-$ have at most $N$ negative leaves. Let $\tau \in \CT_q^-$ be a tree with $n$ negative leaves with scalar skeleton $\tilde{\tau}$. We identify it with its counterpart in $\widehat{\CT}^-_N$. The definition of $\ell^{q; (\eps)}_{\BPHZ}$ in terms of $\ell^{N ; (\eps)}_{\BPHZ}$ and the definition of the extraction-contraction renormalisation \dash which except through $\ell$ is independent of the chosen algebra $\CA$ \dash implies that
	\begin{equ}
		R_{q} \tau = \digamma^N_q \left( R_N \tau \right) \; ,
	\end{equ}
	where $R_q$ it the root renormalisation map in $\cT_q$ and $R_N$ in $\widehat{\cT}_N$. Since $\digamma_q^M$ is an algebra morphism, it follows that
	\begin{equ}
		M_{R_{q}} \tau = \digamma^N_q \left( M_{R_N} \tau \right) \; .
	\end{equ}
	Thus,
	\begin{equ}
		\Pi^{(\eps), R}_x \tau = \iota_q \left( \widehat{\Pi}_x^{N ; (\eps)}\right)
	\end{equ}
	and we  obtain the bound
	\begin{equ}
		\left\| \Pi^{(\eps), R} \right\|_{\CT^-  ; \K , \mfp} \lesssim \left\| \widehat{\Pi}^{N ; (\eps), R} \right\|_{\CT^-  ; \K } \; .
	\end{equ}
	The assertion now follows from Proposition~\ref{prop:NegModBnd}. The case $q = -1$ follows from the analogous properties of $\digamma^N_{(F), b}$ and the random algebra case follows from the deterministic case using Remark~\ref{rem:MixAlg1}.
\end{proof}

To conclude the existence of the limiting models, we therefore only need to establish the existence of distributions $\widehat{\Pi}^{N;R}_x \tau \in \cD'(\spacetime ; \mfA)$ for every $\tau \in \widehat{\CT}_N^-$. Since for all $\eps \in (0,1]$, $\widehat{\Pi}_x^{N; (\eps), R} \tau \in \mfA[n]$. Thus, we only need to check that the components in $\mfA[n , \bpi]$ for all $\bpi \in \CP_{[n]}$ are well-defined in the limit, satisfy the correct scaling bounds, and converge.

Although we believe it possible to construct a noncommutative analogue of the spectral gap approach of \cite{HS24}, we will instead directly use the main result of \cite{HS24}. For every component of every tree $\tau \in \widehat{\CT}^-_N$ we construct a commutative model that controls it. Since we are considering only finitely many trees with finitely many components, this is enough to conclude the result.

For the remainder of this subsection, we restrict ourselves to the case $q \in (-1,1)$, as the projections $P_b$ generically break translational invariance, which is a key ingredient in the proof of the existence of the BPHZ renormalised model in \cite{HS24}.\footnote{Although we have not thoroughly studied the newly published extension to the non-translation-invariant case \cite{BSS25}, we expect that this still would not be enough, as $P_b$ would generically correspond to a non-local term in the differential operator.} We return to the case $q = - 1$ in Section~\ref{sec:BPHZFermi}.

For this purpose, we need the following observation. Let $\tau$ be a negative tree with $n$ negative leaves, which we will denote simply with $i \in [n]$. The norm of the component of $\widehat{\Pi}_x^{N; (\eps), R} \tau$ in $\mfA[n, \bpi]$, henceforth abreviated $\tau^{\eps}_x[n,\bpi]$, tested against any smooth function $\eta$ is the same as the component in the $(n-2|\bpi|)^{\text{th}}$ chaos of the model of a certain $q=1$, i.e.\ commutative, tree $\widetilde\tau$. This tree is the same as $\digamma^N_1 \tau$, where we, however, assume that the (stochastic) noises $i \in [n]$, corresponding to the negative leaves, are all pairwise independent except if $(i,j) \in \bpi$ for $i,j \in [n]$.

This is done by changing $\mfL^-$ to $\bigsqcup_{i = 1}^n \mfL^-$ thus producing independent noises $\Xi_{\mfl}^1$, \dots, $\Xi_{\mfl}^n$ for every old $\Xi_\mfl$. This changes $\mfH$ to $\mfH^n$ as well as the rule to allow any of the $\Xi_{\mfl}^i$ as successors where it only allowed $\Xi_{\mfl}$ before. Since the new noises have the same regularity as the old noises, the new rule is still subcritical. The decorations of the negative leaves of $\tau$ are then replaced with the independent $\Xi_{\mfl}^i$ unless it is a contracted pair. We note that the noises of disjoint pairs in $\bpi$ are still assumed to be independent in this construction.

To establish this we note the following points:
\begin{itemize}
	\item Let $\tilde\xi_1, \dots, \tilde\xi_k$ be $k$ independent (stochastic) $\mfH$-white noises. The subspace of $L^2(\Omega, \d \mu)$ generated by
	\begin{equ}
		\left\{ \tilde\xi_1(f_1) \cdots \tilde\xi_k(f_k) \, \big| \, f_i \in \mfH \right\}
	\end{equ}
	is isometrically isomorphic to $\mfH^{\wotimes_{\alpha} k }$. This is important since per definitionem we can write $\tau_x^{(\eps)}[n , \bpi ](y)$ exactly as
	\begin{equ}
		\int G_{\tau}\left(y ,x, (z_i)_{i \in [n] \setminus \bpi } \right) \bigotimes_{i \in [n] \setminus  \bpi } \rho^{(\eps)}_{\mfl_i}(z_i) \prod_{i \in [n] \setminus  \bpi }\d z_i  \in \mfH^{\wotimes_\alpha (n - 2|\bpi|)}
	\end{equ}
	for some (possibly singular) integration kernel $G_x$. Here $\rho_{\mfl_i}^{(\eps)}(z)$ is the $\widehat{\Pi}^{N ; (\eps)}$-model of the $i^{\text{th}}$ negative leaf $\Xi_\mfl$ with $\mfl \in \mfL^-\times \N^d$.

	By the above isometry, the Hilbert space norm of $\tau_x^{(\eps)}[n,\bpi]$ is the same as that of $\widetilde{\tau}$ if and only if $G_x$ is the same kernel for $\widetilde{\tau}$.

	\item Let $\sigma$ be a contracted version of $\tau$, i.e.\ it carries an extended decoration but is otherwise the same as $\tau$, and let $\widetilde{\sigma}$ be $\widetilde{\tau}$ equipped with the same extended contraction. Then for every $\eps \in (0,1]$ and $x \in \R^d$, $\widehat{\Pi}_x^{N; (\eps)}\sigma [n, \bpi]$ is mapped exactly to the component of $\Pi_x^{1 ; (\eps)} \widetilde{\sigma}$ in the $(n - 2|\bpi|)^{\text{th}}$ chaos.  This is the case by construction, since the inductive formulae for $\Pi^{1;(\eps)}_x$ and $\widehat{\Pi}^{N ; (\eps)}_x$ produce the same kernel $G_x$, cf.\ Proposition~\ref{prop:TreeModIndep}, and
	\begin{equ}
		\Delta(\mfo) \circ \iota_q^{\otimes (n-2|\bpi|)}  = \iota_q \circ \Delta(\mfo)
	\end{equ}
	by the proof of Proposition~\ref{prop:ComDiagAlg}.

	\item Let $\tilde{\ell}$ denote the renormalisation characters of the regularity structure of $\tilde\tau$. First note $\tilde{\ell}^{1; (\eps)}_{\BPHZ}$ is defined in terms $\tilde{\ell}^{N ; (\eps)}_{\BPHZ}$. Furthermore, $\tilde{\ell}^{N ; (\eps)}_{\BPHZ}(\varsigma, \bsigma)$, by construction, can only depend on contractions $\bsigma \subset \bpi$, on which it agrees again by construction with $\ell^{N ; (\eps)}_{\BPHZ}$.

	Finally, apart from the change of $\ell$ the algebraic action of $R$ and $M_R$ on $\widehat{\CT}_N$ and the regularity structure of $\widetilde{\tau}$ are identical for scalar trees, in particular on $\CT^-$, and only the terms with extended decoration $(\varsigma, \bsigma)$, s.t.\ $\bsigma \subset \bpi$, can contribute to $\mfA[n,\bpi]$ and the $(n-2|\pi|)^{\text{th}}$ chaos. For the former, this follows by definition of $\widehat{\Pi}^{N;(\eps)}$ and by construction for the latter, as any other contraction will lead to $0$ since the corresponding contracted noises are independent.

	\item Combining the last two points it follows that $\tau^{(\eps)}_x[n , \bpi]$ and the component in the $(n - 2|\bpi|)^{\text{th}}$ chaos of $\Pi^{1 ; (\eps), R} \widetilde{\tau}$ must agree.
\end{itemize}

Having established this, the following theorem follows directly from \cite[Proposition~5.2~\&~Theorem~6.9]{HS24}, as the Hilbert space norm of the individual chaos components can be controlled by an arbitrary $L^p$-norm in the probability space using hypercontractivity.

\begin{theorem}
\label{thm:BPHZ}
	We assume that $- |\s| < |\mfl|_{\mfs}$ for all $\mfl \in \mfL_-$, that for two trees $\tau, \bar\tau$, $|\tau|_{\s} = |\bar\tau|_{\s} $ implies that $\tau$ and $\bar\tau$ have the same number of negative leaves, and that if $\tau$ has at least two edges then $|\tau|_{\s} > - \frac{|\s|}{2}$.

	Then the limiting model
	\begin{equ}
		\widehat\Pi^{N;R}_x \tau \eqdef \lim_{\eps \downarrow 0} \widehat\Pi^{N;(\eps),R}_x \tau
	\end{equ}
	exists for all $\tau \in \CT^-$ and we have the bound for all $\eps \in [0,1]$
	\begin{equ}
		\bigl\| \widehat\Pi^{N;(\eps),R} \bigr\|_{\CT^- ; \K} \lesssim_{\K} 1
	\end{equ}
	and there exists $\theta > 0$, s.t.\
	\begin{equ}
		\bigl\| \widehat\Pi^{N;R} - \widehat\Pi^{N;(\eps),R} \bigr\|_{\CT^- ; \K} \lesssim \eps^\theta \; .
	\end{equ}
\end{theorem}

\begin{remark}
	Here we have adopted \cite[Assumption~2.31]{HS24}. Since all the noises we consider are Gaussian, they automatically satisfy the spectral gap condition. Furthermore, our noises automatically satisfy the upper bound $|\mfl|_{\s} < \mathrm{scal}(\mfl) - \frac{|\mfs|}{2}$, since $\mathrm{scal}(\mfl) - \frac{|\mfs|}{2}$ exactly corresponds to our definition of $\alpha_{\mfl}$ and $|\mfl|_{\s} = \alpha_{\mfl}- \kappa$.
\end{remark}

\subsection{BPHZ for Fermions}
\label{sec:BPHZFermi}

In order to overcome the issue pointed out in Remark~\ref{rem:NonHomo} we slightly modify the algebras $\mfA^{(F)}(\mfH)$ and $\mfA_N^{(F)}(\mfH)$. As vector spaces, let 
\begin{equs}
	\widetilde{\mfA}(\mfH)^{\circ} &\eqdef Z\left(\cA^{\Cl}_F(\mfH)\right) \otimes \mfA(\mfH)\\
	\widetilde{\mfA}^F(\mfH)^{\circ} &\eqdef Z\left(\cA_F(\mfH)\right) \otimes \mfA^F(\mfH)
\end{equs} 
and analogously for their cut-off versions. We modify the multiplication in $\widetilde{\mfA}(\mfH)^{\circ}$ and $\widetilde{\mfA}^F(\mfH)^{\circ}$ by replacing the contraction terms $\Braket{f,g}_{\mfH}$ and $\Braket{\kappa U f , g}_{\mfH}$ respectively with 
\begin{equ}
	\left[ \balpha(f), \balpha^\dagger(g) \right]_+ \quad \text{and}\quad \left[ \balpha(\kappa U f), \balpha^\dagger(g) \right]_+ \; . 
\end{equ}
This allows us to extend $\pi_b$ for $b \in \Gr(\mfH)$ to an algebra morphism
\begin{equ}
	\tilde\pi_b \colon \widetilde{\mfA}^{(F)}(\mfH) \longrightarrow \mfA^{(F)}(b) \; 
\end{equ}
by setting it to 
\begin{equ}
	\tilde\pi_b \big|_{Z \otimes \mfA^{(F)}[n,\bpi]} \eqdef \pi_b \otimes P_b^{\wotimes_\alpha(n-2|\bpi|)}
\end{equ}
where $Z$ denotes the corresponding centre. Note that $\pi_b$ maps the centre always to the underlying field and we have tacitly identified $\mathbb{K} \otimes \mfA^{(F)}(b)$ with $\mfA^{(F)}(b)$. For $b \in \Gr(\mfH)$ and $N \in \N$ let $\| \bigcdot \|_{N ; b} \eqdef \| \bigcdot \|_N \circ \tilde{\pi}_b$, which is a submultiplicative seminorm. For $n \in \N$ we set 
\begin{equ}
	\| \bigcdot \|_{N ; n} \eqdef \sup_{b \in \Gamma^{(U)}_n} \| \bigcdot \|_{N ; b} 
\end{equ} 
and we denote by $\widetilde\mfA^{(F)}(\mfH)$ the completion of $\widetilde{\mfA}^{(F)}(\mfH)^{\circ}$ w.r.t.\ $\| \bigcdot \|_{N ; \infty} \eqdef \sup_{n \in \N} \| \bigcdot \|_{N ; n}$.

W.r.t.\ this multiplication $\tilde\iota_{-1} \colon \widetilde\mfA(\mfH)^{\circ} \to  \cA_F^{\Cl}(\mfH)$ and $\tilde\iota_{F} \colon \widetilde\mfA^F(\mfH)^{\circ} \to  \cA_F(\mfH)$, with
\begin{equs}
	\tilde\iota_{-1} \big|_{Z \wotimes_{\pi} \mfA[n,\bpi]} &\eqdef (-1)^{\crb(\bpi)} \CM\circ\left(\bone \otimes \bxi_F^{\diamond (n-2|\bpi|)} \right) \; , \\
	\tilde\iota_{F} \big|_{Z \wotimes_{\pi} \mfA^{F}[n,\bpi]} &\eqdef (-1)^{\crb(\bpi)} \CM \circ \left( \bone \otimes \bPsi^{\diamond (n-2|\bpi|)} \right) \; ,
\end{equs}
are now algebra morphisms. Here $\CM \colon Z(\cA^{(\Cl)}_F) \wotimes_\pi \cA^{(\Cl)}_F \to \cA^{(\Cl)}_F$ denotes the multiplication map. It follows directly from the definitions that 
\begin{equs}
	\pi_b \circ \tilde\iota_{-1} & = \iota_{-1} \circ \tilde{\pi}_b \; ,\\
	\pi_b \circ \tilde\iota_{F} & = \iota_{F} \circ \tilde{\pi}_b  \; . 
\end{equs}
and the same holds for the restriction of the maps to $\widetilde{\mfA}^{(F)}_N(\mfH)^{\circ}$ for all $N \in \N$. In particular, this implies that both $\tilde\iota_{-1}$ and $\tilde\iota_{F}$ are continuous, as for example for all $A \in \widetilde\mfA_N(\mfH)^{\circ}$
\begin{equs}
	\| \tilde{\iota}_{-1}(A) \|_{n} &=\sup_{b \in \Gamma_n} \| \tilde{\iota}_{-1}(A) \|_b = \sup_{b \in \Gamma_n} \left\| \pi_b \left( \tilde{\iota}_{-1}(A)\right) \right\| =  \sup_{b \in \Gamma_n} \left\| \iota_{-1}\left( \tilde\pi_b (A)\right) \right\| \lesssim \\
	& \lesssim_n \sup_{b \in \Gamma_n} \left\| \tilde\pi_b (A) \right\|_{N; n} \leqslant \|A\|_{N; \infty} \; , 
\end{equs}
and thus they extend to the completion. Here we used the $N$-contractive estimates for the extended CAR algebra, \cite[Proposition~2.23~\&~2.29]{CHP23}.

We summarise these facts in the following corollary, the proof of which follows analogously to  Propositions~\ref{prop:AlgMorphCont}~\&~\ref{prop:AlgMorphCont2} together with Remark~\ref{rem:NonHomo}.
\begin{corollary}
	The maps $\tilde\iota_{-1}$ and $\tilde\iota_F$ are continuous algebra morphisms. Furthermore, the maps $\iota_1 \otimes \tilde\iota_{-1} \colon \mfA(\mfH_1) \wotimes \widetilde\mfA(\mfH_2) \to \cA_B(\mfH_1) \wotimes \cA_{F}^{\Cl}(\mfH_2)$  and $\iota_1 \otimes \tilde\iota_{F} \colon \mfA(\mfH_1) \wotimes \widetilde\mfA^F(\mfH_2) \to \cA_B(\mfH_1) \wotimes \cA_{F}^{\Cl}(\mfH_2)$ are also continuous algebra morphisms.
\end{corollary}

Let $\widetilde{\cT}_N$ and $\widetilde{\cT}_N^{F}$ denote respectively the $\widetilde{\mfA}_N$-regularity structure and the $\widetilde{\mfA}_N^{F}$-regularity generated by the same rule $\cR$ as $\cT_{-1}$ and $\cT_F$.

%Move def of $\widetilde{\cT}$????

For these regularity structures, we can define continuous maps $\tilde\digamma^N$ and $\tilde\digamma^N_F$ in the same way we defined $\digamma_q^N$ using $\tilde{\iota}_{-1}$ and $\tilde{\iota}_{F}$.
% These maps are continuous and satisfy the relation
% \begin{equs}[eq:DiGamCom]
% 	\pi_b \circ \widetilde\digamma^N & = \digamma^N_{b} \circ \tilde{\pi}_b \; ,\\
% 	\pi_b \circ \widetilde\digamma^N_{F} & = \digamma^N_{F,b} \circ \tilde{\pi}_b  \; . 
% \end{equs}\martinp{Check whether used}
% where we interpret $\tilde{\pi}_b$ and $\pi_b$ as being $\tilde{\pi}_b^{\wotimes n}$ and $\pi_b^{\wotimes n}$ only applied to the algebra decorations $\widetilde\mfA^{\wotimes_\pi n}$ and $\cA_{F}^{\wotimes_\pi n}$ respectively.
We denote the canonical models for these regularity structures by $\widetilde{Z}^{(F), N}_{(\eps)}$. Analogously to Proposition~\ref{prop:ComDiagAlg}, we have for all $\tau \in \widetilde\cT^{(F)}_N$ and $M$ sufficiently large  
\begin{equs}[eq:ExtModCom]
	\tilde{\iota}_{-1} \left( \widetilde{\Pi}_x^{M ; (\eps)} \tau \right) & = \Pi_x^{\Cl ; (\eps)} \left( \tilde{\digamma}^M \tau \right) \; , \\
	\tilde{\iota}_{F} \left( \widetilde{\Pi}_x^{F, M ; (\eps)} \tau \right) & = \Pi_x^{F ; (\eps)} \left( \tilde{\digamma}_F^M \tau \right) \; . 
\end{equs}

We now define the BPHZ characters for $\widetilde{\cT}_N$ and $\widetilde{\cT}^F_N$ analogously to $\ell^{N ; (\eps)}_{\BPHZ}$ and, in particular, in such a way that they are compatible with those of $\widehat{\cT}_N$ and $\widehat{\cT}^F_N$. 
\begin{definition}[BPHZ Characters for Fermions]
	The BPHZ character for $\widetilde{\cT}_N$ and $\widetilde{\cT}^F_N$ w.r.t.\ $\widetilde{Z}^{N}_{(\eps)}$ and $\widetilde{Z}^{F, N}_{(\eps)}$ are defined inductively by setting
	\begin{equs}
		\ell^{\Cl, N ; (\eps)}_{\BPHZ}(\tau, \bpi) &\eqdef - \left( \widetilde{\PPi}^{N;(\eps)} \left( M^\circ_{R_{\ell^{N; (\eps)}_{\BPHZ}}} R'_{\ell^{N; (\eps)}_{\BPHZ}} (\tau) \right)(0)\right)[n, \bpi] \in Z(\cA^{\Cl}_F) \\
		\ell^{F, N ; (\eps)}_{\BPHZ}(\tau, \bpi) &\eqdef - \left( \widetilde{\PPi}^{F, N;(\eps)} \left( M^\circ_{R_{\ell^{F, N; (\eps)}_{\BPHZ}}} R'_{\ell^{F, N; (\eps)}_{\BPHZ}} (\tau) \right)(0)\right)[n, \bpi] \in Z(\cA_F)
	\end{equs}
	for $(\tau, \bpi) \in \mfQ$ if $\tau$ consists only of pairs and $\ell^{N ; (\eps)}_{\BPHZ}(\tau, \bpi) = 0$ otherwise. %Here $F[n, \bpi]$ denote the component of $F \in \mfA$ in $\mfA_N[n,\bpi]$ and we have also canonically identified $\mfA_N[n,\bpi]$ with $\mathbb{K}$.

	The BPHZ character for $\cT_{-1}$ and $\cT_{F}$, w.r.t.\ the canonical model $Z^{\Cl}_{(\eps)}$ and $Z^{F}_{(\eps)}$ are given by
	\begin{equs}
		\ell^{\Cl;(\eps)}_{\BPHZ}(\tau, \bpi) &\eqdef \ell^{\Cl, N ; (\eps)}_{\BPHZ}(\tau,\bpi)\\
		\ell^{F;(\eps)}_{\BPHZ}(\tau, \bpi) &\eqdef \ell^{F, N ; (\eps)}_{\BPHZ}(\tau,\bpi)
  	\end{equs}
	where $N \in \N$ is chosen to be larger than the number of negative leaves of $\tau$ but is otherwise arbitrary.% and -- since $\tau$ is a classic tree -- we have identified it with the corresponding tree in $\widehat{\CT}_N^-$.
\end{definition}
It follows directly from \eqref{eq:ExtModCom} that the ``global'' characters $\ell^{-1 ; (\eps)}_{\BPHZ}$, $\ell^{F;(\eps)}_{\BPHZ}$ are compatible with the ``local'' characters $\ell^{b ; (\eps)}_{\BPHZ}$, $\ell^{b,F;(\eps)}_{\BPHZ}$ in the following way. 
\begin{proposition}
	For all $(\tau, \bpi) \in \mfO$ and all $b \in \Gr(\mfH)$ and $b \in \Gr^U(\mfH)$ we have
	\begin{equs}
		\pi_b \left( \ell^{\Cl ; (\eps)}_{\BPHZ}(\tau,\bpi) \right)  &=   \ell^{b ; (\eps)}_{\BPHZ}(\tau,\bpi)\;, \\
		\pi_b \left( \ell^{F ; (\eps)}_{\BPHZ}(\tau,\bpi) \right)  &=   \ell^{F, b ; (\eps)}_{\BPHZ}(\tau,\bpi)\;. 
	\end{equs}
\end{proposition}
With these definitions and commutation relations at hand, we obtain the analogue of Corollary~\ref{cor:FocktoQbound} for fermions. 
\begin{corollary}\label{cor:FocktoFermBound}
	There exists $N \in \N$ large enough that only depends on the subcritical rule $\cR$, s.t.\ for all $\gamma \in \R$ there exist $k \in \N$ and for all $\K \Subset \R^d$ there exists $\overline\K \Subset \R^d$, s.t.\ for all $n \in \N$ with $n = \dim b$% and $\mfp \in \mfP_q$
	\begin{equs}
		\vvvert Z^{\Cl; R}_{(\eps)} \vvvert_{\gamma ; \mfK, n} &\lesssim_{n} \left( 1+ \bigl\| \widetilde{\Pi}^{N ; (\eps), R} \bigr\|_{\widetilde{\CT}_N^- ; \overline\mfK} \right)^k \\
		\vvvert Z^{F; R}_{(\eps)} \vvvert_{\gamma ; \mfK, n} &\lesssim_{n} \left( 1+ \bigl\| \widetilde{\Pi}^{F , N ; (\eps), R} \bigr\|_{\widetilde{\CT}_N^- ; \overline\mfK} \right)^k 
	\end{equs}
	uniformly in $ \eps \in (0,1]$. %Here $\mfP_q = \{\vvvert \bigcdot \vvvert\}$ for $q \in (-1,1)$ and $\mfP_{-1} = \{ \| \bigcdot \|_n \}_{n \in \N}$. 
	We also have
	\begin{equs}
		\vvvert Z^{\Cl; R}_{(\eps)} &; Z^{\Cl; R}_{(\eps')} \vvvert_{\gamma ; \mfK, n} \lesssim_{n} \\
		&\lesssim\left( 1+ \bigl\| \widetilde{\Pi}^{N ; (\eps), R} \bigr\|_{\widetilde{\CT}_N^- ; \overline\mfK} \vee \bigl\| \widetilde{\Pi}^{N ; (\eps'), R} \bigr\|_ {\widetilde{\CT}_N^- ; \overline\mfK} \right)^k \\
		& \hspace*{4cm} \bigl\| \widetilde{\Pi}^{N ; (\eps), R} - \widetilde{\Pi}^{N ; (\eps'), R} \bigr\|_ {\widetilde{\CT}_N^- ; \overline\mfK} \\
		\vvvert Z^{F; R}_{(\eps)} &; Z^{F; R}_{(\eps')} \vvvert_{\gamma ; \mfK, n} \lesssim_{n} \\
		&\lesssim\left( 1+ \bigl\| \widetilde{\Pi}^{F, N ; (\eps), R} \bigr\|_{\widetilde{\CT}_N^- ; \overline\mfK} \vee \bigl\| \widetilde{\Pi}^{F, N ; (\eps'), R} \bigr\|_ {\widetilde{\CT}_N^- ; \overline\mfK} \right)^k \\
		& \hspace*{4cm} \bigl\| \widetilde{\Pi}^{F, N ; (\eps), R} - \widetilde{\Pi}^{F, N ; (\eps'), R} \bigr\|_ {\widetilde{\CT}_N^- ; \overline\mfK} 
	\end{equs}
	uniformly in $\eps, \eps' \in (0,1]$.

	The analogous bounds hold for the random algebra case. 
\end{corollary}
We are almost ready to consider the convergence of the $\widetilde \Pi$-models as $\eps \downarrow 0$. 
As we mentioned before, applying $\pi_b$ (or more specifically $\tilde{\pi}_b$), needed to compute the seminorms, breaks the translation invariance of the noise. We shall overcome this by specifying a particular filtration $(\Gamma_n)_n$ defining the topology of $\cA_F$. In particular, we want $\Gamma_n$ to be all subspaces of dimension less than $n+1$ spanned by eigenvectors of the group of translations, i.e.\ plane waves $e^{i\Braket{k,x}}$ with $k \in \N^d$. However, in order to make this rigorous, we need to assume that the domain of the function spaces is fully compact, i.e.\ $d_1 = 0$ in the definition of $\mfH$, see Section~\ref{sec:MixedRegSob}. 
\begin{remark}
	In the context of the current paper, this does not lead to any loss of generality since we are only interested in the short-time existence of evolution equations over compact spatial domains. The resulting domain $[0,T] \times \T^d$ can always be embedded in a torus $\T^{d+1}$. \ajay{Is this imposing that we are periodic in time?}\martinp{Only that the noise is periodic in time.}

	In other cases, one typically uses weighted spaces, which again would allow for plane waves. 
\end{remark}

With this assumption at hand we may set $e_k(x) \eqdef e^{2 \pi i\Braket{k,x}}$ for $k \in \Z^d$ and  
\begin{equ}
	\mfB \eqdef \left\{ \left( \lambda_1 e_{k_1}, \dots, \lambda_{|\mfL^-|} e_{k_{|\mfL^-|}}\right) \, \middle| \, k_i \in \Z^d , \lambda_i \in \C \right\}
\end{equ}
which is an (over)complete set of eigenvectors of the unitary group of translations. Furthermore, let 
\begin{equ}
	\mfB^U \eqdef \left\{ f \in \mfB \, \big| \, \exists \lambda \in \C : U f = \lambda f \right\} \; . 
\end{equ}
\begin{remark}
	$\mfB^U$ is nonempty and still a complete set of eigenvectors since we have assumed that $U$ is both unitary and commutes with the group of translations. 
\end{remark}

For $n \in \N$, let 
\begin{equs}
	\Gamma_n &\eqdef \left\{ \spn\{g_1, \dots, g_n\} \, \big| \, g_i \in \mfB \right\} \; ,\\
	\Gamma_n^U &\eqdef \left\{ b + \kappa b U \, \middle| \, b = \spn\{g_1, \dots, g_n\} \land  \forall i \in [n] :  g_i \in \mfB^U  \right\} \; , 
\end{equs}
where we do not assume that $g_i \neq g_j$ for $i \neq j$. Furthermore, let $\Gamma^{(U)}_\infty \eqdef \bigcup_{n \in \N} \Gamma^{(U)}_n$.
\begin{remark}
	We note that since $U$ satisfies $U^\dagger = - \kappa U \kappa $ we have for all $g \in \mfB^U$
	\begin{equ}
		 U \kappa U g = \bar\lambda U \kappa g = \bar \lambda \kappa \kappa U \kappa g = - \bar \lambda \kappa U^\dagger g = - \bar \lambda \kappa \bar \lambda  g  = - \lambda \kappa \lambda g  = - \lambda \kappa U g \; , 
	\end{equ}
	i.e.\ $\kappa U g \in \mfB^U$ again. In particular 
	\begin{equ}
		b + \kappa U b = \spn\left\{ g_1, \kappa U g_1, \dots, g_n , \kappa U g_n \right\}
	\end{equ}
	it is spanned by elements of $\mfB^U$ and has dimension at most $2 \dim b$.
\end{remark}
With these filtrations defining the topologies of $\cA_F^{\Cl}$ and $\cG_F$, the $\CA_F(b)$-valued noises $\pi_b(\bxi)$ and $\pi_b(\bPsi)$ are still translation invariant for all $b \in \bigcup_{n \in \N} \Gamma_n$ and $b \in \bigcup_{n \in \N} \Gamma_n^U$ respectively.

In order to establish the boundedness and convergence of $\widetilde{Z}^{(F),N ; R}_{(\eps)}$ as $\eps \downarrow 0$  we need to show that we can control 
\begin{equ}
	\tilde\pi_b \left( \widetilde\Pi_x^{(F) , N ; (\eps) , R  } \tau \right)  \quad \text{and} \quad \tilde\pi_b \left( \widetilde\Pi_x^{(F) , N ; (\eps) , R  }\tau  \right) - \tilde\pi_b \left( \widetilde\Pi_x^{(F) , N ; (\eps') , R  } \tau \right) 
\end{equ}
uniformly in $\eps, \eps' \in (0,1]$ and(!) $b \in \Gamma_\infty^{(U)}$. However, any scalar tree $\tau$ we have 
\begin{equ}
	\tilde\pi_b \left( \widetilde\Pi_x^{(F) , N ; (\eps) , R  } \tau \right) = \widehat{\Pi}_x^{(F), b , N ; (\eps), R} \tau 
\end{equ}
where we have identified $\tau$ with the corresponding scalar tree in $\widehat{\cT}_{N; b}^{(F)}$. By translation invariance of $\widetilde\pi_b$ we can conclude the existence and convergence for any fixed $b \in \Gamma_\infty^U$ as in Theorem~\ref{thm:BPHZ}. However, since the spectral gap holds with a constant independent of $b \in \Gamma_\infty$, we may at each step of the induction process we can control the Malliavan derivative not only uniformly w.r.t.\ $\eta \in b$ but also any $\eta \in \bigcup \Gamma_\infty$, which yields the uniform bound. We therefore have the well-definedness of the BPHZ renormalised canonical model in the Fermionic case as well. 
\begin{theorem}
\label{thm:BPHZFerm}
	We assume that $- |\s| < |\mfl|_{\mfs}$ for all $\mfl \in \mfL_-$, that for two trees $\tau, \bar\tau$, $|\tau|_{\s} = |\bar\tau|_{\s} $ implies that $\tau$ and $\bar\tau$ have the same number of negative leaves, and that if $\tau$ has at least two edges then $|\tau|_{\s} > - \frac{|\s|}{2}$.

	Furthermore, we assume that the noncommutative points defining the topologies of $\cA_F^{\Cl}$ and $\cG_F$ are translation invariant. 

	Then the limiting models
	\begin{equs}
		\widetilde\Pi^{N;R}_x \tau &\eqdef \lim_{\eps \downarrow 0} \widetilde\Pi^{N;(\eps),R}_x \tau \\
		\widetilde\Pi^{F, N;R}_x \tau &\eqdef \lim_{\eps \downarrow 0} \widetilde\Pi^{F,N;(\eps),R}_x \tau
	\end{equs}
	exists for all $\tau \in \CT^-$ and $\tau \in \CT^-_F$ respectively, and we have the bound for all $\eps \in [0,1]$
	\begin{equs}
		\bigl\| \widehat\Pi^{N;(\eps),R} \bigr\|_{\CT^- ; \K} &\lesssim_{\K} 1 \\
		\bigl\| \widehat\Pi^{F, N;(\eps),R} \bigr\|_{\CT^- ; \K} &\lesssim_{\K} 1
	\end{equs}
	and there exists $\theta > 0$, s.t.\
	\begin{equs}
		\bigl\| \widehat\Pi^{N;R} - \widehat\Pi^{N;(\eps),R} \bigr\|_{\CT^- ; \K} &\lesssim \eps^\theta \; , \\
		\bigl\| \widehat\Pi^{F, N;R} - \widehat\Pi^{F, N;(\eps),R} \bigr\|_{\CT^- ; \K} &\lesssim \eps^\theta \; .
	\end{equs}
\end{theorem}

\section{Examples}
\label{sec:Examples}

\subsection{Nonlinear SDEs for \TitleEquation{q}{q}-Mezdons}\label{sec:NonLinSDE}

In this section, we shall apply the theory of noncommutative regularity structure to construct local-in-time solutions to nonlinear and multiplicative SDEs driven by $n$ independent $q$-$L^2(\R)$-white noise $\xi_q^i$ for all $q \in (-1,1)$. In particular, we will consider the solution to an SDE of the form
\begin{equ}[eq:qSDE]
	\dot X^i = f^i(X) + \sum_{j = 1}^n g^i_{j}(X) \xi^j_q h^i_{j}(X) \; ,
\end{equ}
where we assume that $f^i$, $g^i_{j}$, and $h^i_{j}$ are real-analytic functions.

To construct the usual regularity structure for an $n$-component rough paths set-up, we choose the dimension $d = 1$, scaling $\s = (1)$, Hilbert space $\mfH = L^2(\R)^n$, $\mfL^- = [n]$, with $|\mfl|_\s = - \frac{1}{2}-$ for all $\mfl \in \mfL^-$, $\mfL^+ = \{\mft\}$ with $|\mft|_{\s} = 1$, and rule $\cR$ given by
\begin{gather*}
	\forall \mfk \in \mfL^- \times \N : \cR(\mfk) = \{\emptyset\}\;, \\
	\forall r \geqslant 1 : \cR( (\mft, r))  = \{\emptyset\}\;, \\
	\cR((\mft, 0)) = \bigl\{ \{(\mft, 0)\} \bigr\} \cup \bigcup_{i = 1}^n \Bigl\{  \{(\mft, 0), (i, 0)\}, \{ (i, 0)\} \Bigr\} \; .
\end{gather*}
This rule is obviously normal, and one can easily check that it is subcritical. Cutting the $\cA_q$-regularity structure off at regularity $\frac{1}{2}$, we see that $A = \left\{- \frac{1}{2}-, 0-, 0, \frac{1}{2}-\right\}$. We will call this regularity structure $\cT^q_{\SDE}$. It is spanned by the trees
\begin{equ}
	\1 \, , \; \<0>_i \eqdef \Xi_i \, , \; \<1>_i \eqdef \CI_{\mft}(\Xi_i) \, , \; \<1_0>_{ij} \eqdef \<1>_i \<0>_j \, , \; \<0_1>_{ij} \eqdef \<0>_i \<1>_j \; ,
\end{equ}
for $i,j \in [n]$. These submodules have the following isomorphism types
\begin{equ}
	\CT[\1] \cong \cA_q \, , \; \CT[\<0>_i] \cong \cA_q^{\wotimes_\pi 2} \, , \; \CT[\<1>_i] \cong \cA_q^{\wotimes_\pi 4} \, , \; \CT[\<1_0>_{ij}] \cong \CT[\<0_1>_{ij}] \cong \cA_q^{\wotimes_\pi 5} \; .
\end{equ}

There are two possibilities of how to implement the integration operation $\int_{s}^t \bigcdot \d r$. Either we start out with the Heavyside kernel $\bar K(t) = \vartheta(t)$ and use \cite[Lemma~5.5]{Hai14} to construct a decomposition $\bar K = K + R$ that shows that $\bar K$ is $1$-regularising. However, as $\vartheta$ is sufficiently regular, we can just directly plug it into the formulae in \eqref{eq:CanMod}.

Taking this more straightforward approach we set for $s,t \in [0,1]$
\begin{equs}
	\Pi^{q;R}_s \left( \<0>_i , a_1 \otimes a_2 \right) & \eqdef a_1 \xi^i_q(\,\bigcdot\,) a_2 \\
	\Pi^{q;R}_s \left(\<1>_i, a_{\otimes [4]} \right)(t) & \eqdef a_1a_2 B^i_q(s ,t ) a_3a_4 \; ,
\end{equs}
where $a_i \in \cA_q$ denote the algebra decorations, and we have set
\begin{equ}
	B^i_q(s,t) \eqdef \sgn(t-s)  \xi^i_q\left( \bone_{[s,t]}\right)
\end{equ}
which is well-defined as $\bone_{[s,t]} \in L^2(\R)$. We only need to define the model of their products $\<1_0>_{ij}$ and  $\<0_1>_{ij}$. If we wish to reproduce a $q$-analogue of It\^{o} theory, this is done by setting
% \begin{equs}
% 	\Pi^q_s \<1_0>_{ij}(t) &\eqdef W^{L,ij}_q(s,t) \eqdef \xi^{\diamond 2,ij}_q \left( \bone_{[s,t]} \otimes \delta\left( \bigcdot - t \right)  \right) \; ,\\
% 	\Pi^q_s \<0_1>_{ij}(t) &\eqdef W^{R,ij}_q(s,t) \eqdef \xi^{\diamond 2,ij}_q \left( \delta\left( \bigcdot - t \right)\otimes \bone_{[s,t]}   \right) \; .
% \end{equs}
\begin{equs}[eq:DifLevA]
	\Pi^{q;R}_s \left( \<1_0>_{ij}, a_{\otimes [5]} \right) (t) & \eqdef a_1 a_2 \mathbb{W}_{q}^{L, ij}\left(a_3 a_4\right)(s,t) a_5 \eqdef\\
	& \eqdef   a_1 a_2 \CM^{(1,1)}_q \left( W^{L,ij}_q(s,t) ; a_3 a_4 \right) a_5 \; ,\\
	\Pi^{q;R}_s \left( \<0_1>_{ij}, a_{\otimes [5]} \right) (t) &\eqdef a_1 \mathbb{W}_{q}^{R, ij}\left(a_2  a_3\right)(s,t) a_4 a_5 \eqdef\\
	&\eqdef a_1 \CM^{(1,1)}_q \left( W^{R,ij}_q(s,t) ; a_2 a_3 \right) a_4 a_5\; .
\end{equs}
with the $L^2(\R)^n \wotimes_\alpha L^2(\R)^n$-valued distributions
\begin{equs}
	t \longmapsto W^{L,ij}_q(s,t)(z_1,z_2) &\eqdef  \bone_{[s,t]}(z_1) e_i \otimes \delta\left( z_2 - t \right) e_j   \; ,\\
	t \longmapsto W^{R,ij}_q(s,t)(z_1,z_2) &\eqdef \delta\left( z_1  - t \right) e_i \otimes \bone_{[s,t]}(z_2) e_j  \; .
\end{equs}

Upon integration, \eqref{eq:DifLevA} produces the left and right L\'{e}vy areae given by
\begin{equs}[eq:LevyArea]
	\mathbb{B}_q^{L , ij}(a)(s,t) &= \CM_q^{(1,1)}\left( F^{L,ij}(s,t) ; a \right) \; , \\
	\mathbb{B}_q^{R , ij}(a)(s,t) &= \CM_q^{(1,1)}\left( F^{R,ij}(s,t) ; a \right) \; ,
\end{equs}
with
\begin{equs}
	F^{L,ij}(s,t)(z_1, z_2) & \eqdef \vartheta(z_2 - z_1) \bone_{[s,t]^2}(z_1,z_2) e_i \otimes e_j \; ,\\
	F^{R,ij}(s,t)(z_1, z_2) & \eqdef \vartheta(z_1 - z_2) \bone_{[s,t]^2}(z_1,z_2) e_i \otimes e_j \; .
\end{equs}
% which satisfies
% \begin{equ}
% 	\frac{d}{dt} \mathbb{B}_q^{ij}(s,t) =\frac{1}{2} W^{L,ij}_q(s,t) + \frac{1}{2} W^{R,ij}_q(s,t) \; .
% \end{equ}
Using the formulae \eqref{eq:CanStGr2}, the canonical model is then completed by setting
\begin{equs}
	\Gamma^{q;R}_{st} \1 & \eqdef \1 \; , \\
	\Gamma^{q;R}_{st} \<0>_i & \eqdef \<0>_{i} \; ,\\
	\Gamma^{q;R}_{st} \<1>_i & \eqdef \<1>_i +  B_q^i(s,t) \1 \; , \\
	\Gamma^{q;R}_{st} \<1_0>_{ij} & \eqdef \<1_0>_{ij} + B_q^i(s,t)  \<0>_j \; , \\
	\Gamma^{q;R}_{st} \<0_1>_{ij} & \eqdef \<0_1>_{ij} + \<0>_i B_q^j(s,t)   \; , \\
\end{equs}
and extending these to incorporate algebra decorations in the obvious way. We note here that $M^+$ acts trivially and thus $\Gamma^{q;R} = \Gamma^q$.

One direct consequence of these definitions is the generalisation of the Chen identity to the $q$-mezdonic case. This reads
\begin{equs}
	\mathbb{B}_q^{L, ij}(s,t) - \mathbb{B}_q^{L, ij}(s,u)-\mathbb{B}_q^{L, ij}(u,t) =  B_q^i(s,u) a B_q^j(u,t) \;, \\
	\mathbb{B}_q^{R, ij}(s,t) - \mathbb{B}_q^{R, ij}(s,u)-\mathbb{B}_q^{R, ij}(u,t) =  B_q^i(u,t) a B_q^j(s,u) \;,
\end{equs}
where we have suppressed the algebra argument on the left-hand side and the product on the right-hand side is the multiplication in $\cA_q$.

That the model is well-defined and satisfies the required bounds follows from the usual arguments. For example, $\vvvert B^i_q(s,t) \vvvert = \| \bone_{[s,t]} \|_{L^2} = \sqrt{|t-s|}$ shows the correct scaling for $\Pi^q \<1>_i$, and, for $\phi \in \cD(\R)$, we see, by using Corollary~\ref{cor:RenMultMap}, that
\begin{equs}
	\left\vvvert \Pi^{q;R}_{s} \left( \<0_1>_{ij}, A \right) (\phi) \right\vvvert^2 & \lesssim \vvvert A \vvvert \int \bone_{s,t}(r) \bone_{s,\bar t}(\bar r) \delta(r - t) \delta(\bar r-\bar t) \phi(t)\phi(\bar t) \d r \d \bar r \d t \d \bar t = \\
	&= \vvvert A \vvvert \int \bone_{s,t}(r) \bone_{s,t}(r)  \phi(t)\phi(t) \d r  \d t =\\
	&= \vvvert A \vvvert \int |t-s| \phi(t)^2 \d t \lesssim_{|\supp \phi|} \| \phi \|_{\cC}^2
\end{equs}
which shows that it is a well-defined distribution. Here $A \in \cA_q^{\wotimes_\pi 5}$. Similarly, by testing against $\CS^\lambda_{\s, t} \phi$, we immediately find
\begin{equs}
	\left\vvvert \Pi^{q;R}_{t} \left( \<0_1>_{ij} , A \right) (\CS^\lambda_{\s, t} \phi) \right\vvvert  \lesssim 1 \; .
\end{equs}

Finally, we want to note that this model is constructed according to the BPHZ procedure, with the only caveat being that $\PPi \<1>_i$ does not exist since $\vartheta$ is not integrable. However, a direct calculation shows that the BPHZ character is well-defined when viewing $\vartheta$ as $\lim_{T \to \infty} \bone_{[0,T]}$. Specifically, for all $i \in [n]$
\begin{equs}
	\ell^{q,(\eps)}_{\BPHZ}\left(\<1_0>_{ii}\right) & = \lim_{T \to \infty }\omega_q \left( \int\limits_0^T \xi^{i,(\eps)}_q(s) \d s \, \xi^{i,(\eps)}_q(0)  \right)  =\\
	&= \int\limits_0^\infty \int \rho_{(\eps)} (s-z) \rho_{(\eps)} (z) \d z \d s = \frac{1}{2} \; ,
\end{equs}
and $\ell^{q,(\eps)}_{\BPHZ}\left(\<0_1>_{ii}\right)  = \frac{1}{2}$. A similar calculation shows that, for the $0^{\text{th}}$-order chaos component of $\Pi_s^{(\eps)} \<1_0>_{ij}(t)$, we have
\begin{equ}
	\lim_{\eps \downarrow 0} \omega_q\left( B^{i,(\eps)}_{q}(s,t) \xi^{i,(\eps)}_q(t) \right) = \frac{\delta_{ij}}{2} \; ,
\end{equ}
and thus
\begin{equs}
	\lim_{\eps \downarrow 0} \Pi^{q,(\eps) ; R}_{s} (\<1_0>_{ij}, a_1 \otimes \cdots \otimes a_5) &= \lim_{\eps \downarrow 0} a_1 a_2 B^{i,(\eps)}_{q}(s,\bigcdot ) a_3 a_4 \xi^{j,(\eps)}_q(\bigcdot) a_5 - \\
	& \qquad \qquad - \frac{\delta_{ij}}{2} \Delta_q^{R ; (1,1),(1,2)}\left( a_1a_2, a_3 a_4, a_5 \right) =  \\
	& =  \Delta_q^{R ; (1,1)}\left( W_{q}^{L,ij}(s,\bigcdot) ; a_1a_2, a_3 a_4, a_5 \right) = \\
	&= \Pi^{q ; R}_s(\<1_0>_{ij}, a_1 \otimes \cdots \otimes a_5) \; .
\end{equs}
However, we note that we cancel the $0^{\text{th}}$ chaos component exactly only if $2 \eps \leqslant |t-s|$.

Having established the regularity structure $\cT^q_{\SDE}$, we can solve \eqref{eq:qSDE} as an abstract fixed-point problem in it. According to Theorem~\ref{thm:FncLift} and its extension to modelled distributions with singularities, we can lift the functions $f_i$, $g_j^i$, and $h_j^i$ to maps $\hat{f}^i, \hat{g}_j^i, \hat{h}_j^i \colon \left(\cD_H^{\gamma,0}(V)\right)^n \to \cD_H^{\gamma,0}(V)$ for all $\gamma > 0$. Here $V$ is the sector generated by $\1$ and $\<1>_j$ for $j \in [n]$, and $H$ is the zero-dimensional hyperplane $\{0\}$. We interpret $\phi = (\phi^1, \dots, \phi^n)$ with $\phi^i \in \cD_H^{\gamma,0}(V)$ as an element of $\left(\cD_H^{\gamma,0}(V)\right)^n$.

We can thus define the nonlinearities $F^i \colon \left(\cD_H^{\gamma,0}(V)\right)^n \to \cD_H^{\gamma,0}$
\begin{equ}
	F^i(\phi) = \hat{f}^i \circ \phi + \sum_{j = 1}^n \left( \hat{g}^i_{j} \circ \phi  \right) \<0>_j \bigl( \hat{h}^i_{j} \circ \phi  \bigr)
\end{equ}
where we view $\<0>_j$ as a constant element in $\cD_H^{\infty,\infty}$. Then, a solution to \eqref{eq:qSDE} with initial condition $X(0) = X_0 \in \cA_q^n$ is a fixed-point of
\begin{equ}[eq:SDEFP]
	U_{(\eps)}^i = X_0^i \1  + \CK_{\gamma} \bR^+ F^i\left( U_{(\eps)} \right)
\end{equ}
in $\cD^{\gamma,0}_{H}(V)$ where $\gamma \in \left( \frac{1}{2},1\right)$ and $\CK_\gamma$ is defined as in Definition~\ref{def:CompKern}.

A few comments are in order here. First, $\CK_{\gamma}$ is well-defined since we can define $\CJ(t)$ on the full regularity structure essentially by hand, setting $\CJ(t) \bR^+ \<0>_i = B_q^i(0,t) \eqqcolon B_q^i(t)$ and
\begin{equ}
	\CJ(t) \bR^+  \left( \<1_0>_{ij} , a_1 \otimes \cdots \otimes a_5 \right) = a_1 a_2 \mathbb{W}^{L,ij}_q\left(a_3 a_4\right)(0,t) a_5 \1
\end{equ}
and analogously for $\<0_1>_{ij}$. Here, it is necessary to insert $\bR^+$ because the support of $\vartheta$ is not compact. Furthermore, $\CN_{\gamma}$ is well-defined as soon as $\gamma > \frac{1}{2}$. Secondly, the right-hand side of \eqref{eq:SDEFP} maps $\left(\cD^{\gamma,0}_H (V)\right)^n$ into itself. This holds because $\CK_\gamma$ maps $\cD_H^{\gamma,0}$ into $\cD^{\gamma,0}_H(V)$.

Applying the abstract fixed-point Theorem~\ref{thm:AbsFPM} to \eqref{eq:SDEFP} and its mollified variant with $\CK^{(\eps)}_\gamma$, we obtain a family of solutions $U_{(\eps)} \in \left(\cD^{\gamma,0}_H(V) \right)^n $ for $\eps \in [0,1]$ which can be written as
\begin{equ}
	U_{(\eps)}^i = Y_{(\eps)}^i \1+ \left( \<1>_i , Y'^{i}_{(\eps)} \right)
\end{equ}
where $Y_{(\eps)}^i \colon \R \to \cA_q$ and $Y'^{i}_{(\eps)} \colon \R \to \cA_q^{\wotimes_\pi 2}$. Here we have already performed the multiplication of the left two and right two $\cA_q$ factors in the $\<1>$ terms. 
% By a very serious abuse of notation of we shall denote $Y'^{(\eps)}_{i} ``=" Y^{L, (\eps)}_i \otimes Y^{R, (\eps)}_i$ which is a shorthand for
% \begin{equ}
% 	Y'^{(\eps)}_{i} = \sum_{k} Y^{L, (\eps)}_{i,k} \otimes Y^{R, (\eps)}_{i,k} \; .
% \end{equ}
To obtain the final result, we take a look at the reconstruction of the equation and the solution at positive $\eps > 0$. Let $\CR^{(\eps)}$ denote the reconstruction w.r.t.\ $\Pi^{(\eps) , R}$, then $\CR^{(\eps)} U^{(\eps)} = Y^{(\eps)}$, and $Y^{(\eps)}$ satisfies
\begin{equs}\label{eq:RenSDE}
	\partial_t Y^{i,(\eps)} & = f^i(Y^{(\eps)})  + \sum_{j = 1}^n  g^i_j(Y^{(\eps)}) \xi_q^{j , (\eps)} h^i_j(Y^{(\eps)}) -\\
	& \qquad - \frac{1}{2} \sum_{j = 1}^n \Bigl( \left( D_{j,q}^R g^i_j(Y^{(\eps)})[Y'^j_{(\eps)}] \right)  h^i_j(Y^{(\eps)})  +  g^i_j(Y^{(\eps)}) \left( D_{j,q}^L h^i_j(Y^{(\eps)})[Y'^j_{(\eps)}] \right) \Bigr)
\end{equs}
where the left $D^L_{i,q}$ and right $D^R_{i,q}$ derivatives are defined on $\llbracket A_{\otimes [m+1]} , X^k \rrbracket$ for $A_i \in \cA_q$ and $k \in [n]^m$ by
\begin{equs}
	D^R_{i,q}\llbracket A_{\otimes [m+1]} ; X^k \rrbracket[B_1 \otimes B_2] &= \sum_{\ell  = 1}^m A_{1} X_{k_1} \cdots X_{k_{\ell -1 }} A_\ell B_1 \delta_{i, k_\ell} \\
	& \qquad\qquad \Delta_q \left( B_2 A_{\ell+1} X_{k_{\ell+1}} \cdots A_m X_{k_m} A_{m+1} \right) \; , \\
	D^L_{i,q}\llbracket A_{\otimes [m+1]} ; X^k \rrbracket[B_1 \otimes B_2] &= \sum_{\ell  = 1}^m \Delta_{q} \left( A_{1} X_{k_1} \cdots X_{k_{\ell -1 }} A_\ell B_1 \right) \delta_{i, k_\ell} \\
	& \qquad\qquad  B_2 A_{\ell+1} X_{k_{\ell+1}} \cdots A_m X_{k_m} A_{m+1} \; ,
\end{equs}
which we extend by linearity and continuity to all functions in $\cC^\omega(\cA_q^n ; \cA_q)$ and all derivative directions $\cA_q^{\wotimes_\pi 2}$. By the usual arguments, the radii of convergence of $D^{R/L}_{i,q}f[A]$ are still infinite as the operator norm of $\Delta_q$ is $1$.
\begin{remark}
	We note that from the algebraic point of view, $D^R_q$ and $D^L_q$ are extensions of the right and left supercommuting derivatives that one usually defines on superfunctions.
\end{remark}
% Finally, $C_\eps$ is the usual renormalising constant, i.e.\
% \begin{equ}
% 	C_\eps(t) \eqdef \omega_q \left( \xi_q^{i,(\eps)}(t) B_{q}^{i, (\eps)}(t) \right) = \int\limits_0^t \int \rho_{(\eps)}(s-t-z) \rho_{(\eps)}(z) \d z \d s
% \end{equ}
% where $\rho_{(\eps)} = \CS^\eps_{\s, 0} \rho$ for the mollifier $\rho$. A direct computation shows that $C_\eps(t) \in [0,\frac{1}{2}]$ and in fact for $t \geqslant \eps$ we have $C_\eps(t) = \frac{1}{2}$.

The counterterm in \eqref{eq:RenSDE} is the only part of $\<0_1>_{ij}$ and $\<1_0>_{ij}$ that is not killed by the reconstruction. Explicitly, we can write the component of $F^i(U_{(\eps)})$ in the subspace spanned by $\<0_1>_{ij}$ as
\begin{equ}
	\sum_{\ell = 1}^n \hat g_i^\ell(Y_{(\eps)}) \<0>_i  D_j \hat h_i^\ell(Y_{(\eps)}) \left[ \left( \<1>_j , Y'^{j}_{(\eps)} \right) \right] \; .
\end{equ}
Here, the derivative $D_j$ is the usual derivative w.r.t.\ $X_j$ and one obtains this formula by plugging $U_{(\eps)}$ into $\hat h^i_\ell$ and only keeping monomials, which contain exactly one occurrence of $\<1>_j$. $M_R$ acts non-trivially if and only if $i = j$, and it produces the contraction between those two noise edges. The terms that get plugged into $\Delta_q$ by $D^L_{i,q}$ are exactly the algebra decorations located in between $\<0>_i$ and $\<1>_i$. Since the term that has been contracted contains $\Pi_t \<1>_j (t) = B^j_q(t,t) = 0$ after reconstruction, it gets mapped to zero, and only the counterterm remains.

The above discussion culminates in the following theorem.
\begin{theorem}
	Let $q \in (-1,1)$, $n \in \N$, $f^i, g^i_j, h^i_j \in \cC^\omega(\R^n ; \R)$ for $i,j \in [n]$. Let $\xi^i_q$ for $i \in [n]$ be pairwise independent $q$-$L^2(\R)$-white noises in $\cA_q \eqdef \cA_q(L^2(\R)^n)$, and let $X_0 \in \cA_q^n$.

	For every sequence of $\xi^{i, (\eps)}_q \subset \cD(\R ; \cA^{(1)}_q)$ for $\eps \in (0,1]$ of mollifications of $\xi^i_q$ converging in $\CC^{-\frac{1}{2}-}(\R ; \cA_q)$, there exist a $T> 0$, a unique sequence of functions $Y_{(\eps)} \in  \CC^{\frac{1}{2}-} \left( [0,T) ; \cA_q^n\right)$ and $Y'_{(\eps)} \in  \CC^{\frac{1}{2}-} \left( [0,T) ; \left(\cA_q^{\wotimes_\pi 2} \right)^{n}\right)$ for all $\eps \in [0,1]$, converging as $\eps \downarrow 0$, s.t.\ for all $\eps \in [0,1]$ and all $s,t \in [0,T)$
	\begin{equ}
		\left\vvvert Y_{(\eps)}(t) - Y_{(\eps)}(s) - \Delta_q^{R; (1), \emptyset} (B^{(\eps)}_q(s,t) ; Y'^{(\eps)}(s) )  \right\vvvert_{\cA_q^n} \lesssim |s-t|^{1-}
	\end{equ}
	and, for all $\eps \in (0,1]$, $(Y_{(\eps)}, Y_{(\eps)}')$ satisfies on $[0,T)$
	\begin{equs}\label{eq:RenSDEthm}
		\partial_t Y^{i,(\eps)} & = f^i(Y^{(\eps)})  + \sum_{j = 1}^n  g^i_j(Y^{(\eps)}) \xi_q^{j , (\eps)} h^i_j(Y^{(\eps)}) -\\
		& \qquad - \sum_{j = 1}^n  C_\eps^j \Bigl( \left( D_{j,q}^R g^i_j(Y^{(\eps)})[Y'^j_{(\eps)}] \right)  h^i_j(Y^{(\eps)})  +  g^i_j(Y^{(\eps)}) \left( D_{j,q}^L h^i_j(Y^{(\eps)})[Y'^j_{(\eps)}] \right) \Bigr)
	\end{equs}
	where
	\begin{equ}
		C^j_{\eps} \eqdef \lim_{T \to \infty} \omega_q \left( B^{j,(\eps)}_q(0,T) \xi_q^{j,(\eps)}(0) \right)
	\end{equ}
	as well as $Y_{(\eps)}(0) = X_0$.

\end{theorem}

\begin{remark}[Comparison with \cite{BS98}~\&~\cite{DS18}]
\label{rem:Comparison}
Our present construction in $\cA_q$ for $q=0$ is strictly weaker than the one in \cite{BS98}, in particular their version of the Burkholder-Gundy inequality. If one does not exploit the additional data provided by a filtration, it is \dash to the best of the authors' knowledge \dash impossible to control the operator norm
	of
	\begin{equ}
		\int\limits_0^t \Delta_0^{R ; (1) , \emptyset} \left(  d B_0(0,s) ; U_s \right) \d s
	\end{equ}
	for a function $s \mapsto U_{s} \in \CA_0 \wotimes_{C^*} \CA_0$ or just $\CA_0 \wotimes_{\pi} \CA_0$ without having to control the components of the chaos decomposition of $U_s$ separately but uniformly. This, in turn, necessitates working with a stronger topology such as $\cA_0$, and this requires us to consider only analytic nonlinearities, because without the $C^*$-identity we only have access to holomorphic functional calculus. For a more general theory of It\^o-type SDEs in $\CA_0$, cf.\ \cite[Section~3]{Kar11}, and of Stratonovich-type SDE in $\CA_q$ for $q \in[0,1)$, see \cite{DS18}.

	However, our topological approach is more robust in general \dash we can use a much larger class of approximations. For example, we  reproduce directly both the It\^o and Stratonovich integrals for all $q \in (-1,1)$. Furthermore, our method does not require the special integration by parts for formula, only available for $q = 0$ cf.\ \cite{BS98}, or a similar exact identity coming from a Wong-Zaka\"{i}-type interpolation for $q \geqslant 0$ cf.\ \cite{DS18}, to compute the $L^\infty$-type bounds. For this reason, we can solve fully nonlinear SDEs for all $q \in (-1,1)$ which require an $L^\infty$-type bound for the Burkholder-Gundy inequality. In constrast, \cite{DS18} were only able to achieve a $\CL^2$-bound for general $\CA_q$ integrands.

\end{remark}

As a final application of $\CA$-regularity structures to SDEs, we show that $F(B_t)$ satisfies an It\^o formula for arbitrary $F \in \cC^\omega(\CA^n ; \CA)$. For the sake of generality, we replace $\ell^{q; (\eps)}_{\BPHZ}$ with an arbitrary character $\ell^{q;(\eps)}$ whose values we assume to converge to a set of finite constants $C^q_{ij}$ as $\eps \downarrow 0$.

The lift of the right-hand side of the following equation
\begin{equ}[eq:ItoTriv]
	\partial_t F( B_{q}^{(\eps)} (t) ) = \sum_{i = 1}^n D_i F(B_{q}^{(\eps)} (t))\left[ \xi_q^{i , (\eps)} \right]
\end{equ}
to the regularity structure reads
\begin{equ}
	\sum_{i = 1}^n \widehat{D_i F} \left( \<1> + B_q^{(\eps)}(t) \1 \right) \left[ \<0>_i \right] \; .
\end{equ}
Here, if $F = \llbracket A_1, \dots, A_{m+1} ; X^k \rrbracket$, $\widehat{D_i F}$ should be interpreted as
\begin{equs}
	{}&\sum_{\substack{j, r = 1 \\ j > r }}^{m+1} \delta_{i,k_j} A_1  B_q^{k_1 , (\eps)}(t)  A_2  \cdots A_{r} \<1>_{k_r} A_{r+1}\cdots A_{j} \<0>_i A_{j+1} \cdots\\
	{}& \qquad \quad  \cdots A_{m}  B_q^{k_m , (\eps)}(t)   A_{m+1} + \\
	{}&+\sum_{\substack{j, r = 1 \\ j < r}}^{m+1} \delta_{i,k_j} A_1  B_q^{k_1 , (\eps)}(t)  A_2  \cdots A_{j} \<0>_i A_{j+1}\cdots A_{r} \<1>_{k_r} A_{r+1} \cdots\\
	{}& \qquad \quad \cdots A_{m}  B_q^{k_m , (\eps)}(t)   A_{m+1} + \\
	{}&  + \sum_{j}^{m+1} \delta_{i,k_j} A_1  B_q^{k_1 , (\eps)}(t)  A_2  \cdots A_{j} \<0>_i A_{j+1} \cdots A_{m}  B_q^{k_m , (\eps)}(t)   A_{m+1}
\end{equs}
which extends linearly to all $F \in \cC^\omega$.

$M_R$ acts trivially on the third summand but adds a contraction $\<1>_{k_r}$ and $\<0>_i$ in the first two summands and multiplies them by $\ell^{q;(\eps)}\left(\<1_0>_{k_ri}\right)$ and $\ell^{q;(\eps)}\left(\<0_1>_{ik_r}\right)$ respectively. After reconstruction, the correction term to \eqref{eq:ItoTriv} is exactly
\begin{equ}
	\sum_{i, j = 1}^n \left( \ell^{q;(\eps)}\left(\<1_0>_{ij}\right)+ \ell^{q;(\eps)}\left(\<0_1>_{ij}\right) \right) \Delta_q^{R; (1,1) , (1,2)} \left( D_i D_j F\left(B_q^{(\eps)}(t) \right) \right)
\end{equ}
for all $\eps \in [0,1]$. Here we interpret $D_i D_j F$ as a map $\cA_q \to \cA_q^{\wotimes_\pi 3}$, with
\begin{equs}[eq:NonComDer]
	D_i D_j F(X) = \sum_{\substack{r, s = 1 \\ r \neq s}}^n \delta_{i, k_r} \delta_{i, k_s} & A_1 X_{k_1} A_2 \cdots A_{r} \otimes A_{r+1} \cdots \\
	& \qquad \cdots  A_s \otimes A_{s+1} \cdots A_{m} X_{k_m} A_{m+1} \; .
\end{equs}
In particular, this means that we have established the It\^o formula for arbitrary $q \in (-1,1)$:
\begin{equs}[eq:Ito2]
	\partial_t F( B_{q} (t) ) &= \sum_{i = 1}^n D_i F(B_{q} (t))\left[ \xi_q^{i} \right] + \\
	& \qquad +\sum_{i, j = 1}^n \left( C^q_{ij}+C^q_{ji} \right) \Delta_q^{R; (1,1) , (1,2)} \left( D_i D_j F\left(B_q(t) \right) \right) \; .
\end{equs}
In the It\^o calculus case $C^q_{ij} = \frac{\delta_{ij}}{2}$ whereas in the Stratonovich case $C^q_{ij} = 0$. This extends \cite[Corollary~4.9]{DS18} beyond $q \in [0,1)$ for analytic $F$.
\begin{remark}
	We note that in the notation of \cite{DS18} $\Delta_q^{R; (1,1),(1,2)}$ is exactly $\mathrm{Id} \times \Gamma_q \times \mathrm{Id}$.  As with the previous result, the existence of a filtration was not required for our proof.
\end{remark}

\subsection{The Dynamical \TitleEquation{\Phi^4_3}{p43}-Equation for \TitleEquation{q}{q}-Mezdons}
\label{sec:Phi43Mez}

In this section, we shall show how to apply the theory of noncommutative regularity structures presented so far to solve the more singular cousin of the $\Phi^4_2$ equation, solved in Section~\ref{sec:Phi42}, the $\Phi^4_3$ equation on the torus. Specifically, we seek a solution in $\cD'((0,T) \times \T^3 ; \cA_q)$,  to the equation
\begin{equ}[eq:NonComPhi43]
	(\partial_t - \Delta + m^2) \phi = - \phi^3 + \xi_q \; ,
\end{equ}
with $\xi_q$ being the usual $q$-$L^2(\R^4)$-white noise.

Throughout this section, we shall assume again that $q \in (-1,1)$ and that the Hilbert space $\mfH \eqdef L^2(\R^4)$. Furthermore, we set the scaling to $\s = (2,1,1,1)$, and we let $\bar K \in \cC^\infty(\R^4 \setminus \{0\})$ denote the kernel of $(\partial_t -\Delta + m^2)^{-1}$ on $\R^{1+3}$ with decomposition $\bar K = K + R$. We also fix $\kappa > 0$ sufficiently small.

The required regularity structure is built by setting $\mfL^{-}\eqdef\{\mfl\}$ with $|\mfl|_{\s} \eqdef
 - \frac{5}{2}-\kappa$, $\mfL^+ =\{ \mft \}$ with $|\mft|_{\s} = 2$, and rule $\cR$ given by
\begin{gather*}
	\forall \mfk \in \mfL^- \times \N : \cR(\mfk) = \{\emptyset\} \;, \\
	\forall r \geqslant 1 : \cR( (\mft, r))  = \{\emptyset\}\;, \\
	\cR((\mft, 0)) = \bigl\{ \{(\mft, 0), (\mft, 0) , (\mft, 0)\}, \{(\mft, 0), (\mft, 0) \}, \{(\mft, 0) \}, \{(\mfl, 0) \} , \emptyset  \bigr\} \; .
\end{gather*}
The rule is exactly the minimal normal rule, s.t.\ $\CI_{\mft} \left( U^3 \right)$ is in the regularity structure if $U$ is.  

Cutting the regularity structure off at regularity $1+$ and restricting to a sufficiently large sector, we are left with
\begin{equs}
	A &= \biggl\{ -\frac{5}{2}- \kappa, -\frac{3}{2}-3\kappa, -1-2\kappa,-\frac{1}{2}-5\kappa,\\
	& \qquad \qquad -\frac{1}{2}-\kappa,-4\kappa , -2\kappa , 0, \frac{1}{2}-3\kappa, 1-2\kappa , 1\biggr\}
\end{equs}
spanned by the trees
\begin{equs}
	{} & \1 \, , \;  X_\mu \, , \;  \<0> \, , \;  \<1>\, , \; \<2> \, , \; \<3> \, , \;  \<2> X_\mu \, , \;  \<13_1> \, , \; \<13_2> \, , \; \<23_1>\, , \; \<23_2>\, , \; \\
	{} & \<23_3> \, , \; \<12_1> \, , \; \<12_2> \, , \; \<02> \, , \; \; \<03> \, , \; \<11_1> \, , \; \<11_2> \, , \; \<21_1> \, , \; \<21_2> \, , \; \<21_3> \, , \; \<22_1> \, , \; \<22_2> \, , \; \<22_3> \, .
\end{equs}
where $\<0> = \Xi_{\mfl}$.

Unlike the commutative case, we have to consider trees such as $\<23_1>$ and $\<23_3>$ as separate objects. The model and its BPHZ renormalisation are then constructed as in Section~\ref{sec:CanMod}, implementing integration against $K$. We will not go through the calculations in detail, but we will provide an example. We note here that the calculations in \cite[Section~10.4]{Hai14} can, in fact, be used to directly show that all the Fock space components of all the trees are well-defined without having to rely on Theorem~\ref{thm:BPHZ}.

We shall exemplify this using $\<23_2>$. First, we need to compute the action of $M_R = M^\circ R$ on $\<23_2>$. Starting with $R$, there are two negative trees within $\<23_2>$  on which $\ell_{\BPHZ}^{q ; (\eps)}$ is non-zero, $\<2>$ and $\<22_2>$. The other two trees in the support of $\ell_{\BPHZ}^{q ; (\eps)}$ are $\<22_1>$ and $\<22_3>$ which cannot be embedded into $\<23_2>$. Looking at the embedding
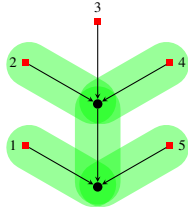
\begin{equ}
	\scalebox{0.5}{
	\begin{tikzpicture}[
   		% --- Styles ---
	    red_square/.style={rectangle, fill=red, inner sep=2.5pt},
    	black_dot/.style={circle, fill=black, inner sep=2.5pt},
    	arrow/.style={->, >=stealth, thick},
    	every edge quotes/.style={font=\small, sloped, fill=white, inner sep=1.5pt},
    	highlight/.style={green, opacity=0.4, line cap=round, line join=round, line width=1.2cm},
    	% The 'spring' style is replaced with a simple 'purple_line' style.
    	purple_line/.style={draw=purple, very thick}
	]
		% --- 1. Define Node positions ---
		%(This section is unchanged)
		\node[black_dot, label=below: ]   (1) at (0,0) {};
		\node[red_square, label=left: 1]     (3) at (-1.9, 1.1) {};
		\node[black_dot, label=right: ]    (5) at (0, 2.2) {};
		\node[red_square, label=left: 2]     (7) at (-1.9, 3.3) {};
		\node[red_square, label=above: 3]    (9) at (0, 4.4) {};
		\node[red_square, label=right: 4]    (11) at (1.9, 3.3) {};
		\node[red_square, label=right: 5]    (13) at (1.9, 1.1) {};

		% --- 2. Draw Highlights on the Background Layer ---
		% (This section is unchanged)
		\begin{pgfonlayer}{background}
		    \draw[highlight] (7) -- (5);
		    \draw[highlight] (11) -- (5);
		    \draw[highlight] (3) -- (1);
		    \draw[highlight] (13) -- (1);
		    \draw[highlight] (5) -- (1);
		\end{pgfonlayer}

		% --- 3. Draw Original Edges ---
		% (This section is unchanged)
		\draw[arrow] (3) edge  (1);
		\draw[arrow] (13) edge (1);
		\draw[arrow] (7) edge (5);
		\draw[arrow] (9) edge (5);
		\draw[arrow] (11) edge (5);
		\draw[arrow] (5) edge (1);

	\end{tikzpicture}
	}
\end{equ}
of $\<22_2>$ into $\<23_2>$ there are three possible contractions
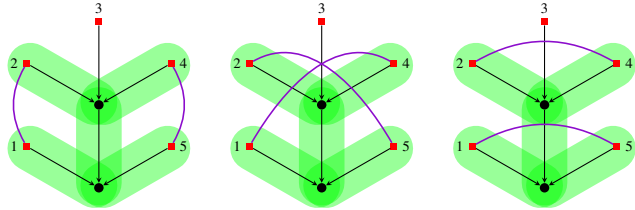
\begin{equ}
\scalebox{0.5}{
\begin{tikzpicture}[
    % --- Styles ---
    red_square/.style={rectangle, fill=red, inner sep=2.5pt},
    black_dot/.style={circle, fill=black, inner sep=2.5pt},
    arrow/.style={->, >=stealth, thick},
    every edge quotes/.style={font=\small, sloped, fill=white, inner sep=1.5pt},
    highlight/.style={green, opacity=0.4, line cap=round, line join=round, line width=1.2cm},
    % The 'spring' style is replaced with a simple 'purple_line' style.
    purple_line/.style={draw=purple, very thick}
]
% --- 1. Define Node positions ---
% (This section is unchanged)
\node[black_dot, label=below: ]   (1) at (0,0) {};
\node[red_square, label=left: 1]     (3) at (-1.9, 1.1) {};
\node[black_dot, label=right: ]    (5) at (0, 2.2) {};
\node[red_square, label=left: 2]     (7) at (-1.9, 3.3) {};
\node[red_square, label=above: 3]    (9) at (0, 4.4) {};
\node[red_square, label=right: 4]    (11) at (1.9, 3.3) {};
\node[red_square, label=right: 5]    (13) at (1.9, 1.1) {};

% --- 2. Draw Highlights on the Background Layer ---
% (This section is unchanged)
\begin{pgfonlayer}{background}
    \draw[highlight] (7) -- (5);
    \draw[highlight] (11) -- (5);
    \draw[highlight] (3) -- (1);
    \draw[highlight] (13) -- (1);
    \draw[highlight] (5) -- (1);
\end{pgfonlayer}

% --- 3. Draw Original Edges ---
% (This section is unchanged)
\draw[arrow] (3) edge  (1);
\draw[arrow] (13) edge (1);
\draw[arrow] (7) edge (5);
\draw[arrow] (9) edge (5);
\draw[arrow] (11) edge (5);
\draw[arrow] (5) edge (1);

% --- 4. Add the new purple line connections ---
% The line connecting 3 and 11 is now bent upwards.
\draw[purple_line] (3) to[out=120, in=-120] (7);
\draw[purple_line] (11) to[out=-60, in=60] (13);

\end{tikzpicture}
} \quad
\scalebox{0.5}{
\begin{tikzpicture}[
    % --- Styles ---
    red_square/.style={rectangle, fill=red, inner sep=2.5pt},
    black_dot/.style={circle, fill=black, inner sep=2.5pt},
    arrow/.style={->, >=stealth, thick},
    every edge quotes/.style={font=\small, sloped, fill=white, inner sep=1.5pt},
    highlight/.style={green, opacity=0.4, line cap=round, line join=round, line width=1.2cm},
    % The 'spring' style is replaced with a simple 'purple_line' style.
    purple_line/.style={draw=purple, very thick}
]
% --- 1. Define Node positions ---
% (This section is unchanged)
\node[black_dot, label=below: ]   (1) at (0,0) {};
\node[red_square, label=left: 1]     (3) at (-1.9, 1.1) {};
\node[black_dot, label=right: ]    (5) at (0, 2.2) {};
\node[red_square, label=left: 2]     (7) at (-1.9, 3.3) {};
\node[red_square, label=above: 3]    (9) at (0, 4.4) {};
\node[red_square, label=right: 4]    (11) at (1.9, 3.3) {};
\node[red_square, label=right: 5]    (13) at (1.9, 1.1) {};

% --- 2. Draw Highlights on the Background Layer ---
% (This section is unchanged)
\begin{pgfonlayer}{background}
    \draw[highlight] (7) -- (5);
    \draw[highlight] (11) -- (5);
    \draw[highlight] (3) -- (1);
    \draw[highlight] (13) -- (1);
    \draw[highlight] (5) -- (1);
\end{pgfonlayer}

% --- 3. Draw Original Edges ---
% (This section is unchanged)
\draw[arrow] (3) edge  (1);
\draw[arrow] (13) edge (1);
\draw[arrow] (7) edge (5);
\draw[arrow] (9) edge (5);
\draw[arrow] (11) edge (5);
\draw[arrow] (5) edge (1);

% --- 4. Add the new purple line connections ---
% The line connecting 3 and 11 is now bent upwards.
\draw[purple_line] (3) to[out=60, in=150] (11);
\draw[purple_line] (7) to[out=30, in=120] (13);

\end{tikzpicture}
} \quad
\scalebox{0.5}{
\begin{tikzpicture}[
    % --- Styles ---
    red_square/.style={rectangle, fill=red, inner sep=2.5pt},
    black_dot/.style={circle, fill=black, inner sep=2.5pt},
    arrow/.style={->, >=stealth, thick},
    every edge quotes/.style={font=\small, sloped, fill=white, inner sep=1.5pt},
    highlight/.style={green, opacity=0.4, line cap=round, line join=round, line width=1.2cm},
    % The 'spring' style is replaced with a simple 'purple_line' style.
    purple_line/.style={draw=purple, very thick}
]
% --- 1. Define Node positions ---
% (This section is unchanged)
\node[black_dot, label=below: ]   (1) at (0,0) {};
\node[red_square, label=left: 1]     (3) at (-1.9, 1.1) {};
\node[black_dot, label=right: ]    (5) at (0, 2.2) {};
\node[red_square, label=left: 2]     (7) at (-1.9, 3.3) {};
\node[red_square, label=above: 3]    (9) at (0, 4.4) {};
\node[red_square, label=right: 4]    (11) at (1.9, 3.3) {};
\node[red_square, label=right: 5]    (13) at (1.9, 1.1) {};

% --- 2. Draw Highlights on the Background Layer ---
% (This section is unchanged)
\begin{pgfonlayer}{background}
    \draw[highlight] (7) -- (5);
    \draw[highlight] (11) -- (5);
    \draw[highlight] (3) -- (1);
    \draw[highlight] (13) -- (1);
    \draw[highlight] (5) -- (1);
\end{pgfonlayer}

% --- 3. Draw Original Edges ---
% (This section is unchanged)
\draw[arrow] (3) edge  (1);
\draw[arrow] (13) edge (1);
\draw[arrow] (7) edge (5);
\draw[arrow] (9) edge (5);
\draw[arrow] (11) edge (5);
\draw[arrow] (5) edge (1);

% --- 4. Add the new purple line connections ---
% The line connecting 3 and 11 is now bent upwards.
\draw[purple_line] (7) to[out=30, in=150] (11);
\draw[purple_line] (3) to[out=30, in=150] (13);

\end{tikzpicture}
}
\end{equ}
$M^\circ$ acts trivially on all of these trees because the unrenormalised part is $\<1>$. Furthermore, $C_\eps^{2} \eqdef \ell_{\BPHZ}^{q ; (\eps)}\left(\<22_2> , \bpi \right)$ yield the same diverging constant for $\bpi = \{(1,2),(4,5)\}$, $\{(1,4),(2,5)\}$. Compared to \cite[Eq.~(10.41)]{Hai14}, our choice of $C_\eps^2$ is exactly half of the usual choice, as we are treating the two contractions $\bpi$ separately. On the other hand, $\ell_{\BPHZ}^{q ; (\eps)}\left(\<22_2> , \{(1,5),(2,4)\} \right) = 0$ by the inductive definition of $\ell^{q;(\eps)}_{\BPHZ}$. Concerning the inclusion of $\<2>$ in $\<23_2>$, we get the tree
\begin{equ}[eq:CountT2]
	\scalebox{0.5}{
	\begin{tikzpicture}[
    % --- Styles ---
    red_square/.style={rectangle, fill=red, inner sep=2.5pt},
    black_dot/.style={circle, fill=black, inner sep=2.5pt},
    arrow/.style={->, >=stealth, thick},
    every edge quotes/.style={font=\small, sloped, fill=white, inner sep=1.5pt},
    highlight/.style={green, opacity=0.4, line cap=round, line join=round, line width=1.2cm},
    % The 'spring' style is replaced with a simple 'purple_line' style.
    purple_line/.style={draw=purple, very thick}
]
% --- 1. Define Node positions ---
% (This section is unchanged)
\node[black_dot, label=below: ]   (1) at (0,0) {};
\node[red_square, label=left: 1]     (3) at (-1.9, 1.1) {};
\node[black_dot, label=right: ]    (5) at (0, 2.2) {};
\node[red_square, label=left: 2]     (7) at (-1.9, 3.3) {};
\node[red_square, label=above: 3]    (9) at (0, 4.4) {};
\node[red_square, label=right: 4]    (11) at (1.9, 3.3) {};
\node[red_square, label=right: 5]    (13) at (1.9, 1.1) {};
\node[label=right: \scalebox{2}{$-C_\eps^1$}] (14) at (-4.6, 2.2) {};

% --- 2. Draw Highlights on the Background Layer ---
% (This section is unchanged)
\begin{pgfonlayer}{background}
    %\draw[highlight] (7) -- (5);
    %\draw[highlight] (11) -- (5);
    \draw[highlight] (3) -- (1);
    \draw[highlight] (13) -- (1);
    %\draw[highlight] (5) -- (1);
\end{pgfonlayer}

% --- 3. Draw Original Edges ---
% (This section is unchanged)
\draw[arrow] (3) edge  (1);
\draw[arrow] (13) edge (1);
\draw[arrow] (7) edge (5);
\draw[arrow] (9) edge (5);
\draw[arrow] (11) edge (5);
\draw[arrow] (5) edge (1);

% --- 4. Add the new purple line connections ---
% The line connecting 3 and 11 is now bent upwards.
%\draw[purple_line] (7) to[out=30, in=150] (11);
\draw[purple_line] (3) to[out=30, in=150] (13);

	\end{tikzpicture}
}
\end{equ}
Of the three trees produced by $M^\circ$ applied to this contracted tree, only one inclusion of $\<2>$ leaves noise $3$ uncontracted. This yields the tree
\begin{equ}[eq:CountT3]
	\scalebox{0.5}{
\begin{tikzpicture}[
    % --- Styles ---
    red_square/.style={rectangle, fill=red, inner sep=2.5pt},
    black_dot/.style={circle, fill=black, inner sep=2.5pt},
    arrow/.style={->, >=stealth, thick},
    every edge quotes/.style={font=\small, sloped, fill=white, inner sep=1.5pt},
    highlight/.style={green, opacity=0.4, line cap=round, line join=round, line width=1.2cm},
    % The 'spring' style is replaced with a simple 'purple_line' style.
    purple_line/.style={draw=purple, very thick}
]
% --- 1. Define Node positions ---
% (This section is unchanged)
\node[black_dot, label=below: ]   (1) at (0,0) {};
\node[red_square, label=left: 1]     (3) at (-1.9, 1.1) {};
\node[black_dot, label=right: ]    (5) at (0, 2.2) {};
\node[red_square, label=left: 2]     (7) at (-1.9, 3.3) {};
\node[red_square, label=above: 3]    (9) at (0, 4.4) {};
\node[red_square, label=right: 4]    (11) at (1.9, 3.3) {};
\node[red_square, label=right: 5]    (13) at (1.9, 1.1) {};
\node[label=right: \scalebox{2}{$\left(C_\eps^1\right)^2$}] (14) at (-4.6, 2.2) {};

% --- 2. Draw Highlights on the Background Layer ---
% (This section is unchanged)
\begin{pgfonlayer}{background}
    \draw[highlight] (7) -- (5);
    \draw[highlight] (11) -- (5);
    \draw[highlight] (3) -- (1);
    \draw[highlight] (13) -- (1);
    %\draw[highlight] (5) -- (1);
\end{pgfonlayer}

% --- 3. Draw Original Edges ---
% (This section is unchanged)
\draw[arrow] (3) edge  (1);
\draw[arrow] (13) edge (1);
\draw[arrow] (7) edge (5);
\draw[arrow] (9) edge (5);
\draw[arrow] (11) edge (5);
\draw[arrow] (5) edge (1);

% --- 4. Add the new purple line connections ---
% The line connecting 3 and 11 is now bent upwards.
\draw[purple_line] (7) to[out=30, in=150] (11);
\draw[purple_line] (3) to[out=30, in=150] (13);

\end{tikzpicture}
}
\end{equ}
Finally, the action of $M^\circ$ on $\<23_2>$ directly produces the tree
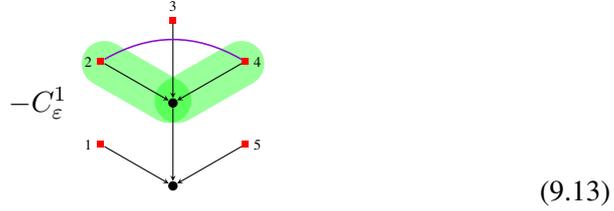
\begin{equ}[eq:CountT4]
	\scalebox{0.5}{
\begin{tikzpicture}[
    % --- Styles ---
    red_square/.style={rectangle, fill=red, inner sep=2.5pt},
    black_dot/.style={circle, fill=black, inner sep=2.5pt},
    arrow/.style={->, >=stealth, thick},
    every edge quotes/.style={font=\small, sloped, fill=white, inner sep=1.5pt},
    highlight/.style={green, opacity=0.4, line cap=round, line join=round, line width=1.2cm},
    % The 'spring' style is replaced with a simple 'purple_line' style.
    purple_line/.style={draw=purple, very thick}
]
% --- 1. Define Node positions ---
% (This section is unchanged)
\node[black_dot, label=below: ]   (1) at (0,0) {};
\node[red_square, label=left: 1]     (3) at (-1.9, 1.1) {};
\node[black_dot, label=right: ]    (5) at (0, 2.2) {};
\node[red_square, label=left: 2]     (7) at (-1.9, 3.3) {};
\node[red_square, label=above: 3]    (9) at (0, 4.4) {};
\node[red_square, label=right: 4]    (11) at (1.9, 3.3) {};
\node[red_square, label=right: 5]    (13) at (1.9, 1.1) {};
\node[label=right: \scalebox{2}{$-C_\eps^1$}] (14) at (-4.6, 2.2) {};

% --- 2. Draw Highlights on the Background Layer ---
% (This section is unchanged)
\begin{pgfonlayer}{background}
    \draw[highlight] (7) -- (5);
    \draw[highlight] (11) -- (5);
    %\draw[highlight] (3) -- (1);
    %\draw[highlight] (13) -- (1);
    %\draw[highlight] (5) -- (1);
\end{pgfonlayer}

% --- 3. Draw Original Edges ---
% (This section is unchanged)
\draw[arrow] (3) edge  (1);
\draw[arrow] (13) edge (1);
\draw[arrow] (7) edge (5);
\draw[arrow] (9) edge (5);
\draw[arrow] (11) edge (5);
\draw[arrow] (5) edge (1);

% --- 4. Add the new purple line connections ---
% The line connecting 3 and 11 is now bent upwards.
\draw[purple_line] (7) to[out=30, in=150] (11);
%\draw[purple_line] (3) to[out=30, in=150] (13);

\end{tikzpicture}
}
\end{equ}
Overall, this leaves us with $5$ trees contributing to the renormalisation of $\<23_2>$, s.t.\ noise $3$ remains uncontracted, the other cases with noises $2$ and $4$ follow symmetrically. We need to check the separate chaos components of $\widehat{\Pi}_x^{N ; (\eps), R} \<23_2>$ for $N \geqslant 5$. We will use the kernel notation of \cite[Section~10.5]{Hai14} from now on, which also shows that they are all well-defined $\mfH^{\wotimes_\alpha k}$-valued distributions. First, we note that in terms of this notation
\begin{equs}
	C^1_\eps &= \begin{tikzpicture}[scale=0.7,baseline=-0.1cm]
		\node at (0,0.5) [icirc] (middle) {};
		\node[circ, label=below: $z$] at (0,-0.5)  (below) {};
		\draw[kepsilon] (middle) to[out = -60, in = 60] (below);
		\draw[kepsilon] (middle) to[out = -120, in = 120] (below);
		%\draw[kepsilon] (above) to (middle);
	\end{tikzpicture}  \; , \qquad
	C^2_\eps = \begin{tikzpicture}[scale=0.7,baseline=-0.1cm]
		\node at (0,0.5) [icirc] (middle) {};
		%\node at (0,1.2) [circ] (above) {};
		%\node at (-0.7,0.8) [circ] (aboveleft) {};
		%\node at (0.7,0.8) [circ] (aboveright) {};
		\node[circ, label=below: $z$] at (0,-0.5)  (below) {};
		\node at (-0.866025,0) [icirc] (left) {};
		\node at (0.866025,0) [icirc] (right) {};
		\draw[kepsilon] (left) to (below);
		\draw[kepsilon] (right) to (below);
		\draw[kernel] (middle) to (below);
		\draw[kepsilon] (left) to (middle);
		\draw[kepsilon] (right) to (middle);
		%\draw[kepsilon] (above) to (middle);
	\end{tikzpicture}  \; .
\end{equs}

The component of $\widehat{\Pi}_x^{N ; (\eps), R} \<23>(z)$ in $[5, \emptyset]$ is
% \begin{equ}
% 	\begin{tikzpicture}[scale=0.7,baseline=0cm]
% 		\node at (0,0) [icirc] (middle) {};
% 		\node at (0,1.2) [circ] (above) {};
% 		\node at (-0.7,0.8) [circ] (aboveleft) {};
% 		\node at (0.7,0.8) [circ] (aboveright) {};
% 		\node[circ, label=below: $z$] at (0,-1)  (below) {};
% 		\node at (-0.7,-0.2) [circ] (left) {};
% 		\node at (0.7,-0.2) [circ] (right) {};
% 		\draw[kepsilon] (left) to (below);
% 		\draw[kepsilon] (right) to (below);
% 		\draw[kernel] (middle) to (below);
% 		\draw[kepsilon] (aboveleft) to (middle);
% 		\draw[kepsilon] (aboveright) to (middle);
% 		\draw[kepsilon] (above) to (middle);
% 	\end{tikzpicture} \; - \; \begin{tikzpicture}[scale=0.7,baseline=0cm]
% 		\node at (0,0) [icirc] (middle) {};
% 		\node at (0,1.2) [circ] (above) {};
% 		\node at (-0.7,0.8) [circ] (aboveleft) {};
% 		\node at (0.7,0.8) [circ] (aboveright) {};
% 		\node[circ, label=below: $x$] at (0,-1)  (below) {};
% 		\node[circ, label=below: $z$] at (-0.4,-1)  (belowleft) {};
% 		\node[circ, label=below: $z$] at (0.4,-1)  (belowright) {};
% 		\node at (-0.4,-0.2) [circ] (left) {};
% 		\node at (0.4,-0.2) [circ] (right) {};
% 		\draw[kepsilon] (left) to (belowleft);
% 		\draw[kepsilon] (right) to (belowright);
% 		\draw[kernel] (middle) to (below);
% 		\draw[kepsilon] (aboveleft) to (middle);
% 		\draw[kepsilon] (aboveright) to (middle);
% 		\draw[kepsilon] (above) to (middle);
% 	\end{tikzpicture}
% \end{equ}

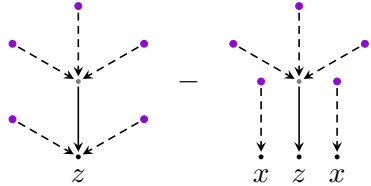
\begin{equ}
  \begin{tikzpicture}[/pgf/fpu/install only={reciprocal}]
    \node at (0,-1) [circ, label=below: $z$] (r) {};
    \node at (-0.866025, -0.5) [bdot] (1) {};
    \node at (0, 0) [icirc] (2) {};
    \node at (0.866025, -0.5) [bdot] (3) {};
    \node at (-0.866025, 0.5) [bdot] (4) {};
    \node at (0, 1) [bdot] (5) {};
    \node at (0.866025, 0.5) [bdot] (6) {};
    \draw[kepsilon] (4) to (2);
    \draw[kepsilon] (5) to (2);
    \draw[kepsilon] (6) to (2);
    \draw[kepsilon] (1) to (r);
    \draw[kernel] (2) to (r);
    \draw[kepsilon] (3) to (r);
  \end{tikzpicture} \quad \raisebox{1.36cm}{$-$} \quad
  \begin{tikzpicture}[/pgf/fpu/install only={reciprocal}]
    \node at (0,-1) [circ, label=below: $z$] (r) {};
    \node at (-0.5,-1) [circ, label=below: $x$] (rl) {};
    \node at (0.5,-1) [circ, label=below: $x$] (rr) {};
    \node at (-0.5, 0) [bdot] (1) {};
    \node at (0, 0) [icirc] (2) {};
    \node at (0.5, 0) [bdot] (3) {};
    \node at (-0.866025, 0.5) [bdot] (4) {};
    \node at (0, 1) [bdot] (5) {};
    \node at (0.866025, 0.5) [bdot] (6) {};
    \draw[kepsilon] (4) to (2);
    \draw[kepsilon] (5) to (2);
    \draw[kepsilon] (6) to (2);
    \draw[kepsilon] (1) to (rl);
    \draw[kernel] (2) to (r);
    \draw[kepsilon] (3) to (rr);
  \end{tikzpicture}
\end{equ}
exactly the kernel $\hat{\CW}^{(\eps ; 5)}$. This is a well-defined $\mfH^{\wotimes_\alpha 5}$-valued distribution for all $\eps \in [0,1]$ by the usual kernel estimates, cf.\ \cite[Section~10.3]{Hai14}, and satisfies the correct scaling behaviour, e.g.\ cf.\ \cite[Proposition~A.11]{CHP23}. In particular, $\hat{\CW}^{(\eps ; 5)} \in \CC^{-1/2}_{\s} \left( \R \times \T^3 ; \mfH^{\wotimes_\alpha 5} \right) \subset \CC^{-1/2}_{\s} \left( \R^4; \mfH^{\wotimes_\alpha 5} \right)$.

There are 6 components with two leaves contracted that are not noise $3$. The components in $[5, \{(1,5)\}]$ and $[5,\{(2,4)\}]$ are $0$, since they are exactly cancelled by the corresponding counterterms \eqref{eq:CountT2} and \eqref{eq:CountT4}. The other components are in $[5, \{(1,2)\}]$, $[5, \{(1,4)\}]$, $[5, \{(2,5)\}]$, and $[5, \{(4,5)\}]$ are well-defined by themselves as they can be represented by (permutations) of the kernel
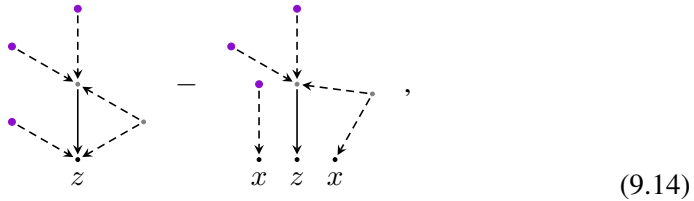
\begin{equ}[eq:Phi43-1Cont]
	% \begin{tikzpicture}[scale=0.7,baseline=0cm]
	% 	\node at (0,0) [icirc] (middle) {};
	% 	\node at (0,1.2) [circ] (above) {};
	% 	%\node at (-0.7,0.8) [circ] (aboveleft) {};
	% 	\node at (0.7,0.8) [circ] (aboveright) {};
	% 	\node[circ, label=below: $z$] at (0,-1)  (below) {};
	% 	\node at (-0.866025,-0.5) [icirc] (left) {};
	% 	\node at (0.7,-0.2) [circ] (right) {};
	% 	\draw[kepsilon] (left) to (below);
	% 	\draw[kepsilon] (right) to (below);
	% 	\draw[kernel] (middle) to (below);
	% 	\draw[kepsilon] (left) to (middle);
	% 	\draw[kepsilon] (aboveright) to (middle);
	% 	\draw[kepsilon] (above) to (middle);
	% \end{tikzpicture} \;\; - \; \begin{tikzpicture}[scale=0.7,baseline=0cm]
	% 	\node at (0,0) [icirc] (middle) {};
	% 	\node at (0,1.2) [circ] (above) {};
	% 	%\node at (-0.7,0.8) [circ] (aboveleft) {};
	% 	\node at (0.7,0.8) [circ] (aboveright) {};
	% 	\node[circ, label=below: $x$] at (0,-1)  (below) {};
	% 	\node[circ, label=below: $z$] at (-0.7,-1)  (belowleft) {};
	% 	\node[circ, label=below: $z$] at (0.4,-1)  (belowright) {};
	% 	\node at (-0.7,-0.2) [icirc] (left) {};
	% 	\node at (0.4,-0.2) [circ] (right) {};
	% 	\draw[kepsilon] (left) to (belowleft);
	% 	\draw[kepsilon] (right) to (belowright);
	% 	\draw[kernel] (middle) to (below);
	% 	\draw[kepsilon] (left) to (middle);
	% 	\draw[kepsilon] (aboveright) to (middle);
	% 	\draw[kepsilon] (above) to (middle);
	% \end{tikzpicture} \; ,
  \begin{tikzpicture}[/pgf/fpu/install only={reciprocal}]
    \node at (0,-1) [circ, label=below: $z$] (r) {};
    \node at (-0.866025, -0.5) [bdot] (1) {};
    \node at (0, 0) [icirc] (2) {};
    \node at (0.866025, -0.5) [icirc] (3) {};
    \node at (-0.866025, 0.5) [bdot] (4) {};
    \node at (0, 1) [bdot] (5) {};
    %\node at (0.866025, 0.5) [icirc] (6) {};
    \draw[kepsilon] (4) to (2);
    \draw[kepsilon] (5) to (2);
    \draw[kepsilon] (3) to (2);
    \draw[kepsilon] (1) to (r);
    \draw[kernel] (2) to (r);
    \draw[kepsilon] (3) to (r);
  \end{tikzpicture} \quad \raisebox{1.36cm}{$-$} \quad
  \begin{tikzpicture}[/pgf/fpu/install only={reciprocal}]
    \node at (0,-1) [circ, label=below: $z$] (r) {};
    \node at (-0.5,-1) [circ, label=below: $x$] (rl) {};
    \node at (0.5,-1) [circ, label=below: $x$] (rr) {};
    \node at (-0.5, 0) [bdot] (1) {};
    \node at (0, 0) [icirc] (2) {};
    %\node at (1, -0.400679) [icirc] (3) {};
    \node at (-0.866025, 0.5) [bdot] (4) {};
    \node at (0, 1) [bdot] (5) {};
    \node at (0.99162, -0.12919) [icirc] (6) {};
    \draw[kepsilon] (4) to (2);
    \draw[kepsilon] (5) to (2);
    \draw[kepsilon] (6) to (2);
    \draw[kepsilon] (1) to (rl);
    \draw[kernel] (2) to (r);
    \draw[kepsilon] (6) to (rr);
    %\draw[vwavy] (3) to (6);
  \end{tikzpicture} \quad \raisebox{1.36cm}{,} 
\end{equ}
which is again the same kernel as $\hat\CW^{(\eps ; 3)}$. It is again well-defined by the usual arguments. The only difference is that we have split up the $4$ copies into their own subspaces, whereas they all fall together in the symmetric case.

Finally, there are $3$ components with $4$ leaves contracted that are not noise $3$. Amongst these $[5,\{(1,5),(2,4)\}]$ is cancelled exactly by \eqref{eq:CountT3}. The other two $[5,\{(1,2),(4,5)\}]$ and $[5,\{(1,4),(2,5)\}]$ are permutations of the kernel $\hat \CW^{(\eps; 1)}$
\begin{equ}[eq:Phi43-2Cont]
  \begin{tikzpicture}[/pgf/fpu/install only={reciprocal}]
    \node at (0,-1) [circ, label=below: $z$] (r) {};
    \node at (-0.866025, -0.5) [icirc] (1) {};
    \node at (0, 0) [icirc] (2) {};
    \node at (0.866025, -0.5) [icirc] (3) {};
    %\node at (-0.866025, 0.5) [icirc] (4) {};
    \node at (0, 1) [bdot] (5) {};
    %\node at (0.866025, 0.5) [icirc] (6) {};
    \draw[kepsilon] (1) to (2);
    \draw[kepsilon] (5) to (2);
    \draw[kepsilon] (3) to (2);
    \draw[kepsilon] (1) to (r);
    \draw[kernel] (2) to (r);
    \draw[kepsilon] (3) to (r);
  \end{tikzpicture}\quad \raisebox{1.36cm}{$-$} \quad \raisebox{0.5cm}{
  \begin{tikzpicture}[/pgf/fpu/install only={reciprocal}]
    \node at (-1.2,0) [circ, label=below: $z$] (d) {} ;
    \node at (-1.2,1) [bdot] (u) {} ;
    \node at (0,0) [circ, label=below: $z$] (r) {};
    \node at (-0.866025, 0.5) [icirc] (1) {} ;
    \node at (0.866025, 0.5) [icirc] (3) {} ;
    %\node at (-1,1) [icirc] (4) {} ;
    \node at (0,1) [icirc] (2) {} ;
    %\node at (1,1) [icirc] (5) {} ;
    \draw[kepsilon] (u) to (d) ;
    \draw[kernel] (2) to (r) ;
    \draw[kepsilon] (1) to (r) ;
    \draw[kepsilon] (3) to (r) ;
    \draw[kepsilon] (1) to (2) ;
    \draw[kepsilon] (3) to (2) ;
  \end{tikzpicture}
  }
  \quad \raisebox{1.36cm}{$-$} \quad 
  \begin{tikzpicture}[/pgf/fpu/install only={reciprocal}]
    \node at (0,-1) [circ, label=below: $z$] (r) {};
    \node at (-0.5,-1) [circ, label=below: $x$] (rl) {};
    \node at (0.5,-1) [circ, label=below: $x$] (rr) {};
    %\node at (-1.30051, -0.400679) [icirc] (1) {};
    \node at (0, 0) [icirc] (2) {};
    %\node at (1.30051, -0.400679) [icirc] (3) {};
    \node at (-0.99162, -0.12919) [icirc] (4) {};
    \node at (0, 1) [bdot] (5) {};
    \node at (0.99162, -0.12919) [icirc] (6) {};
    \draw[kepsilon] (4) to (2);
    \draw[kepsilon] (5) to (2);
    \draw[kepsilon] (6) to (2);
    \draw[kepsilon] (6) to (rr);
    \draw[kernel] (2) to (r);
    \draw[kepsilon] (4) to (rl);
    % \draw[vwavy] (3) to (6);
    % \draw[vwavy] (1) to (4);
  \end{tikzpicture}  \quad \; ,
\end{equ}
which is again well-defined by the usual arguments.

Following similar arguments for
\begin{equs}
	{} & \;  \<0> \, , \;  \<1>\, , \; \<2> \, , \; \<3> \, , \;  \<2> X_\mu \, , \;  \<13_1> \, , \; \<13_2> \, , \; \<23_1>\, , \; \<23_2>\, , \<23_3> \, , \; \<22_1> \, , \; \<22_2> \, , \; \<22_3>
\end{equs}
establishes the existence of the model $(\widehat{\Pi}^{N; (\eps),R}, \widehat{\Gamma}^{N; (\eps),R})$ for all $\eps \in [0,1]$ and thus also $(\Pi^{q; (\eps),R}, \Gamma^{q; (\eps),R})$ for all $q \in (-1,1)$ and $\eps \in [0,1]$.

Having established the existence of all necessary models, we return to the SPDE \eqref{eq:NonComPhi43}. First, let $V$ be the sector of $\cT_q$ spanned by the planted trees, and polynomials and let $H$ be the $\{ t = 0\}$ hyperplane. 

\begin{lemma}
	For $\gamma > 1 + 2 \kappa$ and $\eta \leqslant - \frac{1}{2} - \kappa$, $u \mapsto u^3$ maps 
	\begin{equ}
		\CD_H^{\gamma,\eta}(V) \longrightarrow \CD_H^{\gamma - 1 - 2 \kappa, 3 \eta}
	\end{equ}
	and is locally Lipschitz continuous.
\end{lemma}
\begin{remark}
	The choice of $\gamma$ is such that $u^3$ always has a unique reconstruction. 
\end{remark}
\begin{proof}
	The proof follows directly from Theorem~\ref{thm:RSMult}. The main ingredient consists of recognising that the regularity of $V$ is $-\frac{1}{2} - \kappa$ and thus taking the cube can reduce the regularity of the modelled distribution by at most $2 \left( \frac{1}{2} + \kappa\right)$.
\end{proof}

Given an initial condition $\phi_0 \in \CC^{\eta}(\T^3 ; \cA_q)$ with $ \eta > - \frac{2}{3}$, \eqref{eq:LiftEq} suggest that we should rewrite \eqref{eq:NonComPhi43} as a fixed-point problem in $\CD^{\gamma, \eta}_H(V)$ for $\gamma  > 1 + 2\kappa$, $\bar\gamma = \gamma - 1 - 2\kappa$
\begin{equ}[eq:ModFP]
	u = - (\CK^{(\eps)}_{\bar\gamma} + R_\gamma \CR^{(\eps)}) ( \bR^+ u^ 3 - \bR^+ \Xi) + G \phi_0
\end{equ}
where $G \phi_0$ is the lift of $(t,x) \mapsto e^{-t(m^2-\Delta)} \phi_0(x)$ to a modelled distribution, which by \cite[Lemma~7.5]{Hai14} belongs to $\CD^{\infty, \eta}_H(V)$. The only issue in this construction is that $\bR^+ \Xi$ only belongs to $\CD^{\infty, - \frac{5}{2}- \kappa}$, and thus does not necessarily admit a reconstruction as $ - \frac{5}{2}- \kappa < -2$. However, one can easily show that setting $\CR^{(\eps)}\bR^+ \Xi \eqdef \bone_{\{t>0\}} \xi^{(\eps)}_q$ is well-defined for all $\eps \in [0,1]$ in a sufficiently regular space, cf.\ \cite[Section~9.4]{Hai14}. With this convention, we define $v^{(\eps)} \eqdef (\CK^{(\eps)}_{\bar\gamma} + R_\gamma \CR^{(\eps)}) \bR^+ \Xi$ for which $v^{(\eps)} \in \CD^{\gamma,\eta}_H(V)$ and the assignment $\eps \mapsto v^{(\eps)}$ is continuous. Thus, we arrive at our final abstract form of the $\Phi^4_3$-equation 
\begin{equ}[eq:ModFP2]
	u = - (\CK^{(\eps)}_{\bar\gamma} + R_\gamma \CR^{(\eps)}) ( \bR^+ u^ 3) - v^{(\eps)} + G \phi_0
\end{equ}

For $T>0$, let $\CD^{\gamma,\eta}_{H;T}$ denote the restriction of $\Z^3$-invariant modelled distributions in $\CD^{\gamma,\eta}_{H}$ to $(0,T) \times \R^3$. We have the following theorem describing the abstract solution theory of the $\Phi^4_3$-equation for $q$-mezdons, which is a direct consequence of Theorem~\ref{thm:AbsFPM}.
\begin{theorem}
	Let $\kappa \in (0, \frac{1}{14})$, $\gamma > 1 + 2\kappa$, and $\eta \in (- \frac{2}{3}, -\frac{1}{2}- \kappa)$. For every $\eps \in [0,1]$ and $\phi_0 \in \CC^{\eta}(\T^3 ; \cA_q)$ there exists $T(\eps , \phi_0) > 0 $, s.t.\ \eqref{eq:ModFP2} has a unique solution $\CS_T(\eps, \phi_0)$ in $\CD_{H; T}^{\gamma,\eta}$.  Furthermore, there exists a $\delta > 0$, s.t.\ the map 
	\begin{equs}
		S_{T} \colon B_{\delta}(\eps) \times B_{\delta}(\phi_0) & \longrightarrow \left(\CD^{\gamma,\eta}_{H;T} \right) \\
		(\eps', \phi_0') &\longmapsto S_T(\eps', \phi_0')
	\end{equs}
	is well-defined and continuous.
\end{theorem}

\begin{remark}
	Here $B_\delta(\eps)$ and $B_{\delta}(\phi_0)$ are the balls of radius $\delta$ around $\eps$ and $\phi_0$ in $[0,1]$ and $\CC^\eta(\T^3)$ respectively. 
\end{remark}

\begin{remark}
	By suitably defining blow-up spaces, the above solution map can be extended to the maximal time of existence and remains continuous in those spaces, cf.\ \cite{BCCH21,CCHS2d, CHP23}.
\end{remark}

We are now only left with analysing the relationship between solutions of the original equation \eqref{eq:NonComPhi43} and reconstructions $\phi_q^{(\eps)} \eqdef \CR^{(\eps)} u_{(\eps)}$ of our abstract solution $u_{(\eps)} = S_{T}(\eps, \phi_0)$ for $\eps > 0$. In the commutative case $q=1$, it is well-known that $\phi_{(\eps)}$ satisfies 
\begin{equ}
(\partial_t -\Delta +m^2) \phi_1^{(\eps)} = - \left( \phi_1^{(\eps)} \right)^3 + ( 3 C_1^{(\eps)} - 18 C_2^{(\eps)}) \phi_1^{(\eps)} + \xi_1^{(\eps)} \; ,
\end{equ}
to our conventions for the constants $C^1_{\eps}$ and $C^2_{\eps}$.

Fixing $\eps > 0$, a straighforward Picard iteration argument, \cite[Section~9.4]{Hai14}, shows that there exist functions $\upsilon, \upsilon^i$ for $i \in [3]$, s.t.\
\begin{equ}
	u_{(\eps)} = \<1> + \upsilon \1 - 3 \<03> - \left( \<02_1R> + \<02_2R> + \<02_3R> \right) + \upsilon^i X_i \; ,
\end{equ}
using the Einstein summation convention. Then 
\begin{equs}[eq:RenAEqPhi43]
	\Xi - u_{(\eps)}^3 & = \<0> - \left( \<3> + \upsilon \<2> + \<2_1Rprime> + \<2> \upsilon + \<1> \upsilon^2 + \upsilon \<1> \upsilon + \upsilon^2 \<1> +  \upsilon^3 \1  \right) - \\
	& \qquad - X_i \left( \upsilon^i \<2> + \<2_1Rpi> + \<2> \upsilon^i   \right)  + \\
	&\qquad + \left( \<23_1> + \<23_2> + \<23_3> \right) +\\ 
	& \qquad + \Bigl( \upsilon \<13_1> + \<1_R3_1> + \<13_1> \upsilon + \text{sym.} \Bigr) + \\
	& \qquad + \Bigl( %\<22_1_1R> + \<22_1_2R> + \<22_1_3R> + 
	\<22_2R_1> + \<22_2R_2> + \<22_2R_3> +  \text{sym.} \Bigl) + \rho
\end{equs}
where $\rho$ is a remainder of strictly positive regularity. In the fourth line, we have left 3 terms unwritten -- generated by insertions of $\upsilon$ in $\<13_2>$ -- and in the fifth line, another 6 terms generated by $\<22_1>$ and $\<22_3>$. In the commutative case, these can simply be written as $6 \upsilon \<13_1>$ and $9 \upsilon \<22_2>$; however, because of the noncommutativity, we cannot simply extract the $\upsilon$ from the tree in the mezdonic case. 

The first line in \eqref{eq:RenAEqPhi43} is the same as the one in the 2D case \eqref{eq:Phi42Ren} and thus generates the counterm 
\begin{equ}
	C^1_\eps (2+ \Delta_q) \phi_q^{(\eps)} \; . 
\end{equ}
The new counterm is generated by the contractions in the fifth (and corresponding ones in the third) line. Specialising the case $\<22_2R_2>$ we see that the relevant contractions are 

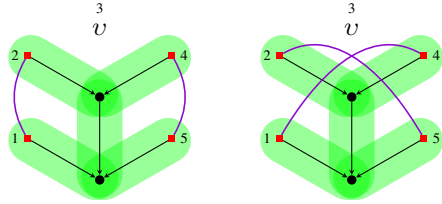
\begin{equ}
\scalebox{0.5}{
\begin{tikzpicture}[
    % --- Styles ---
    red_square/.style={rectangle, fill=red, inner sep=2.5pt},
    black_dot/.style={circle, fill=black, inner sep=2.5pt},
    arrow/.style={->, >=stealth, thick},
    every edge quotes/.style={font=\small, sloped, fill=white, inner sep=1.5pt},
    highlight/.style={green, opacity=0.4, line cap=round, line join=round, line width=1.2cm},
    % The 'spring' style is replaced with a simple 'purple_line' style.
    purple_line/.style={draw=purple, very thick}
]
% --- 1. Define Node positions ---
% (This section is unchanged)
\node[black_dot, label=below: ]   (1) at (0,0) {};
\node[red_square, label=left: 1]     (3) at (-1.9, 1.1) {};
\node[black_dot, label=right: ]    (5) at (0, 2.2) {};
\node[red_square, label=left: 2]     (7) at (-1.9, 3.3) {};
\node[label=above: 3]    (9) at (0, 4) {\scalebox{1.8}{$\upsilon$}};
\node[red_square, label=right: 4]    (11) at (1.9, 3.3) {};
\node[red_square, label=right: 5]    (13) at (1.9, 1.1) {};

% --- 2. Draw Highlights on the Background Layer ---
% (This section is unchanged)
\begin{pgfonlayer}{background}
    \draw[highlight] (7) -- (5);
    \draw[highlight] (11) -- (5);
    \draw[highlight] (3) -- (1);
    \draw[highlight] (13) -- (1);
    \draw[highlight] (5) -- (1);
\end{pgfonlayer}

% --- 3. Draw Original Edges ---
% (This section is unchanged)
\draw[arrow] (3) edge  (1);
\draw[arrow] (13) edge (1);
\draw[arrow] (7) edge (5);
\draw[arrow] (11) edge (5);
\draw[arrow] (5) edge (1);

% --- 4. Add the new purple line connections ---
% The line connecting 3 and 11 is now bent upwards.
\draw[purple_line] (3) to[out=120, in=-120] (7);
\draw[purple_line] (11) to[out=-60, in=60] (13);

\end{tikzpicture}
} \qquad
\scalebox{0.5}{
\begin{tikzpicture}[
    % --- Styles ---
    red_square/.style={rectangle, fill=red, inner sep=2.5pt},
    black_dot/.style={circle, fill=black, inner sep=2.5pt},
    arrow/.style={->, >=stealth, thick},
    every edge quotes/.style={font=\small, sloped, fill=white, inner sep=1.5pt},
    highlight/.style={green, opacity=0.4, line cap=round, line join=round, line width=1.2cm},
    % The 'spring' style is replaced with a simple 'purple_line' style.
    purple_line/.style={draw=purple, very thick}
]
% --- 1. Define Node positions ---
% (This section is unchanged)
\node[black_dot, label=below: ]   (1) at (0,0) {};
\node[red_square, label=left: 1]     (3) at (-1.9, 1.1) {};
\node[black_dot, label=right: ]    (5) at (0, 2.2) {};
\node[red_square, label=left: 2]     (7) at (-1.9, 3.3) {};
\node[label=above: 3]    (9) at (0, 4) {\scalebox{1.8}{$\upsilon$}};
\node[red_square, label=right: 4]    (11) at (1.9, 3.3) {};
\node[red_square, label=right: 5]    (13) at (1.9, 1.1) {};

% --- 2. Draw Highlights on the Background Layer ---
% (This section is unchanged)
\begin{pgfonlayer}{background}
    \draw[highlight] (7) -- (5);
    \draw[highlight] (11) -- (5);
    \draw[highlight] (3) -- (1);
    \draw[highlight] (13) -- (1);
    \draw[highlight] (5) -- (1);
\end{pgfonlayer}

% --- 3. Draw Original Edges ---
% (This section is unchanged)
\draw[arrow] (3) edge  (1);
\draw[arrow] (13) edge (1);
\draw[arrow] (7) edge (5);
\draw[arrow] (11) edge (5);
\draw[arrow] (5) edge (1);

% --- 4. Add the new purple line connections ---
% The line connecting 3 and 11 is now bent upwards.
\draw[purple_line] (3) to[out=60, in=150] (11);
\draw[purple_line] (7) to[out=30, in=120] (13);

\end{tikzpicture}
} 
\end{equ}
For the first contraction $\bpi_1 = \{(1,2),(4,5)\}$ we have $\mathrm{cr}(\bpi_1) = \mathrm{sp}(\bpi_1) = 0$ and for the second $\bpi_2 = \{(1,4),(2,5)\}$ we have $\mathrm{cr}(\bpi_2) = 1$ and $ \mathrm{sp}(\bpi_2) = 2$. Thus, the counterterm corresponding to $\bpi_1$ is $C_\eps^2$ multiplied with 
\begin{equ}
	q^{\mathrm{cr}(\bpi_1)} \Delta_q^{\mathrm{sp}(\bpi_1)}  \phi^{(\eps)}_q = \phi^{(\eps)}_q
\end{equ}
and for $\bpi_2$
\begin{equ}
	 q^{\mathrm{cr}(\bpi_2)} \Delta_q^{\mathrm{sp}(\bpi_2)}  \phi^{(\eps)}_q = q \Delta^2_q \phi_{q}^{(\eps)} \; .
\end{equ}
Summing over all contractions, we are left with the equation
\begin{equ}
		(\partial_{t} - \Delta + m^2) \phi^{(\eps)}_{q} = - \Big( \left(\phi_{q}^{(\eps)}\right) ^{3} - \left(C_\eps^{1} \Delta_q^{(1)}  - C_\eps^{2} \Delta_q^{(2)}\right) \phi^{(\eps)}_{q} \Big) + \xi^{(\eps)}_{q}\;,
\end{equ}
where
\begin{equs}[eq:CTOp]
	\Delta^{(1)}_q &\eqdef 2 + \Delta_q \; ,\\
	\Delta^{(2)}_q &\eqdef 3+2q +(4+4q)\Delta_q + (2+3q) \Delta_q^2 \; .
\end{equs}

\begin{remark}
	As of yet, we do not have a clear interpretation of the meaning of the operators $\Delta_q^{(i)}$ beyond the fact that they are necessary for renormalisation. 
\end{remark}

\subsection{The Clifford \TitleEquation{\Phi^4_3}{Phi43}-Equation}
\label{sec:eq:Phi43Clif}

In this section, we shall apply the theory of noncommutative regularity structures, as well as the work we did in the previous section, to solve the $\Phi^4_3$-equation in the Clifford case
\begin{equ}[eq:Phi43Clif]
	(\partial_t - \Delta + m^2) \phi = - \phi^3 + \bxi \; . 
\end{equ}
with $\bxi \in \cA_F^{\Cl}(\mfH)$ being the extended Clifford noise over the Hilbert space  $\mfH = L^2(\T_2 \times \T^3)$(instead of $L^2(\R \times \T^3)$\footnote{Due to the constraint that the domain of the fermions be compact, necessary for our localised BPHZ estimate, we restrict the temporal interval to $[-2,2]$ and identify it with a one-dimensional torus $\T_2$. This necessitates us to restrict ourselves to a maximal possible time of existence, e.g.\ $T=1$.}).

The regularity structure is defined in a completely analogous fashion. One only needs to slightly modify the renormalisation and kernel estimates to incorporate the localisation procedure. For example, the counterm originating from $\<2>$ now is 
\begin{equ}
	\ell^{\Cl ; (\eps)}_{\BPHZ}(\<2>, \{(1,2)\}) \eqdef -  \left[ \balpha_F\bigg( \begin{tikzpicture}[scale=0.6,baseline=0cm]
		\node at (0,0) [circ, label=below: $0$] (middle) {};
		\node at (0,1) [circ] (above) {};
		\draw[kepsilon] (above) to (middle);
	\end{tikzpicture}  \bigg) , \balpha^\dagger_F\bigg(\begin{tikzpicture}[scale=0.6,baseline=0cm]
		\node at (0,0) [circ, label=below: $0$] (middle) {};
		\node at (0,1) [circ] (above) {};
		\draw[kepsilon] (above) to (middle);
	\end{tikzpicture}  \bigg) \right]_+  \in Z(\cA^{\Cl}_F) \; . 
\end{equ}
For $b \in \Gamma_\infty$, to estimate $\tilde\pi_b \left( \widetilde{\Pi}_x^{5 ; (\eps), R} \<23>\right)$ the kernel in \eqref{eq:Phi43-1Cont} is replaced by 
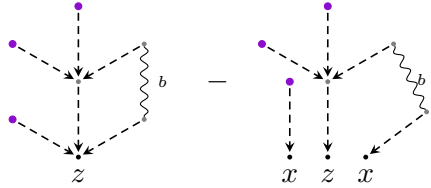
\begin{equ}
  \begin{tikzpicture}[/pgf/fpu/install only={reciprocal}]
    \node at (0,-1) [circ, label=below: $z$] (r) {};
    \node at (-0.866025, -0.5) [bdot] (1) {};
    \node at (0, 0) [icirc] (2) {};
    \node at (0.866025, -0.5) [icirc] (3) {};
    \node at (-0.866025, 0.5) [bdot] (4) {};
    \node at (0, 1) [bdot] (5) {};
    \node at (0.866025, 0.5) [icirc] (6) {};
    \draw[kepsilon] (4) to (2);
    \draw[kepsilon] (5) to (2);
    \draw[kepsilon] (6) to (2);
    \draw[kepsilon] (1) to (r);
    \draw[kepsilon] (2) to (r);
    \draw[kepsilon] (3) to (r);
    \draw[vwavy] (6) to (3);
  \end{tikzpicture} \quad \raisebox{1.36cm}{$-$} \quad
  \begin{tikzpicture}[/pgf/fpu/install only={reciprocal}]
    \node at (0,-1) [circ, label=below: $z$] (r) {};
    \node at (-0.5,-1) [circ, label=below: $x$] (rl) {};
    \node at (0.5,-1) [circ, label=below: $x$] (rr) {};
    \node at (-0.5, 0) [bdot] (1) {};
    \node at (0, 0) [icirc] (2) {};
    \node at (1.30051, -0.400679) [icirc] (3) {};
    \node at (-0.866025, 0.5) [bdot] (4) {};
    \node at (0, 1) [bdot] (5) {};
    \node at (0.866025, 0.5) [icirc] (6) {};
    \draw[kepsilon] (4) to (2);
    \draw[kepsilon] (5) to (2);
    \draw[kepsilon] (6) to (2);
    \draw[kepsilon] (1) to (rl);
    \draw[kepsilon] (2) to (r);
    \draw[kepsilon] (3) to (rr);
    \draw[vwavy] (3) to (6);
  \end{tikzpicture} \quad
\end{equ}
and the one in \eqref{eq:Phi43-2Cont} by 
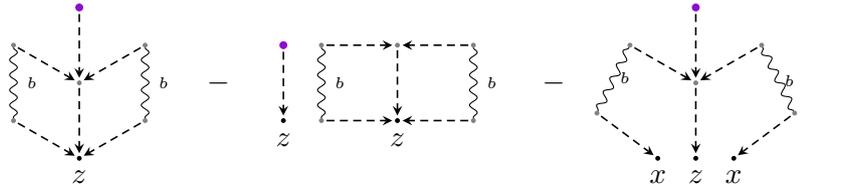
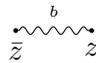
\begin{equs}
  {}&\begin{tikzpicture}[/pgf/fpu/install only={reciprocal}]
    \node at (0,-1) [circ, label=below: $z$] (r) {};
    \node at (-0.866025, -0.5) [icirc] (1) {};
    \node at (0, 0) [icirc] (2) {};
    \node at (0.866025, -0.5) [icirc] (3) {};
    \node at (-0.866025, 0.5) [icirc] (4) {};
    \node at (0, 1) [bdot] (5) {};
    \node at (0.866025, 0.5) [icirc] (6) {};
    \draw[kepsilon] (4) to (2);
    \draw[kepsilon] (5) to (2);
    \draw[kepsilon] (6) to (2);
    \draw[kepsilon] (1) to (r);
    \draw[kepsilon] (2) to (r);
    \draw[kepsilon] (3) to (r);
    \draw[vwavy] (6) to (3);
    \draw[vwavy] (4) to (1);
  \end{tikzpicture}\quad \raisebox{1.36cm}{$-$} \quad \raisebox{0.5cm}{
  \begin{tikzpicture}[/pgf/fpu/install only={reciprocal}]
    \node at (-1.5,0) [circ, label=below: $z$] (d) {} ;
    \node at (-1.5,1) [bdot] (u) {} ;
    \node at (0,0) [circ, label=below: $z$] (r) {};
    \node at (-1,0) [icirc] (1) {} ;
    \node at (1,0) [icirc] (3) {} ;
    \node at (-1,1) [icirc] (4) {} ;
    \node at (0,1) [icirc] (2) {} ;
    \node at (1,1) [icirc] (5) {} ;
    \draw[kepsilon] (u) to (d) ;
    \draw[kepsilon] (2) to (r) ;
    \draw[kepsilon] (1) to (r) ;
    \draw[kepsilon] (3) to (r) ;
    \draw[kepsilon] (4) to (2) ;
    \draw[kepsilon] (5) to (2) ;
    \draw[vwavy] (4) to (1) ; 
    \draw[vwavy] (5) to (3) ;
  \end{tikzpicture}
  }
  \quad \raisebox{1.36cm}{$-$} \quad 
  \begin{tikzpicture}[/pgf/fpu/install only={reciprocal}]
    \node at (0,-1) [circ, label=below: $z$] (r) {};
    \node at (-0.5,-1) [circ, label=below: $x$] (rl) {};
    \node at (0.5,-1) [circ, label=below: $x$] (rr) {};
    \node at (-1.30051, -0.400679) [icirc] (1) {};
    \node at (0, 0) [icirc] (2) {};
    \node at (1.30051, -0.400679) [icirc] (3) {};
    \node at (-0.866025, 0.5) [icirc] (4) {};
    \node at (0, 1) [bdot] (5) {};
    \node at (0.866025, 0.5) [icirc] (6) {};
    \draw[kepsilon] (4) to (2);
    \draw[kepsilon] (5) to (2);
    \draw[kepsilon] (6) to (2);
    \draw[kepsilon] (1) to (rl);
    \draw[kepsilon] (2) to (r);
    \draw[kepsilon] (3) to (rr);
    \draw[vwavy] (3) to (6);
    \draw[vwavy] (1) to (4);
  \end{tikzpicture} \quad \; \raisebox{1.36cm}{.}
\end{equs}
Here 
\begin{equ}
	\begin{tikzpicture}[/pgf/fpu/install only={reciprocal}]
		\node at (0,0) [circ, label = below: $\bar z$] (left) {} ;
		\node at (1,0) [circ, label = below: $z$] (right) {} ;
		\draw[wavy] (left) to (right) ;
	\end{tikzpicture}
\end{equ}
denotes the kernel of projection $P_b$, i.e.\
\begin{equ}
	P_b(x,y) = \sum_{i = 1}^{\dim b} \phi_i(x) \overline{\phi_i(y)}
\end{equ}
for an ONB $\{\phi_i\}$ of $b$. 

Having discussed the regularity structure and model, we can state the local existence of \eqref{eq:Phi43Clif} with the same notation as in the mezdonic case. 
\begin{theorem}
	Let $\kappa \in (0, \frac{1}{14})$ and $\eta \in (- \frac{2}{3}, -\frac{1}{2}- \kappa)$. For every $\eps \in [0,1]$, $n \in \N$, and $\phi_0 \in \CC^{\eta}(\T^3 ; \cA_F^{\Cl})$ there exists $0< T(n, \eps , \phi_0) \leqslant 1 $, s.t.\ 
	\begin{equ}
		u = - (\CK^{(\eps)}_{\bar\gamma} + R_\gamma \CR^{(\eps)}) ( \bR^+ u^ 3) - v^{(\eps)} + G \phi_0
	\end{equ}
	has a unique solution $\CS^n_T(\eps, \phi_0)$ in $\left(\CD_{H; T}^{\gamma,\eta}\right)_n$.  Furthermore, there exists a $\delta_n > 0$, s.t.\ the map 
	\begin{equs}
		S^n_{T} \colon B_{\delta_n}(\eps) \times B_{\delta_n}(\phi_0) & \longrightarrow \left(\CD^{\gamma,\eta}_{H;T} \right)_n \\
		(\eps', \phi_0') &\longmapsto S^n_T(\eps', \phi_0')
	\end{equs}
	is well-defined and continuous.
\end{theorem}
We note that since $\bxi \in \cA^{\Cl}_\infty$ is a bounded noise, we have, for all $\eps >0$, that all the trees in the regularity must also represent bounded operator-valued distributions. In particular, $T(\infty, \eps, \phi_0) \eqdef \inf_{n \in \N}T(n, \eps, \phi_0) > 0$ if $\phi_0 \in \CC^\eta(\T^3 ; \cA^{\Cl}_\infty)$. In this case, there exists a $u_{(\eps)} \in \CD^{\gamma,\eta}_{H;T}$, s.t.\ for all $n \in \N$, $ \pi_n\left( u_{(\eps)} \right)  = \CS_T^n(\eps, \phi_0)$ and $\phi^{(\eps)} \eqdef \CR^{(\eps)} u_{(\eps)}$ satisfies, analogously to the mezdonic case,
\begin{equ}
		(\partial_{t} - \Delta + m^2) \phi^{(\eps)} = - \Big( \left(\phi^{(\eps)}\right) ^{3} - \left(C_\eps^{1} \Delta^{(1)}  - C_\eps^{2} \Delta^{(2)}\right) \phi^{(\eps)} \Big) + \bxi^{(\eps)}\;,
\end{equ}
where $C_{\eps}^1, C_{\eps}^2 \in Z(\cA^{\Cl}_F)$ are the counterterms coming from $\<2>$, $\<22_2>$, and we have used the same expressions for $q=-1$ as in \eqref{eq:CTOp}, i.e.\
\begin{equs}
	\Delta^{(1)} &\eqdef 2 + \Delta_{-1} \; ,\\
	\Delta^{(2)} &\eqdef 1 - \Delta_{-1}^2 = 1-1 = 0  \; .
\end{equs}
In particular, the logarithmic counterterm does not appear in the renormalised equation, although it is necessary for the solution theory. If we assume that $\phi_0$ is contained in the odd part of the algebra, then $\phi^{(\eps)}$ must be as well and therefore $\Delta_{-1} \phi^{(\eps)} = - \phi^{(\eps)}$. Under these assumptions, the renormalised equation reduces to 
\begin{equ}[eq:Phi43ClifRen]
		(\partial_{t} - \Delta + m^2) \phi^{(\eps)} = -  \left(\phi^{(\eps)}\right) ^{3} + C_\eps^{1}  \phi^{(\eps)}  + \bxi^{(\eps)}\; . 
\end{equ}

Finally, we note that the energy functional corresponding to the $\Phi^4_3$-equation
\begin{equ}
	\int_{\T^3} \frac{1}{2} |\nabla \phi|^2 + \frac{m^2}{2} \phi^2 + \frac{1}{4}\phi^4  
\end{equ} 
is a self-adjoint, positive element of $\cA_F^{\Cl}$ (for a self-adjoint initial condition). Thus, there might be hope that one can use energy estimates to show that a solution to \eqref{eq:Phi43ClifRen} remains bounded in the noncommutative localised space $\CL^{2,\infty}(\cA_F^{\Cl}, \bomega_F)$, defined in \cite[Remark~2.21]{CHP23}. In particular, this would show that one can solve this equation in $\CL^2(\CA_F^{\Cl}, \omega_F)$, i.e.\ as unbounded operators affiliated to the original algebra, not just the extended algebra.

\subsection{The Higgs-Yukawa\TitleEquation{{}_2}{2} Model}\label{sec:HiggsYukawa2}

In this section, we shall show how one can apply the theory of noncommutative regularity structures to establish the local in time and space well-posedness of the stochastic quantisation equation of the Higgs-Yukawa$_2$ model
\begin{equ}
	\int\limits_{\R^2} \frac{1}{2} |\nabla \phi|^2 + \frac{m^2}{2} \phi + \Braket{\bar u, (\snabla - M) u}_{\R^2} + \frac{\lambda}{4} \phi^4 + g \phi \Braket{\bar u, u}_{\R^2} \d x \; .
\end{equ}
After some preprocessing, further explained in \cite{CHP23}, the system of PDEs reads
\begin{equs}[eq:YH_model]
    \left( \partial_t  - \Delta + m^2 \right) \phi &= - g  \bilin{(-\overline{\snabla}+M)\bar \upsilon , (\snabla+M)\upsilon} - \lambda \phi^3 +  \xi,	\\
    \left( \partial_t  - \Delta + M^2 \right) \upsilon &=  - g  \phi (\snabla+M) \upsilon  + \psi,\\
    \left( \partial_t  - \Delta + M^2 \right) \bar\upsilon & =  - g \phi (-\overline{\snabla}+M)\bar\upsilon + \bar\psi\;.
\end{equs}
Here $\xi = \xi_B$ is the bosonic $L^2(\R^3)$-white noise, $\psi, \bar\psi$ are the first two and second two components respectively of the four components of the Dirac noise $\Psi = (\psi, \bar \psi)$ over the Hilbert space $\mfH \eqdef \mfh \otimes \C^2$ with $\mfh \eqdef L^2(\T_2) \otimes H^{-\frac{1}{2}}(\T^2) \otimes \C^2$\footnote{Due to the constraint that the domain of the fermions be compact, necessary for our localised BPHZ estimate, we restrict the temporal interval to $[-2,2]$ and identify it with a one-dimensional torus $\T_2$. This necessitates us to restrict ourselves to a maximal possible time of existence, e.g.\ $T=1$.} and covariance
\begin{equ}[def:chargeconj]
		U = \begin{pmatrix}
			0 & V \\
			- \kappa V \kappa & 0
		\end{pmatrix} \eqdef \sqrt{-\Delta+M^2} \begin{pmatrix}
			0 & \big( -\snabla+M \big)^{-1}\\
			-\bigl(\overline{\snabla}+M\bigr)^{-1} & 0
		\end{pmatrix} \;.
\end{equ}
We let $\CA_B\eqdef \CA_B \left( L^2(\R \times \T^2) \right)$ and $\cA_F \eqdef \cA_F \left( \mfH \right)$. 

To formulate this equation in terms of an ($\CA_B$-random) $\cA_F$-regularity structure, we set the scaling to $\s = (2,1,1)$, and we let $\bar K_0 \in \cC^\infty(\R^3 \setminus \{0\})$ denote the kernel of $(\partial_t -\Delta + m^2)^{-1}$ on $\R^{1+3}$ with decomposition $\bar K = K + R$ and we let $\bar K^F_{ij}, \bar K^A_{ij}\in \cC^\infty(\R^3 \setminus \{0\})$ for $i,j \in [2]$ denote the matrix components of the kernels
\begin{equ}
	\frac{\snabla + M}{\partial_t - \Delta + M^2 } \; , \text{ and } \;  \frac{- \overline{\snabla} + M}{\partial_t - \Delta + M^2 }
\end{equ}
respectively; with their own decompositions $\bar K^{F/A}_{ij} = K^{F/A}_{ij} + R^{F/A}_{ij}$.\footnote{We note that, rather than keeping the differential operator $\snabla$ around explicitly, it is more straightforward to directly incorporate it into the integration kernel.} We also fix $\kappa > 0$ sufficiently small.

The required regularity structure is built by setting $\mfL^{-}\eqdef\{\mfb, \mff_1, \mff_2, \bar \mff_1, \bar \mff_2 \}$ with $|\mfb|_{\s} \eqdef
 - 2-\kappa$, $|\mff_i|_{\s} = |\bar\mff_i|_{\s} =  -\frac{3}{2}-\kappa$, and $\mfL^+ =\{ \mft_0 \} \cup \left\{ \mft_{ij} \, \big| \, i,j \in [2] \right\} \cup \left\{ \bar\mft_{ij} \, \big| \, i,j \in [2] \right\}$ with $|\mft_0|_{\s} = 2$, $|\mft_{ij}|_{\s} = |\bar\mft_{ij}|_{\s} = 1$, and rule $\cR$ given by
\begin{gather*}
	\forall \mfk \in \mfL^- \times \N : \cR(\mfk) = \{\emptyset\} \;, \\
	\forall \mft \in \mfL^+ \, \forall r \geqslant 1  : \cR( (\mft, r))  = \{\emptyset\}\;, \\
	\cR((\mft_0, 0)) = \bigl\{ \{(\mft_0, 0), (\mft_0, 0) , (\mft_0, 0)\}, \{(\mft_0, 0), (\mft_0, 0) \}, \{(\mft_0, 0) \}, \{(\mfb, 0) \} , \emptyset  \bigr\} \cup \\
	\qquad \cup \{ \{(\mft_{ij},0)\},  \{(\bar\mft_{ij},0), (\mft_{ik},0)\}  \, \big| \, i,j,k \in [2] \} \; , \\
	\cR((\mft_{ij},0)) = \left\{ \left\{ (\mft_{jk},0) , (\mft_0,0) \right\} , \left\{ (\mft_{jk},0)  \right\} , \left\{ (\mft_0,0) \right\} , \left\{ (\mff_j,0) \right\} \, \middle| \, k \in [2] \right\} \; ,\\
	\cR((\bar\mft_{ij},0)) = \left\{ \left\{ (\bar\mft_{jk},0) , (\mft_0,0) \right\} , \left\{ (\bar\mft_{jk},0)  \right\} , \left\{ (\mft_0,0) \right\} , \left\{ (\bar\mff_j,0) \right\} \, \middle| \, k \in [2] \right\} \; .
\end{gather*}
\begin{remark}
	In fact, this rule allows for all possible $\C$-bilinear nonlinearities $\C^4 \times \C^4 \to \C$ rather than just $\scal{\bigcdot , \bigcdot}_{\R^2}$ that we are using.
\end{remark}
For the sake of legibility, we shall suppress the vector indices of fermionic noises and integration kernels when describing the regularity structure now. Letting $\<0> \eqdef \Xi_{(\mfb,0)}$, $\<0F> \eqdef \Xi_{(\mff_i, 0)}$, and $\<0A> \eqdef \Xi_{(\bar\mff_i, 0)}$, and denoting by $\<E>$ and $\<F>$ denote integration w.r.t.\ $\CI_{(\mft_0, 0)}$, and $\CI_{(\mft_{ij},0)}$ or  $\CI_{(\bar\mft_{ij},0)}$ respectively, the regularity structure $\cT_F$ is spanned by the trees
\begin{equs}
	\1 \, ,  \; \<0> \, ,  \; \<0F> \, , \; \<0A> \, , \; \<1F> \, , \; \<1A> \, ,  \; \<1> \, , \; \<2>  \, ,  \; \<3> \, ,  \; \<2AF> \, ,  \; \<01F> \, , \; \<01A> \,, \; \<02F> \, , \; \<02A> \, , \; \<1F2A> \, ,  \; \<1A2F> \, ,  \; \<2A2F>\, ,  \; \<2F2A> \; .
\end{equs}
Here we have for example identified $\<2AF>$ with $-\<2FA>$ as well as $\<2F>$ with $\<2F_R>$ by enforcing the supercommutativity of the algebra $\CA_B \wotimes \cA_F$ at the level of the trees and we have cut off the regularity structure at regularity $\frac{1}{2} + \kappa +$ and we are left with
\begin{equ}
	A  = \left\{ -2-\kappa, -\frac{3}{2}-\kappa, -1-2\kappa, -\frac{1}{2} - \kappa,  -3\kappa, -2\kappa, - \kappa, 0 , \frac{1}{2} - 2 \kappa , \frac{1}{2}-\kappa\right\}  \; .
\end{equ}
The canonical model is defined as in Section~\ref{sec:CanMod} and the BPHZ character and renormalisation are as in Section~\ref{sec:BPHZFermi}. 

% However, for the fermions we have to define the BPHZ character slightly differently from the commutative and Mezdonic case as 
% \begin{equ}
% 	\Psi_F(f) \Psi_F(g) = \Psi_F^{\diamond 2}(f \otimes g) + \left[ \alpha_F(\kappa U f), \alpha_F(g)^\dagger \right]_+ \neq \Psi_F^{\diamond 2}(f \otimes g) + \Braket{f,g} \bone \; .
% \end{equ}
% which is not correctly reflected the algebra structure $\mfA_N$. In order to proceed in that formalism we would need uniform bounds w.r.t.\ $\mfA_N(P_b \mfH)$ for all $b \in \Gr(\mfH)$, where $P_b \colon \mfH \to b$ is the canonical projection. Since $P_b$ necessarily breaks translational invariance once $\dim b > 1$, we can no longer apply the theory of \cite{HS24} as is. However, due to the fundamental Gaussian structure, we can proceed as usual. 

We will employ a similar kernel notation to the one above. In particular, we have 
\begin{equs}
	\Pi_x^{(\eps)} \<2AF>(z) &= \boldsymbol{\bar\psi}\bigg(\begin{tikzpicture}[scale=0.6,baseline=0cm]
		\node at (0,0) [circ, label=below: $z$] (middle) {};
		\node at (0,1) [circ] (above) {};
		\draw[fkepsilon] (above) to (middle);
	\end{tikzpicture}  \bigg) \bpsi\bigg(\begin{tikzpicture}[scale=0.6,baseline=0cm]
		\node at (0,0) [circ, label=below: $z$] (middle) {};
		\node at (0,1) [circ] (above) {};
		\draw[fkepsilon] (above) to (middle);
	\end{tikzpicture}  \bigg) = \\
	&= \bPsi_F^{\diamond 2}\bigg(\begin{tikzpicture}[scale=0.6,baseline=0cm]
		\node at (0,0) [circ, label=below: $z$] (middle) {};
		\node at (-0.5, 0.866025) [circ] (left) {};
		\node at (0.5, 0.866025) [circ] (right) {};
		\draw[fkepsilon] (left) to (middle);
		\draw[fkepsilon] (right) to (middle);
	\end{tikzpicture}  \bigg) + \left[ \balpha_F\bigg(\kappa U \begin{tikzpicture}[scale=0.6,baseline=0cm]
		\node at (0,0) [circ, label=below: $z$] (middle) {};
		\node at (0,1) [circ] (above) {};
		\draw[fkepsilon] (above) to (middle);
	\end{tikzpicture}  \bigg) , \balpha_F\bigg(\begin{tikzpicture}[scale=0.6,baseline=0cm]
		\node at (0,0) [circ, label=below: $z$] (middle) {};
		\node at (0,1) [circ] (above) {};
		\draw[fkepsilon] (above) to (middle);
	\end{tikzpicture}  \bigg) \right]_+ \; . 
\end{equs}
Here 
\begin{equ}
	\begin{tikzpicture}
		\node at (0,0) [circ, label = below: $\bar z$] (left) {} ;
		\node at (1,0) [circ, label = below: $z$] (right) {} ;
		\draw[fkepsilon] (left) to (right) ;
	\end{tikzpicture}
\end{equ}
denotes integration against one the components of the mollification of $K_{ij}(z -\bar z)$ or its complex conjugate. Similarly, 
\begin{equ}
	\begin{tikzpicture}
		\node at (0,0) [circ, label = below: $\bar z$] (left) {} ;
		\node at (1,0) [circ, label = below: $z$] (right) {} ;
		\draw[fkernel] (left) to (right) ;
	\end{tikzpicture}
\end{equ}
will denote integration against the kernel $K_{ij}(z-\bar z)$ itself or its complex conjugate.

Explicitly, the non-zero components of $\ell^{F;(\eps)}_{\BPHZ}$ are $\ell_{\BPHZ}^{F ; (\eps)}(\<2>)$, defined in the usual way, 
\begin{equ}
	\ell^{F;(\eps)}_{\BPHZ}(\<2AF>) \eqdef -  \left[ \balpha_F\bigg(\kappa U \begin{tikzpicture}[scale=0.6,baseline=0cm]
		\node at (0,0) [circ, label=below: $0$] (middle) {};
		\node at (0,1) [circ] (above) {};
		\draw[fkepsilon] (above) to (middle);
	\end{tikzpicture}  \bigg) , \balpha^\dagger_F\bigg(\begin{tikzpicture}[scale=0.6,baseline=0cm]
		\node at (0,0) [circ, label=below: $0$] (middle) {};
		\node at (0,1) [circ] (above) {};
		\draw[fkepsilon] (above) to (middle);
	\end{tikzpicture}  \bigg) \right]_+ 
\end{equ}
and 
\begin{equ}
	\ell^{F;(\eps)}_{\BPHZ}(\<1F2A>) \eqdef -  \left[ \balpha_F\bigg(\kappa U \begin{tikzpicture}[scale=0.6,baseline=0cm]
		\node at (0,-0.5) [circ, label=below: $0$] (below) {};
		\node at (0.4,0.25) [icirc] (middle) {};
		\node at (0,1) [circ] (above) {};
		\draw[fkepsilon] (above) to (middle);
		\draw[fkernel] (middle) to (below);
	\end{tikzpicture}  \bigg) , \balpha^\dagger_F\bigg(\begin{tikzpicture}[scale=0.6,baseline=0cm]
		\node at (0,0) [circ, label=below: $0$] (middle) {};
		\node at (0,1) [circ] (above) {};
		\draw[fkepsilon] (above) to (middle);
	\end{tikzpicture}  \bigg) \right]_+ 
\end{equ}
with $\ell^{F;(\eps)}_{\BPHZ}(\<1A2F>)$ defined analogously. We have suppressed the contraction in the argument of the character, as there is but one possible contraction in each case.

The existence of and convergence to $ \Pi_x^{F ; R}\<2AF>$, equivalently those of $\widetilde{\Pi}_x^{F , 3 ; R} \<2AF>$, was already directly shown in \cite{CHP23}. We will look at the terms appearing in $\widetilde \Pi_x^{F,(\eps);R} \<2F2A>$. For each $b \in \Gamma_\infty^U$, $\pi_b \left(\Pi_x^{(\eps);R} \<2F2A>\right)$ has components in the $0^{\text{th}}$ and $2^{\text{nd}}$ component w.r.t.\ the fermionic decomposition and only components in the first component w.r.t.\ to the bosonic one. 
These are 
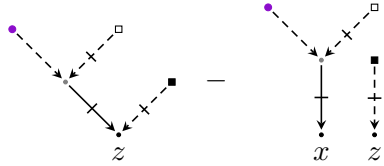
\begin{equ}
  \begin{tikzpicture}[/pgf/fpu/install only={reciprocal}, scale = 0.7]
    \node at (0,-1) [circ, label=below: $z$] (r) {};
    \node at (-1,0) [icirc] (1) {};
    \node at (1,0) [square] (2) {};
    \node at (-2,1) [bdot] (3) {};
    \node at (0,1) [wsquare] (4) {}; 
    \draw[kepsilon] (3) to (1);
    \draw[fkernel] (1) to (r);
    \draw[fkepsilon] (2) to (r);
    \draw[fkepsilon] (4) to (1);
  \end{tikzpicture} \quad \raisebox{1.08cm}{$-$}  \quad \raisebox{0cm}{
  \begin{tikzpicture}[/pgf/fpu/install only={reciprocal}, scale = 0.7]
    \node at (-2, 1) [bdot] (lu) {};
    \node at (-1, -1.41421) [circ, label=below: $x$] (ld) {};
    \node at (0,-1.41421) [circ, label=below: $z$] (r) {};
    \node at (-1,0) [icirc] (1) {};
    \node at (0,0) [square] (2) {};
    \node at (0,1) [wsquare] (4) {};
    \draw[kepsilon] (lu) to (1);
    \draw[fkernel] (1) to (ld);
    \draw[fkepsilon] (2) to (r);
    \draw[fkepsilon] (4) to (1);
  \end{tikzpicture}
  }
\end{equ}
and 

\begin{equ}\label{eq:non-local_divergence}
  \begin{tikzpicture}[/pgf/fpu/install only={reciprocal}, scale = 0.7]
    \node at (0,-1) [circ, label=below: $z$] (r) {};
    \node at (-1,0) [icirc] (1) {};
    \node at (1,0) [icirc] (2) {};
    \node at (-2,1) [bdot] (3) {};
    \node at (0,1) [icirc] (4) {};
    \draw[wavy] (2) to (4) ; 
    \draw[kepsilon] (3) to (1);
    \draw[fkernel] (1) to (r);
    \draw[fkepsilon] (2) to (r);
    \draw[fkepsilon] (4) to (1);
  \end{tikzpicture} \quad \raisebox{1.08cm}{$-$} \quad \begin{tikzpicture}[/pgf/fpu/install only={reciprocal}, scale = 0.7]
    \node at (-1.4, 0.707107) [bdot] (lu) {};
    \node at (-1.4, -0.707107) [circ, label=below: $z$] (ld) {};
    \node at (0,-1) [circ, label=below: $z$] (r) {};
    \node at (-1,0) [icirc] (1) {};
    \node at (1,0) [icirc] (2) {};
    \node at (0,1) [icirc] (4) {};
    \draw[wavy] (2) to (4) ; 
    \draw[kepsilon] (lu) to (ld);
    \draw[fkernel] (1) to (r);
    \draw[fkepsilon] (2) to (r);
    \draw[fkepsilon] (4) to (1);
  \end{tikzpicture}\quad \raisebox{1.08cm}{$-$} \quad \raisebox{0cm}{
  \begin{tikzpicture}[/pgf/fpu/install only={reciprocal}, scale = 0.7]
    \node at (-2, 1) [bdot] (lu) {};
    \node at (-1, -1.41421) [circ, label=below: $x$] (ld) {};
    \node at (1,-1.41421) [circ, label=below: $z$] (r) {};
    \node at (-1,0) [icirc] (1) {};
    \node at (1,0) [icirc] (2) {};
    \node at (0,1) [icirc] (4) {};
    \draw[wavy] (2) to (4) ; 
    \draw[kepsilon] (lu) to (1);
    \draw[fkernel] (1) to (ld);
    \draw[fkepsilon] (2) to (r);
    \draw[fkepsilon] (4) to (1);
  \end{tikzpicture}
  } \; \raisebox{1.08cm}{.}
\end{equ}
Here again 
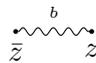
\begin{equ}
	\begin{tikzpicture}[/pgf/fpu/install only={reciprocal}]
		\node at (0,0) [circ, label = below: $\bar z$] (left) {} ;
		\node at (1,0) [circ, label = below: $z$] (right) {} ;
		\draw[wavy] (left) to (right) ;
	\end{tikzpicture}
\end{equ}
denotes integration against the kernel $(1-\Delta)^{-\frac{1}{2}}P_b(1-\Delta)^{-\frac{1}{2}}$ and $P_b$ is the $L^2$-projection
\begin{equ}
	P_b(x,y) = \sum_{i = 1}^{\dim(b)} \left((1-\Delta)^{-\frac{1}{2}} \phi_i \right)(x)  \overline{\left((1-\Delta)^{-\frac{1}{2}} \phi_i \right)(y)} 
\end{equ}
where $\left\{ \phi_i \right\}_i$ is an ONB of $b \subset \mfH$. %When leaving out the subscript $b$, we mean $P_b = \bone_{L^2}$. 

With the regularity structure and the BPHZ renormalised model at hand, we return to the equation. First, let $V_B$ be the sector of $\cT_F$ spanned by bosonically planted trees, i.e.\ the sector spanned by $\1$ and $\<1>$. Let $V_F$ be the sector spanned by fermionically planted trees, i.e.\ spanned by $\1$, $\<1F>$, $\<1A>$, $\<01F>$, $\<01A>$, $\<02F>$, and $\<02A>$.

The mild version of \eqref{eq:YH_model} lifts to the regularity structure as 
\begin{equs}[eq:YH_RS]
	U & = \left( \CK^{(\eps)}_{B ; \bar\gamma_B} + R_{B ; \gamma_B} \CR^{(\eps)} \right) \bR^+\left(  - \lambda U^3 - g \Braket{\bar Y , Y}_{\R^2} +  \<0>\right) + G_B \phi_0 \\
	Y &= \left( \CK^{(\eps)}_{F ; \bar\gamma_F} + R_{F ; \gamma_F} \CR^{(\eps)} \right) \bR^+ \left( -g U Y + \<0F> \right) + G_F \upsilon_0\\
	\overline Y &= \left( \CK^{(\eps)}_{\bar F ; \bar\gamma_F} + R_{\bar F ; \gamma_F} \CR^{(\eps)} \right) \bR^+ \left( -g U \overline Y + \<0A> \right) + G_{\bar F} \bar \upsilon_0
\end{equs}
where $\CK^{(\eps)}_{B ; \bar \gamma}$, $R_{B ; \gamma}$ denote the implementation of $(\partial_t - \Delta + m^2)^{-1}$, $ \CK^{(\eps)}_{F ; \bar\gamma}$,  $ R_{F ; \gamma}  $ denote (the matrix) of implementations of $(\partial_t - \Delta + M^2)^{-1} \left(\snabla + M\right)$ and analogously for $\CK^{(\eps)}_{\bar F ; \bar\gamma}$,  $ R_{\bar F ; \gamma} $. $G_B \phi_0$, $G_F \upsilon_0$, and $G_{\bar F} \bar \upsilon_0$ are the lifts to the regularity structure of the initial conditions respectively in $\CC^{\eta_B}(\T^2 ; \cG^0_F)$, and $\CC^{\eta_F}(\T^2 ; \cG^1_F \otimes \C^2)$.

Here we have set $\gamma > \frac{1}{2}+\kappa$, $\bar\gamma = \gamma - \frac{1}{2}-\kappa$,  $\eta_B \leqslant - \kappa $, and $\eta_F \leqslant - \frac{1}{2} - \kappa$. W.l.o.g.\ we assume that $\eta_F \leqslant \eta_B$. Performing powercounting \`a la Theorem~\ref{thm:RSMult} shows that the bosonic nonlinearity is a continuous map 
\begin{equ}
	\CD^{\gamma, \eta_B}_{H} (V_B) \times \CD^{\gamma, \eta_F}_{H} (V_F) \otimes \C^4 \longrightarrow \CD^{\gamma - \frac{1}{2} - \kappa, 2\eta_F}_H
\end{equ}
and the fermionic nonlinearities are continuous maps 
\begin{equ}
	\CD^{\gamma, \eta_B}_{H} (V_B) \times \CD^{\gamma, \eta_F}_{H} (V_F) \otimes \C^2 \longrightarrow \CD^{\gamma - \frac{1}{2} - \kappa, 2\eta_F}_H \; . 
\end{equ}

After replacing the noises $\<0>$, $\<0F>$, and $\<0A>$ with the correctly cut-off in time versions, as in Section~\ref{sec:Phi43Mez} we obtain the following existence theorem. 
\begin{theorem}
	 Let $\kappa \in \left(0, \frac{1}{6}\right)$, $\gamma > \frac{1}{2}+\kappa$,  $\eta_B \in [\eta_F, - \kappa)$, and $\eta_F \in \big( -1, -\frac{1}{2}-\kappa \big) $. For every $\eps \in [0,1]$, $n \in \N$, $\phi_0 \in \CC^{\eta_B}(\T^2 ; \cG_F^{0})$, and $\upsilon_0, \bar \upsilon_0 \in \CC^{\eta_F}(\T^2 ; \cG^1_F\otimes \C^2 )$ there exists $0< T(n, \eps , \phi_0, \upsilon_0, \bar\upsilon_0) \leqslant 1 $, s.t.\ \eqref{eq:Phi43ClifRen} has a unique solution $\CS^n_T(\eps, \phi_0, \upsilon_0, \bar\upsilon_0)$ in $\left(\CD_{H; T}^{\gamma,\eta}\right)_n$.  Furthermore, there exists a $\delta_n > 0$, s.t.\ the map 
	\begin{equs}
		S^n_{T} \colon B_{\delta_n}(\eps) \times B_{\delta_n}(\phi_0, \upsilon_0 ,\bar\upsilon_0) & \longrightarrow \left( \CD^{\gamma,\eta}_{H;T} \right)_n \\
		(\eps', \phi_0', \upsilon'_0 ,\bar\upsilon'_0) &\longmapsto S^n_T(\eps', \phi_0', \upsilon'_0 ,\bar\upsilon'_0)
	\end{equs}
	is well-defined and continuous in probability.
\end{theorem}
Here $\cG_F^{0}$ and $\cG_F^{1}$ denote the even and odd parts of the algebra $\cG_F$ respectively.

Again, since $\bpsi, \overline{\bpsi} \in \cA_\infty \cap \cG_F$, we have for $\eps > 0$ that $T(\infty, \eps, \dots) \eqdef \inf_{n \in \N} T(n, \eps, \dots) > 0$ if the initial conditions are bounded. Thus, we can find solutions $U_{(\eps)} \in \CD^{\gamma, \eta}_{H ; T}$, $Y_{(\eps)}, \overline{Y}_{(\eps)} \in \CD^{\gamma, \eta}_{H ; T} \otimes \C^2$, s.t.\ for all $n \in \N$, $\pi_n(U_{(\eps)}) = S^n_T(\eps, \phi_0, \upsilon_0, \bar\upsilon_0)$ and analogously for the fermionic components of the equation.

A quick computation shows that the reconstructions $\phi^{(\eps)} \eqdef \CR^{(\eps)} U_{(\eps)}$, $\upsilon^{(\eps)} \eqdef \CR^{(\eps)} Y_{(\eps)}$, $\bar{\upsilon}^{(\eps)} \eqdef \CR^{(\eps)} \overline{Y}_{(\eps)}$ then satisfy
\begin{equs}
	(\partial_t - \Delta + m^2) \phi^{(\eps)} &= - \lambda \left(\phi^{(\eps)}\right)^3  - g \Braket{(- \overline{\snabla}+M)\overline{\upsilon}^{(\eps)}, (\snabla + M) \upsilon^{(\eps)}}_{\R^2} + \\
	& \qquad + \left( \lambda C^1_\eps - g^2 C^3_{\eps}\right) \phi^{(\eps)} + g C^2_{\eps} + \xi^{(\eps)} \\
	(\partial_t - \Delta + M^2) \upsilon^{(\eps)} &= - g \phi^{(\eps)} (\snabla + M) \upsilon^{(\eps)} + \bpsi^{(\eps)}\\
	(\partial_t - \Delta + M^2) \bar\upsilon^{(\eps)} &= - g \phi^{(\eps)} (-\overline\snabla + M) \bar\upsilon^{(\eps)} + \bpsi^{(\eps)}
\end{equs}
with 
\begin{equs}\label{eq:HiggsYukawa_constants}
	C^1_{\eps} & \eqdef \ell^{F;(\eps)}_{\BPHZ}(\<2>) \in \R \; , \\
	C^2_{\eps} & \eqdef \ell^{F;(\eps)}_{\BPHZ}(\<2AF>) \in Z(\cA_\infty) \; ,\\
	C^3_{\eps} & \eqdef \ell^{F;(\eps)}_{\BPHZ}(\<1F2A>)+\ell^{F;(\eps)}_{\BPHZ}(\<1A2F>) \in Z(\cA_\infty) \; .
\end{equs}
Here $C^3_\eps$ is the logarithmically diverging mass renormalisation of the boson. The fact that renormalisation constants do not appear in the fermionic equations is not generic; in the three-dimensional case, such counterterms would appear.

\appendix

\section{Functional Analysis}
\label{appendix:FuncAna}

\subsection{Topological Vector Spaces}

Let $E$ be a topological vector space. A subset $C \subset E$ is called balanced if $\lambda C \subset C$ for all $\lambda \in \mathbb{K}$ with $|\lambda| \leqslant 1$.

A subset $H \subset E'$ is called equicontinuous if there exists an open neighbourhood $U \subset E$ of $0$, s.t.\
\begin{equ}
	\bigcup_{h \in H} h(U) \subset (-1,1) \; .
\end{equ}
The polar of a subset $ A \subset E$ is the set
\begin{equ}[eq:def_of_polar]
	A^\circ \eqdef \left\{ e' \in E' \, \middle| \, \forall e \in A : \left|\scal{e',e}\right| \leqslant 1  \right\}
\end{equ}

\begin{definition}
\label{definition:MackeyTop}
	Let $E$ be a topological vector space. The Mackey topology on the continuous dual $E'$ is generated by the neighbourhood basis of polars $K^\circ$ of $K \subset E$, such that $K$ is weakly compact, convex, and balanced. $E'$ equipped with this topology will be denoted by $E'_\tau$.
\end{definition}

\begin{definition}
\label{definition:ContLinMap}
	Let $E$, $F$ be topological vector spaces. We denote by $\CL(E;F)$ the space of continuous linear maps.

	We denote by $\CL_\eps(E'; F)$ the space $\CL(E';F)$ equipped with the topology generated by the neighbourhood basis of $0$ of sets of the form
	\begin{equ}
		\cU(A,U) \eqdef \left\{ \phi \in \CL(E';F) \, \middle| \, \phi(A) \subset U \right\}
	\end{equ}
	where $A\subset E'$ is equicontinuous and $U \subset F$ is open.
\end{definition}

\begin{definition}[Bilinear Forms]\label{definition:bilinearforms}
	Let $E,F$ be two topological vector spaces. We denote by $B(E,F)$ and $\cB(E,F)$ respectively the space of jointly continuous and separately continuous bilinear forms $E \times F \to \mathbb{K}$.
\end{definition}

\begin{lemma}
\label{lem:CanIso}
	Let $E,F$ be locally convex topological vector spaces. Then the algebraic tensor product $E \otimes F$ is canonically isomorphic with $B(E'_\sigma , F'_\sigma)$ % \martinp{Finish stating isomorphism $E \otimes F$ with $B(E'_\sigma , F'_\sigma)$ and adding note below.}
\end{lemma}
\begin{proof}
	See \cite[Proposition~42.4]{Trev67}.
\end{proof}

\begin{definition}[\TitleEquation{\eps}{eps}-Topology]
	\label{definition:epsTopo}
	Let $E,F$ be two topological vector spaces, and let $E'_\sigma$ and $F'_\sigma$ denote the (continuous) dual spaces equipped with the weak$^*$ topologies. The $\eps$-topology on $B(E'_\sigma,F'_\sigma)$ and $\cB(E'_\sigma,F'_\sigma)$ is defined by the neighbourhood basis of $0$ given by
	\begin{equs}
		\cU (A, B; \eps) &\eqdef \left\{ \phi \in B(E'_\sigma,F'_\sigma) \, \middle| \, \phi(A,B) \subset (-\eps,\eps) \right\} \\
		\cU' (A, B; \eps) &\eqdef \left\{ \phi \in \cB(E'_\sigma,F'_\sigma) \, \middle| \, \phi(A,B) \subset (-\eps,\eps) \right\}
	\end{equs}
	where $\eps > 0$ and $A \subset E'$, $B \subset F'$ are equicontinuous subsets. The spaces of bilinear forms with these topologies will be denoted by  $B_\eps(E'_\sigma,F'_\sigma)$ and $\cB_\eps(E'_\sigma,F'_\sigma)$ respectively.
\end{definition}

\begin{definition}[Injective Tensor Product]
	\label{definition:InjTensProd}
	Let $E,F$ be complete, Hausdorff, locally convex spaces. The injective tensor product of $E$ and $F$ is the closure of $B_\eps(E'_\sigma,F'_\sigma)$ in $\cB_\eps(E'_\sigma,F'_\sigma)$ and will be denoted by $E \wotimes_\eps F$. Equivalently, it is the completion of $E \otimes F$ equipped with the $\eps$-topology using the canonical isomorphism from Lemma~\ref{lem:CanIso}.
\end{definition}\martin{Why not just the completion of $E \otimes F$ for the topology described above? (Wouldn't that be
the same, at least if $E$ and $F$ are complete?) Is the problem that then you wouldn't have $E \wotimes_\eps \R = E$ if
$E$ itself isn't complete?}\martinp{Fixed}

\begin{proposition}
	\label{prop:TensProdProp}
	Let $E,F$ be complete, Hausdorff, locally convex spaces. Then the following hold
	\begin{enumerate}
		\item $\cB_{\eps}(E'_\sigma, F'_\sigma) \cong \CL_\eps(E'_\tau;F)$.
		\item $\cB_{\eps}(E'_\sigma, F'_\sigma)$ and thus $E \wotimes_\eps F$ are complete iff $E$ and $F$ are complete.
	\end{enumerate}
\end{proposition}
\begin{proof}
	See \cite[Proposition~42.2~\&~42.3]{Trev67}
\end{proof}

\subsection{Proofs of Theorem~\ref{thm:PrimTensProd} and Proposition~\ref{prop:TensProdPropSec} }\label{subsec:PrimTensProdproof}

\begin{proof}[Proof of Theorem~\ref{thm:PrimTensProd}]
	We follow the argumentation of \cite[Theorem 44.1]{Trev67} in the case of $\cC^r(\R^d; E)$.

	We first claim that $\CC^\alpha_\s(\R^d; E)$ can be canonically identified with a subspace of $\CL\left(E'_{\tau} ; \CC^\alpha_\s\left(\R^d\right)\right)$, which is itself isomorphic to $\cB_\eps\left(E'_\sigma ; \left(\CC^\alpha_\s\left(\R^d\right)\right)'_\sigma \right)$.
	Here $E_{\tau}'$ denotes the continuous dual of $E$ equipped with the Mackey topology, cf.\ Definition~\ref{definition:MackeyTop}.

	In order to show that this is the case, we need to prove that for all $\phi \in \CC^\alpha_{\s}(\R^d; E)$, the map $\Phi \colon E' \to \CC^\alpha_{\s}(\R^d)$ given by
	\begin{equ}\label{eq:def_of_Phi}
		e' \longmapsto \left(   \cD(\R^d) \ni   \eta \mapsto \scal{e', \phi(\eta)} \in \R  \right)
	\end{equ}
	is continuous. This means that given $U \subset\CC^\alpha_{\s}(\R^d)$ a neighbourhood of $0$, we need to exhibit a weakly compact set $K \in E$, such that $K^{\circ}$ \dash the polar of $K$, see \eqref{eq:def_of_polar} \dash  is mapped into $U$.

	W.l.o.g.\ we may assume that $U = \left\{  \xi \in \CC^\alpha_{\s}(\R^d ) \, \middle| \, \| \xi \|_{\alpha; \K } \leqslant \eps \right\} $ for $\eps > 0$, and $\K \Subset \R^d$. Let $K\subset E $ be the weak closure of the set
	\begin{equ}
		\left\{ \lambda^{-\alpha} \phi\left( \CS^\lambda_{\s,x} \eta \right) \, \middle| \, \lambda \in (0,1] \, , \, x \in \K \, , \, \eta \in \CB^{r,\alpha}_{\s, 0} \right\} \cup \left\{ \phi\left( \CS^1_{\s,x} \eta \right) \, \middle| \, x \in \K \, , \, \eta \in \CB^r_{\s, 0}\right\} \; .
	\end{equ}
	$K$ is convex and balanced as $\CB^{r,\alpha}_{\s, 0}$ and $\CB^{r}_{\s, 0}$ are.
Furthermore, we claim that $K$ is also weakly compact \dash this follows from an argument analogous to the proof of the Banach-Alaoglu Theorem: for each $e' \in E'$, there exist a finite collection of seminorms $\mfp_1, \dots, \mfp_n \in \mfP$ and constants $C_1, \dots, C_n$, such that, for all $e \in E$,
	\begin{equ}
		\left|\scal{e',e}\right| \leqslant \sum_{i = 1}^n C_i \mfp_i(e) \; .
	\end{equ}
	Thus, for all $e' \in E'$ we have
	\begin{equ}
		\sup_{e \in K} \left| \scal{e', e} \right| \leqslant \sum_{i = 1}^n C_i \| \phi \|_{\alpha ; \K , \mfp_i} \; .
	\end{equ}
	This means that each of the sets $\overline{\scal{e', K}}$ are compact in $\C$ and therefore, by Tykhonov's theorem, $\prod_{e' \in E'} \overline{\scal{e', K}}$ is also compact. As $K$ is a weakly closed subset of $\prod_{e' \in E'} \overline{\scal{e', K}}$, $K$ itself is weakly compact.

	Thus, $K^\circ$ and $\eps^{-1} K^\circ$ are neighbourhoods of $0$ in the Mackey topology, and $\Phi( \eps^{-1} K^\circ ) \subset U$ as for all $e' \in \eps^{-1} K^\circ$, all $\lambda \in (0,1]$, $x \in \K$, and $\eta \in \CB^{r,\alpha}_{\s,0}$
	\begin{equ}
		\left| \Braket{e' , \lambda^{-\alpha} \phi\left( \CS^\lambda_{\s,x} \eta \right) } \right| \leqslant \eps
	\end{equ}
	and for all $\eta \in \CB^{r}_{\s, 0}$
	\begin{equ}
		\left| \Braket{e' , \phi\left( \CS^1_{\s,x} \eta \right) } \right| \leqslant \eps  \; ,
	\end{equ}
	i.e.\ $\| \Phi(e') \|_{\alpha; \K}\leqslant \eps$. Therefore, $\CC^\alpha_{\s}(\R^d ; E) \subset \CL(E_{\tau}' ; \CC^\alpha_{\s}(\R^d))$.

	Now we just need to argue that the injective topology on $\CL(E_{\tau}' ; \CC^\alpha_{\s}(\R^d))$ agrees with the topology of $\CC^\alpha_{\s}(\R^d ; E) $. Since we know that $\CC^\alpha_{\s}(\R^d;E)$ is complete and that it trivially contains the algebraic tensor product $\CC^\alpha_{\s}(\R^d)\otimes E$; this concludes the proof.

	For the sake of convenience we shall change from $\| \bigcdot \|_{\alpha; \K , \mfp}$ to the equivalent seminorm
	\begin{equ}
		\vvvert \bigcdot \vvvert_{\alpha; \K , \mfp} = \sup_{x \in \K} \sup_{\eta \in \CB^{r,\alpha}_{\s,0}} \sup_{\lambda \leqslant 1} \lambda^{-\alpha} \mfp \left( \xi \left( \CS^\lambda_{\s,x} \eta \right) \right) \vee  \sup_{x \in \K} \sup_{\eta \in \CB^{r}_{\s,0}}  \mfp \left( \xi \left( \CS^1_{\s,x} \eta \right) \right)
	\end{equ}

	The injective topology is generated by the neighbourhood sub-basis of 0, given by
	\begin{equ}
		\cU(U,V) \eqdef \left\{ \psi \in \CL(E_{\tau}' ; \CC^\alpha_{\s}(\R^d)) \, \middle| \, \psi(U^\circ) \subset V \right\}
	\end{equ}
	for $U = \left\{  e \in E \, \middle| \, \mfp(e) \leqslant \eps \right\} $ for $\eps > 0$, $\mfp \in \mfP$, and $V = \left\{ \phi \in \CC^\alpha_{\s}(\R^d) \, \middle| \, \vvvert\phi\vvvert_{\alpha; \mfK} \leqslant 1 \right\}$. We now see that
	\begin{equs}
		\label{eq:EqTopos}
		\phi \in \big\{ \rho \in \CC^\alpha_{\s}(\R^d ; E) \, & \big| \, \vvvert \rho \vvvert_{\alpha; \mfK, \mfp} \leqslant \eps \big\}   \iff \| \phi \|_{\alpha; \K, \mfp} \leqslant \eps \iff\\
		& \iff \forall e' \in U^\circ : \left\vvvert \eta \mapsto \Braket{ e' , \phi \left( \CS^\lambda_{\s,x} \eta \right) } \right\vvvert_{\alpha;\K} \leqslant 1  \iff\\
		& \iff \forall e' \in U^\circ : \Phi(\phi)(e') \in V \iff \Phi(\phi) \in \cU(U,V) \; ,
	\end{equs}
	i.e.\ the topology coming from $\CC^\alpha_{\s}(\R^d; E)$ and $\CL(E_\tau'; \CC^\alpha_\s(\R^d))$ agree.
	Above, $\Phi$ is as in \eqref{eq:def_of_Phi}.

\end{proof}

\begin{proof}[Proof of Proposition~\ref{prop:TensProdPropSec}]

	The only technical point here is to note that $\Gamma^{\alpha,\gamma}_{\s}(\R^d; E) $ is a complete space. Since $\R^n$ admits countable compact exhaustions and $E$ is Fr\'echet, $\Gamma^{\alpha,\gamma}_{\s}(\R^d; E)$ is metrisable and we can restrict ourselves to Cauchy sequences. Now suppose that $(F_n)_n$ is a Cauchy sequence, then by the usual arguments it converges weakly to a map  $F \colon \R^d \to \cD'(\R^d ; E)$. This map is measurable as it is the limit of a sequence of measurable maps and thus by the Cauchy property it is an element of $\Gamma^{\alpha,\gamma}_{\s}(\R^d; E)$.

	The rest of the proof proceeds as that of Theorem~\ref{thm:PrimTensProd}. One defines a map $\Phi \colon E' \to \Gamma^{\alpha,\gamma}_{\s}(\R^d ; \mathbb{K})$ via
	\begin{equ}
		e' \longmapsto \left( x \mapsto \left( \eta \mapsto \scal{e' , F_x(\eta)} \right) \right)
	\end{equ}
	This is an element of $\CL(E', \R ) \to \Gamma^{\alpha,\gamma}_{\s}(\R^d ; \mathbb{K})$, by the same argument replacing $K$ with the weak closure of
	\begin{equs}
		& \left\{ \frac{(F_y-F_x)\left( \CS^\lambda_{\s,x} \eta \right)}{\lambda^\alpha \left( \left|y-x \right|_{\s} + \lambda \right)^{\gamma-\alpha}  } \, \middle| \, \eta \in \CB^r_{\s,0} \, , \, \lambda \in (0,1] \, , \, x,y \in \K : x\neq y \right\} \cup \\
		& \qquad  \qquad \qquad \qquad \qquad  \cup \left\{ \frac{F_x\left( \CS^\lambda_{\s,x} \eta \right)}{\lambda^\alpha  } \, \middle| \, \eta \in \CB^r_{\s,0} \, , \, \lambda \in (0,1] \, , \, x \in \K  \right\}
	\end{equs}
	Similarly, the injective topology on the image of $\Phi$ coincides with the one defined by $\| \bigcdot \|_{\Gamma^{\alpha,\gamma}; \K , \mfp}$ by the same sequence of equivalences in \eqref{eq:EqTopos} with $U = \left\{  e\in E \, \big| \, \mfp(e) \leqslant \eps \right\}$ and $V = \left\{ \phi \in \Gamma^{\alpha,\gamma}_{\s}(\R^d) \, \big| \, \|\phi\|_{\Gamma^{\alpha,\gamma}_{\s}; \K} \leqslant 1 \right\}$.

\end{proof}

\subsection{Proof of Lemma~\ref{lemma:2SYoung}}
\label{sec:NonComYng}

\begin{proof}[Proof of Lemma~\ref{lemma:2SYoung}]
	We shall proceed similarly to the proof of Theorem~\ref{thm:YoungMult}. We build an $\cA_q$-regularity structure $\cT$ from the symbols $X^k$, $\<1>X^k$, and $\<1>a\<1>X^k$ for $k \in \N^3$ and $a \in \cA_q$. We assign all these symbols the regularity
	\begin{equ}
		\left| X^k \right| = \left| \<1> X^k \right| = \left| \<1>a\<1>X^k \right| = |k|_{\s}
	\end{equ}
	and we have the structure group be the usual one for polynomial regularity structure, i.e.\
	\begin{equs}
		\Gamma_h X^k &= (X+h)^k \; ,\\
		\Gamma_h \left( \<1> X^k \right) &= \<1> (X+h)^k \; , \\
		\Gamma_h \left( \<1> a \<1> X^k\right) &= \<1> a \<1> (X+h)^k \; .
	\end{equs}
	Finally, we define the model to be
	\begin{equs}
		\left( \Pi_x X^k \right)(y) &= (y-x)^k \; , \\
		\left( \Pi_x \<1> X^k \right)(y) &= \<1>(y) (y-x)^k \; , \\
		\left( \Pi_x \<1> a \<1> X^k \right)(y) &= \CM_q^{(1,1)}\left({\mfZ[\<2>](y)} ; a \right) (y-x)^k \; , \\
	\end{equs}
	which are well-defined as $\<2>,\<1> \in \CC^0_{\s}(\spacetime ; \cA_q)$.

	Now a function $\upsilon \in \CC^\alpha_\s(\spacetime; \cA_q)$ with $\alpha \in (0,\infty) \setminus \N$ uniquely lifts to a modelled distribution $\hat \upsilon \in \CD^\alpha(\cT_{3,\s}) \subset \CD^\alpha$ and $\<1> \in \CD^\infty$, when viewed as the constant modelled distribution. Thus, $\<1> \hat \upsilon \<1> \in \CD^\alpha$ by Theorem~\ref{thm:RSMult}, and $\<2_1Rprime> \eqdef \CR \left( \<1> \hat \upsilon \<1> \right) \in \CC^0_\s$ yields the desired extension of the pointwise multiplication.
\end{proof}

\subsection{$\CA$-Analytic Functions}

\begin{definition}[$\CA$-Analytic Functions]
\label{def:AAnaFct}
	Let $n,m \in \N$ and $k \in [n]^m$. We define the $t$-continuous map $\CA^{m+1} \times \CA^n \to \CA$
	\begin{equ}
		\left\llbracket A_1, \dots,  A_{m+1} ; X^k \right\rrbracket \eqdef A_1 X_{k_1} A_2 \cdots A_{m} X_{k_m} A_{m+1}
	\end{equ}
	with $A_i, X_j \in \CA$. This map is multilinear in its first $m+1$ arguments and extends to a $t$-continuous map $\CA^{\wotimes_\pi (m+1)} \times \CA^n \to \CA$.

	We say that a function $F\colon \CA^n \to \CA$ is $\CA$-analytic if there exist $A_{m,k} \in \CA^{\wotimes_\pi (m+1)}$, s.t.\
	\begin{equ}[eq:APowSer]
		F(X_1, \dots, X_n) = \sum_{m = 0}^\infty \sum_{k \in [n]^m} \llbracket A_{m,k} ; X^k \rrbracket \; .
	\end{equ}
	We say that it is entire if $U = \CA$.

	We will denote these function spaces by $\cC^\omega(\CA^n ; \CA)$ and $\cC^\omega(U^n ; \CA)$ respectively.
\end{definition}

\begin{proposition}
	For a seminorm $\mfp \in \mfP$ of $\CA$, we set the $\mfp$-convergence radius $R_{\mfp}$ of a formal power series with coefficients $\left(A_{m,k}\right)_{m \in \N ,k \in [n]^m}$ with $A_{m,k} \in \CA^{\wotimes_\pi (m+1)}$ to be
	\begin{equ}
		R_{\mfp} \eqdef \Bigl( \limsup_{m \to \infty} \Bigl( \sum_{k \in [n]^m} \mfp^{\wotimes_\pi (m+1)} \left( A_{m,k}\right) \Bigr)^{\frac{1}{m}}\Bigr)^{-1} \; .
	\end{equ}
	The function $F$ defined by $\left(A_{m,k}\right)_{m,k}$ as in \eqref{eq:APowSer} is well-defined and entire if $R_{\mfp} = \infty$ for all $\mfp \in \mfP$. Furthermore, $F \in \cC^\infty(\CA^n ; \CA)$ and $t$-continuous.

	If $\CA$ is a Banach space, then $F \colon U^n \to \CA$ is well-defined and $\CA$-analytic with $U = B_{R}(0)$ where $R$ is the unique convergence radius of $\left(A_{m,k}\right)_{m,k}$.
\end{proposition}
Within this paper, we will only consider $F$ that are themselves lifts of functions $f \colon \mathbb{K}^n \to \mathbb{K}$, which we shall do using holomorphic functional calculus.
\begin{corollary}
	Let $f \colon \mathbb{K}^n \to \mathbb{K}$ be an entire $\mathbb{K}$-analytic function. Then $f$ lifts to a unique $t$-continuous smooth function $F \in \cC^\infty(\CA; \CA)$. If $\CA$ is furthermore a Banach algebra, then we need only assume that $f \colon U \to \mathbb{K}$ is analytic for some open $0 \in U \subset \mathbb{K}$, i.e.\ the power series defining $f$ need only have a finite radius of convergence.
\end{corollary}
\begin{proof}
	Let $(c_\ell)_{\ell \in \N^n}$, s.t.\
	\begin{equ}
		f(x) = \sum_{\ell \in \N^n} c_\ell x^\ell \; .
	\end{equ}
	For each $\ell \in \N^n$ choose some $k \in [n]^{|\ell|}$, s.t.\ $\ell_i  = \left|\left\{j \in [|\ell|] \, \big| \, k_j = i\right\}\right|$. Set $A_{m, k} = c_\ell \bone^{\otimes (m+1)}$ and $A_{m, k'} = 0$ for all $k' \neq k$. Defining $F \colon \CA^n \to \CA$
	\begin{equ}
		F(X) = \sum_{m = 0}^\infty \sum_{k \in [n]^m} [ A_{m,k} ; X^k]  \; ,
	\end{equ}
	it is a straightforward calculation to verify that $R_{\mfp} = \infty$ for all $\mfp \in \mfP$.

	If $\CA$ is a Banach algebra, suppose that $f$ has radius of convergence $R < \infty$ around $0$. Then, $F$ has radius of convergence $R$.
\end{proof}
% \begin{remark}
% 	We remind the reader here that $DF(a)[b]$ is given by
% 	\begin{equ}
% 		DF(a)[b] \sum_{n \in \N} \frac{k_n}{n!} \sum_{i = 0}^{n-1} a^i b a^{n-1-i} \; .
% 	\end{equ}
% \end{remark}
\begin{remark}
	The reason we need to restrict ourselves to Banach algebras when considering functions that are not entire is that we have to be able to restrict ourselves to the set
	\begin{equ}
		\left\{ a \in \CA \, \big| \, \forall \mfp \in \mfP : \mfp(a) < R \right\}
	\end{equ}
	which is open if and only if the topology of $\CA$ can be generated by a single norm.
\end{remark}
\begin{remark}
	To go beyond functions that are analytic, we need to assume at least a locally $C^*$-algebraic structure in order to employ non-holomorphic functional calculus.
\end{remark}

\section{Proofs of Algebraic Propositions}
\label{app:alg}

\begin{proof}[Proof of Lemma~\ref{lemma:CharHomo}]
    We shall prove these statements by induction on the tree structure. For $\1$, $X_\mu$, $\Xi_\mfl$, and multiplication, the group and action properties follow trivially. Furthermore, the action property also follows trivially by induction on trees that are root-renormalised, as
	\begin{equs}
		\Gamma_g \Gamma_{\bar g} \tau & = \Gamma_g \Bigl( \Bigl(\Gamma_{\bar g} (\CF) \otimes \bigotimes_{v \in N_T} \Gamma_{\bar g} \left( X^{\mfn(v)} \right) \Bigr) \curvearrowright \tau_\rho \Bigr) = \\
		& =   \Bigl(\Gamma_g \Gamma_{\bar g} (\CF) \otimes \bigotimes_{v \in N_T} \Gamma_g \Gamma_{\bar g} \left( X^{\mfn(v)} \right) \Bigr) \curvearrowright \tau_\rho  = \\
		&= \Gamma_{g \cdot \bar g} \tau \; .
	\end{equs}
	Therefore, consider an element $\CI_{\mfk} \tau \in \CH$. We have for all $g, \bar g \in \CG(\CH^+)$
	\begin{equs}
		\Gamma_{g} \Gamma_{\bar g} \CI_{\mfk} \tau &= \Gamma_g \left(\CI_{\mfk}(\Gamma_{\bar g} \tau) + \sum_{\ell} \frac{X^\ell}{\ell!} \bar g \left( \CJ_{\mfk+\ell} \tau\right)\right) = \\
		&= \CI_{\mfk}(\Gamma_g \Gamma_{\bar g} \tau) + \sum_{\ell} \frac{X^\ell}{\ell!} g \left( \CJ_{\mfk+\ell} \Gamma_{\bar g} \tau\right) + \sum_{\ell} \frac{\Gamma_{g} \left( X^\ell \right)}{\ell!} \bar g \left( \CJ_{\mfk+\ell} \tau\right) = \\
		&= \CI_{\mfk}( \Gamma_{g \cdot \bar g} \tau) + \sum_{\ell} \frac{X^\ell}{\ell!} g \left( \CJ_{\mfk+\ell} \Gamma_{\bar g} \tau\right) + \sum_{\ell} \frac{\left(  X+g(X) \right)^\ell}{\ell!} \bar g \left( \CJ_{\mfk+\ell} \tau\right)\\
		\Gamma_{g \cdot \bar g} \CI_{\mfk} \tau & = \CI_{\mfk} \left(\Gamma_{g \cdot \bar g} \tau \right) + \sum_{\ell} \frac{X^\ell}{ \ell !} \left(g \cdot \bar g\right)(\CJ_{\mfk + \ell} \tau) =\\
		&= \CI_{\mfk} \left(\Gamma_{g \cdot \bar g} \tau \right) + \sum_{\ell} \frac{X^\ell}{ \ell !}  \biggl(  g\left( \CJ_{\mfk+\ell} \left( \Gamma_{\bar g} \tau \right) \right) +  \sum_{k} \frac{g(X)^k}{k!} \bar g\left( \CJ_{\mfk +k+ \ell} \tau \right)  \biggr) = \\
		&= \CI_{\mfk} \left(\Gamma_{g \cdot \bar g} \tau \right) + \sum_{\ell} \frac{X^\ell}{ \ell !}   g\left( \CJ_{\mfk+\ell} \left( \Gamma_{\bar g} \tau \right) \right) +  \sum_{k} \frac{(X+g(X))^k}{k!} \bar g\left( \CJ_{\mfk +k} \tau \right)  \; ,
	\end{equs}
	where in the last line we used $g(X)$ and $X$ are in the centre of $\CH^+$ and $\CA$.
	\begin{equs}
		\Gamma_{\1^*} \left(\CI_{\mfk} \tau \right) &= \CI_{\mfk} \bigl( \underbrace{\Gamma_{\1^*} \tau}_{ = \tau} \bigr) + \sum_{\ell} \frac{X^\ell}{\ell!} \underbrace{\1^* \left( \CJ_{\mfk + \ell} \tau\right)}_{ = 0} = \CI_{\mfk} \tau\\
		\Gamma_{g^{-1}} \Gamma_{g} \left(\CI_{\mfk} \tau \right) & = \Gamma_{g^{-1}} \biggl( \CI_{\mfk}(\Gamma_g \tau) + \sum_{\ell} \frac{X^\ell}{\ell!} g \left( \CJ_{\mfk+\ell} \tau\right)\biggr) = \\
		&= \CI_{\mfk} \bigl( \underbrace{ \Gamma_{g^{-1}} \Gamma_g \tau}_{ = \tau} \bigr) + \sum_{\ell} \frac{X^\ell}{\ell!} g^{-1} \left(\CJ_{\mfk + \ell} \Gamma_g \tau \right) +\\
		& \qquad + \sum_{\ell!} \frac{\Gamma_{g^{-1}} \left( X^\ell \right)  }{\ell!} g(\CJ_{\mfk+\ell}\tau) = \\
		&= \CI_{\mfk} \bigl( \tau \bigr) -\sum_{k,\ell} \frac{X^\ell}{\ell!} \frac{(-g(X))^k}{k!} g \bigl(\CJ_{\mfk+k + \ell} \bigl( \underbrace{\Gamma_{g^{-1}} \Gamma_g \tau}_{ = \tau} \bigr)\bigr) + \\
		&  \qquad +\sum_{\ell!} \frac{\left( X - g(X) \right)^\ell   }{\ell!} g(\CJ_{\mfk+\ell}\tau) = \\
		&= \CI_{\mfk} \bigl( \tau \bigr) -\sum_{k,\ell} \frac{(g-g(X))^\ell}{\ell!} g \bigl(\CJ_{\mfk+ \ell} \bigl( \underbrace{\Gamma_{g^{-1}} \Gamma_g \tau}_{ = \tau} \bigr)\bigr) + \\
		&  \qquad +\sum_{\ell!} \frac{\left( X - g(X) \right)^\ell   }{\ell!} g(\CJ_{\mfk+\ell}\tau) = \CI_\mfk(\tau)\\
		\Gamma_{g} \Gamma_{g^{-1}} \left(\CI_{\mfk} \tau \right) & = \Gamma_g \biggl(\CI_{\mfk}(\Gamma_{g^{-1}} \tau) + \sum_{\ell} \frac{X^\ell}{\ell!} g^{-1} \left( \CJ_{\mfk+\ell} \tau\right)  \biggr) = \\
		& = \CI_{\mfk}(\tau) + \sum_{\ell} \frac{X^\ell}{\ell!} g \left( \CJ_{\mfk+\ell} \left( \Gamma_{g^{-1}} \tau \right) \right)  -\\
		&\qquad - \sum_{k,\ell} \frac{\bigl(-g(X)\bigr)^k }{k!} \frac{\bigl(\Gamma_{g}(X)\bigr)^\ell}{\ell!} g \left( \CJ_{\mfk+k+\ell} \left( \Gamma_{g^{-1}} \tau \right) \right)  =
		\\
		& = \CI_{\mfk}(\tau) + \sum_{\ell} \frac{X^\ell}{\ell!} g \left( \CJ_{\mfk+\ell} \left( \Gamma_{g^{-1}} \tau \right) \right)  -\\
		&\qquad - \sum_{\ell}  \frac{\bigl(\Gamma_{g}(X)-g(X)\bigr)^\ell}{\ell!} g \left( \CJ_{\mfk+k+\ell} \left( \Gamma_{g^{-1}} \tau \right) \right)  = \CI_{\mfk}(\tau)
	\end{equs}
	Finally, to check the group properties, we only need to consider an element $\CJ_{\mfk} \tau \in \CH^+$.
	\begin{itemize}
	\item Associativity:
  	\begin{equs}
    	(g\cdot (\bar g \cdot \bar{\bar g}))(\CJ_{\mfk} \tau ) &= g\left( \CJ_\mfk \left( \Gamma_{\bar g \cdot \bar{\bar g}} \tau \right) \right) +  \sum_{\ell} \frac{g(X)^\ell}{\ell!} \left( \bar g\cdot  \bar{\bar g}\right)\left( \CJ_{\mfk + \ell} \tau \right) =\\
		&= g\left( \CJ_\mfk \left( \Gamma_{\bar g }\Gamma_{\bar{\bar g}} \tau \right) \right) +  \sum_{\ell} \frac{g(X)^\ell}{\ell!} \bar g\left( \CJ_{\mfk+\ell} \left( \Gamma_{\bar{\bar g}} \tau \right) \right) +\\
		& \qquad +  \sum_{k,\ell}  \frac{g(X)^\ell}{\ell!}\frac{\bar g(X)^k}{k!} \bar{\bar g}\left( \CJ_{\mfk + k+\ell} \tau \right) \\
		&= g\left( \CJ_\mfk \left( \Gamma_{\bar g }\Gamma_{\bar{\bar g}} \tau \right) \right) +  \sum_{\ell} \frac{g(X)^\ell}{\ell!} \bar g\left( \CJ_{\mfk+\ell} \left( \Gamma_{\bar{\bar g}} \tau \right) \right) +\\
		& \qquad +  \sum_{\ell}  \frac{(g(X)+\bar g(X))^\ell}{\ell!}\bar{\bar g}\left( \CJ_{\mfk + \ell} \tau \right) \\
		((g\cdot \bar g) \cdot \bar{\bar g})(\CJ_{\mfk} \tau ) &=(g \cdot \bar g) \left(\CJ_{\mfk} \left(\Gamma_{\bar{\bar g}} \tau\right) \right) + \sum_{\ell} \frac{\bigl((g\cdot \bar g)(X)\bigr)^\ell}{\ell!} \bar{\bar g}(\CJ_{\mfk + \ell} \tau) = \\
		&=g\left( \CJ_\mfk \left( \Gamma_{\bar g}  \Gamma_{\bar{\bar g}} \tau \right) \right) +  \sum_{\ell} \frac{g(X^\ell)}{\ell!} \bar g\left( \CJ_{\mfk + \ell} \left( \Gamma_{\bar{\bar g}} \tau \right) \right)  +\\
		& \qquad + \sum_{\ell} \frac{\bigl(g(X)+ \bar g(X) \bigr)^\ell}{\ell!} \bar{\bar g}(\CJ_{\mfk + \ell} \tau)
  	\end{equs}

 	\item Unit:
  	\begin{equs}
    	(g \cdot \1^*)(\CJ_\mfk \tau) &=  g\bigl( \CJ_\mfk \bigl( \underbrace{\Gamma_{\1^*} \tau}_{ = \tau} \bigr) \bigr) +  \sum_{\ell} \frac{g(X^\ell)}{\ell!} \underbrace{\1^*\left( \CJ_{\mfk + \ell} \tau \right)}_{ = 0}  = g(\CJ_{\mfk}\tau)  \\
    	(\1^*\cdot g)(\CJ_\mfk \tau) &=  \underbrace{\1^*\left( \CJ_\mfk \left( \Gamma_{g} \tau \right) \right)}_{ = 0} +  \sum_{\ell} \underbrace{\frac{\1^*(X^\ell)}{\ell!}}_{ = \delta_{\ell,0}} g\left( \CJ_{\mfk + \ell} \tau \right) = g\left(\CJ_{\mfk} \tau\right)
  	\end{equs}

	\item Inverse:
  	\begin{equs}
    	\left( g \cdot g^{-1} \right)(\CJ_\mfk \tau) &= g \left( \CJ_{\mfk} \left( \Gamma_{g^{-1}} \tau \right) \right) + \sum_{\ell} \frac{g(X^\ell)}{\ell!} g^{-1} \left( \CJ_{\mfk + \ell} \tau \right)  = \\
		&= g \left( \CJ_{\mfk} \left( \Gamma_{g^{-1}} \tau \right) \right) - \sum_{k, \ell} \frac{(g(X))^\ell}{\ell!} \frac{(-g(X))^k}{k!} g\left( \CJ_{\mfk +k+ \ell} \left(\Gamma_{g^{-1}} \tau\right) \right) =\\
		&= g \left( \CJ_{\mfk} \left( \Gamma_{g^{-1}} \tau \right) \right) - g\left( \CJ_{\mfk} \left(\Gamma_{g^{-1}} \tau\right) \right) = 0 = \1^*(\CJ_{\mfk} \tau)\\
		\left( g^{-1} \cdot g \right)(\CJ_\mfk \tau) &= g^{-1} \left( \CJ_{\mfk} \left( \Gamma_{g} \tau \right) \right) + \sum_{\ell} \frac{g^{-1}(X^\ell)}{\ell!} g \left( \CJ_{\mfk + \ell} \tau \right)  =\\
		&= - \sum_{\ell} \frac{(-g(X))^\ell}{\ell!} g \left( \CJ_{\mfk + \ell} \left( \Gamma_{g^{-1}} \Gamma_g \tau \right) \right) +\\
		& \qquad + \sum_{\ell} \frac{(-g(X))^\ell}{\ell!} g \left( \CJ_{\mfk + \ell} \tau \right) =\\
		&= - \sum_{\ell} \frac{(-g(X))^\ell}{\ell!} g \left( \CJ_{\mfk + \ell}  \tau  \right) + \sum_{\ell} \frac{(-g(X))^\ell}{\ell!} g \left( \CJ_{\mfk + \ell} \tau \right) =0
	\end{equs}
\end{itemize}
\end{proof}

\begin{proof}[Proof of Proposition~\ref{prop:RootDerCom}]
	We shall proceed by induction. First we have for all $i,j \in [d]$, $\mfk \in \mfL^- \times \N^d$, and $g \in \CG(\CH^+)$ we have
	\begin{equs}
		\Gamma_g D_i \1 &= \Gamma_g 0 = 0 = D_i \1  = D_i \Gamma_g \1 \; , \\
		\Gamma_g D_i X_j &= \delta_{ij} \Gamma_g \1 = \delta_{ij} \1  = D_i \left( X_j + g(X_j) \1 \right) = D_i \Gamma_g X_j \; , \\
		\Gamma_g D_i \Xi_{\mfk}  &= \Gamma_g \Xi_{\mfk + e_i} = \Xi_{\mfk + e_i} = D_i \Xi_{\mfk} =  D_i \Gamma_g \Xi_{\mfk} \; ,
	\end{equs}
	and therefore by induction we have $\Gamma_g D^k = D^k \Gamma_g$ for all base cases and $k \in \N^d$.

	Now supposing that $\Gamma_g$ and $D^k$ commute for $\tau_1$ and $\tau_2$ we have
	\begin{equs}
		D^k \Gamma_g(\tau_1 \tau_2)  & = D^k \left( \Gamma_g (\tau_1) \Gamma_g(\tau_2)\right) = \sum_{\substack{\ell_1, \ell_2 \in \N^d \\ \ell_1 + \ell_2 = k}} \binom{k}{\ell_1, \ell_2} D^{\ell_1} \Gamma_g(\tau_1) D^{\ell_2} \Gamma_g(\tau_2) = \\
		&= \sum_{\substack{\ell_1, \ell_2 \in \N^d \\ \ell_1 + \ell_2 = k}} \binom{k}{\ell_1, \ell_2}  \Gamma_g(D^{\ell_1} \tau_1)  \Gamma_g(D^{\ell_2} \tau_2) = \\
		&= \Gamma_g \Bigl(  \sum_{\substack{\ell_1, \ell_2 \in \N^d \\ \ell_1 + \ell_2 = k}} \binom{k}{\ell_1, \ell_2}  D^{\ell_1} \tau_1 D^{\ell_2} \tau_2  \Bigr) = \Gamma_g \left( D^k \left( \tau_1 \tau_2 \right)\right)
	\end{equs}
	and finally let $\mfk \in \mfL^+ \times \N^d$
	\begin{equs}
		 D^k \Gamma_g \CI_{\mfk} \tau &= D^k \Bigl( \CI_{\mfk} \left( \Gamma_g \tau \right) + \sum_{\ell \in \N^d} \frac{1}{\ell!}  X^\ell g \left( \CJ_{\mfk + \ell} \tau \right)\Bigr) =\\
		 &= \CI_{\mfk + k} \left( \Gamma_g \tau\right) + \sum_{\substack{\ell \in \N^d \\ \ell \geqslant k}} \frac{1}{(\ell - k)!} X^{\ell - k} g \left( \CJ_{\mfk + \ell \tau } \right) = \\
		 & = \CI_{\mfk + k} \left( \Gamma_g \tau\right) + \sum_{\ell \in \N^d } \frac{1}{\ell!} X^{\ell } g \left( \CJ_{\mfk + k+ \ell \tau } \right) = \\
		 & = \Gamma_g \left(  \CI_{\mfk + k} \tau \right) = \Gamma_g \left( D^k \left( \CI_{\mfk} \tau \right) \right)
	\end{equs}
\end{proof}

\begin{proof}[Proof of Lemma~\ref{lemma:StGrForm}]
	\begin{equs}
		\Gamma_{xy} \CI_{\mfk} \tau & = \CI_{\mfk} \Gamma_{xy} \tau + \sum_{\ell} \frac{X^\ell}{\ell!} \left( g_x^{-1} \cdot g_y \right) \left( \CJ_{\mfk + \ell} \tau \right)	= \\
		&= \CI_{\mfk} \Gamma_{xy} \tau + \sum_{\ell} \frac{X^\ell}{\ell!}  g_x^{-1}  \left( \CJ_{\mfk + \ell} \Gamma_{g_y} \tau \right) + \sum_{k,\ell} \frac{X^\ell}{\ell!} \frac{x^k}{k!} g_y  \left( \CJ_{\mfk + k+ \ell} \tau \right)	= \\
		&= \CI_{\mfk} \Gamma_{xy} \tau - \sum_{k, \ell} \frac{X^\ell}{\ell!} \frac{x^k}{k!}  g_x  \left( \CJ_{\mfk + k+ \ell} \Gamma_{g_x^{-1}} \Gamma_{g_y} \tau \right) + \\
		 &\hspace{5cm}+ \sum_{k,\ell} \frac{X^\ell}{\ell!} \frac{x^k}{k!} g_y  \left( \CJ_{\mfk + k+ \ell} \tau \right)	= \\
		%&= \CI_{\mfk} \Gamma_{xy} \tau - \sum_{k, \ell} \frac{X^\ell}{\ell!} \frac{x^k}{k!}  g_x  \left( \CJ_{\mfk + k+ \ell} \Gamma_{xy} \tau \right) - \sum_{k,k', \ell} \frac{X^\ell}{\ell!} \frac{x^k}{k!}\frac{(-y)^{k'}}{k'!} f_y  \left( \CJ_{\mfk + k+k'+\ell} \tau \right)	= \\
		&= \CI_{\mfk} \Gamma_{xy} \tau + \sum_{k,k', \ell} \frac{X^\ell}{\ell!} \frac{x^k}{k!} \frac{(-x)^{k'}}{k'!}  f_x  \left( \CJ_{\mfk + k+k'+ \ell} \Gamma_{xy} \tau \right) - \\
		& \hspace{5cm} - \sum_{k, \ell} \frac{X^\ell}{\ell!} \frac{(x-y)^k}{k!} f_y  \left( \CJ_{\mfk +k+ \ell} \tau \right)	= \\
		%&= \CI_{\mfk} \Gamma_{xy} \tau + \sum_{\ell} \frac{X^\ell}{\ell!}  f_x  \left( \CJ_{\mfk + \ell} \Gamma_{xy} \tau \right) - \sum_{\ell}  \frac{(X+x-y)^\ell}{\ell!} f_y  \left( \CJ_{\mfk + \ell} \tau \right) = \\
		&= \CI_{\mfk} \Gamma_{xy} \tau + \sum_{\ell} \frac{X^\ell}{\ell!}  \left( f_x  \left( \CJ_{\mfk + \ell} \Gamma_{xy} \tau \right) - \sum_{k}  \frac{(x-y)^k}{k!} f_y  \left( \CJ_{\mfk+k+ \ell} \tau \right) \right) \; . 
	\end{equs}
	Computing the coefficient of $X^\ell$ we see that 
	\begin{equs}
		\ell! \left( \Gamma_{xy} \CI_{\mfk} \tau \right)_{\ell} & = \int D^{\mfk_2 + \ell}_1 K_{\mfk_1}(x,z) \left( \Pi_x \Gamma_{xy} \tau \right) (z) \d z - \\
		& \qquad  -\sum_{\substack{k \in \N^d \\ |k|_{\s} < |\mfk + \ell|_{\s} + |\tau|_{\s} }} \frac{(x-y)^k}{k!} \int D^{\mfk_2 + \ell + k}_1 K_{\mfk_1}(y,z) \left( \Pi_y \tau \right)(z) \d z = \\
		& = \int K_{ \mfk+ \ell,|\tau|_{\s}} (x, y ;  z) \left( \Pi_y \tau \right)(z) \d z \; . 
	\end{equs}
	Here the condition that $|k|_{\s} < |\mfk + \ell|_{\s} +|\tau|_{\s}$ appears since $\CJ_{\mfk + k + \ell} \tau$ is non-zero if and only if the regularity of $\CI_{\mfk + k + \ell} \tau$ is positive.
\end{proof}

\begin{proof}[Proof of Lemma~\ref{lemma:StatMod}]
	We only need to check this for the cases $\Xi_\mfl$, $X_\mu$, and $\CI_{\mfk}\tau$. It holds trivially for $\Xi_{\mfl}$, for $X_\mu$ we have 
	\begin{equ}
		\PPi \left( \Gamma_{g_x} X_\mu\right)(y) =   \PPi \left( X - x_\mu \1\right) (y) = \left(y_\mu - x_\mu\right) \bone_{\CA} \; ,
	\end{equ}
	and for $\CI_{\mfk} \tau$ we have 
	\begin{equs}
		\PPi \left( \Gamma_{g_x} \CI_{\mfk} \tau \right)(y)& = \PPi \left( \CI_{\mfk} \Gamma_{g_x} \tau + \sum_{\ell} \frac{X^\ell}{\ell!} g_x \left( \CJ_{\mfk + \ell} \tau \right)  \right) (y) = \\
		& = \PPi \left( \CI_{\mfk} \Gamma_{g_x} \tau - \sum_{k,\ell} \frac{X^\ell}{\ell!} \frac{(-x)^k}{k!} f_x \left( \CJ_{\mfk +k+ \ell} \tau \right)  \right) (y) =\\
		& = \PPi \left( \CI_{\mfk} \Gamma_{g_x} \tau - \sum_{\ell} \frac{(X-x)^\ell}{\ell!}  f_x \left( \CJ_{\mfk + \ell} \tau \right)  \right) (y) =  \\
		& = \int D^{\mfk_2} K_{\mfk_1}(y,z) \PPi \left( \Gamma_{g_x} \tau \right) (z) \d z - \sum_{\ell} \frac{(y-x)^\ell}{\ell!} f_x\left( \CJ_{\mfk+\ell} \tau \right) =  \\
		& = \int D^{\mfk_2} K_{\mfk_1}(y,z) \Pi_x \tau  (z) \d z - \sum_{\ell} \frac{(y-x)^\ell}{\ell!} f_x\left( \CJ_{\mfk+\ell} \tau \right) \; . 
	\end{equs}
	This proves the assertion. 
\end{proof}

\section{BPHZ Supplement}

\begin{proposition}
\label{prop:TreeModIndep}
	Let $\cT = \cT(\cR)$ be the $\CA$-regularity structure built from a subcritical rule $\cR$. Let $\left(\xi_{\mfl}\right)_{\mfl \in \mfL^-} \subset \cD'(\R^d)$ be a choice of noise realisations and $\left( K_{\mft} \right)_{\mft \in \mfL^+} \subset \cC^\infty( \R^d \setminus \{0\} )$ of scalar-valued integration kernels satisfying the usual assumptions and let $Z_{(\eps)}$ be the canonical model w.r.t.\ these.

	For every scalar tree $\tau$ with $n$ uncontracted negative leaves there is a possibly singular integration $G^{(\eps)}_{\tau} \in \cD'( \R^{d (n+2)})$ independent of the algebra $\CA$, s.t.\
	\begin{equ}
		\Pi_x^{(\eps)} \tau(y)  =
		\int\limits_{\R^{dn}} G_\tau \left(y , x, z_{1}, \dots, z_n \right) \Delta(\mfo) \Bigl(\bigotimes_{i =1 }^n \Pi_0^{(\eps)}\Xi_{\mfl_i} (z_i)\Bigr) \prod_{i = 1}^n \d z_i
	\end{equ}
	with $\Delta(\mfo) \in \CB^n( \CA ; \CA)$ independent of the variables $x,y, z_i$.
\end{proposition}

\begin{proof}
	We proceed by straightforward induction. For $\tau = \1$ or $X_\mu$ it is obvious, as is the case for $\tau = \Xi_{(\mfl, k)}$ with $G_\tau = (y,x,z) = \delta^{(k)}(y-z)$, with $\Delta(\mfo) \equiv \bone$ in all of these cases. Similarly, for $\tau = \tau_1 \tau_2$, we have $G_{\tau}(y,x,z_1, \dots, z_n) = G_{\tau_1}(y,x, z_1, \dots, z_k)G_{\tau_2}(y,x , z_{k+1}, \dots, z_n)$ and $\Delta(\mfo) = \Delta(\mfo_1) \otimes \Delta(\mfo_2)$.

	For $\tau = \CI_{(\mft , k)} \tau'$, one sets
	\begin{equs}
		G_{\tau}(y,x, z_1, \dots, z_n) &\eqdef \int\limits_{\R^d}  K_{\mft+k, |\tau'|_{\s}}(y,x;w)   G_{\tau'}(w,x, z_1, \dots, z_n) \d w \; .
	\end{equs}
	For root-renormalised trees we have $G_{\tau}(y,x,z_1, \dots, z_n) = \prod_{i = 1}^k G_{\tau_k}(y,x, z_{I_i})$, where $\tau_i$ are the (scalar) trees that are attached to the root-renormalised tree in $\tau  = (\tau_1 \otimes \cdots \otimes \tau_k) \curvearrowright \tau_\rho$. In the grafting decomposition $\CF = \tau_1 \otimes \cdots \otimes \tau_k$, since $\tau$ is scalar. Letting $\Delta_i \in \CB^{|I_i|}(\CA;\CA)$ be the linear maps of the model of $\tau_i$, we have
	\begin{equs}
		\Pi_x^{\eps} \tau(y) = \int\limits_{\R^{dn}} G_{\tau}(y,x,z_1, \dots, z_n) \prod_{i  = 1}^k \Delta_i \Bigr( \bigotimes_{j \in I_i} \Pi_0^{(\eps)} \Xi_{\mfl_j}(z_j) \Bigr) \prod_{i = 1}^n \d z_i \; ,
	\end{equs}
	i.e.\ $\Delta(\mfo) = \bigotimes_{i  = 1}^k \Delta_i$. This finishes the induction.
\end{proof}

%\section{Glossary}
%\label{app:notation}
%
%
%\begin{center}
%\begin{longtable}{p{.12\textwidth}p{.65\textwidth}p{.12\textwidth}}
%\toprule
%Object & Meaning & Ref. \\
%\midrule
%\endhead
%\bottomrule
%\endfoot
%$\CA_q(\mfH)$ & & \\
%$\CA_1(\mfH)$ & & \\
%$\CA_F(\mfH)$ & & \\
%$\cA_q(\mfH)$ & & \\
%$\cA_1(\mfH)$ & & \\
%$\cA_F(\mfH)$ & & \\
%$\cG_F(\mfH)$ & & \\
%$\mfS_n$ & Group of permutations on $[n] = \{1,\dots,n\}$ & p.~\pageref{symbol:permutation}\\
%$\mfS_{n,k}$ & $\mfS_{n,k} \eqdef \mfS_{n}/\mfS_k \times \mfS_{n-k}$ & p.~\pageref{symbol:twopermutation}\\
%$\otimes \slash \wotimes_\pi \slash \wotimes_\eps \slash \wotimes_\alpha $ & algebraic \slash projective \slash injective \slash Hilbert tensor product (subscript sometimes dropped in last case) & p.~\pageref{symbol:tensor}\\
%\end{longtable}
%\end{center}

\endappendix

\bibliographystyle{Martin}
\bibliography{./refs}

\end{document}